\documentclass{article}[14pt]
\usepackage{latexsym}
\usepackage{amsmath}
\usepackage{amsfonts}
\usepackage{amssymb}
\usepackage{graphicx}
\usepackage{amsthm}
\usepackage[top=2.5cm, bottom=2.5cm, left=2.5cm, right=2.5cm]{geometry}
\usepackage{hyperref}
\usepackage{slashed,mathtools}
\usepackage[titletoc,toc]{appendix}
\usepackage{float}
\usepackage{color}

\newcommand{\oH}{\overline{H}}
\newcommand{\tH}{\tilde{H}}
\newcommand{\E}{{\bf E}}
\newcommand{\te}{\tilde{e}}

\newcommand{\A}{{\bf A}}
\newcommand{\K}{{\bf K}}
\newcommand{\ttt}{\tilde{t}}

\newcommand{\oa}{\overline{a}}
\newcommand{\beq}{\begin{equation}}
\newcommand{\eeq}{\end{equation}}
\newcommand{\mc}[1]{\mathcal{#1}}  
\newcommand{\e}{\epsilon}
\newcommand{\g}{{\bf g}}
\newcommand{\tttt}{{\tilde{t}}}
\newcommand{\ttheta}{{\tilde{\theta}}}
\newcommand{\bs}{\begin{split}}
\newcommand{\es}{\end{split}}


\theoremstyle{plain}
\newtheorem{theorem}{Theorem}[section]
\newtheorem{proposition}[theorem]{Proposition}
\newtheorem{lemma}[theorem]{Lemma}
\newtheorem{corollary}[theorem]{Corollary}

\theoremstyle{definition}
\newtheorem{definition}[theorem]{Definition}

\newtheorem{remark}[theorem]{Remark}

\numberwithin{equation}{section}

\allowdisplaybreaks[4] 

\swapnumbers
\begin{document}

\title{Stable space-like singularity formation for axi-symmetric and polarized near-Schwarzschild black hole interiors.}

\author{Spyros Alexakis and Grigorios Fournodavlos}

\date{}

\maketitle

\parskip = 0 pt

\abstract{ We show a stability result for the Schwarzschild singularity 
(inside the black hole region)
 for the Einstein {vacuum} equations (EVE). The result is proven in the class of 
 polarized axial 
 symmetry, under  perturbations of the Schwarzschild data induced  on a
 hypersurface $\{r=\e\}$, $\e<<2M$. 
  Our result is only partly a stability result, in that we show that while a (space-like) singularity 
 persists under perturbations as above, the behaviour of the metric approaching the 
 singularity is much more  involved than for the Schwarzschild solution. Indeed, we find that the 
 solution displays asymptocially-velocity-term-dominated dynamics and  approaches  a different Kasner solution at \emph{each} point of the singularity. 
 These Kasner-type asymptotics are very far from isotropic, since (as in Schwarzschild) 
 there are two contracting directions and one expanding one. 
 Our proof relies on energy methods and 
 on a new 
 approach to the {EVE}  in axial symmetry, which we believe 
 has wider applicability: In this symmetry class and under a suitable \emph{geodesic} 
 gauge, the EVE can be studied as a free wave  coupled to 
(nonlinear)  ODEs, which couple the geometry of the projected, 2+1 space-time to the
free  wave. 
The fact that
 the 
nonlinear part of the Einstein equations is described by ODEs   
lies at the heart of how one can  overcome  a certain \emph{linear instability} exhibited by the singularity.}
\tableofcontents

\section{Introduction}

We study the problem of the stability of the singularity of the Schwarzschild black hole
from the point of view of the forwards-in-time initial value problem for the Einstein vacuum equations (EVE): We consider perturbations of the initial data of Schwarzschild, {along a space-like hypersurface in the
black hole  interior}, and wish to 
understand 
the maximal future  hyperbolic development of the solution, up to any singularities that 
might form, with detailed asymptotics at the singularity.

 We restrict {our} attention to polarized axially symmetric perturbations.
 Within that class we find that the maximal hyperbolic development 
 of any sufficiently small perturbation of the Scwharzschild initial data terminates 
 at a space-like singularity with very rich dynamics. 
As we review below, at each point on its ``final'' {singular hypersurface}, the Schwarzschild solution
exhibits a \emph{collapsing} behaviour of the metric in two principal directions and an \emph{expanding} behaviour in a remaining, third principal direction.  In this regard,
 the  behaviour of our solutions 
 is qualitatively similar to that of the Schwarzschild solution, in that the singularity it forms 
 is still space-like, and moreover at each point there are still two collapsing and one expanding principal directions. From this point of view, 
{our result can 
 be seen as a stability result}. However the \emph{rates} of the two contractions and {expansion}
 are different at each point on the final {singular hypersurface} and also 
  generically different from those of Schwarzschild. Thus,
  the result we derive 
should be thought of as a stability property, albeit holding only 
in a broad sense.

We next present a rough version of our result, and then situate it in the context of singularity formation in black hole interiors and cosmological space-times. We then provide a broad outline of some of the ideas in this
paper. 
 
 \subsection{The result.}
 
We will exclusively be studying space-times $({\cal M}^{1+3}, g)$ which are axially  symmetric and 
the axial symmetry is \emph{polarized}, with the Killing field corresponding to a rotation.  In particular for our space-times there exists a system of coordinates 
${r\in(0,2\epsilon)},t\in (-\infty,\infty), \theta\in (0,\pi),\phi\in [0,2\pi)$ with $\partial_\phi$ being the Killing field. The polarization condition is equivalent to the requirement that:
\[
\partial_\phi\perp {\rm Span}\langle\partial_t,\partial_r,\partial_\theta\rangle
\]
Equivalently, the metric components $g_{t\phi}, g_{\theta\phi}, 
g_{r\phi}$ all vanish.
\medskip
 
For the sake of comparison, recall the Schwarzschild metric in the standard coordinates $r,t,\theta,\phi$:
\begin{align}\label{Schmetric.pre}
g_{\rm S}=-(\frac{2M}{r}-1)^{-1}dr^2+(\frac{2M}{r}-1)dt^2+r^2(d\theta^2+\sin^2\theta d\phi^2),&&M>0.
\end{align} 
We note that in the interior of the black hole, $r<2M$, the characters of 
the coordinate 
vector fields $\partial_r,\partial_t$ are reversed, namely, $\partial_r$ is 
timelike and 
$\partial_t$ is space-like. The (true) singularity is thus at $\{r=0\}$, 
where  the 
metric components 
$(g_{\rm S})_{\phi\phi}=r^2 \sin^2\theta$, $(g_{\rm S})_{\theta\theta}=r^2$ \emph{collapse}, as 
$r\to 0^+$, 
while $(g_{\rm S})_{tt}=(\frac{2M}{r}-1)$ \emph{expands} as $r\to 0^+$.

Now, consider the hypersurface $\Sigma_{\e}:=\{r=\e\}$, for some constant 
$\e>0$ that 
will be chosen suitably small {further down}.
 { In the Schwarzschild space-time, 
 {let us denote} by $\g_{\rm S}$ be the induced metric on this hypersurface and by 
 $\K_{\rm S}$ its
the second fundamental form}. It follows straightforwardly from 
 \eqref{Schmetric.pre} that the second fundamental form $\K_{\rm S}$  
 on 
 these slices is  of magnitude 
 $\e^{-3/2}$, {see \eqref{SK} below}. The metric components 
  $(\g_{\rm S})_{\theta\theta}$,
  $(\g_{\rm S})_{\phi\phi}$, {$(\g_{\rm S})_{tt}$}
are of magnitude  $\e^2, \e^2sin^2\theta,{ (\frac{2M}{\epsilon}-1)}$ respectively.

The initial data $(\g_{\rm init}, \K_{\rm init})$ that we consider in this paper will 
be (polarized and axially symmetric) perturbations of the Schwarzschild 
background data $(\g_{\rm S}, \K_{\rm S})$. The closeness will be measured in suitable Sobolev spaces, and the closeness (in these spaces)
will be captured by a parameter $\eta>0$. The precise assumption will be 
formulated {further down} in subsection \ref{sec_REVESNGG}.
\medskip

It is for the space of initial data for which both $\epsilon,\eta$ are 
small enough (to be
 determined later) that we obtain our result. Our main finding will be 
 that in a suitable coordinate system  $r,T,\Theta$,  the solution exists all the way 
 up to a space-like singularity that occurs at $\{r=0\}$.

 We provide a first, rough formulation of our main result here: 
 \medskip

\begin{theorem}
\label{thm_rough}
Consider a perturbation $(\g, \K)$ of the Schwarzschild initial data 
 $(\g_{\rm S}$,  $\K_{\rm S})$ on the hypersurface $r=\e$. 
   Let $\eta>0$ capture the size of the perturbation (in a re-normalized
    sense to be specified in {subsection \ref{sec_REVESNGG}), in $H^s\times H^{s-1}(\mathbb{S}^2\times \mathbb{R})$ spaces,
    for some $s\in\mathbb{N}$,} to be chosen suitably large below. 
    Assume the perturbation preserves the polarized axi-symmetric 
    structure of the Schwarzschild background and solves the vacuum 
    constraint equations within this symmetry class. 
    
     Then for $\epsilon, \eta>0$ small enough this initial data admits a future maximal hyperbolic
      development 
     which terminates at a space-like singularity, {where the Kretschmann scalar blows up}. 
 Near {the singularity}, the space-time metric $g^{3+1}$ has an asymptotic {profile}
 of the following form, in suitable coordinates 
 $r, \phi, T, \Theta$: 
 \begin{align}
 \begin{split}
 \label{form}
&g=-(\frac{2M}{r}-1)^{-1}dr^2+ g_{\phi\phi}(r,T,\Theta) d\phi^2+ 
g_{\Theta\Theta}(r,T,\Theta) d\Theta^2+
 g_{TT}(r,T,\Theta) dT^2+
 g_{T\Theta}(r,T,\Theta) dT d\Theta
 \\&+g_{r\Theta}(r,T,\Theta)
  drd\Theta+g_{rT}(r,T,\Theta) drdT,
 \end{split}
\end{align}
where the metric components {admit the asymptotic expansions}:
\begin{align}
  \begin{split}
  \label{cl}
&    g_{\phi\phi}(r,T,\Theta)= A(T,\Theta)r^{2\alpha(T,\Theta)}
    (1+O(r^{\frac{1}{4}})) \sin^2\Theta, \quad
    g_{TT}(r,T,\Theta)= B(T,\Theta) r^{2\beta(T,\Theta)}(2M+O(r^{\frac{1}{4}})),  
\\&    g_{\Theta\Theta}(r,T,\Theta)= C(T,\Theta) r^{2\delta(T,\Theta)}
    (1+O(r^{\frac{1}{4}})), \qquad g_{T \Theta}= D_1(T,\Theta) r^{\vartheta_1(T,\Theta)}(1+O(r^\frac{1}{4})), \\
&  g_{rT}(r,T,\Theta)= D_2(\tilde{t},\tilde{\theta})r^{\vartheta_2(T,\Theta)}(1+O(r^\frac{1}{4})),\qquad g_{r\Theta}(r,T,\Theta) \equiv0,\;\;r\in(0,\frac{\epsilon}{2}),
\end{split} 
\end{align}
{as $\rightarrow0^+$}.
  Here the function $\alpha(T,\Theta)$ is everywhere close to 1, 
  $\delta(T,\Theta)$ is also close 
  to 1, $\beta(T,\Theta)$ is everywhere close to $-1/2$. 
       Moreover the coefficients 
  $A(T,\Theta), B(T,\Theta), 
  C(T,\Theta)$ are everywhere close to 1 also.

      In fact, the value of $\alpha$ at each point $(T,\Theta)$ uniquely 
     determines the value of the other 
  powers $\beta(t,\theta),\delta(t,\theta)$.   
   The remaining metric components are less singular, in the sense that the corresponding exponents satisfy 
{  $\vartheta_1(T,\Theta)\ge 1+\frac{1}{8}$,  
  $\vartheta_2(T,\Theta)\ge\frac{3}{8}$.}
  
\end{theorem}

\begin{remark}
The same expansions hold for up to {a certain number of $\partial_{T},\partial_\Theta$ derivatives of the metric, which henceforth we will denote by ${\rm low}$, ${\rm low}\ll s$}. (At these orders, the Kretschmann scalar blows up like $r^{-6}$, as $r\rightarrow0$).  We expand on this in the stricter formulation of our result.   
\end{remark}

\begin{remark}
 Note that the metric $g$ in this form is evidently axisymmetric and  
 polarized, due to the absence of cross-terms 
 $g_{T\phi},g_{r\phi}, g_{\phi\Theta}$, and also since all metric 
 coefficients are
  independent of $\phi$.  
 However there is an extra gauge normalization, 
captured in this form: The integral curves of $\partial_r$  are geodesics 
which are (asymptotically as $r\to 0^+$) 
orthogonal to 
$\partial_{T},\partial_{\Theta},\partial_\phi$. Purely for 
comparison reasons, the 
parameter 
$r$ has been chosen to agree with the corresponding parameter in the 
Schwarzschild 
background.
 \end{remark}




\begin{remark}
Observe that as in the Schwarzschild background,
   the directions $\partial_\phi,\partial_{\Theta}$ are \emph{collapsing}, while the direction 
  $\partial_{T}$ is \emph{expanding}. The terms {of order}
  $r^{\vartheta_1(T,\Theta)}, r^{\vartheta_2(T,\Theta)}$ 
  should  be seen to be \emph{less singular} off-diagonal terms (which
  vanish for the  Schwarzschild metric).
    \end{remark}

While our paper is entirely concerned with black hole interiors, one can 
state a corollary that links it with recent studies of perturbations of the 
Schwarzschild black hole \emph{exterior} regions (for two-ended initial 
data).  
 
In particular, our main theorem complements the recent breakthrough stability result of the exterior region by Klainerman-Szeftel \cite{KS} and that of 
the inner red-shift region,\footnote{{In fact, the stability of the inner red-shift region, announced in \cite{DL2}, concerns general space-times that 
converge to a Kerr along the horizon, which is only simpler in the absence of rotation.}} announced by Dafermos-Luk \cite{DL2}, which combined give the 
full picture of near-Schwarzschild (double-ended) space-times in polarized axi-symmetry. 
\begin{corollary}\label{thm:stabPen}
Dynamical space-times, arising from sufficiently regular and small perturbations
of the Schwarzschild initial data of mass $M_0$, on a global Cauchy hypersurface $\Sigma$, have the Penrose diagram depicted in Figure \ref{stabPen}.
\begin{figure}[h]
	\centering
	\def\svgwidth{9cm}
	\includegraphics[scale=1.2]{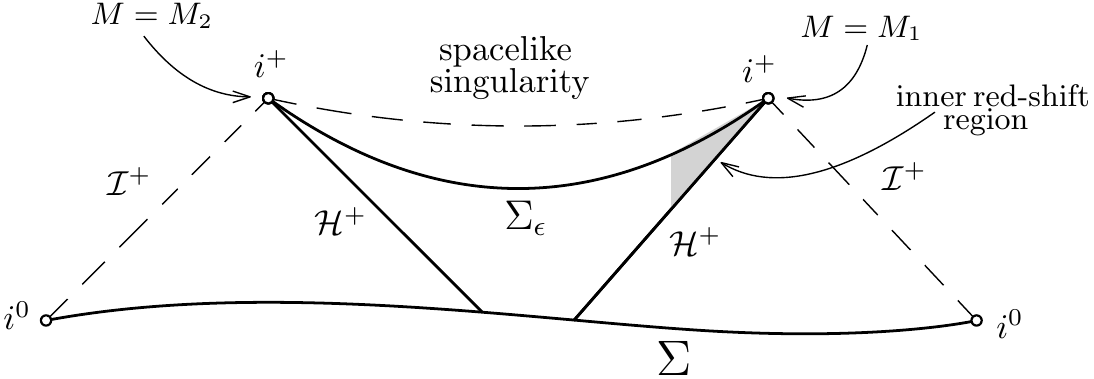}
	\caption{The Penrose diagram of dynamical, polarized axi-symmetric, near-Schwarzschild space-times.} \label{stabPen}
\end{figure}

\noindent
In the exterior regions, they are globally defined, having complete null infinities, and converge to Schwarzschild metrics at timelike infinities, of  masses 
$M_1,M_2$ that are close to that of the background mass $M_0$. Moreover, the inner boundary of  their black hole is entirely space-like, singular, and 
the asymptotic behaviour of the metrics 
 towards the singularity are as in Theorem \ref{thm_rough}.  In particular, both the weak and strong cosmic censorship conjectures 
 are valid in the axi-symmetric, polarized, near-Schwarzschild regime.  
\end{corollary}  

 The proof of this corollary is carried out in \S\ref{app:stabPen}. 
 \medskip

  \subsubsection{AVTD behaviour of our solutions.}
     
     Let us comment here on the asymptotically velocity term dominated {(AVTD)}
     behaviour of our solutions. One can formulate this property in different (essentially equivalent) ways. We here present the property purely in terms of the behaviour of the kinetic part of the energy of the 
 metric components (gravitational field components) relative to the potential part of the same energy, as in  Kichenassamy-Rendal  \cite{KR98}, where they constructed analytic Gowdy space-times with Big Bang singularities exhibiting such 
 behaviour, as well as section 3 in \cite{IM02}.
 The term was originally coined in the first construction of AVTD space-times
by Isenberg and Moncrief in 1990, \cite{IM90} for polarized Gowdy space-times. 
We note that the notion appears in different guise already in Eardley, Liang and 
Sachs' work in 
1972, \cite{ELM72}. 
\medskip

In our {space-times}, $\partial_r$ is time-like and 
$\partial_T,\partial_\Theta$ are space-like. Let us denote by 
$e_0, e_1, e_2$ the unit length vector fields in those directions. 

In the Schwarzschild solution (where $t=T,\theta=\Theta$), the
gravitational field component $(g_{\rm S})_{\phi\phi}=r^2 \sin^2\theta$ has the following  
property: 
	
$e_0(g_{\rm S})_{\phi\phi}\sim r^{1/2}{\rm sin}^2\theta$,
 while $e_1(g_{\rm S})_{\phi\phi}=0, e_2(g_{\rm S})_{\phi\phi}=O(r)$.
For the rest of the principal gravitational fields components 
$(g_{\rm S})_{\theta\theta}, (g_{\rm S})_{tt}$ we also have that: 
\[
e_0(g_{\rm S})_{\theta\theta}\sim r^{1/2}, e_1(g_{\rm S})_{\theta\theta}=e_2(g_{\rm S})_{\theta\theta}=0, 
\]
and also 
\[
e_0(g_{\rm S})_{tt}\sim r^{-5/2}, e_1(g_{\rm S})_{tt}=e_2(g_{\rm S})_{tt}=0
\]
In particular for each of these components  the kinetic (the $e_0$)-part of the energy dominates the 
potential parts (the $e_1, e_2$-parts). 

For the {space-times} that we construct we have a {similar} behaviour. 
We illustrate this for the function $\tilde{\gamma}:=\frac{1}{2}\log g_{\phi\phi}-\log\sin\Theta$.\footnote{This is a modification of the logarithm of the 
axi-symmetric gravitational field component.} 
{Up to lower-order terms in 
 $r$ in the RHSs, we will see that}:
 \beq
 \label{AVTD1}
 e_0(\tilde{\gamma})\sim -\alpha(\tilde{t},\tilde{\theta})\sqrt{2M} r^{-3/2}, \quad|e_1(\tilde{\gamma})|\le C 
 {r^{-1-\frac{1}{4}}}, \quad
 |e_2(\tilde{\gamma})|\le C r^{-1-\frac{1}{4}}.
 \eeq
In terms of the energy, this implies that
\beq\label{energy.def}
E_{\{r=r_0\}}[\tilde{\gamma}]= \int_{r=r_0} |e_0\tilde{\gamma}|^2+|e_1\tilde{\gamma}|^2+|e_2\tilde{\gamma}|^2 
sin\tilde{\theta} d\tilde{\theta} d\tilde{t}=
(r_0)^{-3}\int_{r=r_0}  2M|\alpha(\tilde{t},\tilde{\theta})|^2sin\tilde{\theta} d\tilde{\theta}
d\tilde{t}+O(r_0^{-2-3/4})
\eeq
In particular, the energy of $\tilde{\gamma}$ is all  
asymptotically concentrated 
in the $e_0$ direction, i.~e.~the spatial (potential) components of the energy 
$e_1(\tilde{\gamma}), e_2(\tilde{\gamma})$ are strictly \emph{less singular} than the 
time-like (kinetic) component 
 $e_0(\tilde{\gamma})$ of the energy. 
 
 This captures the 
 {AVTD} behaviour of the 
component $g_{\phi\phi}$ of the gravitational field.

As we will see, a consequence of this behaviour of $g_{\phi\phi}$ 
is that the principal components 
$g_{TT}, g_{\Theta\Theta}$  and their derivatives 
   display a behaviour that is consistent with 
    \eqref{cl}. This in particular implies that the 
    derivatives in the $e_0\parallel \partial_r$ and $e_1, e_2\perp e_0$ 
    directions of  
    $\log g_{TT}, \log g_{\Theta\Theta}$ have the property that: 
    \begin{align*}
    e_0(\log  g_{TT}) \sim r^{-3/2}, \quad|e_1(\log  g_{TT})|
    \le C{r^{-1-\frac{1}{4}}},\quad |e_2 g_{TT}| 
    \le C r^{-1-\frac{1}{4}},\\
    e_0(\log  g_{\Theta\Theta}) \sim r^{-3/2},\quad |e_1(\log 
     g_{\Theta\Theta})|
    \le C{r^{-1-\frac{1}{4}}},\quad |e_2 g_{\Theta\Theta}| 
    \le C r^{-1-\frac{1}{4}}. 
    \end{align*}
    Notably letting $f$ to stand for any of the components $g_{AA}$ of the gravitational field $g$ (where $A$ takes on one of the values 
    $T,\Theta, \phi$), the energy of that field component 
\[
E[f](r=r_0)\int_{-\infty}^\infty \int_0^\pi \int_0^{2\pi}  [|e_0 f|^2+|e_1f|^2+|e_2f|^2](r=r_0) d\phi sin\theta d\theta dt 
\]    
is asymptotically, as $r_0\to 0$,  concentrated entirely in the $e_0$-direction.  
In other words the \emph{kinetic} part of the energy dominates the potential parts of the energy. 
In particular in the {evolution equations} further down, if one were to \emph{drop} terms involving 
spatial derivatives of all fields, one would still derive the correct \emph{leading order} behaviour of all gravitational fields in $r$, 
from the resulting (dramatically simplified) equations. \medskip
 
This velocity-term dominated behaviour of the gravitational field for the solutions that we obtain {is consistent with the predictions} in the physics literature, as we explain next:

\subsection{Singularity formation in black holes {and} cosmological 
  singularities: Predictions and results.}
  
The question of whether and how singularities form in the evolution of smooth initial data 
is a central question for all non-linear evolutionary  PDEs. In the Einstein equations specifically, it is intimately linked to 
the question of \emph{strong cosmic censorship}; in the usual formulation this predicts that inside black holes, generically, the 
space-time metric terminates at a final singularity, past which it is inextendible.\footnote{The relation with the Penrose singularity (alternatively,  incompleteness)
theorem is discussed in \cite{RS3}.} The nature of the singularity is not formally part of the conjecture.

 We recall some results on singularity formation in black hole interiors and in Big
 Bang settings in the next few sections.  A more extensive discussion of these examples can be found in \cite{RS3} and the references therein.

\subsubsection{Singularity formation in black hole interiors and the \emph{instability} of the Schwarzschild singularity.}

   A brief comment is in order concerning the possibilities of viewing our result beyond 
   polarized axial symmetry: On one hand, one  sees that our stability result 
\emph{ceases} to hold by merely removing 
   the condition of polarization from the perturbations: Indeed, one can consider 
   the family of Kerr solutions $g_{\mc{K}, M, a}$, $a\ne 0$ bifurcating off of a 
   background 
   Schwarzschild solution $g_{\mc{K}, M, 0}$; as is well-known, the
   future 
maximal hyperbolic development of 
   the data on 
   ${\{r=\e\}}$ is then \emph{smooth}.
 (In fact, that maximal hyperbolic development even admits a smooth extension 
{past} a Cauchy horizon ${\cal CH}^+$, which can be attached as a boundary to this 
maximal hyperbolic development).    
    In other words, the Schwarzschild 
   singularity 
   entirely \emph{disappears} under (still axially symmetric!) perturbations 
   that introduce angular momentum. 
   {To our knowledge, the only known examples of vacuum space-times that exhibit a Schwarzschild-type singularity, without any symmetry assumptions, are the ones constructed by the second author \cite{F}. These spacetimes contain a singularity at a collapsed 2-sphere, where their asymptotic behaviour agrees with that of Schwarzschild at a suitably {\it high} order. This special requirement provided  an early 
   indication of the thinness of the set of perturbed initial data, which can lead to Schwarzschild-type singularity.}
   \medskip
   
{Therefore, one comes to the conclusion} that the Schwarzschild singularity, as it appears 
in the maximal analytic (two-ended) extension of the Schwarschild solution, 
is  
\emph{unstable} from the point of view of the initial value problem in 
full generality (i.e. with no symmetry assumptions imposed).       
The stability of the (double-ended) Kerr maximal hyperbolic development was 
also studied in the recent {breakthrough} paper of Dafermos-Luk \cite{DL,DL3}. It was shown there 
that general perturbations of a rotating Kerr solution   $g_{\mc{K}, M, a}$, 
$a\ne 0$, \emph{when the perturbation is small enough to rule out proximity 
to a 
(non-rotating) Schwarzschild solution}, {still form Cauchy horizons in the interior, along which the metric is ${\cal C}^0$-extendible}. A very interesting, \emph{weaker} 
type of singularity, is expected  to emerge in that context, see also \cite{DafCMP} and references therein 
for an earlier result on such weak null singularities {in spherical symmetry}, as well as \cite{Luk} for examples of weak null singularities in vacuum without symmetries.

A brief comparison of the two types of singularities (space-like and null) 
is in order: The space-like singularity in Schwarzschild (and in our 
solutions also)  has a locality property, in that each compact set on 
the final hypersurface $\{r=0\}$ depends on a \emph{compact} subset $\cal D$ of any 
Cauchy hypersurface in the entire space-time, as in the next picture. 
 In contrast, any given point 
on the weak null singularity depends on a a non-compact set of a Cauchy 
hypersurface. In particular, it depends on the entire 
future event horizon to which the weak null hypersurface is ``attached''.

\begin{figure}[h]
	\centering
	\def\svgwidth{9cm}
	\includegraphics[scale=1.5]{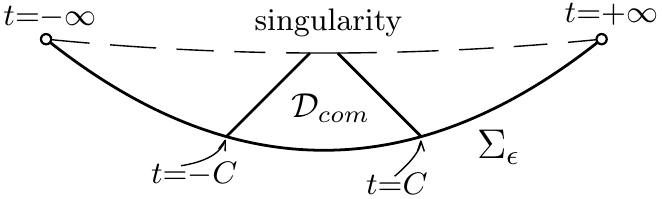}
	\caption{The region to the future of $\Sigma_\e$.} \label{twoDom}
\end{figure}

This in particular allows for the possibility of studying  space-like 
singularities \emph{locally}, by prescribing initial data (which perturbs the Schwarzschild data {within polarized axi-symmetry}) on an incomplete initial data hypersurface; one can 
reduce oneself to the setting of our theorem by constructing an artificial 
extension to a complete initial data set (satisfying the constraints) which 
asymptotes to the Schwarzschild data on $\{r=\e\}$ 
so as to  fulfil the assumptions of our theorem (in particular it would be asymptotically cylindrical). {If 
one can do this,} then our 
result would yield a space-like singularity, a portion of which is 
independent of the extension we constructed. 

The construction of such an extension is a matter of solving the constraint equations with asymptotically  cylindrical data. 
We are not aware that this has been done in the literature; 
pursuing this here is beyond the scope of our paper. 
\medskip

 This locality property also makes the space-time singularity indistinguishable (up to
 a time reversal), whether it occurs in a black hole or 
 at  an \emph{initial} Big-Bang type-singularity.   
This latter class has been extensively studied and we wish to make some connections of our result to that {part of the} literature:

\subsubsection{Big Bang type singularities.}

The nature of the Big-Bang type singularities is in general a wide-open 
and very interesting question. Many of the beliefs surrounding this 
question stem from the explicit family of Kasner solutions, which in 1+3 dimensions have the form (in a local system of 
spatial coordinates $x_1, x_2, x_3$): 
\begin{align}\label{Kasnermetric}
g=-dt^2+(\omega_1)^2t^{2p_1}+(\omega_2)^2t^{2p_2}+(\omega_3)^2t^{2p_3},&&\lim_{t\rightarrow0}\omega_i=\sum_{j=1}^3c_{ij}dx_j,
\end{align}
where $t\in(0,T]$ is a time function synchronizing the singularity at $t=0$ and $p_i,c_{ij}:\Sigma\to\mathbb{R}$ are functions of the spatial coordinates $x_1,x_2,x_3$.
Moreover, in {\it vacuum}, the $p_i$'s must satisfy the Kasner relations:
\begin{align}\label{sumpi}
\sum_{i=1}^3p_i=\sum_{i=1}^3p_i^2=1.
\end{align}  
Note that in view of the space-like nature of the singularity, for $t$ small enough, different points $(t,x_1, x_2, x_3)$ near different
 points on the singularity cannot be joined by time-like curves. {In our case, different points have Kasner dynamics with exponents close to their
 Schwarzschild counterparts,\footnote{{Setting $ds=(\frac{2M}{r}-1)^{-\frac{1}{2}}dr$, $s\sim r^\frac{3}{2}$, the Schwarzschild metric \eqref{Schmetric.pre} 
takes the form \eqref{Kasnermetric}, where $(g_{\rm S})_{tt}\sim s^{-2/3}$, $(g_{\rm S})_{\theta\theta}\sim s^{4/3}$, $(g_{\rm S})_{\phi\phi}\sim s^{4/3}\sin^2\theta$.}} 
$p_1=-\frac{1}{3}$, $p_2=p_3=\frac{2}{3}$, where the topology of $\Sigma$ is 
$\mathbb{S}^2\times\mathbb{R}$.}

  Note also that locally, 
 up to a time reversal $t\to (-t)$ one cannot ``see'' 
 whether the singularity occurs as an \emph{initial} Big Bang singularity or \emph{terminally}, inside a black hole. In $1+3$ vacuum, at least one of the $p_i$'s has to be negative, in view of \eqref{sumpi}. So in particular, isotropic 
 solutions or nearly-isotropic (where all $p_i$'s in \eqref{Kasnermetric}
 are close to equal) are not possible in $1+3$ vacuum.

 The question of how general solutions that exhibit Kasner-type behaviour
 are within the class of all big-bang singularities  
   has been
   studied in the mathematical literature in two main directions: In one
    direction one  \emph{constructs} classes of solutions with the prescribed 
    asymptotics at the singularity. 
    All of the examples constructed in this way 
    directly display an AVTD behaviour towards a (different) Kasner solution at each point on the singularity hypersurface:
    \medskip

{\it Constructions of AVTD singularities}: There are various constructions 
in the literature of AVTD space-times, 
applying  Fuchsian techniques to a first order reduction of 
the Einstein equations in order to produce Kasner-type
singularities, where the Kasner exponents depend on the spatial coordinates 
of every point on the singularity. {Apart from a recent work of the second author with J. Luk \cite{FL}, where Kasner-type singularities were constructed in 1+3 vacuum without symmetries or analyticity, all such other examples} are either in the analytic class, or with extra symmetries imposed (or both):  

The first such construction of AVTD space-times was given by Kichenassamy-Rendall \cite{KR98}, in the analytic Gowdy class. Other such results in the literature include: 
analytic AVTD space-times, without symmetries, by Anderson-Rendall \cite{AndRend} for the Einstein-scalar field model or a stiff-fluid; smooth AVTD space-times
 in the Gowdy class by Rendall \cite{Rend}; analytic $U(1)$-symmetric polarized\footnote{Polarization here is defined at the singularity in a function counting
 sense and is different from ours. However, from a function counting point of view, the gravitational degrees of freedom are the same. } and half-polarized 
AVTD space-times in vacuum by Isenberg-Moncrief \cite{IM02} and Choquet-Bruhat-Isenberg-Moncrief \cite{BIM04}; higher dimensional, analytic, AVTD
 space-times in vacuum, without symmetries, by Damour et al. \cite{DHRW02}, for space-time dimensions $1+d$, $d\ge 10$; {analytic AVTD space-times in $1+3$ dimensions without symmetries by P. Klinger \cite{Kl}}.  {In the other direction, one 
studies the stability of Kasner-like singularities. Such results, which go beyond the previous constructions via Fuchsian techniques, have only been fairly recently obtained:} 
\medskip

{\it Stability of Big Bang singularities without symmetries}: In breakthrough works by Rodnianski-Speck \cite{RS1,RS2}, it was proved that 
for models with certain special matter models (massless scalar fields and stiff fluids), 
near-\emph{isotropic} 
FLRW-Big Bang singularities are non-linearly stable, for $\mathbb{T}^3$ topologies and later by Speck \cite{Speck} for $\mathbb{S}^3$. Moreover,
Rodnianski-Speck \cite{RS3} verified the stable Big Bang formation of Kasner type, in vacuum, for an open set of Kasner exponents in space-time dimensions $1+d\ge 39$, perturbing off of close-to-isotropic 
explicit solutions with spatial topology $\mathbb{T}^d$. These works did not include any symmetry assumptions. In particular, the {authors}
consider \emph{any } sufficiently small perturbation of Kasner data
 at  {constant $t_0>0$} hypersurface
and then solve towards { the singularity at $t=0$}. The space-times one thus obtains (within the models considered) are therefore unrestricted  stability results, 
in that the results hold for open sets of data that are prescribed 
on a hypersurface ${\Sigma}$ off of the singularity. 
In contrast, the space-times obtained from {constructions} by Fuchsian techniques
are not apriori known to cover such an open set of data on a 
 hypersurface off of the singularity, {even if they enjoy all the gravitational degrees of freedom from a function counting point of view}--in principle they could be  a very 
 thin set in the {moduli} space of allowable initial data on ${\Sigma}$. 

The results in \cite{RS1, RS2, Speck, RS3}  are perturbative results (as is ours), but the techniques used differ from ours substantially. 
The main difference of our setting  with those considered in   \cite{RS1, RS2, Speck, RS3} on the face of them, 
 is that our background is \emph{highly} anisotropic; in fact it contains an \emph{expanding} direction in addition to the two contracting ones. 
One could thus speculate on the 
universality of this  \emph{extra}
instability in (3+1) Kasner vacua, which are of necessity highly an-isotropic. One can also wonder whether 
stability results within more restricted symmetry classes (like the ones considered here) might be true for other (3+1) Kasner vacua. 
While these questions are
 not pursued here, we note that the methods we develop seem very robust in that regard.

On the other hand,  the stability results in \cite{RS1, RS2, Speck, RS3} encompass \emph{all} sufficiently 
close perturbations of the background solution. In our setting such a result is \emph{not true}, as we saw by virtue of the Kerr examples. 
The fact that we are able to prove a stability result in the axial-symmetry class considered here 
utilizes many geometric features of this symmetry class, but most essentially 
a way to re-write the equations as a \emph{free} wave coupled with 1st order ODEs. 
(This special free-wave-ODE system which we will introduce and exploit here is not available in settings outside (polarized)  axial symmetry).

A brief comparison of some of the methods in  \cite{RS1, RS2, Speck, RS3} and this work 
is also due: 
The former works use a CMC foliation of the space-times and utilize that gauge in 
the study of Einstein's equations. 

We use instead a \emph{geodesic} gauge (not used before in the study of the Einstein's equations outside 
the analytic class, as far as we are aware); as we will explain below, we \emph{derive} 
an approximate CMC property of the geodesic parameter
which is \emph{central} in obtaining our result. (We are unable to explain on prior grounds why our geodesic gauge should display this additional behaviour). 
On a more analytic level, 
in the former works \cite{RS1,RS2,Speck}, the authors relied on an approximate monotonicity, i.e., good signs of error terms in the main estimates after the use of a combination of identities. 
In \cite{RS3} this approximate monotonicity is not available, and the authors allow a much more singular behaviour of the solutions and higher derivatives, coupled with a weights-descent scheme to derive optimal estimates at the lower ones. 
This is in fact similar to what we perform here (and indeed descent techniques have been used in other problems in nonlinear waves). But the method presented here is different; the source 
of the descent scheme here is traced directly back to the AVTD behaviour displayed in some geometric parameters of the space-time.  
\medskip

The fact that our result (the only such singularity 
formation result in $(3+1)$-vacua with just one degree of symmetry) holds in (polarized) axial symmetry, but manifestly \emph{is false}
with no symmetry assumptions  also relates to predictions by 
 {\it Belinski\u{i}-Khalatnikov-Lifshitz (BKL)} \cite{BKL70,BKL82}:

A long-standing, if controversial, proposal on the \emph{generic} 
behaviour of 
Big Bang (and, by extension, black hole interior {spacelike}) singularities was put 
forward in \cite{BKL70,BKL82}. The prediction there was that generically (in 1+3 vacuum), the space-time should experience rapid oscillations around 
different Kasner {epochs}, as one approaches any fixed point on  the 
 singular hypersurface;
this proposed generic behaviour is often called `mix-master' type. 
The ``generic'' part of the statement is based on a formal 
analysis that identifies settings where this mixmaster 
behavior should not be 
true, 
but instead {AVTD behavior occurs}.
  This class of solutions where the mix-master behavior is 
 ``turned off'', due to eliminating gravitational degrees of freedom, for 
 example,  includes 
 the polarized models studied here, cf. \cite{IM02}.
 
{Finally, we should note that} the mix-master type dynamics have only been rigorously derived in the 
 spatially isotropic (meaning three degrees of symmetry) setting of Bianchi 
 IX space-times by Ringstr\"om \cite{Rin01}. Further numerical investigations have 
 appeared in \cite{Gar,WIB98}. We refer to \cite{IM02} for a more detailed discussion 
 of such results.
The extent to which the mix-master proposal should be trusted to be generic 
is  a matter of discussion, and we do not take a position here.
 One must note, however that non-mix-master (but rather, AVTD-type behavior) \emph{was} predicted in the literature
 for all the settings for which it has now been proven, in particular \cite{RS1, RS2, Speck, RS3} and the present paper--see \cite{RS3} 
for a discussion of the relevant literature. 
  \medskip

We close the discussion of the literature by
comparing the ``strength'' of the singularities occurring here with those that have been studied 
 prior:

\subsubsection{The strength of the singularity.}


 The ``strength'' of a singularity can be measured with respect to the behaviour of the space-time curvature, measured against a suitably propagated orthonormal frame. Altenratively, as often used in
 spherical symmetry one can consider the blow-up of the Kretschmann scalar $R^{abcd}(g)\cdot R_{abcd}(g)$. 
The blow-up of such componenents against the frame we introduce below signifies the inextendibility of the metric past the singularity in a ${\cal C}^2$ 
sense. In fact we strongly expect that the proof of Sbierski \cite{Sb} extends to this setting to show the ${\cal C}^0$ inextendibility also. 

The choice  of the \emph{parameter} relative to which the blow-up rate is determined can be made in different ways: 
If one had a CMC foliation of our space-time foliation that terminates at the singularity, one could measure curvature components or the Kretschmann 
scalar with respect to that parameter; this is done in \cite{RS1, RS2}, for example. We here have an \emph{asymptotically CMC} foliation, 
given by a parameter $r$ in Theorem \ref{thm_rough}; as we will see the 
mean curvature ${\rm tr}_g K$ of level sets of $r$ satisfies: 

\[
{\rm tr}_g K= -\frac{3\sqrt{2M}}{2} r^{-3/2} +O(r^{-\frac{3}{2}+\frac{1}{4}}). 
\]
Relative to this parameter the Kretschmann scalar blows up like $r^{-6}$. This is in complete agreement with the 
asymptotic behaviour of this scalar in the higher-dimensional vacuum
 space-times in \cite{RS3}.
 (The parameter $t$ referred to there corresponds to $r^{3/2}$ here). 
In fact each of the curvature components $R_{0101}, R_{0202}, R_{0303}$ for an orthonormal frame $<e_0, e_1, e_2 e_3>$
 where $e_0$ is time-like 
blow up like $r^{-3}$, where $r$ is the  parameter that appears in Theorem \ref{thm_rough}. 
 
 It is also useful to make an analogy with the parameter usually used in spherical symmetry. There, the natural parameter is the \emph{area radius} 
 of the spheres of symmetry;
in our settings the analogues of these are 
$\{ {\cal S}_{\tau, r_0} \}:=\{T=\tau, r=r_0\}$. Given the extreme  an-isotropy of these spheres,  one can instead consider the area element 
at any  point $T,\Theta$ on such spheres (with a suitable renormalization to account for the degeneracy of the induced metrics at the poles),
 as the correct localized analogue of the area radius. 
It follows from \eqref{cl} that the area element $({\rm sin}\Theta)^{-1}\sqrt{ g_{\Theta\Theta}g_{\phi\phi}-(g_{\Theta\phi})^2}:={\rm rad}(T,\Theta)$ 
coefficient behaves like $(\sqrt{r})^{\alpha(T,\Theta)+\delta (T,\Theta)}$. 
In particular, the blow-up rate of the curvature components $R_{0101}, R_{0202}, R_{0303}$ 
will behave like: 

\[
[({\rm sin}\theta)^{-1}\cdot \sqrt{ g_{\Theta\Theta}g_{\phi\phi}-(g_{\Theta\phi})^2}]^{-\frac{6}{[\alpha(T,\Theta)+\delta (T,\Theta)]}}.
\]
As we will see below the function $\delta(T,\Theta)$ is an explicit function of 
$\alpha$: $\delta(T,\Theta)=-d_2(\alpha(T,\Theta))$, 
given by the explicit formula \eqref{d1.d2} below. In particular, since the 
function $\alpha(T,\Theta)$ is close to 1 everywhere,
 the rate of blow-up relative to this (localized) area radius is a \emph{function}, 
 whose rate  of blow up depends on the point $(T,\Theta)$ on the final singularity; 
we note that this exponent function ${\rm rad} ^{-\zeta(T,\Theta)}$  has 
$\zeta(T,\Theta)$  varying continuously in $t,\theta$, but its minimum value is 3.  

We note that this is in agreement with the blow-up behaviour exhibited by the massless 
scalar fields in spherical symmetry considered by Christodoulou  in \cite{Christ87, Christ91}, where the 
rate of blow-up of these components is \emph{at least} like ${\rm rad}^{-3}$. (See also  \cite{An} where upper bounds for the same matter model, again in spherical symmetry 
are established). 

\subsubsection{Outlook: Results beyond two degrees of symmetry?}

In the classical $(1+3)$-dimensional space-times, 
many of the settings in which an understanding of the entire maximal 
hyperbolic developments of solutions to the Einstein equations 
(including in black hole interior regions) has been obtained, {concern} 
 space-times with two degrees of symmetry imposed, such as the spherical 
symmetry or $\mathbb{T}^2$ and Gowdy symmetry classes. 
{In the first case, this is always in the presence of  matter fields, in view of Birkhoff's theorem}. 

A wealth of literature on such space-times exists over the past decades. 
Mathematically, 
two degrees of symmetry result in a quotient space-time of 1+1 dimensions. 
These are especially well-suited for analysis since the resulting 
quotient space-time is locally described by two scalar-valued functions; these are  the area radius (usually denoted by
 $r$--\emph{not to be confused} with the function $r$ here) and the conformal factor $\Omega$ in the $(1+1)$-quotient. Moreover, 
in many matter models one is able  to close estimates at the level of first-derivative norms of the matters fields and of
 r, $\Omega$.\footnote{In particular, the structure equations themselves can be studied directly, without considering further derivatives thereof.} 
This allows for large-data results in these symmetry classes, 
capturing the behaviour up to (frequently space-like) singularities, with 
remarkably ingenious techniques. 

{In the cosmological setting, we single out the resolution of strong cosmic censorship for unpolarized $\mathbb{T}^3$-Gowdy space-times,
 in the seminal work of Ringstrom \cite{Rin09}, previously known in the polarized\footnote{Polarized here has a different meaning than ours.}
 case by Chru\'sciel-Isenberg-Moncrief \cite{CIM90}. Towards the expanding direction, global existence is known in the polarized $\mathbb{T}^2$ class \cite{BCIM97}, 
even for weakly regular spacetimes \cite{LS15,LS16}, including precise late-time asymptotics, see also \cite{Rin15}.}

In spherical symmetry, we recall the seminal work of
Christodoulou \cite{Christ87,Christ91} for 
the Einstein-massless scalar field model, where he provided a complete 
classification of all solutions arising from one-ended initial data. In 
particular, he showed that under open conditions on the initial data, a 
black hole forms, containing a space-like singularity at $r=0$, where the 
Kretschmann scalar $|\text{Riem}|^2$ blows up no slower than $r^{-6}$, $r$ 
being the area radius function.\footnote{See recent work \cite{An}, where 
the authors derived upper bounds for the Kretschmann scalar.}
Large portions of the singularities there, are expected to be very similar 
to the Schwarzschild one. Of course, one also has 
 the classical 
   Oppenheimer-Snyder solution and the Vaidya space-times display portions 
of their final singularities to be \emph{isometric} to a portion of the 
Schwarzschild space-time. 
\medskip

Naturally, one would like to be able to obtain results of this nature 
 outside  the 
2-degree symmetry class. 
We believe the techniques developed here should be very helpful in studying 
perturbations of any of the solutions obtained in spherical symmetry in 
just 
the axi-symmetric setting, {where the Einstein equations admit a wave map formulation \cite{Wein}.}

{We note also that the setting of polarized axi-symmetry has recently attracted  attention in other settings. We recall} the problem of full non-linear stability of the 
Schwarzschild exterior in vacuum, resolved in polarized axi-symmetry 
in \cite{KS}, as well as the construction of solutions to the Einstein-Vlasov system 
out of suitable limits of pure vacuum solutions by Huneau-Luk \cite{HL19,HL18}.  

It is hoped that the methods developed herein can serve as a powerful tool
 to study many of the other questions surrounding 
 Einstein's equations in the presence of just one Killing field.


  

\subsection{An outline of the ideas.}
  
  We briefly outline some of the main ideas in this paper and how we overcome the
  central 
  challenges that come up. First we  recall   the reduced Einstein equations.  
  \medskip

\subsubsection{The reduced Einstein vacuum equations under polarized axi-symmetry}\label{axipol}

Let $(\mathcal{M}^{1+3},g)$ be a Lorentzian manifold. We work in the class of polarized 
axial symmetry, that is, in the class of metrics which admit a hypersurface orthogonal, 
spatial Killing vector field $\partial_{\phi}$ with closed $\mathbb{S}^1$ orbits. Under 
this symmetry assumption, the space-time metric takes the form
\begin{align}\label{metric}
g=\sum_{a,b=0,1,2}h_{ab}dx_adx_b+e^{2\gamma}d\phi^2,
\end{align}
where $h_{ij},\gamma$ are independent of $\phi$. Define the projected $1+2$ metric on the Lorentzian hypersurfaces orthogonal to $\partial_{\phi}$:
\begin{align}\label{h}
h=\sum_{a,b=0,1,2}h_{ab}dx_adx_b
\end{align}
In this context, the {EVE}
\begin{align}\label{EVE}
\text{R}_{ab}(g)=0,\qquad\qquad a,b=0,1,2,3,
\end{align}
are equivalent to the system \cite[Appendix VII]{ChoqBook}
\begin{align}
\label{redEVEwav}\square_g\gamma=&\;0\\
\label{redEVERic}\text{R}_{ab}(h)
=&\;\nabla_{ab}\gamma+\partial_a\gamma\partial_b\gamma,
\end{align}
for $a,b=0,1,2$, where $\nabla$ is the Levi-Civita connection of 
the $(1+2)$-metric $h$. {Note that \eqref{redEVEwav} can also be written as a non-linear wave equation with respect to $h$:
\begin{align}\label{redEVEwav2}
\square_h\gamma=-|\nabla\gamma|^2_h.
\end{align}
}The following lemma plays a central role in our approach.
\begin{lemma}\label{lem:Weyl}
In the reduction (\ref{redEVEwav}),(\ref{redEVERic}) of the EVE, the Riemann curvature of the $1+2$ metric $h$ is locally determined by its Ricci tensor, due to the vanishing of the Weyl tensor in 3 dimensions.
 In particular, the following identities are valid:
\begin{align}
\label{Rh=0}\mathrm{R}(h)=&\,0\\
\label{Riem=Ric}\mathrm{R}_{abcd}(h)=&\,2\big(h_{a[c}\mathrm{R}_{d]b}(h)-h_{b[c}\mathrm{R}_{d]a}(h)
\big)
\end{align}
for all indices $a,b=0,1,2$.\footnote{Here we adopt the notation $T_{a[bc]d}=\frac{1}{2}(T_{abcd}-T_{acbd})$ for the antisymmetrization of indices between square brackets.}
\end{lemma}
The proof is presented in \S\ref{pfWeyl}.

The reduced system (\ref{redEVEwav}),(\ref{redEVERic}) consists of seven equations, when 
\eqref{redEVERic} is expressed with respect to an $h$-orthonormal frame $\{e_0,e_1,e_2\}$ with $h_{00}=-1$. As we will see,  the $\{0b\}$-
components of (\ref{redEVERic}) should be thought of as constraint equations, {which satisfy separate propagations equations (see \S
\ref{retrEVE}), leaving} four genuine \emph{evolution} equations: 1) The scalar wave equation for 
$\gamma$ and 2) the equation (\ref{redEVERic}) for $(a;b)=(1;1),(1;2),(2;2)$ which we 
replace by corresponding formulas for $\text{R}_{0i0j}(h)$, $i,j=1,2$, using 
(\ref{Riem=Ric}):
\begin{align}\label{R0i0j}
\text{R}_{0i0j}(h)=-\text{R}_{ij}(h)+h_{ij}\text{R}_{00}(h)=-\nabla_{ij}\gamma-
e_i\gamma e_j\gamma
+\delta_{ij}\big(\nabla_{00}\gamma+(\nabla_0\gamma)^2\big).
\end{align}

  {\bf Choice of gauge:}
 We wish to use a \emph{geodesic} gauge, with 
time-like geodesics that end at different points on the singularity surface 
$\{r=0\}$. The main advantage of this choice  is that the second set of equations \eqref{redEVERic} 
gives rise to a system of (non-linear) ODEs, whose forcing terms are determined 
from the polarized field $\gamma$. This uses in an essential way the 
fact that $h$ is a $(1+2)$-dimensional metric, and thus the curvature tensor is locally 
expressible in terms of the Ricci curvature via \eqref{Riem=Ric}. 
(The latter being directly expressible in 
terms of the field $\gamma$).  
\medskip

 We put down some key equations schematically here for the purposes of this outline 
 of ideas: We will be choosing the vector field $e_0$ to be affine: 
 \[
 \nabla_{e_0}e_0=0,
 \]
 and associate to it a parameter $r$ which we choose to agree with the one in Schwarzchild, in the sense that: 
 \[
 e_0= -(\frac{2M}{r}-1)^{\frac{1}{2}}\partial_r.
 \]
 We will be considering an $h$-orthonormal frame $e_0, e_1, e_2$ which is propagated
  along $e_0$ through  a modification of parallel transport, which we specify {in \eqref{almpar.trans}}. 

The metric $h$ (and thus also the metric $g$, in view of formula \eqref{metric}) will 
then be encoded in the connection coefficients of this frame. Certain key connection coefficients we wish to highlight for this introduction  are the components  
\[
K_{ij}(r,t,\theta):={h(\nabla_{e_i}e_0, e_j),\qquad i,j =1,2. } 
\]  
 These satisfy the Riccati equations-see 
 \eqref{finredEVERic11pre}-\eqref{finredEVERic12pre} below. 

Our frame will be partly 
 ``initialized'' \emph{at} the singularity, $r=0$, and partly on the 
 initial data hypersurface. The two key
directions $e_1, e_2$ will be normalized by requiring that: 
\begin{itemize}

\item $e_2$ should capture the  collapsing direction, i.e. $K_{22}(r,t,\theta)<0$ as $r\to 0^+$, and 

\item $K_{11}(r,t,\theta)$ should capture the normal to $e_2$ 
 expanding direction, i.e. $K_{11}>0$ as $r\to 0^+$. 
 
\item  The mixed component $K_{12}(r,t,\theta)$ should {be less singular}, capturing thus an asymptotic diagonalization of $K_{ij}$, and in particular will satisfy:
 
 \beq\label{diag.asympt}
 K_{12}(r,t,\theta)\cdot K^{-1}_{11}(r,t,\theta))=o(1),  \quad
 K_{12}(r,t,\theta)\cdot K^{-1}_{22}
 (r,t,\theta))=o(1)
\eeq
 as $r\to 0^+$. 
 
 \item $e_2$ should be ``tangent'' to the singularity in a suitable sense, and $e_1$ should be tangent to the initial data hypersurface. 
\end{itemize} 


One would expect this geodesic gauge to not be  a suitable framework for the Einstein 
equations, due to the apparent \emph{loss of derivatives} of the metric 
relative to the 
curvature in the directions $e_1, e_2$ in such  a gauge. 
In particular, one might expect to \emph{not} be able to prove energy 
estimates 
that close for the reduced Einstein equations. This however turns out 
\emph{not} to be 
the case, as we will explain below. 

The upshot of all this is that the EVE system in this symmetry class and in the geometric parameters  that we introduce below 
can be seen as a coupled system 
 of a  free wave with a system of transport equations for the connection 
 coefficients, the most important of which are {the}  non-linear  
 Riccati-type system
 \eqref{finredEVERic11pre},  
 \eqref{finredEVERic22pre}, \eqref{finredEVERic12pre} below. 
 The metric 
 $g$
 in the  wave equation \eqref{redEVEwav} of $\gamma$ is of course coupled to
 $\gamma$, since $g$ is determined by the Riccati equation with  
 the forcing term depending on $\nabla\gamma,\nabla^2\gamma$. So \eqref{redEVEwav}
 can schematically be expressed as: 
 \[
 \square_{g(\gamma)} \gamma=0. 
 \]
 This is the main equation where the quasi-linearity of the Einstein 
 equations  is manifested in our setting. 
 \medskip
 
 A first difficulty appears here already, in that due to the contracting direction 
 $\partial_{\Theta}$ 
 of the space-time metric $g$, one expects a generic family of time-like geodesics 
 would develop caustics \emph{long before} they reach the singularity. So in particular,
 the
 solutions to the  Riccati equations  would blow up before $\{r=0\}$, resulting in a 
 \emph{gauge} 
breakdown which would impede the study of the \emph{true} singularity which lies at 
$\{r=0\}$.

 Secondly, we note that both the Schwarzschild metric itself  and the solutions we 
 eventually
  obtain by perturbing its initial data are highly an-isotropic: There are two 
  \emph{contracting}  directions $\partial_\phi,{\partial_{\tilde{\Theta}}}$ and an \emph{expanding} 
  direction ${\partial_{T}}$. If one hopes to obtain uniform estimates consistent with the 
  desired conclusion \eqref{cl}, 
  one must control the metric strongly enough (at least in lower norms),
  in order to 
  capture the collapsing directions and separate them from the expanding 
  {direction}.  
  We review this in the overview of the Riccati equations. 
 
\subsubsection*{The wave equation on very singular backgrounds, and the 
``asymptotically CMC, 
for free'' property of our space-times} Our 
analysis of the wave 
equation proceeds  via energy estimates, using a weighted version of the 
 affine vector field $e_0\parallel \partial_r$ as a 
multiplier. 
A first important observation at this point involves the mean curvature 
of level sets of 
$r$: 
 It follows easily from the asymptotics in \eqref{form}-\eqref{cl} 
  that the mean curvature of each level set $\Sigma_r$ will be of the form: 
 \[
 H(r,t,\theta)\sim G(t,\theta)r^{-3/2}. 
 \]
However, we will show  that $G(t,\theta)$ is in fact a \emph{constant}, \emph{independently} 
of the value of $\alpha(t,\theta)$. In particular, as we 
see  below, \emph{in this geodesic gauge}, the mean curvature 
of $\Sigma_r$ is automatically 
\emph{constant to leading order}. 
\medskip

This ``asymptotically  CMC'' feature of the geodesic gauge  is absolutely 
essential 
in deriving suitable estimates for the free wave 
$\gamma$ and its 
derivatives. 
\emph{If} this property had not held, then for the \emph{linear} equation 
$\Box_g\gamma=0$, 
the energy of $\gamma$ would \emph{not} behave as predicted by the first 
term in 
\eqref{cl}, but would rather blow up exponentially.

In order to take advantage of this feature, whenever we study the wave 
equation
 $\Box_g\psi=F$ (for $\psi$ 
being $\gamma$ and its suitable derivatives), we will always be using 
 energy currents, whose associated  multiplier  vector field 
will always be  
$r$-dependent 
re-scalings of $e_0$. The vector fields by which we seek to commute the 
equation are 
(for the most part) chosen to commute with $e_0$, so as to take advantage 
of the 
asymptotically CMC behaviour of the mean curvature that provides some key 
cancellations. 

The analysis of the wave equation is carried out in a separate section, 
and further details 
of the ideas are given in
 the (brief) introduction of that section. Notably, the use of the AVTD 
 behaviour of the solution {is important}
 to establishing a \emph{{weights}-descent scheme} to derive improved 
 estimates for all our parameters at lower orders (where 
 the estimates are optimal, in that they are fully consistent with
  \eqref{cl}) compared to higher orders, where the estimates are much weaker than the claimed behaviour in \eqref{cl}. 

\subsubsection*{The Riccati equations and the singular branch of
 the solutions.}
A very central  challenge to the stability result we derive,
 appears at the level 
of the (non-linear) Riccati equations. 
{Requiring that} the frame $e_1, e_2$ 
is transported according to the law:\footnote{{It is easy to see that \eqref{almpar.trans} defines an orthonormal frame, provided $e_0,e_1,e_2$ are orthonormal initially.}}
\beq
\label{almpar.trans}
\nabla_{e_0}e_0=0,\qquad
\nabla_{e_0}e_1=-K_{12}e_2,\qquad \nabla_{e_0}e_2=K_{12}e_1;
\eeq
the corresponding 
connection 
coefficients {$K_{ij}$, $i,j=1,2$,}
solve the following system:\footnote{{The Riccati system \eqref{finredEVERic11pre}-\eqref{finredEVERic12pre} is a consequence of \eqref{R0i0j} and \eqref{almpar.trans}, implying the formulas 
\begin{align}\label{R0i0j2}
\begin{split}
\text{R}_{0101}(h)=-e_0K_{11}-(K_{11})^2-3(K_{12})^2,&\quad \text{R}_{0202}(h)=-e_0K_{22}-(K_{22})^2+(K_{12})^2,\\
\text{R}_{0102}(h)=&-e_0K_{12}-2K_{22}K_{12}.
\end{split}
\end{align}
}} 
\begin{align}
\label{finredEVERic11pre}e_0K_{11}+(K_{11})^2+3(K_{12})^2+e_0\gamma K_{11}=&\,\overline{\nabla}_{11}\gamma
+(e_1\gamma)^2
-e_0^2\gamma-(e_0\gamma)^2\\
\label{finredEVERic22pre} e_0K_{22}+(K_{22})^2-(K_{12})^2+e_0\gamma K_{22}=&\,\overline{\nabla}_{22}\gamma+(e_2\gamma)^2-
e^2_0\gamma-(e_0\gamma)^2\\
\label{finredEVERic12pre} e_0K_{12}+(2K_{22}+e_0\gamma)K_{12}=&\,\overline{\nabla}_{12}\gamma+e_1
\gamma e_2\gamma
\end{align}
(Here $\overline{\nabla}$ stands for a connection on the space spanned by 
of 
$e_1, e_2$--this connection is defined to be the projection of the
 Levi-Civita connection $\nabla$ onto ${\rm Span}\langle e_1, e_2\rangle$).

 In terms of singular behaviour in $r$, we will see below that 
  given the expected asymptotic behaviour \eqref{AVTD1} of $\gamma$, as $r\to 0^+$, \emph{and} the AVTD property that implies that 
all $e_0$-derivatives are more singular in $r$ than spatial derivatives, 
  these equations admit formal solutions with the following asymptotic expansion: 
\beq
\label{zero.asympts}
K_{22}(r,t,\theta)\sim d_2(t,\theta)\sqrt{2M} r^{-\frac{3}{2}},\qquad K_{11}(r,t,\theta)\sim d_1(t,\theta)\sqrt{2M} r^{-3/2}, \qquad|K_{12}(r,t,\theta)|\lesssim r^{1-\frac{1}{4}},
\eeq
where the functions $d_1(t,\theta)=d_1[\alpha(t,\theta)]$ and $d_2(t,\theta)=d_2[\alpha(t,\theta)]$   are given by the explicit  formulas of the parameter $\alpha(t,\theta)$ in \eqref{AVTD1}:
\beq
\label{d1.d2} 
\begin{split}
&d_1(t,\theta):=\frac{\alpha(t,\theta)-\frac{3}{2}+\sqrt{(\alpha
(t,\theta)
-\frac{3}{2})^2
+6\alpha(t,\theta)-4|\alpha(t,\theta)|^2}}{2},
\\& d_2(t,\theta):=\frac{\alpha(t,\theta)-\frac{3}{2}
-\sqrt{(\alpha(t,\theta)-\frac{3}{2})^2+6\alpha(t,\theta)
-4|\alpha(t,\theta)|^2}}{2}.
\end{split}
\eeq

   In fact, these two possible leading-order behaviours correspond to the (unique) ``collapsing'' 
 direction, which we will choose to be $e_2$ and the dual (principal) 
 ``expanding'' direction $e_1$. In the above derivation we have implicitly 
 normalized  the expanding direction by requiring it to be 
 (asymptotically) orthogonal to the collapsing direction, to suitably 
 high order, implicitly imposing \eqref{diag.asympt}

 Given that the connection coefficients $K_{ij}$ appear in the wave equation $\Box_g \gamma{=0}$, 
 and especially in the derivatives of this equation, we need to control the tensor $K_{ij}$ 
 in higher order Sobolev spaces $H^k_{t,\theta}(\Sigma_r)$. In particular, we need to consider
  \emph{derivatives} (with respect to $\partial_t,\partial_\theta$)
   of $K_{ij}$ and derive bounds for them that are consistent with the asymptotic 
  behaviour \eqref{cl} (at least for a low enough number of derivatives).

 It is \emph{here} that an essential (and unexpected) difficulty in this problem arises: 
Assume that $\gamma$ and its (low enough) derivatives 
display a behaviour in $r$ that is consistent with the 
asymptotics  of ${g_{\phi\phi}(r,t,\theta)}$ in \eqref{cl}. We then need to derive bounds for 
$K_{ij}$ 
and its (low enough) derivatives that would be consistent with the asymptotics for \eqref{zero.asympts}. In particular (low enough) derivatives of $K_{ij}$ should behave (in Sobolev and the $L^\infty$ spaces)
as in \eqref{zero.asympts}. 

Here there is a dichotomy:  For the \emph{un-differentiated} terms $K_{ij}$, 
indeed assuming that up to two of  the derivatives of $\gamma$ satisfy pointwise bounds that 
are consistent 
with {\eqref{AVTD1}}, we can \emph{derive} the asymptotics of $K_{ij}$ consistent with \eqref{cl},
via a
 Fuchsian-type analysis of the nonlinear ODEs.

However, once we consider the \emph{differentiated} 
equations (here $\partial$ stands for $\partial_t,\partial_\theta$) 
\begin{align}
\label{finredEVERic11pre.diff}e_0\partial K_{11}+(2K_{11}+e_0\gamma) \partial  K_{11} +6K_{12}\partial K_{12}=&\,\partial [\overline{\nabla}_{11}\gamma
+(e_1\gamma)^2
-e_0^2\gamma-(e_0\gamma)^2]-(e_0\partial \gamma) K_{11} \\
\label{finredEVERic22pre.diff} e_0\partial  K_{22}+(2K_{22}+e_0\gamma) \partial K_{22}
-2K_{12}\partial  K_{12}=&\,\partial 
[ \overline{\nabla}_{22}\gamma+(e_2\gamma)^2-
e^2_0\gamma-(e_0\gamma)^2]- (e_0\partial\gamma) K_{22}\\
\label{finredEVERic12pre.diff} e_0\partial K_{12}+(2K_{22}+e_0\gamma)\partial K_{12}=&\,\partial [\overline{\nabla}_{12}\gamma
+e_1\gamma e_2\gamma]-(2\partial K_{22}
+e_0\partial\gamma) K_{12},
\end{align}
 we find that {the {\it free branches} of the solutions} of these  linear equations are
 \beq
 \begin{split}
& \partial K_{11}^{\rm free}=  c_{11}(t,\theta)
  r^{2d_1(t,\theta)-\alpha(t,\theta)}, \qquad
 \partial K_{22}^{\rm free}= c_{22}(t,\theta)
  r^{2d_2(t,\theta)-\alpha(t,\theta)},  
  \\&\partial K_{12}^{\rm free}= c_{12}(t,\theta)
  r^{2d_2(t,\theta)-\alpha(t,\theta)}.
\end{split}
 \eeq
Recalling that $\alpha(t,\theta)$ is close to 1 in $L^\infty$ and thus $d_2(t,\theta)$ is close to {$-1$} and $d_1(t,\theta)$ close to 
$\frac{1}{2}$,  we note that 
  the first of these free branches (for $K_{11}$) is \emph{less} singular than the leading order behaviour 
$r^{-3/2}$, and thus does not 
  impede the proof that $K_{11}$ and its high derivatives satisfy bounds  consistent 
  with \eqref{cl}. 
  The \emph{second} branch of the free solution of $K_{22}$ is potentially
   detrimental: 
  If the 
  $\partial_t,\partial_\theta$ 
  derivatives of $K_{22}$  \emph{do indeed} behave in this much more singular way $c(t,\theta) r^{-3+\e(t,\theta)}$, 
  $|\e(t,\theta)|\le \frac{1}{8}$, 
  (whereas 
  the undifferentiated $K_{22}$ behaves like 
  $r^{-3/2}$), then one has absolutely no hope  of eventually \emph{proving}
  asymptotics of the form \eqref{cl}, and hence,  our result. 
  (We remark that this would even kill the hope of deriving 
  estimates for the linearized Einstein
   equations which would be consistent with the asymptotics that we prove here)--thus this feature of the equation can be termed 
a \emph{linear instability} of the EVE around the Schwarzschild singularity, at least in this gauge.


   The only hope therefore  is that this very singular branch, 
   allowed by  the differentiated equations, is somehow \emph{not there}. 
This hope is actually validated.    
   It is to prove this part that the ODE character of the second branch of our 
   equations (in this gauge) is used in an essential way: 
   
   Whereas if one were to solve the equations \eqref{finredEVERic22pre}, 
   \eqref{finredEVERic12pre} \emph{forwards}, towards the singularity, one 
   cannot rule out the possibility of this very singular behaviour, 
   one \emph{can} set this singular branch to \emph{zero} when one solves 
   the Riccatti 
   equations \emph{backwards} from the singularity.
      Skipping some technical issues, that is possible to do if one 
      considered the Riccati equations 
  \eqref{finredEVERic11pre}, \eqref{finredEVERic12pre}, 
  \eqref{finredEVERic22pre}  separately, taking the RHS as being 
      \emph{given}, and consistent with the behaviour \eqref{cl}. 
\medskip

{\bf The iteration scheme:}  Taking this challenge into account, we resort to an 
\emph{iteration scheme} for solving the system \eqref{redEVEwav}, \eqref{redEVERic},
 producing a sequence of 
 solutions $(\gamma^m, h^m)$: Taking the previous step $(\gamma^{m-1}, h^{m-1})$ as given,  
 we need to produce a new pair $(\gamma^m,h^m)$. 

 We first solve the free wave equation     
 \beq
 \label{wave}
 \square_{g^{m-1}}\gamma^m=0
 \eeq
  \emph{forwards}, 
 towards the singularity; we next solve for the metric $h^m$ via its connection 
 coefficients $K^m_{ij}$. 
 We  solve the two of the 
connection coefficients $K^m_{22}(r,t,\theta), K^m_{12}(r,t,\theta)$
\emph{backwards from the 
 singularity}, \emph{setting the singular {branches} of the 
 {solutions} for $\partial K^m_{22},{\partial K_{12}^m}$ to zero}. This is completed by showing that it is possible 
 to solve for the remaining connection coefficients so that the metric and second fundamental 
 form induced on a suitable  hypersurface
  $\Sigma_{r_*^m}=\{r=r_*^m(t,\theta)\}$ (which is to be determined)
matches the initial data that we have prescribed.      
   \medskip
   
     In other words, we prove our result \emph{not} by a bootstrap 
     argument, but by a 
      (forwards-backwards) Picard-type iteration. 
      ({We note that in
      recent 
       works by Hintz-Vasy, for example in their breakthrough proof of the Kerr-de Sitter stability problem \cite{HV}, the authors solved the Einstein equations via an iteration}--however the underlying reasons there are 
       entirely different). 
     Certain technical difficulties that this gives rise to will be 
     discussed in the main body of the 
     paper.
     
      For now, we wish to discuss the final main
      challenges that we need to overcome, to eventually prove the result, 
      and 
      establish the asymptotics 
      \eqref{cl}. 
      \medskip

\subsubsection*{Closing the EVE in {a} geodesic gauge}
 
 Working in a geodesic gauge presents certain challenges in terms of deriving suitable 
 bounds for all the quantities that govern our space-time. In particular, there is a clear danger  
of \emph{losing derivatives}, which would not allow the derivation of our estimates in the next step of our iteration. This is in fact 
true even for the local-in-time problem, independently of singularity formation. To distinguish these two 
challenges (singular behaviour in $r$ and regularity in fixed Sobolev spaces), we introduce the following convention: 
 
{\bf Language Convention:} 
Given any parameter $f(r,t,\theta)$ in our problem, we will use the term \emph{regularity} 
to refer to suitably many derivatives of $f$ lying in $L^2_{t,\theta}$. The term
 \emph{singularity} will refer to the behaviour in $r$ 
of different norms of $f(r,t,\theta)$ that  we keep track of (e.g. 
$L^\infty_{t,\theta}$ norms, Sobolev norms $H^k_{t,\theta}$),
 as $r\to 0^+$.  
\medskip

The well-known loss of derivatives that occurs in a geodesic gauge, e.g. in 
Fermi or exponential coordinates, 
does not make the geodesic gauge 
suitable, \emph{in general}, for a study of this initial value problem. 
Nonetheless, the special structure of the equations 
makes this possible in our case: 

In our approach to this problem, the wave $\gamma$ is treated as the main 
part of the evolution. 
The equation \eqref{redEVEwav} is however non-linear in $\gamma$, since the 
metric 
$g$ is related to $\gamma$ via the 
Ricci curvature of $h$ \eqref{redEVERic}.

The relation between the curvature of the 
projected $(2+1)$-metric $h$ and the free wave $\gamma$ is utilized via the 
Riccati equations, and also becomes manifest whenever we commute the 
equation with derivatives.  
It is important at this point that we always use the vector field $e_0$ as a multiplier. Also, at the top order 
of derivatives, $e_0$ is necessarily 
one of the commutation vector fields for our 
equation. 

From the point of view of regularity, it is clear that that the direction 
$e_0$ is \emph{privileged}: The Riccati equations show that $e_0K_{ij}$ 
is on the same level as $\overline{\nabla}^2\gamma$, while the derivatives 
$\partial_tK_{ij}, \partial_\theta K_{ij}$ are on the same level as 
$\partial_t\overline{\nabla}^2\gamma,\partial_\theta\overline{\nabla}^2\gamma$. 

A delicate balance is struck here: From the point of view of not loosing derivatives, use of the $e_0$ vector field is
good, because it brings out (differentiated) metric and connection  terms that are at the correct number of derivatives in terms of the wave $\gamma$. 
However, from the point of view of deriving asymptotics up to the singularity it is dangerous, since 
(in view of the asymptotics \eqref{cl} that we seek to establish) it generates terms that are more singular 
in terms of powers of $r$.
How this balance is achieved
 is  explained in more detail in the main body of the proof.  
The closing occurs in function spaces which at the very top orders use the
 vector field $e_0$ as a commutator up to two times. 

\subsubsection*{The location of the initial data, and regularity 
at the poles}

Two more novel aspects of our technique that we wish to highlight here 
relate to the issue of \emph{identifying} the position of the initial data 
hypersurface in the geodesic gauge we have chosen, and also some 
(technical) issues related to the regularity of our space-times at the 
axes. 

In contrast to \cite{RS1, RS2, RS3}, or the spherical symmetry setting 
we are not in a position to choose a space-like foliation which  
``synchronizes'' our approach to the singularity (via for example CMC 
surfaces). Nor do we have the option of using an \emph{area radius} 
parameter for 2-spheres as is often done in spherical symmetry. 
Rather, the approach to the singularity in our gauge is governed by a
(non-affine)   parameter $r$ along our time like $e_0$-geodesics.   
 This in particular implies that the \emph{location} of the hypersurface (expressed in terms of the coordinate $r$) 
 that is to carry the initial data must be solved for. This reduces to a 
 2x2
 system, which relies on connection coefficients that are solved for 
 \emph{starting} from the singularity. The solvability of the resulting 
 system is far from evident (at least to the authors); in fact it is to 
 obtain such a solvable system that the requirement of tangency of $e_1$
  (but not of $e_2$!) to the initial data hypersurface was imposed.  

A further challenge is related to the fact that we split our analysis 
between the (3+1)-dimensional wave equation $\Box_{g^{3+1}}\gamma=0$ and
the quotient metric $h^{2+1}$. Indeed,  the metric $h^{2+1}$ lives over a 
\emph{manifold-with-boundary} (over the coordinates $\theta, t, r$), with the boundary being at 
$\theta=0, \theta=\pi$. For various parameters in our inductive procedure 
we must \emph{impose} or \emph{derive} a certain vanishing of transverse
 derivatives to those two  boundaries. These vanishing conditions capture the 
 regularity of the resulting (3+1)-dimensional space-time. 
\medskip

More technical aspects of our analysis  will be discussed 
in separate introductions of the separate sections.

\begin{quote}
\textbf{Acknowledgements.} {\small We would like to thank Jonathan Luk, Igor Rodnianski, 
Jared Speck for useful discussions. S.A.~was supported by an NSERC  discovery grant
and an Ontario ER Award.  G.F.~was supported by the \texttt{ERC grant 714408 GEOWAKI},
 under the European Union's Horizon 2020 research and innovation program.}
\end{quote}

\section{The precise formulation of the result. }

We will be introducing the precise gauge in which the theorem is proven. 
It is useful to consider the Schwarzschild metric and a canonical frame 
associated to that metric.

\subsection{The Schwarzschild metric}\label{subsec:Schw}

The Schwarzschild solution \eqref{Schmetric.pre}, being spherically 
symmetric, belongs in the axi-symmetric polarized class and satisfies the EVE 
(\ref{redEVEwav}), (\ref{redEVERic}) for
\begin{align}\label{Schsol}
\gamma_{\rm S}=\log r+\log \sin\theta,&& h_{\rm S}=-(\frac{2M}{r}-1)^{-1}dr^2+(\frac{2M}{r}-1)dt^2+r^2d\theta^2.
\end{align}

The interior region $\{r<2M\}$ is naturally foliated by space-like hypersurfaces $\Sigma_r$, $r\in(0,2M)$, the level sets of the coordinate (and area radius) function $r$. The limiting slice $\Sigma_0$ is the hypersurface $\{r=0\}$, where the singularity occurs and where the 
 curvature invariants, such as the Kretschmann scalar, blow up. Also, across $\Sigma_0$, the space-time metric is $\mathcal{C}^0$-inextendible \cite{Sb}.\\
Consider the orthonormal frame
\begin{align}\label{Sframe}
e_0=-(\frac{2M}{r}-1)^\frac{1}{2}\partial_r,&&e_1^{\rm S}=(\frac{2M}{r}-1)^{-\frac{1}{2}}\partial_t,&&e_2^{\rm S}=\frac{1}{r}\partial_\theta,&&e_3^{\rm S}=\frac{1}{r\sin\theta}\partial_\phi.
\end{align}
In this frame, the second fundamental form $K_{\rm S}$ of the constant $r$ hypersurfaces $\Sigma_r$ is given by
\begin{align}\label{SK}
(K_{\rm S})_{11}=\frac{M}{r^2}(\frac{2M}{r}-1)^{-\frac{1}{2}},\qquad (K_{\rm S})_{22}=(K_{\rm S})_{33}=-\frac{1}{r}(\frac{2M}{r}-1)^\frac{1}{2}.
\end{align}
A direct computation also shows that 
\begin{align}\label{SRiem}
-\text{R}_{0101}(g_{\rm S})=2\text{R}_{0202}(g_{\rm S})=2\text{R}_{0202}(g_{\rm S})=\frac{2M}{r^3}.
\end{align}
%

%

\subsection{The initial data for our problem.}

The space-times we will study in this paper will arise as the future 
maximal hyperbolic developments of initial data sets that correspond
to \emph{perturbations}  of the initial data set 
$(\g_{\rm S}, \K_{\rm S})$ on $\Sigma_{\epsilon}$;
 the latter corresponds to the metric 
and second fundamental form induced on the hypersurface 
$\Sigma_\e:=\{r=\e\}$ in the Schwarzschild space-time.

 The closeness of our data to the Schwarzschild background 
  will be encoded in a parameter $\eta>0$, whose 
smallness  will also be specified below. We frequently denote by $\g,\K$ the 
initial data for brevity; also, all quantities in bold-faced letters will 
be related to the abstract initial data.

\medskip

Specifically, we consider a spatial metric $\g$ and a second fundamental form $\K$ 
(expressed in coordinates $t,\theta,\phi$) which satisfy the
vacuum constraint equations and the following
 polarized-axisymmetric condition, for any fixed component $\g_{ab}, \K_{ab}$, expressed  
with respect to the coordinate vector fields $\partial_\phi,\partial_\theta,\partial_t$: 
 \[
 \partial_\phi \g_{ab}=\partial_\phi \K_{ab}=0,\quad a,b=t,\theta,\phi,\qquad \K_{\phi a}=\g_{\phi a}=0,\quad a=t,
 \theta. 
 \]
For definiteness,  we will be normalizing the coordinates $t,\theta$ 
 by requiring that 
 \[
 \K_{t\theta}(t,\theta)=0,\quad \g_{t\theta}=0. 
 \]
 This requirement only 
 specifies the \emph{level sets} of the  coordinates $t,\theta$. We 
 impose an extra gauge 
  normalization to ensure that the poles occur at $\theta=0,\theta=\pi$, 
  and that 
for each fixed $t=t_0$ the set $\{\theta\in (0,\pi), \phi\in [0,2\pi)\}$ 
should 
extend to a 
smooth sphere 
at the poles $\theta=0,\theta=\pi$.

We note that this coordinate normalization 
 implies that  
the vector field $\partial_\theta$ must be normal to $\partial_t$. 
At the poles $\theta=0,\pi$ it also implies 
 that $\partial_\theta$ must be mapped to $-\partial_\theta$ after 
 flowing by $\pi$ 
 along $\partial_\phi$, and moreover the flow of $\partial_\theta$ at any of the two poles  defines a 2-dimensional space. Then 
 $\partial_t$ must be 
 invariant under the flow of $\partial_\phi$ at the two poles 
 $\theta=0,\pi$, since it must be 
 the unique vector field (up to choice of direction) that is normal to the  
 2-dimensional
  space that is left invariant under the flow of $\partial_\phi$.  

\medskip

Within this gauge normalization, 
we require that this initial data  be close to the corresponding 
Schwarzschild data in a suitably high Sobolev norm $H^s$.

Our assumptions on the initial data $(\g,\K)$ 
 will be formulated in terms of 
Sobolev spaces defined relative to the coordinates $t,\theta,\phi$. 
Also, the assumption on the component $\g_{\phi\phi}=e^{2\gamma}$ 
will be separate 
from that of the normal-to-$\partial_\phi$ part of the 
 metric ${\bf g}$: For the former we will impose initial data on 
 $\gamma=\frac{1}{2}\log (\g_{\phi\phi})$, while 
 initial data on the latter will be  treated in term of the 
 relevant components of  $\g, \K$ in the coordinates $t,\theta$. 
We write $\gamma_{\rm init}$ for the initial datum of this scalar parameter $\gamma$, for notational simplicity.  
 \medskip
 
\begin{remark}
We recall that for the abstract initial data, we also consider the 
(abstract) normal vector field $n$, which is normal to our
 initial data set. This vector field is used in defining the 
 initial energy of waves \emph{on} the initial data set. 
 
 In view of this, it makes sense to consider the 
 formal jet of the solution metric (given the prescribed $(\g,\K)$) 
 off of $\Sigma$--this makes sense in a coordinate $\tau$ that satisfies 
 $n(\tau)=1$. With this formulation, it makes sense to consider 
 the energy $E$ of $\gamma$, but also the energy of 
  $\gamma_{\rm init}-\gamma_{\rm S}$, {where $\gamma_{\rm S}$ is as in \eqref{Schsol}}. (The energy of a function  is defined using formula \eqref{energy.def}--note that we use the volume form
 ${\rm sin}\theta d\theta dt$ in its definition--this is for uniformity with our measurement of other quantities). 
\end{remark}

With this in mind, we require that for some $s\in\mathbb{N}$, 
large enough to be chosen below, and for all $k_1+k_2\le s-1$:\footnote{These assumptions can in fact be weakened. The requirements
 imposed here should only hold for what we will later call the \emph{lower derivatives} of the parameters; the derivatives beyond this are allowed to be more singular (in terms of powers of $\e$). 
This follows from the proof further down straightforwardly, but we do not make this weakening of the assumptions here for the sake of brevity. }

\beq
\label{tg.bds}
E[ \partial^{k_1}_{t\dots t}\partial^{k_2}_{\theta\dots\theta} 
[\gamma_{\rm init}-\gamma_{\rm S}]]\le  \eta^2 \e^{-3},
\eeq
and also for the function $\gamma_{\rm init}$ itself we require: 

\beq
\label{gamma.init.bds}
{\int_{-\infty}^{+\infty}\int_0^\pi}|\partial^{k_1}_{t,\dots t}\partial^{k_2}_{\theta\dots\theta} 
[\gamma_{\rm init}-\gamma_{\rm S}]|^2 \sin\theta d\theta dt \le \eta^2 (\log \e)^2.
\eeq

For the remaining two non-zero components
 of the metric $\g$ on the initial data set we assume that the 
components $\g_{tt}, \g_{\theta\theta}$ satisfy the bounds, for all $k_1+k_2\le s$:
\beq
\label{g.bds}
\begin{split}
{\int_{-\infty}^{+\infty}\int_0^\pi |\partial^{k_1}_{t\dots t}\partial^{k_2}_{\theta\dots\theta} 
[\log(\frac{\g_{tt}}{(\g_{\rm S})_{tt}})]|^2\sin\theta  dt 
d\theta\le \eta^2 (\log\e)^2, }\\
{\int_{-\infty}^{+\infty}\int_0^\pi |\partial^{k_1}_{t\dots t}\partial^{k_2}_{\theta\dots\theta} 
[\log(\frac{\g_{\theta\theta}}{(\g_{\rm S})_{\theta\theta}})]|^2 \sin\theta dt
d\theta\le \eta^2 (\log\e)^2.}
\end{split}
\eeq


Next, we  define 
\beq
\label{basic.frame}
E_1=\g_{tt}^{-1/2} \partial_t, \qquad E_2= \g_{\theta\theta}^{-1/2}
 \partial_\theta
\eeq
and consider the components of $\K$ with respect to this frame,  
$\K_{22},\K_{11}$. We require then for all $k_1+k_2\le {s-2}$: 

\beq
\label{K.bds}
{\int_{-\infty}^{+\infty}\int_0^\pi} |\partial^{k_1}_{t\dots t}\partial^{k_2}_{\theta\dots\theta} 
[\K_{ab}-(\K_{\rm S})_{ab}]|^2 \sin \theta d\theta dt\le  \eta^2 \e^{-3},\qquad (a;b)=(1;1), (2;2). 
\eeq
The four conditions above capture the $\eta$-closeness of our initial
data to $(\g_{\rm S}, \K_{\rm S})$ on $\{r=\e\}$.

\subsection{The result, properly formulated.}

The space-times $(M, g^{3+1})$ that we construct will be considered  both in terms 
of coordinates (and the metric components expressed in terms of these coordinates), but also 
in terms of connection coefficients of certain special frames. 

We present a more descriptive version of our result in coordinates here: 

 Consider the coordinates $\phi\in [0,2\pi)_,t\in (-\infty,\infty),\theta\in (0,\pi)$
  constructed on our initial data $(\Sigma,\g,\K)$ above. Our maximal future hyperbolic 
  development will involve a fourth (time) coordinate $r$:

  The future maximal hyperbolic development will live over a domain:
  \[
\{  \phi\in [0,2\pi)_,t\in (-\infty,\infty),\theta\in (0,\pi), 
r\in (0,r_*(t,\theta))\},
  \] 
  where the function $r_*(t,\theta)$ is one of the parameters that will 
  be solved for in 
  the problem. The abstract initial data $(\g,\K)$ are induced by 
  ${g}$ 
  on the hypersurface $\Sigma_{r_*(t,\theta)}:=\{r=r_*(t,\theta)\}$. In 
  particular, 
  the restriction of the metric ${g}$ to $\Sigma_{r_*}$,
  \emph{expressed in these
   same coordinates} $t,\theta,\phi$, will \emph{match exactly} the 
   prescribed initial 
   metric $\g$. In other words, the metric $g|_{\Sigma_{r_*}}$ 
   expressed in 
   these coordinates is assumed to be  \emph{equal} (not just up to a coordinate 
   transformation) to our 
   prescribed  $\g$.

  \begin{theorem}
\label{thm_strict}  
  
  Consider an abstract initial data set $(\Sigma,\g,\K)$ which is an 
  $\eta$-perturbation of the Schwarzschild space-time data on
   $\Sigma_\e=\{r=\e\}$, as defined in \eqref{K.bds}, \eqref{g.bds},
   \eqref{tg.bds}. Assume $s\in\mathbb{N}$ is large enough and that $\e,\eta>0$ are 
   small enough. (How large and small these parameters  is 
   derived below).
   
   Then the maximal future  hyperbolic development ${(\mathcal{M}^{1+3},g)}$ 
   of this initial data set 
   can be described as follows: There exists a fourth coordinate $r$
  so that $g$ lives over  
  \[
  \phi\in [0,2\pi)_,t\in (-\infty,\infty),\theta\in (0,\pi), 
r\in (0,r_*(t,\theta)),
  \]   
  and {it acquires the form}: 
  \beq
  \label{final.metric.form}
  \begin{split}
&  g= -(\frac{2M}{r}-1)^{-1}dr^2+  g_{\phi\phi}(r,t,\theta)d\phi^2+ 
  g_{\theta\theta}(r,t,\theta)d\theta^2+
g_{tt}(r,t,\theta) dt^2  
  \\&   +g_{t\theta}(r,t,\theta) dtd\theta
  +g_{r\theta}(r,t,\theta)drd\theta+g_{rt}(r,t,\theta) drdt. 
   \end{split}
  \eeq
  Here the integral curves of $\partial_r$ are time-like geodesics. 
Also the abstract initial data $(\Sigma,\g,\K)$ are induced by $g$ onto 
$\Sigma_{r_*}:=\{r=r_*(t,\theta)\}$. 
\medskip

The metric $g$  exists as an $H^s$-smooth Lorenzian metric until 
$\{r=0\}$.  The asymptotic expansion of the components in the 
${\cal C}^k$, $k\le {\rm low}-2$ (for some ${\rm low}<s$ to be specified below)  norms are as follows:  
There will exist a change of coordinates 
  $(T, \Theta)=(T(t,\theta),\theta)$ 
    so that  with respect to the coordinate system 
  \[
  T,\Theta, \phi,r
  \]
  the components of the metric $g$ have  the expansion:
  \beq
  \begin{split}
  \label{g.asymptotics}
&g_{\phi\phi}(r,T,\Theta)=A(T,\Theta)r^{2\alpha(T,\Theta)}
    (1+O(r^{\frac{1}{4}})) \sin^2\Theta,\quad
  g_{\Theta\Theta}(r,T,\Theta)
   = B(T,\Theta) r^{2\beta(T,\Theta)}(1+O(r^{\frac{1}{4}})), 
\\&  g_{TT}= C(T,\Theta) r^{2\delta(T,\Theta)} (2M+O(r^{\frac{1}{4}})), \qquad 
  g_{T\Theta} (r,T,\Theta) ={O( r^{1+\frac{1}{8}}),} 
  \\& g_{T r} = {O(r^{3/8})}, \qquad g_{\Theta r}\equiv0,\;\;r\in(0,\frac{\epsilon}{2}). 
  \end{split}
  \eeq
Here the exponent functions $\alpha(T,\Theta), \beta (T,\Theta), \delta (T,\Theta)$
depend on the point $T, \Theta$; they are all pointwise close to their values for the Schwarzschild metric, in particular: 
\beq
|\alpha(T,\Theta)-1|, |\beta(T,\Theta)-1|, |\delta(T,\Theta)+\frac{1}{2}|\le \frac{1}{8}.
\eeq
The coefficients $A(T,\Theta), B(T,\Theta), C(T,\Theta)$ are also pointwise close to their values for the Schwarzschild 
metric: 

\[
|A(T,\Theta)-1|, |B(T,\Theta)-1|, |C(T,\Theta)-1|
\le \frac{1}{8}.
\]
Moreover all these functions $\alpha,\beta,\delta, A, B, C$ are 
${\cal C}^{{\rm low} -2}$-functions, and satisfy 
similar bounds in the norms ${\cal C}^k, k\le  {\rm low} -2$. 

At the higher norms the behaviour of the metric components is more 
singular;
 however we do not write those bounds  out here. 
 \end{theorem}
 
 \begin{remark}
  The claims we made above are \emph{optimal} for {the leading orders of 
  the first three terms} in \eqref{g.asymptotics}, in the gauge that we consider. 
  For the other terms they are in fact not optimal, yet they are sufficient 
  for our purposes. 
  \end{remark}  

\begin{remark}  
Once our theorem has been proven, we will note in the last appendix that we can express 
the same metric in  a different (still geodesic) gauge 
$\{\tilde{r},\tttt, \ttheta,\phi\}$; \emph{relative to this gauge} the coordinate 
vector fields $\partial_{\tilde{r}},\partial_\tttt,\partial_\ttheta,\partial_\phi$,
capture the \emph{principal} directions of contractions and expansion of the metric $g$ at the singularity; the cross terms in this coordinate system 
satisfy much stronger (optimal) decay properties. 
 \end{remark}

To prove our theorem \ref{thm_strict}, we find it more convenient to introduce  frames
 in 
conjunction with coordinates. We will make our claim in terms of connection 
coefficients corresponding to a gauge-normalized 
frame, {see Theorem \ref{thm_REVESNGG} below}. We will then derive Theorem 
\ref{thm_strict} {from Theorem \ref{thm_REVESNGG},} as a consequence 
of estimates derived in  Section \ref{sec_fun.spaces}, {see \S
\ref{subsec:REVESNGGimpl}} in the Appendix.

We introduce 
our  frame and its gauge normalization
(and the corresponding equations that stem from the EVE) in the next
 subsection.

\subsection{The geodesic gauge: Reduction of the EVE to free wave-ODE system}
\subsubsection{The orthonormal frames and their propagation.} 


Given a $g$-orthonormal frame $\{e_a\}^3_0$, $e_3=e^{-\gamma}\partial_\phi$, on a 3-dim space-like hypersurface $\Sigma$ with $e_0$ transversal to $\Sigma$, we may extend the frame off of $\Sigma$ {via the propagation rule \eqref{almpar.trans}.}
{Along a fixed $e_0$ geodesic}, this uniquely determines the orthonormal 
frame, once the frame has been prescribed at one point on the geodesic. 

Choosing $e_3=e^{-\gamma}\partial_\phi$, in which case the equation $D_{e_0}e_3=0$ is 
automatic, where $D$ is the Levi-Civita connection of $g$. We may thus restrict the frame $\{e_0,e_1,e_2\}$ in the $1+2$ projected manifold 
($\mathcal{M}^{1+3}/\mathbb{S}^1,h)$.\\

{\bf Coordinate Normalization:}
We will be expressing the metric $h$ in a system of coordinates $t,\theta,r$. 
The coordinates $t,\theta$ exist on the initial data set and give rise to 
a coordinate system on $(\mathcal{M}^{1+3}/\mathbb{S}^1,h)$ as follows:

\begin{itemize}
\item The coordinates $t,\theta$ are required to satisfy 
\beq
\label{e0te0theta} 
e_0(t)=e_0(\theta)=0.
\eeq

 \item The  coordinate function $r$ satisfies
   \begin{align}\label{e0}
e_0=-(\frac{2M}{r}-1)^\frac{1}{2}\partial_r,
\end{align}

and is  normalized so that 
$r=0$ on the 
singularity.

\end{itemize}
Some key  non-trivial connection coefficients of $h$ are defined via: 
\beq
\label{conn.def}
K_{ij}:=h(\nabla_{e_i}e_0,e_j)=K_{ji}.
\eeq
These connection coefficients must satisfy the system \eqref{finredEVERic11pre}, \eqref{finredEVERic12pre}, \eqref{finredEVERic22pre}; 
the $\gamma$
on the RHS solves the wave equation \eqref{redEVEwav}.

Now, the rest of the $(2+1)$-metric $h$ will be captured in coordinates:

\subsection{The space-time metric, expressed in terms of the 
orthonormal frame.}

To complete our set of unknowns, we will be fixing  a system of coordinates 
$(\rho, t,\theta)$, where $\rho$ is a re-parametrization of $r$, defined by the 
equation: 

\begin{align}
{\rho}={\rho}(r{,t,\theta})= r-\chi(r)({r}_*-\epsilon),&&\chi\in C^\infty([0,2\epsilon]),\;\;\chi\big|_{[0,\frac{\epsilon}{2}]\cup[\frac{3\epsilon}{2},2\epsilon]}\equiv0,\;\chi\big|_{[\frac{3\epsilon}{4},\frac{5\epsilon}{4}]}\equiv1;
\end{align}
(note in particular that $\{r=r_*(t, \theta)\}$ corresponds to $\{\rho=\e\}$). 

The three functions $\rho,t,\theta$ define a system of coordinates; for this section 
 $\partial_\rho, \partial_t,\partial_\theta$ will be the coordinate 
 vector fields for this system of coordinates.

 We will seek to express the space-time metric (with respect to this system of 
 coordinates) in terms of the frame 
 $(e_0,e_1,e_2,e_3)$ 
 constructed above. First, let us introduce  a modification of $e_1, e_2$ into 
 vector fields $\overline{e}_1,\overline{e}_2$ which are tangent to the level sets of $\rho$:

	\begin{align}\label{eibar.pre}
	\begin{split}
	\overline{e}_i:=&\,e_i-e_i(\rho)
	\partial_{\rho}\\
	=&\,e_i+e_i(\rho)
	(\frac{2M}{r}-1)^{-\frac{1}{2}}[1-\partial_r\chi(r)
	(r_*-\epsilon)]^{-1}e_0\in T\Sigma_{\rho},
	\end{split}
	\end{align}
	for $i=1,2$.

  Now, the scalars that will connect our space-time metric $h$ in terms of 
  $e_0, e_1, e_2$ are the scalars that define the coordinate vector fields 
  $\partial_t,\partial_\theta$ as linear combinations of 
  $\overline{e}_1,\overline{e}_2$. 
  
  In particular we define the functions $a_{Ai}(\rho,t,\theta)$, $A=t,\theta, i=1,2$ via the equations:

\begin{align}\label{tthetatransebar.pre}
\partial_t=a_{t1}\overline{e}_1+a_{t2}
\overline{e}_2,\qquad\partial_\theta= a_{\theta 1}
\overline{e}_1+a_{\theta 2}\overline{e}_2.
\end{align}

We also remark how the scalar-valued functions $a_{Ai}$ along with the scalar functions $e_i(\rho)$ determine the metric $h$ (the $(2+1)$-part of $g$):
\beq
\label{g.from.a.pre} 
\begin{split}
& h_{\theta\theta}  =  \sum_{i=1,2}   [a_{\theta i} ]^2 \cdot\bigg{[}1- [e_i(\rho) (\frac{2M}{r}-1)^{-1}
[1-\partial_r\chi(r)(r_*-\e)]^{-1}]^2\bigg]
\\&-2a_{\theta 1}a_{\theta 2} e_1(\rho)e_2(\rho)(\frac{2M}{r}-1)^{-1/2}[1-\partial_r\chi(r)(r_*-\e)]^{-2}, 
\\&h_{tt}=\sum_{i=1,2}[a_{ t i}]^2 \cdot\bigg{[}1- [e_i(\rho) (\frac{2M}{r}-1)^{-1}
[1-\partial_r\chi(r)(r_*-\e)]^{-1}]^2\bigg{]}
\\&-2a_{t 1}a_{t 2} e_1(\rho)e_2(\rho)(\frac{2M}{r}-1)^{-1}[1-\partial_r\chi(r)(r_*-\e)]^{-2}, 
\\&h_{t\theta}=\sum_{i=1,2}  [a_{ \theta i}\cdot a_{ t i} ] \cdot\bigg{[} 1- [e_i(\rho) (\frac{2M}{r}-1)^{-1/2}
[1-\partial_r\chi(r)(r_*-\e)]^{-1}]^2\bigg{]}
\\&-[a_{\theta_1}a_{t_2}+a_{t 1}a_{\theta 2}] e_1(\rho)e_2(\rho)
(\frac{2M}{r}-1)^{-1}[1-\partial_r\chi(r)(r_*-\e)]^{-2},
\\&h_{\rho \theta}=\sum_{i=1,2}(\frac{2M}{r}-1)^{-1}
a_{\theta i}e_i(\rho)\cdot [1-\partial_r\chi(r)(r_*-\e)]^{-2}, 
\\&h_{\rho t}=\sum_{i=1,2}(\frac{2M}{r}-1)^{-1}
a_{t i}e_i(\rho)\cdot [1-\partial_r\chi(r)(r_*-\e)]^{-2}.
\end{split}
\eeq

  We will also note an evolution equation on the functions $a_{Ai}(\rho,t,\theta)$, which is derived
  below, in the iterative step, just above \eqref{e0.a}. 
  
\begin{align}\label{e0.a.pre}
	\begin{split}
	e_0a_{t1}-K_{11}a_{t1}=0,\qquad e_0a_{t2}-K_{22}a_{t2}=2K_{12}a_{t1},\\
	e_0a_{\theta 1}-K_{11}a_{\theta 1}=0,\qquad e_0a_{\theta 2}-K_{22}a_{\theta 2}=2K_{12}a_{\theta 1}.
	\end{split}
\end{align}

\subsubsection{Gauge fixing for the frame $e_0, e_1, e_2$, relative to 
the 
singularity and the initial data hypersurface:}
  
First, we specify a parameter $r$ along each integral
 curve of $e_0$ via the condition \eqref{e0}. This parameter commences at 
 $r=0$, for each of the 
 integral curves of $e_0$. The family of integral curves (geodesics) 
 itself is parametrized by two functions 
 $t,\theta$  with $t\in (-\infty,+\infty)$ and ${\theta\in (0,\pi)}$. We 
 denote these
  geodesics by $l_{t,\theta}$; considering them as parametrized curves 
  with 
  parameter $r$ we  denote them by $l_{t,\theta}(r)$. We will be 
  requiring 
  that any two different 
   geodesic segments $l_{t,\theta}(r),r\in (0,2\epsilon]$, for some 
  $\epsilon>0$ small enough that will be fixed later, do not intersect.  
  (This will turn out to hold for all the space-times we construct, and 
  thus, $r,t,\theta$ and $\phi$ define  a system of coordinates for our 
  space-time).

 As we will remark below, imposing this condition is in fact 
 a partial  gauge normalization of the affine vector fields $e_0$, 
 in that they do not form focal points \emph{before} the singularity;
equivalently, there is no break-down of the geodesic gauge \emph{prior} to
 the singularity. Concretely, we require that  
all connection coefficients 
$K_{11}(r,t,\theta), K_{12}(r,t,\theta), K_{22}(r,t,\theta)$ are smooth up to $r=0$, and that 
  $e_1, e_2$ diagonalize 
 $K_{ij}(r,t,\theta)$ \emph{asymptotically} as $r\to 0$; in particular: 
 \beq
 \label{distinctK}
 K_{12}(r,t,\theta)\cdot K_{11}^{-1}(r,t,\theta)\to 0,\qquad K_{12}(r,t,\theta)
 \cdot K_{22}^{-1}(r,t,\theta)\to 0
 \eeq
  as $r\to 0$. The choice of $e_1, e_2$ is finally fixed by requiring that:
  
  \beq\label{Ksigns} K_{11}>0,\qquad
   K_{22}<0.  
   \eeq
\medskip

 Beyond these normalizations of the connection coefficients $K_{ij}$, we impose certain normalizations to the
  frame elements $e_2, e_1$ themselves: 
 
We impose that   $e_2$ is asymptotically tangent to the singularity,
as $r\to 0$, along any of the curves $l_{t,\theta}(r)$. In particular, we require (recall the definition \eqref{d1.d2} of $d_2(t,\theta)$): 
\beq
\label{gaugenorme2}
e_2(r)=o(r^{-\frac{1}{2}+d_2(t,\theta)}).
\eeq
In conjunction with the Riccati equations and \eqref{almpar.trans}, we will see that this implies that $e_2(r)=0$ on the entire space-time that we construct. (This is equivalent to 
 requiring that $e_2$ is tangent to all level sets of $r$).

Moreover we require that  the vector field $e_1$  be tangent to the entire hypersurface 
$\Sigma_{r_*}$, on which the initial data will be induced (how  this hypersurface  is found is made precise in the next subsection).
We note that these normalizations, along with the requirement that $e_0\perp {\rm Span}\langle e_1, e_2\rangle$, uniquely  specifies the locations  of the geodesics $l_{t,\theta}$ 
in the solved-for space-time.  
 \medskip

  \begin{figure}[h]
	\centering
	\def\svgwidth{9cm}
	\includegraphics[scale=1.2]{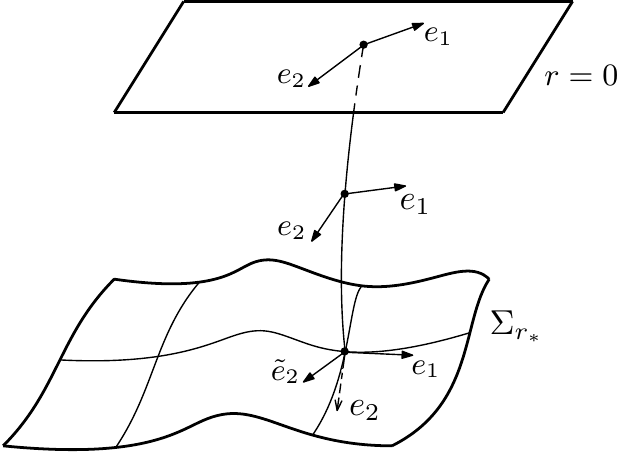}
	\caption{The frame $e_1,e_2$ normalized so that $e_2$ is ``tangent to the singularity'' and $e_1$ is tangent to the initial data hypersurface.} \label{Pict1}
\end{figure}
 
 We will see that 
 \eqref{distinctK} implicitly imposes initial conditions on 
 $K_{12}, K_{22}$ \emph{on} the singularity; as explained, these conditions should be thought of purely a 
 gauge-fixing requirements. Now, the \emph{rest} of the parameters need to 
 be prescribed on some hypersurface $\Sigma_{r_*}$, which is meant to carry the abstract initial data
$(\g,\K)$. Moreover the \emph{position} of this hypersurface (expressed
  graphically in the coordinates $r,t,\theta$ via the function 
  $r_*(t,\theta)$)
 is also to be determined. 

We discuss this next:
 \medskip
 
 \subsection{The matching of the prescribed initial data. }
 
 The prescribed initial data $(\Sigma,\g, \K)$ in our problem must be  \emph{induced} on 
 a hypersurface
  \beq
  \label{Sigmar*}
\Sigma_{r_*}:=\{r=r_*(t,\theta)\},
\eeq
for some unique function $r_*(t,\theta)$. 
 
 \subsubsection*{The induced connection on $\Sigma_{r_*}$.}
 
To make  this requirement precise, 
we firstly identify a canonical rotation of the frame $e_0, e_2$ to yield a new 
orthonormal 
frame $\te_0, \te_2$ that will be \emph{adapted} to the 
hypersurface $\Sigma_{r_*}$, on which the initial data are to live. (Recall that $e_1$ is 
required to be tangent to $\Sigma_{r_*}$).  

\begin{figure}[h]
	\centering
	\def\svgwidth{9cm}
	\includegraphics[scale=1.2]{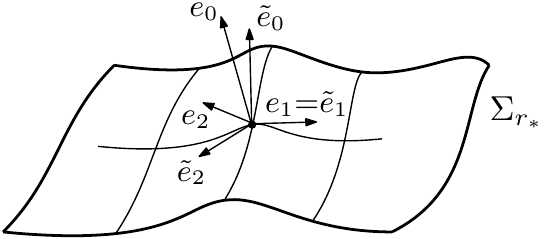}
	\caption{The frame $e_0, e_1, e_2$ and its rotation 
	$\te_0, \te_1=e_1, \te_2$.} 
\end{figure}

\begin{definition}
\label{adapted}
An orthonormal frame ${e}^\sharp_0, {e}^\sharp_1, {e}^\sharp_2$, defined 
over a 
hypersurface $\Sigma_{r_*}$, is called adapted to the hypersurface iff 
${e}^\sharp_1,{e}^\sharp_2$ are both tangent to the hypersurface (and 
thus 
${e}^\sharp_0$ is normal to the hypersurface). 
\end{definition}

We consider the (unique, within small rotation angles) rotation of the frame 
$(e_0,  e_2)$ to a new frame 
$(\te_0, \te_2)$ that makes $\te_0$ normal to $\Sigma_{r_*}$ and
 $\te_2$ tangent to $\Sigma_{r_*}$
(and we  preserve $e_1$). 
The new frame $\te_0, \te_2$ on $\Sigma_{r_*}$ is given by
the following formulas:

\begin{align}\label{tildee}
{\tilde{e}_2}:=q\,{e_2}-(\frac{2M}{r_*}-1)^{-\frac{1}{2}}({\tilde{e}_2}{r_*})e_0,\qquad
{\tilde{e}_0}=qe_0-(\frac{2M}{r_*}-1)^{-\frac{1}{2}}({\tilde{e}_2}{r_*}){e_2},
\end{align}
where 
 
\begin{align}\label{e2rstar.pre}
q=\sqrt{1+(\frac{2M}{r_*}-1)^{-1}({\tilde{e}_2}r_*)^2}.
\end{align}

Inverting (\ref{tildee}), we obtain:
\begin{align}\label{tildeeinv}
{e_2}=q\,{\tilde{e}_2}+(\frac{2M}{r_*}-1)^{-\frac{1}{2}}({\tilde{e}_2}{r_*}){\tilde{e}_0},&& e_0=q\,{\tilde{e}_0}
+(\frac{2M}{r_*}-1)^{-\frac{1}{2}}({\tilde{e}_2}{r_*}){\tilde{e}_2}.
\end{align}
 The  connection coefficients $\tilde{K}_{ij}:=g({\bf D}_{\tilde{e}_i}\tilde{e}_0,\tilde{e}_j)$, $\tilde{A}_{ij,l}:={\bf g}({\bf D}_{\tilde{e}_i}\tilde{e}_j,\tilde{e}_l)$, for this new frame $\te_0,e_1,\te_2$ on 
 $\Sigma_{r_*}$, (where $\bf D$ stands for the abstract connection of the 2-dimensional metric on the inital data set) 
are related to the connection coefficients
  $K_{ij}(r_*(t,\theta),t,\theta)$ (restricted to 
 $\Sigma_{r_*}$) 
as follows: 
\begin{align}\label{Korthexp.pre}
q^2K_{22}=&\,q\tilde{K}_{22}+(\frac{2M}{{r_*}}-1)^{-\frac{1}{2}}{\tilde{e}_2}{\tilde{e}_2}{r_*}+(\frac{2M}{r_*}-1)^{-\frac{3}{2}}\frac{2M}{r_*^2}({\tilde{e}_2}{r_*})^2,\\
\label{K11exp.pre}
K_{11}=&\,q\tilde{K}_{11}+(\frac{2M}{r_*}-1)^{-\frac{1}{2}}({\tilde{e}_2}{r_*})
\tilde{A}_{11,2},
\\ 
\label{K12exp.pre}
K_{21}=&\,
\tilde{K}_{21}+q^{-1}(\frac{2M}{r_*}-1)^{-\frac{1}{2}}({\tilde{e}_2}{r_*})
\tilde{A}_{22,1}
\end{align}

 Next, the requirement that $\tilde{K}, \tilde{A}$ should ``match'' the 
 prescribed initial data $(\Sigma, \g, \K)$ on 
  $\Sigma_{r_*}$ needs to be imposed.

   This requirement will  implicitly determine the hypersurface 
  $\Sigma_{r_*}$. 
    Before imposing this condition we will study how our initial
     data can be 
    realized with respect to different frames:

\subsubsection*{The abstract initial data realized in different frames. }
 
 First, we recall some standard formulas: 
 Consider the initial data $(\Sigma, \g, \K)$ in terms of the background 
 coordinates 
 $t,\theta$
 and the background frame ${n, E_1, E_2}$ defined in \eqref{basic.frame}.
  This initial data can 
 equivalently be expressed in 
 terms of the connection coefficients 
 $\A_{ij,l}, \K_{ij}, i,j,l\in \{1,2\}$ 
 of the background frame ${n, E_1, E_2}$. Furthermore, we can consider 
 rotations of the frame elements $E_1, E_2$ (tangent to the initial data 
 surface $\Sigma$).
 This will yield a new frame  ${(n}, E^\varphi_1, E^\varphi_2)$, on 
 the initial data hypersurface given by the formulas
\begin{align}\label{framerot.pre}
E_1^\varphi:=\cos\varphi E_1+\sin\varphi E_2,&&E^\varphi_2=-\sin\varphi E_1+\cos\varphi 
E_2.
\end{align}
 We then consider the connection coefficients, denoted by 
 $\A^\varphi, \K^\varphi$ 
 for short, in this new rotated 
 frame. The components of $\K^\varphi$ relative to $E^\varphi_1, E^\varphi_2$
   are given by 
 standard  transformation formulas: 
\begin{align}\label{KEvarphi.pre}
\notag \K^\varphi_{11}=&\,\cos^2\varphi\K_{11}+
\sin^2\varphi \K_{22}+2\sin\varphi\cos\varphi  \K_{12}\\
\K^\varphi_{22}=&\,\sin^2\varphi\K_{11}+
\cos^2\varphi 
\K_{22}-2\sin\varphi\cos\varphi  \K_{12}\\
\notag \K_{12}^\varphi=&\,\K_{12}
+\sin\varphi\cos\varphi
\big[\K_{22}- \K_{11}\big]
\end{align}

On the other hand, the spatial connection coefficients are given by
the following. (Recall that $\K_{12}=0$ by construction). 
\begin{align}
\label{AE112varphi.pre}{\bf A}_{11,2}^\varphi=&\,
\g({\bf D}_{\cos\varphi E_1+\sin\varphi E_2}(\cos
\varphi E_1+\sin\varphi E_2),-\sin\varphi E_1+\cos\varphi E_2)\\
\notag=&\,\cos\varphi\sin^2\varphi (E_1\varphi)+\cos^3\varphi
{\bf A}_{11,2}+\cos^3\varphi (E_1\varphi)
-\cos\varphi\sin^2\varphi {\bf A}_{12,1}\\
\notag&+\sin^3\varphi (E_2\varphi)+\sin\varphi\cos^2\varphi 
{\bf A}_{21,2}
+\sin\varphi\cos^2\varphi (E_2\varphi)-\sin^3\varphi {\bf A}_{22,1}\\
\notag=&\,\cos\varphi(E_1\varphi)+\sin\varphi (E_2\varphi)
+\cos\varphi {\bf A}_{11,2}-\sin\varphi {\bf A}_{22,1}\\
\notag=& E_1^\varphi (\varphi)+\cos\varphi {\bf A}_{11,2}-\sin
\varphi {\bf A}_{22,1},\\
\label{AE221varphi.pre}{\bf A}_{22,1}^\varphi=&\,\g
({\bf D}_{-\sin\varphi E_1+\cos\varphi E_2}(-\sin
\varphi E_1+
\cos\varphi E_2),\cos\varphi E_1+\sin\varphi E_2)\\
\notag=&\sin\varphi (E_1\varphi)-\cos\varphi(E_2\varphi)
+\cos\varphi {\bf A}_{22,1}+\sin\varphi {\bf A}_{11,2}\\
\notag=& -E_2^\varphi (\varphi)+\cos\varphi {\bf A}_{22,1}+\sin
\varphi {\bf A}_{11,2}.
\end{align}

In view of these transformation laws,  we now define:

\begin{definition}
\label{equiv.metr}
Consider a symmetric 2x2-matrix valued function $\tilde{K}_{ij}(t,\theta)$ and a 3x2-matrix valued 
function $\tilde{A}_{ij,k}(t,\theta)$, $i,j,k\in \{1,2\}$. 

We say that these matrix-valued data agree with the prescribed initial data $(\Sigma, \A, \K)$ up to a gauge 
transformation, provided there exists a function $\varphi(t,\theta)$ so that:
\begin{align}\label{KEvarphi'}
\notag \tilde{K}_{11}=&\,\cos^2\varphi\K_{11}+\sin^2\varphi \K_{2 2}+2\sin\varphi\cos\varphi  \K_{1 2}\\
\tilde{K}_{22}=&\,\sin^2\varphi\K_{1 1}+\cos^2\varphi \K_{2 2}-2\sin\varphi\cos\varphi  \K_{1 2}\\
\notag \tilde{K}_{12}=&\,\K_{1 2}+\sin\varphi\cos\varphi\big[\K_{2 2}- \K_{1 1}\big]
\end{align}
(in the last equation we recall that $\K_{12}=0$, yet we include it for completeness), 
and for the spatial connection coefficients:
\begin{align}\label{AEiijvarphi'}
\begin{split}
\tilde{A}_{11,2}=&\,{\tilde{e}_1(\varphi)}
+\cos\varphi\A_{11,2}-\sin\varphi\A_{22,1}\\
\tilde{A}_{22,1}=&{-\tilde{e}_2 (\varphi)}
+\cos\varphi\A_{22,1}+\sin\varphi\A_{11,2}
\end{split}
\end{align}
\end{definition}

\begin{figure}[h]
	\centering
	\def\svgwidth{9cm}
	\includegraphics[scale=1.2]{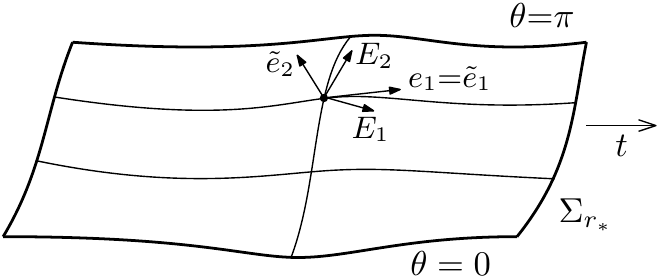}
	\caption{The adapted frame $(\te_1,\te_2)$ as a rotation of the fixed 
	background frame $(E_1, E_2)$.} \label{Pict3}
\end{figure}

 We also make a remark for future reference: 
\begin{remark} 
 The value of $\varphi(t,\theta)$ is (uniquely, up to adding an integer multiple of $\pi$) 
  fixed by the value of the tensor  
  $\tilde{K}_{12}(t,\theta)=[\tilde{K}(\te_1,\te_2)](t,\theta)$, {see \eqref{KEvarphi'}}.
  In particular: 
\beq
\label{tilF.act}
\varphi(t,\theta)={\frac{1}{2}
\sin^{-1}(\frac{2\tilde{K}_{12}}{\K_{22}-\K_{11}}).}
\eeq
\end{remark}
(Recall that $\K_{22}-\K_{11}$ is a \emph{fixed, smooth} function, 
which is fully determined by our initial data).

Note that if we consider connection coefficients $\tilde{A}, \tilde{K}$ 
solving 
\eqref{KEvarphi'} and \eqref{AEiijvarphi'} for some function 
$\varphi(t,\theta)$, then 
the metric $g$ induced by $\tilde{A}_{ij,k}$ is identical to the prescribed 
$\g$ in the $t,\theta$
coordinates. Also, the prescribed second fundamental form $\tilde{K}$ is 
the same 
(as a tensor) with the prescribed second fundamental form 
${\bf K}$ of our problem. In particular, the value of the component 
$\tilde{K}_{12}(t,\theta)=\K(E_1^\varphi,E_2^\varphi)$ uniquely forces the
 values of the connection coefficients 
 $\tilde{K}_{11}(t,\theta), \tilde{K}_{22}(t,\theta)$, 
 $\tilde{A}_{11,2}(t,\theta),\tilde{A}_{22,1}(t,\theta)$ of our initial
  data, relative to the frame $E^\varphi_1, E^\varphi_2$. 
\medskip


 

\medskip

 In other words, for any trace of the form \eqref{framerot.pre}
 in our abstract initial data then the value of the component 
 $\tilde{K}_{12}(t,\theta)=[\K(\te_1,\te_2)](t,\theta)$ uniquely specifies the gauge-rotation angle $\varphi(t,\theta)$. 
 Therefore, all other components of the initial data, relative to the frame $\te_1,\te_2, \te_0$, should be expressible in terms of 
 $\tilde{K}_{12}(t,\theta)=[\K(\te_1,\te_2)](t,\theta)$. 

We obtain these relations in the next subsection.

\subsubsection*{Relations between geometric quantities on the initial 
data set.}

Assume that we have 
a pair of matrix-valued functions 
 $[\tilde{K}_{ij}](t,\theta), [\tilde{A}_{ij,l}](t,\theta)$, which 
 matches the prescribed initial data $(\Sigma, \g, \K)$ up to a gauge 
 transformation, 
 in the sense of definition \ref{equiv.metr}.

In particular,
 $[\tilde{K}_{ij}](t,\theta)$ equals 
$[\K^\varphi_{ij}](t,\theta)$ for some (apriori not 
specified) function 
$\varphi(t,\theta)$. 

Then there exists a  \emph{fixed} function 
$F^{11}_{t,\theta}(\cdot):(-\delta,\delta)\to \mathbb{R}$, $\delta>0$, so that:
\beq
\label{G.def}
\tilde{K}_{11}(t,\theta)= \K^\varphi_{11}(t,\theta)= 
F^{11}_{t,\theta}\big{(} \K^\varphi_{12}(t,\theta)\big{)}=F^{11}_{t,\theta}\big(\tilde{K}_{12}(t,\theta)\big).
\eeq
There is also another fixed function 
 $F^{22}_{t,\theta}:(-\delta,\delta)\to\mathbb{R}$,
 so that: 
 \beq
\label{tG.def}
\tilde{K}_{22}(t,\theta)=\K^\varphi_{22}(t,\theta)= 
F^{22}_{t,\theta}\big{(} 
\K^\varphi_{12}(t,\theta)\big{)}=F^{22}_{t,\theta}\big(\tilde{K}_{12}(t,\theta)\big).
\eeq
These functions can  in fact be calculated explicitly, using the formulas 
\eqref{KEvarphi'},\eqref{tilF.act} and the trigonometric identities $\cos^2\varphi=\frac{1}{2}(1+\cos 2\varphi)$, $\sin^2\varphi=\frac{1}{2}(1-\cos2\varphi)$ to find: 
\begin{align}
\label{F11.form}
F^{11}_{t,\theta}\big{(} \tilde{K}_{12}(t,\theta)\big{)}={\frac{1}{2}\bigg[1+\sqrt{1-\frac{4\tilde{K}_{12}^2}{({\bf K}_{22}-{\bf K_{11}})^2}}\bigg]}\K_{11}+
 \frac{1}{2}\bigg[1-\sqrt{1-\frac{4\tilde{K}_{12}^2}{({\bf K}_{22}-{\bf K_{11}})^2}}\bigg]\K_{22},\\
\label{F22.form}
F^{22}_{t,\theta}\big{(} \tilde{K}_{12}(t,\theta)\big{)}=
\frac{1}{2}\bigg[1+\sqrt{1-\frac{4\tilde{K}_{12}^2}{({\bf K}_{22}-{\bf K_{11}})^2}}\bigg]\K_{22}+
 \frac{1}{2}\bigg[1-\sqrt{1-\frac{4\tilde{K}_{12}^2}{({\bf K}_{22}-{\bf K_{11}})^2}}\bigg]\K_{11}.
\end{align}

In fact, the converse is also seen to be true: Assuming that a symmetric 
tensor $\tilde{K}_{ij}(t,\theta)$ satisfies the properties \eqref{F11.form}, \eqref{F22.form},
then it agrees with the prescribed second fundamental form ${\bf K}$ up to a gauge 
transformation, encoded in a function $\varphi(t,\theta)$. Moreover, that 
gauge function $\varphi(t,\theta)$ can be determined from the component 
$\tilde{K}_{12}(t,\theta)$ via the formula \eqref{tilF.act}. 
\medskip

Using formulas \eqref{AE112varphi.pre}, \eqref{AE221varphi.pre} and \eqref{tilF.act}
 we then observe that $\A^\varphi_{11,2}$, $\A^\varphi_{22,1}$ can also be expressed 
 in terms of $\tilde{K}_{12}$ via explicit formulas. 

Thus, if a pair $\tilde{A}_{11,2}, \tilde{A}_{22,1}$
arises via \eqref{AEiijvarphi'} from the background frame 
via  
a rotation by $\varphi(t,\theta)$ (and $\varphi(t,\theta)$ is given from 
$\tilde{K}_{12}(t,\theta)$ via \eqref{tilF.act}), then  
$\tilde{A}_{11,2}, \tilde{A}_{22,1}$ can also be expressed via the formulas: 

\beq
\label{tAs.def}
\begin{split}
&\tilde{A}_{11,2}= 
{\bigg(1-\frac{4\tilde{K}_{12}^2}{({\bf K}_{22}-{\bf K_{11}})^2}\bigg)^{-\frac{1}{2}}}\te_1(\frac{\tilde{K}_{12}}{\K_{22}-\K_{11}})
+F_1(\frac{\tilde{K}_{12}}{\K_{22}-\K_{11}}),
\\&\tilde{A}_{22,1}=\bigg(1-\frac{4\tilde{K}_{12}^2}{({\bf K}_{22}-{\bf K_{11}})^2}\bigg)^{-\frac{1}{2}}
 \te_2(\frac{\tilde{K}_{12}}{\K_{22}-\K_{11}})
+F_2(\frac{\tilde{K}_{12}}{\K_{22}-\K_{11}}).
\end{split}
\eeq
Here the functions $F_1, F_2{:(-\delta,\delta)\to\mathbb{R}}$ are explicit smooth functions that depend only on the 
prescribed 
initial data:
\begin{align}
\label{F1.form}
F_1(x)
=&\,{\frac{1}{2}\sqrt{1+\sqrt{1-4x^2}}}
\A_{11,2}-{\frac{{\rm sign}(x)}{2}\sqrt{1-\sqrt{1-4x^2}}}\A_{22,1},\\
\label{F2.form}
F_2(x)
=&\,\frac{1}{2}\sqrt{1+\sqrt{1-4x^2}}
\A_{22,1}+\frac{{\rm sign}(x)}{2}\sqrt{1-\sqrt{1-4x^2}}\A_{11,2}.
\end{align}

We also note that the vector fields  $\te_1, \te_2$ on the hypersurface $\Sigma_{r_*}$ can also be expressed in terms of the fixed background coordinates 
$\partial_t,\partial_\theta$, with coefficients that are determined by the value of 
$\tilde{K}_{12}$. 

This follows merely {from our choice \eqref{basic.frame} of the initial abstract frame $E_1,E_2$}, along with the formula \eqref{framerot.pre}:
\begin{align}\label{L.rel}
\begin{split}
&\te_1=E_1^\varphi=\cos\varphi (\g_{tt})^{-1/2}\partial_t +\sin\varphi
 (\g_{\theta\theta})^{-1/2}\partial_\theta,\\
&\te_2=E_2^\varphi= \cos\varphi(\g_{\theta\theta})^{-1/2}\partial_\theta-\sin\varphi
(\g_{tt})^{-1/2} \partial_t, 
\end{split}
\end{align}
which after replacing $\varphi$ in favour of $\tilde{K}_{12}$, via \eqref{tilF.act}, yields the formulas:
   \beq
   \begin{split}
   \label{frames.coords}
& \te_1={\frac{1}{2}\sqrt{1+\sqrt{1-4x^2}}}
(\g_{tt})^{-1/2} \partial_t+ {\frac{{\rm sign}(x)}{2}\sqrt{1-\sqrt{1-4x^2}}}
(\g_{\theta\theta})^{-1/2}\partial_\theta,
\\
& \te_2= \frac{1}{2}\sqrt{1+\sqrt{1-4x^2}}
(\g_{\theta\theta})^{-1/2}\partial_\theta-\frac{{\rm sign}(x)}{2}\sqrt{1-\sqrt{1-4x^2}}
(\g_{tt})^{-1/2} \partial_t,
 \end{split}
   \eeq
where $x=\tilde{K}_{12}({\bf K}_{22}-{\bf K}_{11})^{-1}$.

For future reference, let us 
note here that  the values of 
${a_{ti},a_{\theta i}}$ defined in \eqref{tthetatransebar.pre}
on 
$\Sigma_{r_*}$
are precisely determined from the coefficients that appear in \eqref{frames.coords}
(see also  \eqref{te1m.tK12}{-\eqref{te1m.tK12inv}}).

 \beq
   \begin{split}
   \label{a.init}
& a_{t1}(r_*(t,\theta), t,\theta)= \frac{1}{2}\sqrt{1+\sqrt{1-4(\frac{\tilde{K}_{12}}{{\bf K}_{22}-{\bf K}_{11}})^2}}(\g_{tt})^{1/2},
\\&a_{t2}(r_*(t,\theta), t,\theta)=-\frac{{\rm sign}(\frac{\tilde{K}_{12}}{{\bf K}_{22}-{\bf K}_{11}})}{2}
\sqrt{1-\sqrt{1-4(\frac{\tilde{K}_{12}}{{\bf K}_{22}-{\bf K}_{11}})^2}}(\g_{tt})^{1/2},
\\&  a_{\theta 2}(r_*(t,\theta), t,\theta)= \frac{1}{2}\sqrt{1+\sqrt{1-4(\frac{\tilde{K}_{12}}{{\bf K}_{22}-{\bf K}_{11}})^2}}(\g_{\theta\theta})^{1/2},
\\& a_{\theta 1}(r_*(t,\theta), t,\theta)=\frac{{\rm sign}(\frac{\tilde{K}_{12}}{{\bf K}_{22}-{\bf K}_{11}})}{2}\sqrt{1-
\sqrt{1-4(\frac{\tilde{K}_{12}}{{\bf K}_{22}-{\bf K}_{11}})^2}} (\g_{\theta\theta})^{1/2}. 
 \end{split}
   \eeq


This leads to a notion of solutions to the system of the equations 
\eqref{finredEVERic11pre}, \eqref{finredEVERic22pre}, 
\eqref{finredEVERic12pre}, and \eqref{e0.a.pre}
 capturing the 
prescribed initial data on some hypersurface:

\subsubsection{The system for the initial data}

 \begin{definition}
 \label{met.induced}
 We say that a solution 
 {$K_{12}(r,t,\theta), K_{22}(r,t,\theta), K_{11}(r,t,\theta)$} to the set of equations 
 \eqref{finredEVERic22pre},  
 \eqref{finredEVERic12pre}  \eqref{finredEVERic11pre},
\eqref{e0.a.pre}
 in the gauge introduced 
  above, captures our prescribed initial data $(\Sigma, \g, \K)$ 
  on some  hypersurface $\Sigma_{r_*}$, as in \eqref{Sigmar*}, provided:
 
In addition to the function $r_*(t,\theta)$ there exists a function 
$\tilde{K}_{12}(t,\theta) $  
 so that:
 
 \begin{itemize}
 \item  if we define the functions
  $\tilde{K}_{22}(t,\theta), \tilde{A}_{22,1}(t,\theta)$, 
  $\tilde{A}_{11,2}(t,\theta)$ 
 on $\Sigma_{r_*}$ 
via the formulas \eqref{G.def}, \eqref{tG.def}, 
\eqref{tAs.def},  then  on $\Sigma_{r_*}$ 
 the formulas \eqref{Korthexp.pre},  \eqref{K12exp.pre}, \eqref{K11exp.pre}
hold, 

\item The coefficients $a_{Ai}(r_*(t,\theta),t,\theta), A=t,\theta$, $i=1,2$  
\emph{on} $\Sigma_{r_*}$ satisfy the relations \eqref{a.init}. 
\end{itemize}
\end{definition}
The above requirements ensure that the first and second fundamental forms induced by 
our solution to the system  \eqref{finredEVERic22pre},  
 \eqref{finredEVERic12pre}  \eqref{finredEVERic11pre},
\eqref{e0.a.pre} onto $\Sigma_{r_*}$ agree with the prescribed 
 initial data 
 $(\Sigma,{\bf g}, {\bf K})$ up to a gauge transformation, as defined in 
 Definition \ref{equiv.metr}. The gauge transformation is captured precisely in $\tilde{K}_{12}(t,\theta)$, via \eqref{tilF.act}.
\medskip

\subsection{The reduced Einstein equations in geodesic gauge, normalized at the singularity. } 
\label{sec:system} 

 What we have studied so far is 
 a solution $g$ of the EVE under polarized axial symmetry with abstract initial conditions 
$(\Sigma, {\bf g}, {\bf K})$, expressed in a special geodesic gauge. This gauge exists provided the 
space-time 
 admits a non-singular congruence of time-like 
geodesics, which
 emanate from the singularity at $\{r=0\}$, are normal to the collapsing direction 
 $e_2$ on the singularity 
 (in the sense that $e_2(r)=0$), and normal to the hypersurface $\Sigma_{r_*}$
 (on which the abstract initial data live)  in the direction $e_1\perp{\rm Span}\langle e_2, e_0\rangle$. 

Such a space-time 
 yields a solution to the equation 
 \[
 \square_g \gamma=0, 
 \]
 along with a system of 
 transport equations 
 in the connection and coordinates-to-frame  
parameters $$K_{22}, K_{11}, K_{12}, a_{\theta 1}, a_{\theta 2}, a_{t1}, a_{t2}$$
 which are 
functions in $r,t,\theta$. These functions satisfy initial conditions either 
 at $r=0$ ($K_{12}, K_{22}$ satisfy conditions there), 
 or at $r=r_*(t,\theta)$ (all the rest of the parameters satisfy conditions there
  in terms of 
 $\te_2(r_*),  \tilde{K}_{12}$). The initial data  at $\{r=r_*(t,\theta)\}$ for $\gamma$ are given by $\frac{1}{2}\log ({\bf g}_{\phi\phi}),{\bf K}_{33}$.\\ 
 On the other hand,
 the initial data  at $\{r=r_*(t,\theta)\}$ {for the variables $K_{11}, a_{Ai}$, satisfying the evolution equations} \eqref{finredEVERic11pre}, 
\eqref{e0.a.pre}, {are prescribed via the relations
\eqref{K11exp.pre}, \eqref{a.init},  through the explicit formulas \eqref{G.def}, \eqref{F11.form}, \eqref{a.init} (replacing $\varphi$ by $\tilde{K}_{12}$ via \eqref{tilF.act}).}

Finally, the equations that determine the values of $r_*(t,\theta)$ 
(that defines the 
hypersurface 
$\Sigma_{r_*}$ 
on which the initial data are induced) and of 
$\tilde{K}_{12}(t,\theta)$ 
(which determines the gauge parameter $\varphi(t,\theta)$ on 
$\Sigma_{r_*}$) are
 \eqref{Korthexp.pre}, \eqref{K12exp.pre}, coupled to \eqref{tG.def}--the latter being coupled to  \eqref{F22.form}. 
 The initial data for $K_{22}, K_{12}$
 have been fixed 
at $r=0$, thus, the values of $K_{22}(r,t,\theta), K_{12}(r,t,\theta)$
are in principle determined by $\gamma$ alone, via the Riccati
 equations \eqref{finredEVERic22pre}, \eqref{finredEVERic12pre}.
Therefore, the system of equations \eqref{Korthexp.pre}, \eqref{K12exp.pre}, with these 
substitutions of terms, becomes a 2x2 system on the unknowns 
$r_*(t,\theta), \tilde{K}_{12}(t,\theta)$, if we could treat the RHSs of the Riccati 
equations \eqref{finredEVERic22pre}, \eqref{finredEVERic12pre} as ``given''.  

For any solution of this 2x2 system in the (2+1)-metric $h$ to yield a smooth 
hypersurface in the (3+1)-dimensional picture we note that the condition $\te_2(r_*)=0$ 
must be imposed at the poles $\theta=0,\theta=\pi$.

Furthermore, if  this 2x2 
system could be solved separately (equivalently, if the
 functions $r_*(t,\theta), \tilde{K}_{12}(t,\theta)$ were known to us), then, {as we discussed in the second paragraph above},
 the  values of 
$\tilde{K}_{11}, \tilde{A}_{11,2}, \tilde{A}_{22,1}$  
on  
the hypersurface $\Sigma_{r_*}$, are determined from $\tilde{K}_{12}$ 
on $\Sigma_{r_*}$. In turn, these values, together with $\te_2(r_*)(t,\theta)$,
 determine $K_{11}, a_{\theta 1}, a_{\theta 2}, a_{t 1}, a_{t2}$ 
on $\Sigma_{r_*}$.

\begin{remark} 
\label{K22.gauge.0}
We note that $\lim_{r\to 0^+} {K}_{22}(r,t,\theta)$ has \emph{not} been 
prescribed an initial value,
as opposed to $K_{12}$ which has been prescribed an asymptotic expansion  
at 
$r=0$ via  
\eqref{distinctK}, and $\tilde{K}_{11}$ which has been prescribed one on 
the 
hypersurface $\Sigma_{r_*}$, via the requirement \eqref{F11.form}. 
However, we \emph{are} requiring that $K_{22}(r,t,\theta)$ be smooth all 
the way to 
$r=0$.
\emph{This} is the prescription of data on $K_{22}(r,t,\theta)$ at the 
singularity; 
in fact, from the point of view of solving the 
Riccati equation \eqref{finredEVERic22pre}
 forward in time, there is a \emph{unique} (but 
implicitly defined) initial datum for $K_{22}$ at $r=r_*$, from which the 
solution to 
that 
equation does not blow up prior to $r=0$. \end{remark}  
  
It follows readily that a solution of the system \eqref{redEVEwav}
\eqref{finredEVERic11pre}, \eqref{finredEVERic22pre}, 
\eqref{finredEVERic22pre}, \eqref{e0.a.pre}, that also satisfies 
the conditions  \eqref{Korthexp.pre}, \eqref{K11exp.pre}, 
\eqref{K12exp.pre}, \eqref{G.def}, \eqref{tG.def}, \eqref{a.init}
that involve the additional functions 
$r_*(t,\theta), \tilde{K}_{12}(t,\theta)$ 
 gives rise to a 
(unique) axially symmetric solution of the EVE with the prescribed initial data; this is shown in the Appendix, \S\ref{retrEVE} . In  
addition, the existence of such a solution shows that a 
smooth congruence of time-like geodesics, which terminates at the singularity $r=0$, 
exists. 
 This is the system we will study in this paper. Proving an existence result 
for this system will prove Theorem \ref{thm_strict},
in the geodesic-normalized gauge we have imposed.

In sum,  the initial value problem for the EVE, under polarized axial symmetry, has been reduced
 to the  system of equations 
\eqref{redEVEwav}, \eqref{finredEVERic11pre}, \eqref{finredEVERic22pre}, 
\eqref{finredEVERic12pre}, \eqref{e0.a.pre},
\eqref{K11exp.pre}, \eqref{K12exp.pre}, 
 \eqref{F11.form}, \eqref{F22.form}, \eqref{tAs.def},
 \eqref{frames.coords}.

We call this system in the 
unknowns 
$\gamma, K_{22}, K_{12}, K_{11}, a_{\theta 1}, a_{\theta 2}, a_{t1}, a_{t2}$
(which depend on $r,t,\theta$) and 
$r_*(t,\theta), \tilde{K}_{12}(t,\theta)$ (which depend only on 
$t,\theta$),
the ``reduced Einstein vacuum equations in the singularity-normalized geodesic 
gauge''. 
We refer to it {as} REVESNGG. 
\medskip

We end this discussion with a key remark: 
Due to the highly anisotropic nature of the singularity, 
it will be necessary for our analysis to express the space-time metric 
relative to a \emph{new} coordinate system $r,T,\Theta$ instead of the system 
$r,t,\theta$ constructed above.  In particular we will be 
constructing  a new coordinate $T=T(t,\theta)$ and preserving the old coordinate 
$\theta$, 
so $\Theta=\theta$. We will then be expressing the metric $g$ with respect to the 
\emph{new} system of coordinates $\{r,T,\Theta,\phi\}$ as opposed to the old one 
$\{r,t,\theta,\phi\}$. The frame $(e_0, e_1, e_2)$ will \emph{still be the same}. 
However in view of the change of coordinates, the coordinate-to-frame
scalars $a_{Ti}, a_{\Theta i}$, $i=1,2$ will now \emph{change}, as will the expression 
of the space-time metric $g$ with respect to the new coordinates.

This, however should be seen as a gauge transformation of our REVESNGG system; in 
particular the new system of equations thus obtained is manifestly \emph{equivalent} to 
the original system.
The reason this change of gauge is performed is to  allow for the optimal 
estimates for the free wave to be derived; this requires the suitable \emph{adaptation} 
of one of the coordinate vector fields to the direction of collapse \emph{at} the 
singularity. The coordinates $T,\Theta$ achieve an \emph{alignment} 
of $\partial_\Theta$ with $e_2$ \emph{at} the singularity.

Now, in addition to the new coordinate $T(t,\theta)$
certain other parameters (notably the scalar valued functions  $e_1(r), e_2(r)$)
will enter our analysis below. However these parameters are readily solved for in terms 
of the ``main variables'' in the REVESNGG system; in this sense they are of secondary 
importance and are not recorded along with the main variables.

\subsection{The theorem re-cast in terms of the REVESNGG.}

Theorem \ref{thm_strict} refers to metric quantities expressed in 
terms of a coordinate system; in particular, it refers to a system of 
coordinates $T,\Theta, r, \phi$. 

Here, we present our theorem in terms of the connection coefficients and 
coordinate-to-frame components of the REVESNGG system. 
This is the result we show in the bulk of this paper. 
We 
show in \S\ref{subsec:REVESNGGimpl} in the appendix, how the next 
formulation implies 
our 
original Theorem \ref{thm_strict}. 

\begin{theorem}
\label{thm_REVESNGG}
Consider  polarized and axi-symmetric 
initial data $(\g, \K)$ which are perturbations of the Schwarzschild data 
$({\bf g}_{\rm S}, {\bf K}_{\rm S})$ at $r=\e$, in the sense that assumptions presented 
in Theorem \ref{thm_strict} hold. 

 Then there exists a coordinate function $T(t,\theta)$, and a coordinate
  $\Theta=\theta$ so that  the REVESNGG system (where the coordinate-to-frame components $a_{Ai}$ are defined with respect to these coordinates)  has 
 a unique solution 
 \[
\gamma(r,t,\theta), K_{ij}(r,t,\theta), a_{Ai}(r,t,\theta), i,j=1,2, A=T,\Theta
\]
and $r_*(t,\theta), \tilde{K}_{12}(t,\theta)$.
These variables satisfy the bounds presented in subsection 
\ref{sec_REVESNGG} below. 

This solution uniquely determines 
the expression for the vector fields ${e_1,e_2}$
 in terms of the coordinate vector fields 
 $\partial_r, \partial_T,\partial_{\Theta}$ via the parameters
  $a_{Ai}(\rho,t,\theta)$, $i=1,2$, $A=T,\Theta$; 
for $r\le \e/2$ these are defined by formulas {\eqref{eibar},
 \eqref{tthetatransebar}}.\footnote{With the index $m-1$ suppressed.}

\end{theorem}

In all these estimates the parameters that appear are 
$s, C, c, \eta, \e$. The constant $\e>0$ determines the 
hypersurface $\{r=\e\}$ \emph{in Schwarzschild}, 
over which we consider the (re-normalized) 
perturbation of the Schwarzschild data. 
$\eta>0$ captures the (post-renormalization) closeness of our initial 
data to 
that of the Schwarzschild background.  $s\in\mathbb{N}$ denotes the 
(Sobolev space) 
order  at which we measure the initial data and our solution. 
$c>0$ captures the order ${\rm low}:=s-3-4c$ at which we provide 
\emph{optimal} 
estimates for our key parameters $\gamma, K_{ij}, a_{Ai}$ that are fully in agreement 
with the claim of Theorem \ref{thm_strict}. The constant $C>1$ 
 captures the \emph{growth multiple} of the (renormalized) 
norms of the evolution parameters $\gamma, K$ between the initial data and the 
final singularity at $r=0$. The parameters $\eta,\e$ must satisfy certain 
smallness conditions which we list out in detail in  
\S\ref{sec_csts_params}. Here we highlight that $C\cdot \eta>0$ must be 
small enough to ensure that the explicit function $d_2(\alpha)$, {defined in \eqref{d1.d2}}, is
bounded {in absolute value} by $1+\frac{1}{8}$ for $\alpha=1+C\cdot \eta$. Moreover, 
$\e>0$ must satisfy {the inequality} $\e<(\frac{C\eta}{{2}B})^4$. The full set of bounds we impose on our parameters is speled out in 
the subsection ``Key constants'' below.

\section{The Iteration scheme. }
\label{suc.iter}

\subsection{Overview.}

Our method is to solve the system REVESNGG using 
 an iteration scheme rather than treating it as a  coupled system 
 directly. In particular, we produce a sequence of metrics $g^m$ 
 in the form \eqref{gm.form}  and these 
  will converge, as the parameter $m\to\infty$, 
to a solution of the system REVESNGG. Thus, we obtain a solution of the EVE 
only in the limit $m\to\infty$. (Note in particular that the individual metrics $g^m$ do \emph{not} solve the vacuum Einstein equations).  

Let us spell out a few features of the iterated metrics: 

We set $\gamma^0=\gamma_{\rm S}$ (the value of the function in Schwarzschild), and $h^0=h_{\rm S}$ 
(the value of the metric in Schwarzschild), then the subsequent iterates 
$\gamma^m, h^m$, $m\ge1$,  define  a sequence of space-time metrics

\beq
\label{gm.form}
g^m= e^{2\gamma^m}d\phi^2+h^m(r,t,\theta).
\eeq
These solve a recursive system that we discuss next. 
\begin{itemize}
\item $\gamma^m$ is solved-for first at each step of the iteration. It is 
required to solve the (linear) free wave equation \eqref{finredEVEwavit},
relative to the \emph{previous} metric in the iteration. The 
(abstractly prescribed) 
initial data for $\gamma^m$ live on a hypersurface $\Sigma_{r^{m-1}_*}$, 
defined at the previous step in the iteration.

\item The geometry of the (2+1)-metrics $^mh$ is encoded in a suitable orthonormal frame 
$e_0, e_1^m, e_2^m$. This frame is expressible in terms of fixed background coordinates
 $(r,t,\theta)$, as specified by formulas \eqref{tthetatransebar}  below. 
 $e_1^m$ is required to be tangent to the hypersurface 
 $\Sigma_{{r^m_*}}:=\{r=r^m_*(t,\theta)\}$,
  on which the initial data will live.  $e_2^m$ is required to satisfy
\beq  \label{e2r.van.it}
   e_2^m(r)=o(r^{-\frac{1}{2}+d_2^m(t,\theta)}). 
   \eeq
      (See the 
  discussion further down {in subsection \ref{subsec:rstarm}} on how $r^m_*(t,\theta)$ is to 
  be determined).

\item In addition to the metric component $\gamma^m$, the ``remaining'' parts $h^m$ of 
the metric  $g^m$ are encoded in the independent connection and coordinate-to-frame
 coefficients of \emph{this} frame. 
These connection coefficients  (that we will 
solve for) are 
\beq K^m_{11}(r,t,\theta), K^m_{22}(r,t,\theta), K^m_{12}(r,t,\theta).
\eeq
The coordinate-to-frame coefficients are $a^m_{Ai}$, $A=t,\theta$ and $i=1,2$.
In fact we will be solving for an equivalent system for coefficients 
 $a^m_{Ai}$, $A=T,\Theta$, where $\Theta,T$ are coordinate vector fields constructed out of the old coordinates by an explicit transformation.  

\item The connection coefficients $K^m_{ij}$ solve first order ODEs 
 which are an iterated version of the propagation equations \eqref{finredEVERic11pre}, \eqref{finredEVERic22pre},
\eqref{finredEVERic12pre};  all these evolution equations involve forcing terms in their 
RHSs that contain  covariant derivatives for the
just-solved-for $\gamma^m$, evaluated against the \emph{previous} metric $h^{m-1}$ and its associated previous frame.   
In particular, these connection coefficients are
scalar-valued functions over the coordinates  $\{r,t,\theta\}$. The restriction we 
impose on 
these is that $K^m_{12}=o(r^{2d^m_2(t,\theta)-\alpha^m(t,\theta)})$ 
 as $r\to 0$, for all $t,\theta$, and that 
both $K^m_{22}, K^m_{11}$ remain smooth until $r=0$. (This is in fact a gauge
 normalization on $K^m_{22}$, as discussed in the previous section). 

\item The coordinate-to-frame coefficients $a^m_{Ai}$ 
themselves solve first order ODEs, 
with coefficients depending on the just-solved-for connection coefficients $K^m_{ij}$. 


Thus, in particular, we 
 de-couple the free wave $\gamma$ from the $(2+1)$-metric $h$ 
by performing an iteration, to find a sequence of free waves and 
$(2+1)$-metrics, 
$(\gamma^m, h^m)$ (which are indexed by a parameter $m\in\mathbb{N}$). 
We note that
the equations described above are all evolution equations, yet we have 
not prescribed 
initial data for these parameters $K^m_{11},
 a^m_{\theta 1}, a^m_{\theta 2}, a^m_{t1}, a^m_{t2}$, 
 nor prescribed the hypersurface where this initial 
data are to
be induced. 
\medskip

To determine the solution one needs to prescribe the initial data 
$(\Sigma, \g, \K)$ for these variables 
somewhere.  
In particular, the prescribed initial data $(\Sigma, \g, \K)$ 
are to live on an $m$-dependent hypersurface 
$\Sigma_{r^m_*}=\{ r=r^m_*(t,\theta)\}$, for some function 
$r^m_*(t,\theta)$ that is to be 
solved-for.  
The equations that prescribe the function $r^m_*(t,\theta)$ are coupled to a special 
 component $\tilde{K}^m_{12}(t,\theta)= \tilde{K}^m_{12}(r^m_*(t,\theta),t,\theta)= [\tilde{K}^m(\te^m_1,\te^m_2)](r^m_*(t,\theta),t,\theta)$ of the second fundamental form 
 of that surface. This then yields a coupled 2x2 system in the unknowns 
 ${r^m_*}(t,\theta), \tilde{K}^m_{12}(t,\theta)$, which determines 
 both these 
 parameters. 
 More precisely:  
 \medskip

\item The prescribed initial data $(\Sigma,\g,\K)$ are induced on a to-be-determined 
hypersurface 
\beq
\label{Sigma.r*m}
\Sigma_{r^m_*}:=\{r=r^m_*(t,\theta)\}.
\eeq
``Induced'' here means the following: 
The frame element $e^m_1$ is tangent to $\Sigma_{r^m_*}$; and then if one 
considers the 
rotation $\{\te^m_0,e^m_1,\te^m_2\}$ of the frame 
$\{e_0, e^m_1, e^m_2\}$,
which makes $\te^m_2$ tangent to $\Sigma_{r^m_*}$ 
 and 
specifies connection
coefficients $\tilde{K}^m$ on this rotated frame, according to 
the formulas \eqref{Korthexp}, \eqref{K11exp}, \eqref{K12exp}, 
  below, as well as the  coordniate-to-frame coefficients $a^m_{Ai}$
then the so-defined $\tilde{K}^m_{ij}, a^m_{Ai}$ solve 
the requirements \eqref{tAm.def}, \eqref{F22.F11} (which are extrapolations of \eqref{G.def}, \eqref{tG.def}, {\eqref{tAs.def}} in the coupled case).
Thus, both $\tilde{K}^m_{ij}(t,\theta)$ and $a^m_{Ai}(t,\theta)$
correspond to the prescribed initial data $(\g,\K)$ 
up to 
a gauge transformation (as described in Definition \ref{equiv.metr} 
above). 

\item The initial data for the connection coefficient
$K^m_{11}$, as well as for 
$a^m_{t1}, a^m_{t2}, a^m_{\theta 1}, a^m_{\theta 2}$, 
are determined on 
the just-solved-for hypersurface $\Sigma_{{r^m_*}}$, in terms of the 
parameters
$\tilde{K}^m_{12}(t,\theta),$ and $\te^m_2(r^m_*)$ and the 
abstract 
initial data $(\Sigma, \g, \K)$ by precisely the (iterated analogues of) 
formulas 
\eqref{K11exp.pre}, \eqref{a.init}, where now 
$r^m_*(t,\theta), \tilde{K}^m_{12}(t,\theta)$ have
just been solved for.

\end{itemize}

In particular, 
the parameters that we will be solving for at each step  in the iteration 
are functions of $r,t,\theta$, and some functions of $t,\theta$. The 
former are: 
\begin{align}\label{evolparam}
\gamma^m, K^m_{11}, K^m_{22}, K^m_{12}, a^m_{Ai}, A=t,\theta, i=1,2, 
\end{align}
while the latter are  the function $r_*^m(t,\theta)$ that determines the 
hypersurface $\Sigma_{r^m_*}$, on which the initial data will be induced, 
as well as 
 the function $\tilde{K}^m_{12}(t,\theta)$ on $\Sigma_{r^m_*}$. The 
 latter encodes the 
 (gauge)
  choice of the orthonormal frame
 $\te_1^m, \te_2^m$
on the initial data hypersurface. 
\medskip

We list out the equations that govern the evolutions of the parameters  
{in \eqref{evolparam}} that depend on $r,t,\theta$. 
These  will be
 \eqref{finredEVEwavit},  \eqref{finredEVERic11it}, \eqref{finredEVERic22it}, 
 \eqref{finredEVERic12it},
\eqref{e0.a.new}. 
Notably, the equation {\eqref{finredEVEwavit}} of $\gamma^m$ is a free wave equation, which we call the free wave part 
of the system. 
The equations \eqref{finredEVERic11it}, \eqref{finredEVERic22it}, \eqref{finredEVERic12it}, \eqref{e0.a.new}
 on the time-like connection 
coefficients and (spatial) frame-to-coordinate components of $h^m$ 
are 1st order ODEs. 

\begin{definition}
We call the set of first three equations (which is de-coupled from the 
remaining ones) the Riccati part of the system;
we note that they are non-linear first order ODEs. 
Equations \eqref{e0.a.new} are linear first-order ODEs. 
We call these the spatial components part of the 
system. 
\end{definition}


We now explain our iteration scheme in more detail:

\subsection{The recursive equations  for $\gamma^m$ and $K_{ij}^m$.}
\label{recur.gam-K}

Recall from above that at each step $m$ of the iteration, there is a 
hypersurface 
$\Sigma_{{r^m_*}}$, as in \eqref{Sigma.r*m}, on which the prescribed initial 
data 
$(\Sigma, \g, \K)$ live. In particular, at the $m^{\rm th}$ step of the 
iteration, there 
exists a hypersurface $\Sigma_{r_*^{m-1}}:=\{ r=r_*^{m-1}(t,\theta)\}$ on 
which the 
\emph{previous} metric $h^{m-1}$ induces the initial data 
$(\Sigma, \g, \K)$. 
\medskip

The component $\gamma^m$ is required to solve:
\beq
\label{finredEVEwavit}\square_{g^{m-1}}\gamma^m=0,
\eeq
and the initial Cauchy data $(\gamma_{\rm init}, n(\gamma_{\rm init}))$ 
for $\gamma^m$  live on $\Sigma_{r^{m-1}_*}$.\footnote{{These correspond to $\frac{1}{2}\log({\bf g}_{\phi\phi}),{\bf K}_{33}$ respectively.}} 
Next we determine $h^m$:
\medskip

 Each iterated metric $h^m$ (and $g^m$),  comes equipped with an 
 $m$-dependent
  frame 
 $e_0, e^m_1, e^m_2$ (and {$e^m_3:=e^{-\gamma^m}\partial_\phi$} for $g^m$). 
 This orthonormal frame is fixed by the requirement that 
 $e_2^m(r)=o(r^{-\frac{1}{2}+d_2^m(t,\theta)})$
  and 
 $e^m_1$ should be 
 tangent to the hypersurface $\Sigma_{r_*^m}$, on which the initial data 
 are to be
  induced, along with the following propagation conditions:  
  
   $e_0$ will be time-like and affine for each iterate: 
\begin{align}\label{e0it}
\nabla^m_{e_0}e_0=0,\qquad e_0=-(\frac{2M}{r}-1)^\frac{1}{2}\partial_r
\end{align}
and the vector fields $e^m_1, e^m_2$ will be transported along $e^m_0$
 according 
to the rule: 
\begin{align}\label{gaugeit}
\nabla^m_{e_0^m}e^m_1=-K_{12}^me^m_2,\qquad \nabla^m_{e^m_0}e^m_2=
K^m_{12}e^m_1,
\end{align}
{where $\nabla^m$ is the connection intrinsic to $h^m$}.\footnote{{Note that for $g^m$, 
we also have automatically the propagation relation $D^m_{e_0}e_3^m=0$, due to the 
symmetry, $D^m$ being the connection intrinsic to $g^m$.}}
 (Each such frame $e_0, e^m_1, e^m_2$ 
can be expressed in terms of the fixed background coordinates $\{r,t,\theta\}$. This is 
encoded via 
coefficients, defined in formulas \eqref{tthetatransebar} below). The vector fields 
$e^1_m, e^m_2$ are normalized by the requirements that $e^m_1$ should be tangent to 
$\Sigma_{r^m_*}:=\{ r=r^m_*(t,\theta)\}$ and $e^m_2$ by the requirement 
$e_2^m(r)=o(r^{-\frac{1}{2}+d_2^m(t,\theta)})$.

Denote by $K^m_{ij}$ the connection coefficients 
\begin{align}\label{Kijm}
K_{ij}^m:=h^m(\nabla^m_{e_i^m}e_0,e_j^m)=K_{ji}^m.
\end{align}
For these connection coefficients we impose the equations:
\begin{align}
\label{finredEVERic11it}e_0K_{11}^m+(K_{11}^m)^2+3(K_{12}^m)^2+e_0\gamma^m K_{11}^m=&\,\overline{\nabla}^{m-1}_{11}\gamma^m+\overline{\nabla}_1\gamma^{m-1}\overline{\nabla}_1\gamma^m
-e_0^2\gamma^m-(e_0\gamma^m)^2,\\
\label{finredEVERic22it} e_0K_{22}^m+(K_{22}^m)^2-
(K_{12}^{m-1})^2+e_0\gamma^m K^m_{22}=&\,
\overline{\nabla}^{m-1}_{22}\gamma^m+\overline{\nabla}_2\gamma^{m-1}
 \overline{\nabla}_2\gamma^m-
e^2_0\gamma^m-(e_0\gamma^m)^2,\\
\label{finredEVERic12it} e_0K_{12}^m+(2K_{22}^m+e_0\gamma^m)K_{12}^m=&\,\overline{\nabla}^{m-1}_{12}\gamma^m+\frac{1}{2}
[\overline{\nabla}_1\gamma^{m-1} \overline{\nabla}_2\gamma^m+
\overline{\nabla}_1\gamma^{m} \overline{\nabla}_2\gamma^{m-1}]
\end{align}
As before $\overline{\nabla}^{m-1}$ stands for the projection of the connection $\nabla^{m-1}$ onto
 ${\rm Span}\langle e^{m-1}_1, e^{m-1}_2\rangle$. It is thus the ``spatial part'' of the 
Levi-Civita connection. 

In these ODEs, $K_{ij}^m$ are seen as  simply three scalar-valued functions of $r,t,\theta$. 
In the RHS we again consider the function $\gamma^m$ that was just solved for. The 
covariant derivatives and vector indices $1,2$ that appear in the above RHS are with respect to the connection of $h^{m-1}$, projected onto 
${\rm Span}\langle e_1^{m-1},e_2^{m-1}\rangle$,\footnote{{We denote this connection by $\overline{\nabla}^{m-1}$ and omit the index $m-1$ wherever it is evident for simplicity.}} and with respect to the frame
 $e_1^{m-1}, e_2^{m-1}$, associated to the metric $h^{m-1}$. 
 The equations above are evaluated at points $(r,t,\theta)$. We note also that in the second equation above we have entered the previously-solved for scalar $K^{m-1}_{12}$ instead of $K^m_{12}$; this is for technical convenience only, as  the equations \eqref{finredEVERic22it}
 and \eqref{finredEVERic12it} then become completely de-coupled. 
\medskip



 Note that (\ref{finredEVEwavit}) is linear in $\gamma^m$, whereas the decoupled Riccati ODEs 
 (\ref{finredEVERic11it})-(\ref{finredEVERic12it}) that we impose {\it remain} non-linear in 
 $K^m_{ij}$, $i=1,2$.\footnote{Solving a non-linear system of ODEs for $K^m_{ij}$ is crucial for 
 our argument to close. An attempt to work with a linearised version of 
 (\ref{finredEVERic11it})-
 (\ref{finredEVERic12it}) would fail to capture the
 correct rate of 
 the blow up, and would make it impossible to close the required estimates.} 
\medskip

We will be restricting our attention to solutions of the above equations that do not blow up 
in the
 $\mc{C}^0$ norm prior to $\{r=0\}$. In addition,  
we will be imposing the asymptotic diagonalisation condition that 
\beq
\label{init.cond.m}
K_{12}^m\cdot(K_{11}^m)^{-1}\to 0,\qquad K_{12}^m\cdot (K_{22}^m)^{-1}\to 0,
\eeq
 as $r\to 0$, 
as well as suitable variants of this for derivatives of $K^m_{12}$. 
To distinguish the two directions $e_1^m, e_2^m$, we choose that $K_{11}^m>0, K^m_{22}<0$ near the 
singularity.

We will see that these requirements allow us to \emph{uniquely} solve for 
$K^m_{22}(r,t,\theta), K^m_{12}(r,t,\theta)$ via the ODEs \eqref{finredEVERic22it}, \eqref{finredEVERic12it} above, subject to {the}
initial
 condition \emph{at} $r=0$, 
 $K^m_{12}(r,t,\theta)=o(r^{2d^m_2(t,\theta)-\alpha^m(t,\theta)} )$.

Having solved for these parameters $K^m_{22}, K^m_{12}$ \emph{separately} from all other connections 
coefficients, 
we next solve  for the hypersurface $\Sigma_{{r^m_*}}$ on which the initial data are to be 
induced.
This hypersurface is 
 defined by a function 
${r^m_*}(t,\theta)$, which is solved for along 
 with the component $\tilde{K}^m_{12}(t,\theta)$ on  
$\Sigma_{{r^m_*}}$. As discussed above, $\tilde{K}^m_{12}(t,\theta)$  ``sees'' 
the frame $e_1^m,\te_2^m$ that is induced by $e_1^m, e_2^m$ onto $\Sigma_{{r^m_*}}$ by rotation. 

We discuss this system in the next subsection, after a brief remark:

\subsubsection{Rotation formulas and some useful calculations. }

 Recall  that at each step in our iteration we will have 
 $e^m_2r=0$. On the to-be-determined hypersurface $\Sigma_{r^m_*}$ 
 we will evaluate $\tilde{K}^m$ against the following 
 frame, which is adapted to $\Sigma_{{r^m_*}}$:
\begin{align}\label{tildee.m}
\begin{split}
\tilde{e}_1^m:=e_1^m\in T\Sigma_{r^m_*},\qquad{\tilde{e}_2^m}:=q[e_2^m-(\frac{2M}{r_*^m}-1)^{-\frac{1}{2}}(e_2^m{{r^m_*}})e_0]\in T\Sigma_{{r^m_*}},\\
\tilde{e}_0^m=q[e_0-(\frac{2M}{r_*^m}-1)^{-\frac{1}{2}}(e_2^m{{r^m_*}}){e_2^m}]\in T\Sigma_{{r^m_*}}^\perp,\qquad  q=\big(1-(\frac{2M}{r_*^m}-1)^{-1}({e^m_2}{{r^m_*}})^2\big)^{-\frac{1}{2}}
\end{split}
\end{align}
In the latter two formulas $r^m_*$ is extended to be constant along the integral curves of $e_0$, therefore, 
$e_2^mr_*^m$ makes sense.
Note that
\begin{align}\label{e2rstar}
{\tilde{e}_2^m}{{r^m_*}}=q({e^m_2}{{r^m_*}}),&&q=\sqrt{1+(\frac{2M}{r_*^m}-1)^{-1}
({\tilde{e}^m_2}r_*^m)^2}
\end{align}
which implies that 
\begin{align}
\label{tildeedef}
{\tilde{e}^m_2}:=q{e^m_2}-
(\frac{2M}{r_*^m}-1)^{-\frac{1}{2}}({\tilde{e}^m_2}{r^m_*})e_0,\qquad
{\tilde{e}_0^m}=qe_0-(\frac{2M}{r_*^m}-1)^{-\frac{1}{2}}({\tilde{e}^m_2}{r^m_*}){e^m_2}.
\end{align}
Inverting (\ref{tildee}) we obtain:
\begin{align}\label{tildeeinvit}
{e^m_2}=q{\tilde{e}_2^m}+(\frac{2M}{r_*^m}-1)^{-\frac{1}{2}}
({\tilde{e}^m_2}{r^m_*}){\tilde{e}^m_0},&& e_0=q{\tilde{e}^m_0}+(\frac{2M}{r_*^m}-1)^{-\frac{1}{2}}({\tilde{e}^m_2}{r^m_*}){\tilde{e}^m_2}.
\end{align}

The connection coefficients $K^m_{ij}$ are then fixed by requiring 
that they should \emph{induce} the required initial data $(\Sigma,\g,\K)$ on the hypersurface 
$\Sigma_{r^m_*}$. Inducing here means the following:

 We require that if we consider 
 the orthonormal frame $\te^m_0, \te^m_2, e^m_1$ defined by the formulas
  \eqref{tildeedef},  and \emph{define}  connection coefficients 
  $\tilde{K}^m,\tilde{A}^m$ 
  on that frame given by the formulas: 
\begin{align}
\label{Korthexp}
q^2(K^m_{22})=&\,q\tilde{K}^m_{22}+(\frac{2M}{r^m_*}-1)^{-\frac{1}{2}}({\tilde{e}^m_2}{\tilde{e}^m_2}{r^m_*})+(\frac{2M}{r^m_*}-1)^{-\frac{3}{2}}\frac{2M}{r^m_*}
({\tilde{e}^m_2}{r^m_*})^2\\
\label{K11exp}
K^m_{11}=&\,q\tilde{K}^m_{11}+(\frac{2M}{r^m_*}-1)^{-\frac{1}{2}}({\tilde{e}^m_2}
{r^m_*})
\tilde{A}^m_{11,2}\\
\label{K12exp}
K^m_{21}=&\,
\tilde{K}^m_{21}+q^{-1}(\frac{2M}{r^m_*}-1)^{-\frac{1}{2}}
({\tilde{e}^m_2}{r^m_*})
\tilde{A}^m_{22,1},
\end{align}
as well as \eqref{tAm.def} right below, then the first
and second fundamental forms induced by  
$\tilde{A}^m,\tilde{K}^m$ on $\Sigma_{r_*^m}$, should both be equivalent to the background metric and
 second fundamental 
form  (in the sense of
 definition \ref{equiv.metr}) via a rotation by a function 
 $\varphi^m(t,\theta)$. (The function $\varphi^m(t,\theta)$ is also 
 to be solved for).

Next, we derive the equations that determine the initial data hypersurface $\Sigma_{r^m_*}$, along with 
the connection coefficients $\tilde{K}^m_{ij}$, $\tilde{A}^m_{ij,k}$ 
\emph{on} $\Sigma_{r^m_*}$. 

\subsection{Determination of the initial data hypersurface ${r^m_*}(t,\theta)$ and the 
connection and curvature components adapted to that hypersurface.}\label{subsec:rstarm}

The sought-after parameters $r_*^m(t,\theta)$ and $\tilde{K}^m_{12}(t,\theta)$ will be 
fixed by imposing the equations \eqref{Korthexp}, \eqref{K12exp}, with suitable 
substitutions 
for certain terms, which are derived by making use of other necessary conditions.

We define $\tilde{A}^m_{11,2}, \tilde{A}^m_{22,1}$  \emph{on} 
$\Sigma_{r^m_*}$ as functions of $\tilde{K}^m_{12}$, via the formulas: 
\beq
\label{tAm.def}
\begin{split}
&\tilde{A}^m_{11,2}={\bigg(1-\frac{4(\tilde{K}^m_{12})^2}{({\bf K}_{22}-{\bf K_{11}})^2}\bigg)^{-\frac{1}{2}}}\te_1^m(\frac{\tilde{K}^m_{12}}{{\bf K}_{22}-{\bf K}_{11}})
+F_1(\frac{\tilde{K}^m_{12}}{{\bf K}_{22}-{\bf K}_{11}}),
\\&\tilde{A}^m_{22,1}= {\bigg(1-\frac{4(\tilde{K}_{12}^m)^2}{({\bf K}_{22}-{\bf K_{11}})^2}\bigg)^{-\frac{1}{2}}}\te_2^m(\frac{\tilde{K}^m_{12}}{{\bf K}_{22}-{\bf K}_{11}})+F_2(\frac{\tilde{K}_{12}}{{\bf K}_{22}-{\bf K}_{11}}),
\end{split}
\eeq
where the functions $F_1, F_2$ are given by formulas \eqref{F1.form}, \eqref{F2.form}. 
\medskip


Moreover, the frame elements $\te^m_1(t,\theta), \te^m_2(t,\theta)$ are given from the
 background frame ${E}_1(t,\theta), { E}_2(t,\theta)$ by a rotation of angle 
 $\varphi^m(t,\theta)$. The rotation angle $\varphi^m(t,\theta)$ is given by the 
 value of $\tilde{K}^m_{12}(t,\theta)$ via the formula: 
 \beq
 \label{varphi.tilK.m}
  \varphi^m(t,\theta)={
\frac{1}{2}\sin^{-1}(\frac{2\tilde{K}^m_{12}}{\K_{22}-\K_{11}}).}
 \eeq

Thus, making use of formula \eqref{frames.coords}, 
we find that the  frame elements $\te_1^m(t,\theta),\te_2^m(t,\theta)$ are given from the
 background frame ${\bf E}_1(t,\theta), {\bf E}_2(t,\theta)$ and the 
 value of $\tilde{K}^m_{12}(t,\theta)$ by the formulas: 
 \beq
   \begin{split}
   \label{te1m.tK12}
& \te_1^m={\frac{1}{2}\sqrt{1+\sqrt{1-4x^2}}}
(\g_{tt})^{-1/2} \partial_t+ {\frac{{\rm sign}(x)}{2}\sqrt{1-\sqrt{1-4x^2}}}
(\g_{\theta\theta})^{-1/2}\partial_\theta,
\\
& \te_2^m= \frac{1}{2}\sqrt{1+\sqrt{1-4x^2}}
(\g_{\theta\theta})^{-1/2}\partial_\theta-\frac{{\rm sign}(x)}{2}\sqrt{1-\sqrt{1-4x^2}}
(\g_{tt})^{-1/2} \partial_t,
 \end{split}
   \eeq
where $x=\tilde{K}^m_{12}({\bf K}_{22}-{\bf K}_{11})^{-1}$. {We also write, for future reference, the inverse transformation:
 \beq
   \begin{split}
   \label{te1m.tK12inv}
& \partial_t= \frac{1}{2}\sqrt{1+\sqrt{1-4x^2}}(\g_{tt})^{1/2}\te_1^m-\frac{{\rm sign}(x)}{2}\sqrt{1-\sqrt{1-4x^2}}(\g_{tt})^{1/2}\te_2^m,
\\
& \partial_\theta= \frac{1}{2}\sqrt{1+\sqrt{1-4x^2}}(\g_{\theta\theta})^{1/2}\te_2^m +\frac{{\rm sign}(x)}{2}\sqrt{1-\sqrt{1-4x^2}} (\g_{\theta\theta})^{1/2}\te_1^m. 
 \end{split}
   \eeq
}
\medskip

{\bf Determination of $\tilde{K}^m_{ij}(t,\theta)$ and ${r^m_*}(t,\theta)$.}
The tensor $\tilde{K}^m_{ij}$ is required to satisfy that 
 $\tilde{K}_{11},\tilde{K}_{22}$ are determined in terms 
of $\tilde{K}_{12}$ via the formulas: 
\beq
\label{F22.F11}
\tilde{K}^m_{22}(t,\theta)= F^{22}_{t,\theta}[\tilde{K}^m_{12}(t,\theta)],\qquad
\tilde{K}^m_{11}(t,\theta)= F^{11}_{t,\theta}[\tilde{K}^m_{12}(t,\theta)].
\eeq
Here the functions $F^{22}[\cdot], F^{11}[\cdot]$ are given by \eqref{F22.form}, \eqref{F11.form}.  

\begin{remark}
\label{K.obtained}
Observe that in the language  of Definition \ref{equiv.metr}, 
the tensor $[\tilde{K}_{ij}^m](t,\theta)$ obtained as
 above,
is gauge-equivalent to the prescribed initial second fundamental form ${\bf K}$.
\end{remark}

Then $\tilde{K}^m_{12}$ and $r^m_*(t,\theta)$ are determined via a system of two 
equations in these two unknowns. The equations arise from imposing {\eqref{Korthexp}, \eqref{K12exp}. Plugging the first formula of \eqref{F22.F11} into \eqref{Korthexp} gives:}
\beq
\label{key.link2}
\begin{split}
&-q^{-2}(\frac{2M}{{r^m_*}}-1)^{-\frac{1}{2}}(^m{\tilde{e}_2}^m{\tilde{e}_2}{{r^m_*}})-q^{-2}
(\frac{2M}{{r^m_*}}-1)^{-\frac{3}{2}}\frac{2M}{{r^m_*}^2}
(^m{\tilde{e}_2}{{r^m_*}})^2+(K^m_{22})(r^m_*(t,\theta),t,\theta)
\\&=q^{-1}F^{22}_{t,\theta}[\tilde{K}^m_{12}(t,
\theta)](r^m_*(t,\theta),t,\theta).
\end{split}
\eeq
In the above equation, 
 $K_{22}^m(r,t,\theta)$ solves the {evolution 
 equation \eqref{finredEVERic22it}  with 
``zero free data'' at $r=0$}. In particular, it is the unique solution of this
first order  ODE, 
and thus, (given that the RHS of this equation has already been solved for at this 
stage), 
for each fixed $t,\theta$,
 $K_{22}^m(r_*(t,\theta),t,\theta)$ is a  
  function of $r_*^m(t,\theta)$ \emph{alone}. 

Thus, \eqref{key.link2} 
is a second order equation  on the sought-after  ${r^m_*}$, with derivatives 
in the direction of the vector 
field 
$\te_2^m$. We also note that the smoothness of $\Sigma_{r_*^m}$ in the resulting 
$(3+1)$-dimensional space-time  forces that $\partial_\theta r^m_*=0$ at the two poles 
$\theta=0,\pi$. The equation itself then forces that all $\partial^{k_1}_{\theta\dots\theta}\partial^{k_2}_{t\dots t} r^m_*$ with $k_1+r_2\le s$ and $k_1$ odd must vanish 
at those two poles, in the weighted $L^2$ sense defined in our assumptions. 
\eqref{key.link2} is complemented by the unknown $\tilde{K}_{12}^m$ which also appears in the 
equation. However, we can relate this quantity to the sought-after $r_*^m$ via equation
 \eqref{K12exp}, {plugging in \eqref{tAm.def}:} 
\begin{align}
\label{key.link2'}
\begin{split}
&\tilde{K}^m_{12}(r^m_*(t,\theta),t,\theta)+q^{-1}(\frac{2M}{r^m_*}-1)^{-\frac{1}{2}}
({\tilde{e}^m_2}{r^m_*})
{\bigg(1-\frac{4(\tilde{K}_{12}^m)^2}{({\bf K}_{22}-{\bf K_{11}})^2}\bigg)^{-\frac{1}{2}}}\te_2^m(\frac{\tilde{K}^m_{12}}{{\bf K}_{22}-{\bf K}_{11}})\\
&+q^{-1}(\frac{2M}{r^m_*}-1)^{-\frac{1}{2}}
({\tilde{e}^m_2}{r^m_*})F_2(\frac{\tilde{K}_{12}}{{\bf K}_{22}-{\bf K}_{11}})=K^m_{12}(r^m_*(t,\theta),t,\theta)
\end{split}
\end{align}
{We note that the equation \eqref{finredEVERic12it},} coupled with the imposed initial 
condition $K^m_{12}\sim o(r^{2d_2^m(t,\theta)-\alpha^m})$,
implies  that  $K^m_{12}(r_*(t,\theta),t,\theta)$ is a function of $r_*^m(t,\theta)$ 
alone.

Therefore, the system of equations \eqref{key.link2}, \eqref{key.link2'} provides a system 
of 
two equations in the two unknowns $r^m_*(t,\theta), \tilde{K}^m_{12}(t,\theta)$. 
The next step in the iteration process is to produce a unique solution of this 2x2 
system. As we will see below the solvability of this system is not obvious (at least to the authors), and required a special weak formulation 
to obtain existence. 
\medskip

Thus, at this point, the variables 
$K^m_{22}(r,t,\theta), K^m_{12}(r,t,\theta)$ have been solved for everywhere. Moreover, the initial parameters 
 $\tilde{K}^m_{22}(t,\theta)$, 
$\tilde{K}^m_{12}(t,\theta)$, as well as $r^m_*(t,\theta)$,
have been determined on the initial data hypersurface 
$\Sigma_{{r^m_*}(t,\theta)}$. Next, we define 
the functions $\tilde{K}^m_{11}(t,\theta)$
$a^m_{Ai}(\rho^m=\e,t,\theta)$, 
 on $\Sigma_{{r^m_*}(t,\theta)}$, via the second equation in \eqref{F22.F11} and the formula \eqref{tAm.def}, as well as \eqref{te1m.tK12inv}.

We then define $K^m_{11}(r_*(t,\theta),t,\theta)$ via the formula \eqref{K11exp} \emph{on} $\Sigma_{{r^m_*}}$. With this initial value, we determine $K^m_{11}(r,t,\theta)$ \emph{everywhere} by solving the
Riccati equation \eqref{finredEVERic11it} forwards-in-time. 
\medskip

Having solved for the components 
$K^m_{ij}(r,t,\theta)$ we can solve for $a^m_{Ai}(\rho,t,\theta)$ via \eqref{e0.a.new}, 
and the inital conditions for these parameters. 
This will complete the determination of the next metric iterate 
$h^m$ (and thus the next $g^m$ also).

\subsubsection{The system of unknowns and the system of equations.}

To summarize, the system of functions that we solve for at the $m^{\rm th}$ step is as follows: 
The parameters $\gamma^m(r,t,\theta), K^m_{22}(r,t,\theta)$, 
$K^m_{12}(r,t,\theta), K_{11}^m(r,t,\theta)$, $a^m_{t1}(r,t,\theta), a^m_{t2}(r,t,\theta)$, $a^m_{\theta 1}(r,t,\theta)$, $a^m_{\theta 2}(r,t,\theta) $, which all 
depend on $r,t,\theta$;  
the parameters 
$\tilde{K}^m_{12}, {r_*^m}$ 
that only depend on 
$(t,\theta)$. 
\medskip

The system of equations is \eqref{finredEVEwavit}, 
\eqref{finredEVERic11it},
\eqref{finredEVERic22it}, \eqref{finredEVERic12it}, \eqref{e0.a.new}
 (these are the evolution 
equations), \eqref{tAm.def}, 
\eqref{F22.F11} (these are the equations used to 
capture the initial data); the latter two equations imply 
the system \eqref{key.link2}, \eqref{key.link2'}. And finally,
the two systems are linked by the equations \eqref{K11exp}, \eqref{tAm.def}, \eqref{te1m.tK12inv}
 which provide initial data on the hypersurface $\Sigma_{r_*^m}$ for the 
 parameters 
 $K_{11}^m(r,t,\theta)$, 
 $ a^m_{Ai}(r,t,\theta)$.  
 
  \section{The function spaces and bounds for the key variables
  of our reduced system. }
\label{sec_fun.spaces}

\subsection{Regularity spaces for the parameters.} 
%
 We present the spaces in which we will derive estimates for the 
 variables that describe  the
  space-time metric $g^m$ we deal with. 
  (The field $\gamma^m$, and the connection coefficients
   $K^m$ of the metric $h^m$ along with the 
coordinate-to-frame coefficients). 
 
It is well-known that expressing the space-time metric $g$ in a geodesic 
(Fermi-type) gauge
 leads to a loss of derivatives, in that one expects the metric components  to enjoy 
 \emph{less} regularity in the spatial directions  ($e_1, e_2$) relative 
 to the 
 \emph{special, affine} time direction $e_0$ that defines our Fermi 
 coordinates. From this point of view, we can think of the affine 
 direction $e_0$ as being 
 \emph{privileged} in terms of  regularity. 
 While one would worry that this would impede the closure of our 
 estimates 
 in a fixed function space, we do find function spaces that allow us to 
 close our estimates. The algebraic structure of our equations, with a 
 free wave and transport equations (where the free wave supplies 
 the forcing term) is very important in this regard. 
 
More specifically, the relevant variables 
$\gamma^m, K^m_{ij}, a^m_{Ai}$  
\emph{and a suitable number of derivatives thereof}, will be shown to lie 
in 
$L^2$-{based} energy spaces 
on level sets of the function $r$ and of certain variants $\rho^m$ of 
the function $r$ that we introduce. These are 
properly defined in the next subsection. Here we highlight a few  
features of the bounds we 
derive: 
\medskip

{\bf Hierarchy of Regularities:} In view of the loss
 of spatial regularity 
for the geometric parameters $K^m, a^m,$ in our geodesic gauge, 
the derivatives of the key parameters $\gamma^m, K^m, a^m$ that we control 
come in a certain hierarchy: The free wave $\gamma^m$ will have a total of 
$s-1$ 
derivatives in the energy space, \emph{however}, at the top order,  two of 
those 
derivatives must be the ``privileged'' $e_0$-direction. The variables $K^m$ 
and $a^m$ will have $s-3$ 
derivatives lying in $L^2$, while the Christoffel symbols  $\Gamma_{AB}^C$
 will have just 
$s-4$ 
derivatives lying in $L^2$. Moreover because of the singular nature of the 
functions $\gamma^m$ at the poles $\theta=0,\pi$, certain derivatives of 
these parameters 
will lie in the same spaces, but with an additional \emph{singular} weight 
${\rm cot}\theta$ which 
blows up at the poles. (Regularity with respect to these enhanced spaces 
captures 
the smoothness of the resulting space-time at those poles).

How we are able to close the energy estimates for $\gamma^m$
in a higher regularity class, relative to that of the coefficients 
$a^m_{Ai}, K_{ij}^m$, and how we can recover these singularly-weighted space estimates for the latter parameters 
will be described in {Sections \ref{sec:itergamma}, \ref{sec:iterh}, which 
deal with  the iterates $\gamma^m,h^m$ respectively.} 
\medskip

{\bf More singular estimates for the higher  derivatives:} As described 
earlier, at the very top 
orders, the estimates we derive for the energies 
of $\partial^{I}e_0 e_0\gamma^m$, $\partial^{I} K^m$ and
 $\partial^{I}a^m$, $|I|=s-3$  are much \emph{worse} (in terms of their singular
  behaviour in $r$) 
 than the bounds we derive at the lower orders. As we will see, 
 beyond  the 
 number ${\rm low}=s-3-4c$  of derivatives, where we obtain the optimal behaviour 
 (fully consistent with 
 the asymptotics \eqref{cl}), there is a \emph{descent scheme} where for each 
 order of regularity 
 {$|I|$}, the bounds we derive are improved by a power 
 $r^{1/4}$ relative to the order {$|I|+1$}. This 
 order-dependent behaviour beyond the lower orders 
  is captured precisely in the function spaces we introduce in the next 
  subsection.

\subsection{Key Constants.}
\label{sec_csts_params}

We discuss here certain key parameters that will be appearing below, in 
our claims on the various parameters that we keep track of, and in our 
derivation of the bounds further down. 
These constants will be universal and in 
 particular, independent of $m\in\mathbb{N}$.
We recall the parameters here, and put down the inequalities that we will
 be imposing on them.

We have already introduced the (small) parameter $\e>0$, which captures
 the hypersurface $\{r=\e\}$ in the Schwarzschild space-time, whose 
 induced data we are perturbing. 

We have also introduced a second (small) parameter 
 $\eta>0$, which captures 
the closeness of our abstract initial data to the Schwarzschild data. 
In particular, $\eta>0$ captures the smallness of the 
\emph{difference} in a renormalized energy space
between our initial data and those of the Schwarzschild background. 
We refer to this quantity as the perturbation size. 

A further constant that will appear is some \emph{fixed, large} number $B\gg1$. $B$ depends on the algebraic forms of the equations, 
(via, for example, the number of terms generated upon commuting our equations with suitable  vector fields below). 
It is also allowed to depend on the mass parameter $M$ of the background Schwarzschild solution that we perturb. 
This constant will never be explicitly calculated, although in principle this is certainly possible.

The next key constant is $C>1$, which captures the \emph{growth} factor of the norms of key parameters in the REVESNGG system. 
$C$ depends on $\e$ and $B$ in an explicit way
\beq
\label{C.def}
C=e^{\int_0^\e 10B^2\tau^{-1+\frac{1}{4}}d\tau}.
\eeq
 Since $B\gg1$ is fixed and independent of any other choice we make, we think of 
$C$ as a function of $\e$: $C=C(\e)$; $C(\e=0)=1$ and $C(\e)$ is a continuous increasing function in $\e$.  

The first key inequality that we demand on $C,\eta$, is that the product 
$C\cdot \eta$ should satisfy an absolute  smallness  bound. To 
present this bound, let us  recall the explicit functions of a 
parameter $\alpha\in\mathbb{R}$
\beq
\label{d2}
d_2(\alpha):=
\frac{\alpha-\frac{3}{2}-\sqrt{(\alpha-\frac{3}{2})^2+6\alpha
-4\alpha^2}}{2},
\eeq
We also consider the parameter $d_1(\alpha)$ 
\beq
\label{d1}
d_1(\alpha):=
\frac{\alpha-\frac{3}{2}+\sqrt{(\alpha-\frac{3}{2})^2+6\alpha
-4\alpha^2}}{2}
\eeq

(Note that $d_2(1)=-1, d_1(1)=\frac{1}{2}$; the significance of $d_2(\alpha)$ in terms of the asymptotics 
of $K_{22}(r,t,\theta)$ has been highlighted in the introduction). 
We then require that for all $\alpha\in [1-C\eta, 1+C\eta]$,
\beq
\label{Ceta.bds}
|d_2(\alpha)+1|\le 1/8, |d_1(\alpha)-\frac{1}{2}|\le 1/8. 
\eeq

As will become manifest in the proof,
 it is this requirement (in fact the first of the two) that ensures the AVTD behaviour of 
our solutions holds. It is also  responsible for the ``gain'' of a power at least $r^{\frac{1}{4}}$ 
of various less singular terms in our inductive estimates below, relative to the ``principal'' singular behaviour of the same terms.

To make this gain manifest further down, we let $D$ to be the sup of the Lipschitz norms of $d_2, d_1$ over $ [1-C\eta, 1+C\eta]$, so in particular: 
\beq
\label{D.def}
|d_2(1+x)+1|\le D|x|, \qquad|d_1(1+x)-\frac{1}{2}|\le D|x|.
\eeq

\begin{remark}
The requirement \eqref{Ceta.bds}
is in fact stronger than what really needs to be imposed, to derive the AVTD behaviour of our solution (and 
to show our result); however, a bound of this
 type \emph{does} need to be imposed; in particular 
the methods here do \emph{not} work for \emph{any} (polarized, axi-symmetric) large perturbation of Schwarzschild. In particular, it is
necessary  for our methods to impose that for some fixed $\delta<\frac{1}{2}$, the inequality  $|d_2(\alpha)+1|\le \delta$, 
holds for all $\alpha\in [1-C\eta, 1+C\eta]$. 
Up to some technical modifications, we believe this follows by essentially the proof we have here, but we do not pursue it in this paper. 
\end{remark}

\medskip
The final key constant that plays a role in our analysis is a constant $c>0$: 
The constant $c>0$ captures the growth of the renormalization power 
(in $r$) at the higher norms (in particular at the top norm).  
In particular, 
 $c>0$ is chosen large enough in order to absorb certain dangerous terms
  at the top order and `close the estimates'. How large $c>0$ is taken 
  depends on the coefficients of the equations and it is determined 
  further down. In particular, we will require: 
\beq
\label{c.bd}
c>20.
\eeq
$c>0$ also determines the energy space in which we will need to bound our 
  initial data, and also in which we will derive bounds for our 
  parameters. In particular, the number $s\in\mathbb{N}$ that deterines the Sobolev spaces $H^s,H^{s-1}$, in which our 
inital data metric $\g$ and $\K$ are to live, is chosen so that: 

\beq
\label{s.bd}
\frac{s-3}{2}{<}s-3-4c\iff s>{8c+3> 163}. 
\eeq

\begin{definition}
  We let ${\rm low}=s-3-4c$; we use ${\rm low}, s-3-4c$
   interchangeably. 
\end{definition}

A second bound on $\e$  that we need to impose (relative to the other parameters $B, \eta$ we have already introduced)
is:

\beq
\label{b.bd}
\frac{C\eta\e^{-\frac{3}{2}}}{2}>B\e^{-1-\frac{1}{4}}\iff \frac{C\eta}{{2}} >B\e^{1/4}.
\eeq
\medskip
In fact, for various technical reasons 
we will strengthen the bound to: 

\beq
\label{e.bd} B^3\e^{1/8}<10^{-1}C\eta.
\eeq

\subsection{Preparatory steps: The interpolating function $\rho^m$ and its 
adapted frames.}
{\bf Orthonormal frame and coordinates}: Recall that we have chosen
 $e_0$ at every step $m$ in the iteration to satisfy the same relation relative to the 
 coordinate $r$, as for the Schwarzschild space-time:
 
 \beq
 \label{e0r}
 e_0(r)= -(\frac{2M}{r}-1)^{\frac{1}{2}}.
 \eeq
Also, given $t,\theta$ coordinate functions on the initial hypersurface, identified for every step $m$, we extended them via
\eqref{e0te0theta}.

These two coordinates $t,\theta$ 
along with the coordinate $r$ provide a coordinate 
system for the $2+1$ metric $h^m$. However, for technical reasons we sometimes need to replace the coordinate $r$ by an $m$-dependent modification: 
\medskip

{\bf The coordinate function $\rho^m$, and the regularity spaces on its level sets.}

The space-like hypersurfaces $\Sigma_r$ (level sets of $r$) are suitable for deriving energy estimates in a neighborhood of the singularity at $r=0$. However, at each step in the iteration, we must adjust our foliation to include the  hypersurface 
$\Sigma_{r_*^m}$, $r_*^m:=r_*^m(t,\theta)\sim \epsilon$,
on which the initial data are to live, and where we are to `start' 
most of our estimates. For this reason, we introduce a modification of the function $r$ near $\{r=\epsilon\}$ to capture this. Let:
\begin{align}\label{rho}
{\rho^m}={\rho}^m(r{,t,\theta})= r-\chi(r)({r}_*^m-\epsilon),&&\chi\in C^\infty([0,2\epsilon]),\;\;\chi\big|_{[0,\frac{\epsilon}{2}]\cup[\frac{3\epsilon}{2},2\epsilon]}\equiv0,\;\chi\big|_{[\frac{3\epsilon}{4},\frac{5\epsilon}{4}]}\equiv1,
\end{align}
for $m\ge0$, ${\rho^0}=r$, and use the level sets of ${\rho^m}$, denoted by $\Sigma_{\rho^m}$, to foliate the region $\{0<r<2\epsilon\}$. 
We will often consider the coordinates $\{\rho^m, t, \theta\}$. {Note that for $r_*^m\sim\epsilon$, the correspondence $\rho^m\leftrightarrow r$ is one to one, for fixed $t,\theta$.}

By definition (\ref{e0te0theta}), the functions $\rho^m,t,\theta$ also constitute a coordinate system.
In this coordinate system:
\[
e_0(t)=0, e_0(\theta)=0.
\]
Note that by definition $\Sigma_{\rho^m}=\Sigma_r$, for $r\in[0,\frac{\epsilon}{2}]\cup[\frac{3\epsilon}{2},2\epsilon]$ and $\{r=r^m_*\}=\{\rho^m=\epsilon\}$. Also, along an $e_0$ geodesic we have
\begin{align}\label{partialrho}
\partial_{\rho^m}=[1-\partial_r\chi(r)(r^m_*-\epsilon)]^{-1}\partial_r=
-(\frac{2M}{r}-1)^{-\frac{1}{2}}[1-\partial_r\chi(r)(r^m_*-\epsilon)]^{-1}e_0.
\end{align}
On the other hand, the future directed ${g^m}$-unit normal to $\Sigma_{\rho^m}$ is given by the ${g}^m$-normalised gradient of ${\rho^m}$:
\begin{align}\label{n}
{n^m}={-}\frac{{\rm grad}_{g^m}{(\rho^m)}}{\sqrt{-{g^m}({\nabla}^m{\rho^m},{\nabla}^m{\rho^m})}}.
\end{align}
In the regions $r\in[0,\frac{\epsilon}{2}]\cup[\frac{3\epsilon}{2},2\epsilon]$, ${n^m}$ coincides with the unit normal to $\Sigma_r$ which can be viewed as a perturbation of $e_0$:
\begin{align}\label{nSigmar}
n^m\big|_{\{0<r<\frac{3M}{4}\}\cup\{\frac{5M}{4}<r<\frac{3M}{2}\}}=\frac{e_0-(\frac{2M}{r}-1)^{-\frac{1}{2}}(e^m_1r)e^m_1-(\frac{2M}{r}-1)^{-\frac{1}{2}}(e^m_2r)e_2^m}{\sqrt{1-(\frac{2M}{r}-1)^{-1}(e^m_1r)^2-(\frac{2M}{r}-1)^{-1}(e^m_2r)^2}}
\end{align}
Generally the future ${g^m}$-unit normal to $\Sigma_{\rho^m}$ reads:
\begin{align}\label{nSigmarho}
{n^m}=\frac{e_0-(e_0{\rho^m})^{-1}(e^m_1{\rho^m})e^m_1-(e_0{\rho^m})^{-1}(e^m_2{\rho^m})e^m_2}{\sqrt{1-(e_0{\rho^m})^{-2}(e^m_1{\rho^m})^2-(e_0{\rho^m})^{-2}(e^m_2{\rho^m})^2}},
\end{align}
while the lapse of the foliation $\Sigma_{\rho^m}$ equals:
\begin{align}\label{Phi}
\Phi^m:=\frac{1}{\sqrt{(e_0{\rho^m})^{2}-(e^m_1{\rho^m})^2-
(e^m_2{\rho^m})^2}}. 
\end{align}
We note here that one of the reasons for requiring the tangency of $e_2^m$ 
to the singularity (in the asymptotic sense \eqref{e2r.van.it}) 
is already apparent here. Had that condition not been imposed, then the 
coefficient of $e_2$ would have been \emph{much} too singular, and would in 
fact be more dominant  in the  energy of $\gamma^m$ than the vector field 
$e_0$, making impossible (and in fact false!)  the derivation of our 
inductive claims. Thus our (gauge) condition forces out 
 this potentially 
more singular  coefficient.

\subsection{Regularity spaces.}\label{Indhyp}
We introduce the spaces in which the various parameters will be measured. 
Recall that 
$\gamma^m$ is studied in the $(3+1)$-dimensional space-time and the bounds 
will be 
using $L^2$-based energies on that space 
$\mathbb{R}\times\mathbb{S}^2\times (0,2\epsilon)$. The parameters 
$K^m_{ij}, a^m_{Ai}$
 will be studied 
on the $(2+1)$-dimensional space $\mathbb{R}\times(0,\pi)\times(0,2\epsilon)$. These
parameters also will be bounded in $L^2$-based spaces, with respect to the volume form $sin\theta d\theta dt$. In the instances where we use a different volume form, we will spell it out explicitly. 

\medskip

{\bf Notation}: We will defining $L^2$-based spaces on level sets of 
$\rho^m$. Thus, the functions will depend on 
$t\in (-\infty,+\infty),\theta\in (0,\pi),\phi\in [0,2\pi)$; 
 all functions will be $\phi$-independent, so sometimes we will omit 
 $\phi$ altogether.  We specify that the volume form will be 
 $\sin\theta d\theta dt d\phi$, unless otherwise stated (the canonical
  volume form on $\mathbb{S}^2\times\mathbb{R}$, for $\phi$-independent functions). We sometimes denote this also by {${\rm vol}_{\rm Euc}$}. 
  \medskip

Given a smooth function $\psi({\rho^m},t,\theta):\{\Sigma_{\rho^m}\}_{{\rho^m}\in(0,2\epsilon]}\to\mathbb{R}$, we define the energy
\begin{align}\label{Epsi}
E[\psi({\rho}^m,t,\theta)]=\int_{\Sigma_{\rho^m}}\big[(e_0\psi)^2+|\overline{\nabla}^m\psi|^2\big]\mathrm{vol}_{Euc}.
\end{align}
Here 
$|\overline{\nabla}^m\psi|^2$ stands for 
$|{e}_1^m\psi|^2+|{e}_2^m\psi|^2$, for any 
$h^m$-orthonormal frame ${e}_1^m, {e}_2^m$ {orthogonal to $e_0$}. We also
 define the $H^l$ norm
\begin{align}\label{Hsnormpsi}
\|\psi\|_{H^l[{\rho^m}]}:=\bigg(\sum_{|I|\leq l}
\int_{\Sigma_{\rho^m}}(\partial^I\psi)^2\mathrm{vol}_{Euc}
\bigg)^\frac{1}{2},
\end{align}
{where $\partial^I$ stands for a combination of $\partial_t,\partial_\theta$ derivatives dictated by the multi-index $I$.}\\
 \medskip

For some of our parameters we will be using a slight variant of these standard 
 Sobolev spaces; the variants are taylored to capture some delicate behaviour of our parameters at the poles $\theta=0,\theta=\pi$. (Had we been studying our 
system away from the poles, the Sobolev spaces introduced just above would have been sufficient). 
 
\newcommand{\DS}{{\Delta_{\mathbb{S}^2}}} 
 
 Let us introduce 
 the 1-st order operator 
 \newcommand{\opartial}{\overline{\partial}}
\beq \label{opartial}
\tilde{\partial}_\theta=\partial_\theta+\frac{cos\theta}{sin\theta},
\eeq 
as well as the second order operator 
\[
\DS= \tilde{\partial}_\theta \partial_\theta. 
\]
Note that $\DS$  agrees with the standard round Laplacian acting on $\phi$-independent functions on $\mathbb{S}^2$. 

Based on this, we will introduce the higher-order operators $\opartial^I_{\theta \dots \theta t \dots t}$, for any multi-index $I$ consisting of an \emph{even} number of $\theta$'s and any number of $t$'s: 
Letting $2k_1$ be the (even number) of $\theta$'s and $k_2$ be the number of $t$'s in the multi-index $I$, we define: 

\beq
\label{opartial.ext}
\opartial^I_{\theta \dots \theta t\dots t}v=(\DS)^{k_1}\partial^{k_2}_{t\dots t}v. 
\eeq
In other words, $\opartial^I$ differs from $\partial$ only for the $\theta$-indices;
for the $t$-indices it agrees with $\partial_t$. We will see further down how $L^2_{sin\theta d\theta dt}$ control of the $\opartial^I$ derivatives of a function (usually 
$\gamma^m_{\rm rest}$ for most of this paper) yields control of the same function in the standard Sobolev spaces $H^k_{sin\theta \theta dt}$, 
$k=|I|$. 
We also introduce the associated norm to this operator: 

\beq
\label{oHk}
\|\psi\|_{\oH^k}[\rho^m]= \sum_{|I|\le k, I=(2k_1, k_2)}\int_{\Sigma_{\rho^m}}(\opartial^I\psi)^2sin\theta d\theta dt. 
\eeq
The definition of the analogous homogenous norm $\dot{\oH}^k$ is immediate. 
\medskip
  We will note in Lemma \ref{lem:interpolation} the equivalence 
  of this $\oH^k$  norm with the standard norm $H^k$. 

      \newcommand{\tpartial}{\widetilde{\partial}}

Finally, we make a final convention: For technical reasons, we will be requiring that
 $s\in\mathbb{N}$ be an odd number, and in particular $s-3$ is an even number. 

\subsection{The Inductive claim for all parameters in the REVESNGG.}
\label{sec_REVESNGG}

We present here the inductive claims on the parameters we solve for in
 the REVESNGG. These claims are verified trivially at the 0th-step by the Schwarzschild variables 
$\gamma_{\rm S}=\gamma^0,(K_{\rm S})_{ij}=K^0_{ij},(a_{\rm S})_{Ai}=a^0_{Ai}$,
$i,j=1,2$, where 
$e^0_1\parallel\partial_t$, $e^0_2\parallel\partial_\theta$, 
cf. \S\ref{subsec:Schw} and \S\ref{subsubsec:0step}.
 
  In all parameters that are functions of $r,t,\theta$, there 
 will be a key distinction between the lower orders $k\le {\rm low}$ and 
 the higher derivatives. 
 
 At the lower orders, the inductive claim is substantially 
 stronger and it involves proving \emph{optimal} asymptotic behaviours, as $r\to 0$. 
 At the higher derivatives, the bounds we claim are weaker; in fact, 
 for each derivative beyond the lower ones, the bounds we claim 
 become more 
 singular by a fixed amount. Even in these very singular spaces, 
 the closeness to the Schwarzschild background is part of what is being
  claimed, albeit in a weaker sense compared to the lower norms. 
\medskip

 A few general comments: Firstly, we present the claims for the step 
 $m-1$. We will then verify the validity of the claim for
  $m\in\mathbb{N}$. Secondly, most of the estimates in our inductive 
  assumptions will be broken in three categories, depending on the number 
  of derivatives on the various quantities: 
  
   There will be the \emph{low orders} where there are a total of 
   ${\rm low}:=s-3-4c$ derivatives on the various quantities. At those orders we 
   claim what is (for $\gamma^{{m-1}}$ and $K^{{m-1}}_{ij}$) the optimal behaviour {(for their leading orders)}.
   At the next `higher' orders, where $s-3-4c<l\le s-4$, we claim 
   bounds which are more singular in terms of powers of $r$; each 
   derivative beyond the optimal orders ``costs'' a power $-\frac{1}{4}$ 
   in $r$. Finally, at the top orders, we take $s-3$ 
   coordinate derivatives, 
   and then up to two $e_0$ derivatives.
\medskip

{\bf Convention:} All the inductive statements we write below will be for 
the step $m-1$ in the induction. The statements will then have to be 
verified for the $m^{\rm th}$ step. We note that  all estimates below 
(for the step 
$m-1$) will be assumed to hold on level sets of the function 
$\rho^{m-1}$. (We will see in the validation of the inductive step $m$ below that some of the bounds will be derived on 
the level sets of $\rho^{m-1}$ or of $r$; however once the function $r^m_*$ has been sooved for, we can derive that the same bounds will hold on level sets of 
$\rho^m$).
Moreover, we will be taking the difference of these 
parameters from the corresponding values in the  Schwarzschild 
space-time. Here $\gamma^S(t,\theta,\rho^{m-1})$ evaluated at values
$t=a,\theta=b, \rho^{m-1}=c$ is identified with the 
value $\gamma_{\rm S}$ at $(t=a, \theta=b, r=c)$ in the standard $t,\theta, r$ coordinates. 
The same convention applies to all other quantities 
($K^{m-1}_{ij}$ etc.) below.

\subsubsection{Inductive claim for $\gamma^{m-1}$}

We assume certain  energy estimates are satisfied by $\gamma^{m-1}$ 
across all level sets of the functions $r$ and $\rho^{m-2}$, $m\ge2$.
Precisely the same estimates are true on level sets of $\rho^{m-1}$,
with $r$ replaced by $(\rho^{m-1})$ in the RHSs. 
\begin{align}
\label{inductiongammaopt}\sqrt{E[\partial^I(\gamma^{m-1}-\gamma^S)]}
\leq C\cdot\eta\cdot r^{-\frac{3}{2}},\qquad\text{for}\;\;\; |I|\leq s-3-4c,\\
\label{inductiongammalow}
\sqrt{E[\partial^I(\gamma^{m-1}-\gamma^S)]}
\leq C\cdot \eta\cdot r^{-\frac{3}{2}+(s-3-|I|)\frac{1}{4}-c},\qquad
\text{for}\;\;\; s-3-4c<|I|\leq s-4,\\
\label{inductiongammatopmixed}
\sqrt{E
[\partial^Ie_0^{J_0}(\gamma^{m-1}-\gamma^S)]}
\leq C\cdot\eta\cdot r^{-\frac{3}{2}-\frac{3}{2}|J_0| -c},\;\;|J_0|\leq2, \text{ }
{\rm for}\text{ } |I|=s-3, I\ne (T, T, \dots, T).
\end{align}
\emph{Note:} At the very last estimates, we \emph{exclude} the top order derivatives, 
when all the $s-3$ derivatives are in one ``less regular'' direction 
$\partial_T=\partial_{T^{m-2}}$ direction\footnote{This $\partial_{T^{m-1}}$ is to illustrate that 
the direction depends on the step $m-1$ in our induction} that we will introduce in \S\ref{sec:met.Christ.bds} below.  
%

\begin{remark}
We note that there is an $r^{-c}$, $c>0$, loss in the blow up behavior of 
$\gamma^{m-1}$ at the top derivatives, which is improved at each lower order until the order $s-3-4c$. 
At that (low) order the behaviour in terms of powers of $r$ is optimal. 
\end{remark}

\begin{definition}
Below we will be using the notation
\[
\gamma^{m-1}_{\rm rest}(\rho^{m-2},t,\theta):=
\gamma^{m-1}(\rho^{m-2},t,\theta)-\gamma^{\rm S}(\rho^{m-2},t,\theta).
\]
\end{definition}
(This implicitly also defines $\gamma^m_{\rm rest}$ in the coordinates $r,t,\theta$ also--this transition of considering the 
variables here with respct to the coordniate systems $\{ \rho^{m-2},t,\theta\}$, or $\{r,t,\theta \}$). 

Furthermore, at the lower orders we make a stronger claim: 
We claim that $\gamma^{m-1}$ has the following expansion at the 
lower orders $l\le s-3-4c$:
\begin{align}\label{gammam-1exp}
\begin{split}
\gamma^{m-1}(r,t,\theta)=&\,\alpha^{m-1}(t,\theta)\log r+
\gamma_{1}^{m-1}(r,t,\theta),\\
e_0 \gamma^{m-1}(r,t,\theta)=&-{(\frac{2M}{r}-1)^\frac{1}{2}\frac{\alpha^{m-1}(t,\theta)}{r}}+
{e_0}\gamma_{1}^{m-1}(r,t,\theta),
\end{split}
\end{align}
for a function $\alpha^{m-1}(t,\theta)\in H^{s-3-4c}$ that verifies the pointwise bound
\begin{align}\label{alpham-1assum}
\|\alpha^{m-1}(t,\theta)-1\|_{L^\infty}\leq C \eta \ll1.
\end{align}
Moreover  $\gamma_1^{m-1}(r,t,\theta)$ (the `leftover term')
satisfies 
\begin{align}\label{gamma1m-1enest}
r^3\|e_0\gamma^{m-1}_1\|^2_{H^{s-3-4c}}\leq&\, 
 B^2 
r^\frac{1}{2},
\end{align}
for all $r\in(0,2\epsilon]$. In particular, the `leftover term'
is strictly less singular (at the lower derivatives) 
 than the `main term' $\alpha^{m-1}(t,\theta)\log r$ in
  \eqref{gammam-1exp}.

\medskip

\emph{Improved behaviour of the Hessian terms}  

We make certain further claims on the functions $\gamma^{m-1}$, which 
stem from the AVTD behaviour of these solutions (at the orders below the top ones). 
These will also be important
in making the optimal inductive claims on the asymptotic behaviours 
of the connection coefficients $K^{m-1}_{ij}$, {since they appear in the RHS of the Riccati system \eqref{finredEVERic11it}-\eqref{finredEVERic12it}.} 

The terms we will seek to bound at the step $m$ are: 
\newcommand{\onabla}{\overline{\nabla}}
\beq
\begin{split}
\label{Riccati.RHSs}
{{\rm Hess}(\gamma^m)\in}\,\big\{&\,\onabla_{22}^{m-1}\gamma^m+({e}^{m-1}_2\gamma^{m-1})
({e}^{m-1}_2
\gamma^m),\;\;\onabla_{11}^{m-1}\gamma^m+({e}^{m-1}_1\gamma^{m-1})
({e}^{m-1}_1
\gamma^m),
\\
&\onabla_{12}^{m-1}\gamma^m+\frac{1}{2}
[({e}^{m-1}_1\gamma^{m-1})
({e}^{m-1}_2
\gamma^m)+({e}^{m-1}_2\gamma^{m-1})
({e}^{m-1}_1\gamma^m)]\,\big\}.
\end{split}
\eeq
We seek to bound these expressions in the spaces $H^l$, $l\le s-4$.

{The bounds we claim, for the corresponding quantities of the step $m-1$, are} as follows:  {
\begin{align}\label{inductionHess}
\begin{split}
\|{\rm Hess}(\gamma^{m-1})\|_{H^{s-3-4c}}\leq&\, B r^{-2-\frac{3}{4}},\\
\|{\rm Hess}(\gamma^{m-1})\|_{\dot{H}^l}\leq&\,B r^{-2-\frac{3}{4}+(s-3-l)\frac{1}{4}-c},\qquad s-3-4c< l\leq s-4,
\end{split}
\end{align}
cf. Lemma \ref{hessian.bounds.again}.}

We remark that these estimates are not optimal, but they suffice for us to prove the results we want. 
We note for example that we can prove that at the lower orders $l\le {\rm low}-{2}$, the third term in \eqref{Riccati.RHSs} above satisfies the stronger estimate: 
\beq
\label{K12.future}
\|\onabla_{12}^{m-1}\gamma^m+\frac{1}{2}
[({e}^{m-1}_1\gamma^{m-1})
({e}^{m-1}_2
\gamma^m)+({e}^{m-1}_2\gamma^{m-1})
({e}^{m-1}_1\gamma^m)]\|_{{H^{{\rm low}-2}}}\le B {r^{-\frac{1}{2}-\frac{1}{4}}},
\eeq  
{which would yield a much more improved behaviour for $K_{12}^{m-1}$ than the one we claim below, cf. \S\ref{subsec:REVESNGGimpl}.}
However, this is not needed to close our estimates (i.e., to prove our induction). 

\subsubsection{Inductive claims on $r^{m-1}_*(t,\theta)$, 
$\tilde{K}^{m-1}_{12}(t,\theta)$.}\label{subsec:indrstar}

The functions $r^{m-1}_*(t,\theta)$ are required to satisfy 
$\partial_\theta r^m_*=0$ at the two poles $\theta=0,\theta=\pi$. Then 
the functions $r^{m-1}_*(t,\theta)$, $\tilde{K}^{m-1}_{12}(t,\theta)$
 will satisfy the following bounds at the lower derivatives: 

For $k_1+k_2=k\le \rm low$ we have:
\begin{align}
\label{r*.ind.claim.low}
\sqrt{\int_{-\infty}^\infty\int_0^\pi |\partial^{k_1+k_2}_{t\dots t\theta\dots\theta}
(r^{m-1}_*-\e)|^2\sin\theta d\theta dt} \le&\, (D+1)C\eta \epsilon,\qquad k_1+k_2=k\le {\rm low},\\
\label{tK12.ind.claim.low}
\sqrt{\int_{-\infty}^\infty\int_0^\pi |\partial^{k_1+k_2}_{t\dots t\theta\dots
\theta}
(\tilde{K}^{m-1}_{12})|^2\sin\theta d\theta dt} \le&\, C\eta 
\epsilon^{-3/2+\frac{1}{4}},\qquad k_1+k_2=k\le {\rm low}. 
\end{align}

For the higher derivatives ${\rm low}< k\le s-4$, letting 
$h=k-(s-3-4c)$ we 
have: 
\begin{align}
\label{r*.ind.claim.high}
\sqrt{\int_{-\infty}^\infty\int_0^\pi  |\partial^{k_1+k_2}_{t\dots t\theta
\dots
\theta}(r^{m-1}_*-\e)
|^2\sin\theta d\theta dt} \le &\,3C\eta \epsilon^{1-\frac{1}{4}\cdot h},\\
\label{tK12.ind.claim.high}
\sqrt{\int_{-\infty}^\infty\int_0^\pi
|\partial^{k_1+k_2}_{t\dots t
 \theta\dots\theta}(\tilde{K}^{m-1}_{12})
|^2\sin\theta d\theta dt} \le &\,C\eta \epsilon^{-3/2+\frac{1}{4}-\frac{1}{4}\cdot h}.
\end{align}
Moreover we claim the following top-order estimate, when
 $k_1+k_2=s-3$ on the first two lines and $k_1+k_2$ on the last two lines:

\begin{align}
\label{r*.ind.claim.top}
\sqrt{\int_{-\infty}^\infty\int_0^\pi  |\partial^{k_1+k_2}_{t\dots t\theta
\dots
\theta}(r^{m-1}_*-\e)
|^2\sin\theta d\theta dt} \le &\,3C\eta \epsilon^{1-c},\\
\label{tK12.ind.claim.high}
\sqrt{\int_{-\infty}^\infty\int_0^\pi
|\partial^{k_1+k_2}_{t\dots t
 \theta\dots\theta}(\tilde{K}^{m-1}_{12})
|^2\sin\theta d\theta dt} \le &\,C\eta \epsilon^{-3/2-c},\\
\notag\sqrt{\int_{-\infty}^\infty\int_0^\pi | 
\partial^{k_1+k_2}_{t\dots t\theta\dots
\theta}[\partial_\theta(r^{m-1}_*-\e)
\cdot {\rm cot}\theta ]|^2\sin\theta d\theta dt} \le &\,3C\eta
 \epsilon^{1-c},\\
\label{tK12.ind.claim.top}
\sqrt{\int_{-\infty}^\infty\int_0^\pi
 | \partial^{k_1+k_2}_{t\dots t
 \theta\dots\theta}[(\tilde{K}^{m-1}_{12})\cdot {\rm cot}\theta ]
|^2\sin\theta d\theta dt} \le &\,C\eta \epsilon^{-3/2-c}.
\end{align}
\medskip
(the last inequality holds provided $k_2>0$). 

The  bounds claimed for $\tilde{K}^{m-1}_{12}(t,\theta)$ are also 
claimed for the expression 
$[\te^{m-1}_2 \tilde{K}^{m-1}_{12}(t,\theta)\cdot \te^{m-1}_2(r^{m-1}_*)(t,\theta)]$, 
with a factor of 2 in the RHS.

We also have some inductive claims on the $\te^{m-1}_2$-derivative of the 
function $r^m_*$; at the low derivatives $k\le \rm low$ 
our claim is as follows: 
\beq
\label{te2.r*.low}
\sqrt{\int_{-\infty}^\infty\int_0^\pi |\partial^{k_1+k_2}_{t\dots t\theta
\dots\theta}
[\te_2^{m-1}(r_*^{m-1}-\e)]|^2\sin\theta d\theta dt}\le 3C\eta 
\e^{-\frac{1}{2}}.
\eeq
At the high derivatives ${\rm low}< k\le s-3$ the corresponding claim is: 
\beq
\label{te2.r*.high}
\sqrt{\int_{-\infty}^\infty\int_0^\pi  |\partial^{k_1+k_2}_{t\dots t\theta
\dots
\theta}\te^{m-1}_2({r^{m-1}_*}-\e)
|^2\sin\theta d\theta dt} \le 3C\eta \epsilon^{{-\frac{1}{2}}-\frac{1}{4}
\cdot h}.
\eeq

 Note that
 the 
equations \eqref{key.link2}, 
\eqref{key.link2'}  for $r^{m-1}_*$ in the REVESNGG  then imply that the 
function $r^{m-1}_*$ must  have an even  expansion (in $\theta$) 
at the poles $\theta=0,\pi$.  We also note that as a consequence of the 
equations,  the function $\tilde{K}^{m-1}_{12}$ will also vanish 
at the poles, along with all its derivatives which are even in $\theta$. 

\subsubsection{Inductive claims for $K^{m-1}$}
\label{sec:Kassns}

We commence here with the behaviour of the components 
$K^{m-1}_{ij}(r,t,\theta)$ at the lower orders.  

We observe that in view of the AVTD-type assumption \eqref{inductionHess} at the end of the 
inductive claim on $\gamma^{m-1}$, the most singular terms in the RHSs of 
\eqref{finredEVERic11it},  \eqref{finredEVERic22it}, 
\eqref{finredEVERic12it} (in terms of behaviour in $r$) are the terms 
$e_0^2\gamma^m, (e_0\gamma^m)^2$ ({provided we can confirm the above inductive claims for $\gamma^m$}). This is true at all orders below the 
top ones. 
\medskip

At the lowest orders, {the validity of the} inductive claims 
\eqref{gammam-1exp}, 
\eqref{gamma1m-1enest}, as well as \eqref{inductionHess}, {at the $m^{\rm th}$ step}, 
imply that the RHSs of these equations 
satisfy the following asymptotic expansion in $H^{low}$:
\begin{align*}
\|{\rm RHS}[\eqref{finredEVERic11it}]{-[\frac{3}{2}\alpha^m-
(\alpha^m)^2]} 2Mr^{-3}\|_{H^{\rm low}}=&\, {O(r^{-3+\frac{1}{4}}),}\\
\|{\rm RHS}[\eqref{finredEVERic22it}]-[\frac{3}{2}\alpha^m-
(\alpha^m)^2] 2Mr^{-3} \|_{H^{\rm low}}= &\,O(r^{-3+\frac{1}{4}}),\\
\|{\rm RHS}[\eqref{finredEVERic12it}]\|_{H^{\rm low}}= &\,O(r^{-3+\frac{1}{4}}).
\end{align*}
In particular, this implies that \emph{formally}, 
the equations \eqref{finredEVERic11it},  \eqref{finredEVERic22it}, 
\eqref{finredEVERic12it} admit solutions of the form: 
\begin{align}\label{K11heur}
K_{11}^{m-1}(r,t,\theta)=:&\,\frac{d_1^{m-1}(t,\theta)\sqrt{2M}}{r^\frac{3}{2}}+u_{11}^{m-1}(r,t,\theta), \\
\label{K22heur}
 K_{22}^{m-1}(r,t,\theta)=:&\,\frac{d_2^{m-1}(t,\theta)\sqrt{2M}}{r^\frac{3}{2}}+u_{22}^{m-1}(r,t,\theta),
\end{align}

\begin{align}
 \label{K12heur}
 K_{12}^{m-1}(r,t,\theta)=:u_{12}^{m-1}(r,t,\theta),
\end{align}
where
\begin{align}
\label{d1.m} 
d_1^{m-1}(t,
\theta):=&\,\frac{\alpha^{m-1}(t,\theta)-\frac{3}{2}+\sqrt{(\alpha^{m-1}
(t,\theta)
-\frac{3}{2})^2
+6\alpha^{m-1}(t,\theta)-4|\alpha^{m-1}(t,\theta)|^2}}{2},\\
\label{d2.m} d_2^{m-1}(t,\theta):=&\,
\frac{\alpha^{m-1}(t,\theta)-\frac{3}{2}-\sqrt{(\alpha^{m-1}(t,\theta)
-\frac{3}{2})^2+6\alpha^{m-1}(t,\theta)
-4|\alpha^{m-1}(t,\theta)|^2}}{2},
\end{align}
and furthermore,
 the ``remainder'' 
 terms $u_{22}^{m-1}(r,t,\theta), u_{11}^{m-1}(r,t,\theta)$,
$u_{12}^{m-1}(r,t,\theta)$ are $O(r^{-\frac{3}{2}+{\frac{1}{4}}})$ in 
$H^{\rm low}$. 
\medskip

 Our claim at the lower orders is that this \emph{formal} solution is in fact true. 
This will be part of our inductive claim at the lower 
orders:\footnote{At the 
higher orders, our estimates do not distinguish between a \emph{leading order} and a \emph{remainder} .}

Then the inductive claim that we make for $K^{m-1}_{ij}(r,t,\theta)$, 
$i,j=1,2$, is at the lower orders: 
\begin{align}
\label{inductiontrKopt}\|u^{m-1}_{ij}(r,t,\theta)\|_{\dot{H}^l}\leq 5
\cdot B r^{-\frac{3}{2}+\frac{1}{4}},\qquad i,j=1,2,\quad l \le {s-3-4c},
\end{align}
while at the higher orders we claim instead:
\begin{align}
\label{inductiontrKlow}&\|(K_{ii}^{m-1}-K_{ii}^{S})(r,t,\theta)\|_{\dot{H}^l}\leq 2\cdot C{\eta}\cdot  r^{-\frac{3}{2}+(s-3-l)\frac{1}{4}-c},\qquad\text{for}\;\;\;s-3-4c<l\leq s-4\\
\notag&\|K_{12}^{m-1}(r,t,\theta)\|_{\dot{H}^l}\leq 2\cdot C{\eta}\cdot  
r^{-\frac{3}{2}+\frac{1}{4}+(s-3-l)\frac{1}{4}-c},\qquad\text{for}\;\;\;s-3-4c<l\leq s-4.
\end{align}
(In particular there is no division into a ``principal term'' $d^{m-1}_i(t,\theta)r^{-\frac{3}{2}}$ and a 
``remainder term'' $u^{m-1}_{ij}$ beyond   the lower orders). 
The top order estimates for $K^{m-1}$ are when $|I|=s-3$ 
\emph{and} we allow the  possibility  of
allowing the singular weight ${\rm cot}\theta$ in our norm.

Also, at the top order we make separate claims for 
$K^{m-1}_{22},K^{m-1}_{22}$, and $K_{11}^{m-1}$.
 For the first two we claim, for all $J_0\le 2$:
\begin{align}\label{inductiontrKtopmixed}
\begin{split}
\|e_0^{J_0}(K_{22}^{m-1}-K_{22}^{S})(r,t,\theta)\|_{\dot{H}^{s-3}}\leq&\, 3\cdot 
C {\eta}\cdot 
 r^{-\frac{3}{2}-\frac{3}{2}(J_0) -c},\\
 \|e_0^{J_0}(K_{12}^{m-1})(r,t,\theta)\|_{\dot{H}^{s-3}}\leq&\, 3\cdot 
C {\eta}\cdot 
 r^{-\frac{3}{2}-\frac{3}{2}(J_0) -c}.
 \end{split}
\end{align}
The enhanced version of this claim with the singular weight at the poles 
is as follows: 
\begin{align}\label{inductiontrKtopmixed.sing}
\begin{split}
\|e_0^{J_0}\partial_\Theta(K_{22}^{m-1}-K_{22}^{S})(r,t,\theta)\cdot 
{\rm cot}\theta \|_{\dot{H}^{s-4}}\leq&\, 3\cdot 
C {\eta}\cdot 
 r^{-\frac{3}{2}-\frac{3}{2}(J_0) -c},\\
 \|e_0^{J_0}K_{12}^{m-1}(r,t,\theta)\cdot 
{\rm cot}\theta 
 \|_{\dot{H}^{s-3}}\leq&\, 3\cdot 
C {\eta}\cdot 
 r^{-\frac{3}{2}-\frac{3}{2}(J_0) -c}.
 \end{split}
\end{align}
\medskip

The claim at the top order for $K^{m-1}_{11}(r,t,\theta)$ is slightly weaker: In particular we claim the above bounds, for all top-order derivatives 
$\partial^IK^{m-1}_{11}$, $|I|=s-3$ 
\emph{except} when all the $s-3$ derivatives are in a special direction 
$\partial_T=\partial_{T^{m-1}}$ which will be specified in  \S\ref{sec:met.Christ.bds}.  
In particular we claim:

\begin{align}\label{inductiontrKtopmixed2}
\begin{split}
{ \big\| e_0^{J_0}\big[\partial^I
[K_{11}^{m-1}-K_{11}^{S}]]\big\|_{L^2}}
\leq 3\cdot C \eta\cdot 
r^{-\frac{3}{2}-\frac{3}{2}|J_0| -c},
\end{split}
\end{align}
for all $|J_0|\le 2$ and for all derivatives of order $s-3$ \emph{except} for the case 
where all $s-3$ directions are in the direction $\partial_T=\partial_{T^{m-1}}$ 
introduced in \S\ref{sec:met.Christ.bds}. 

Furthermore, in analogy with the enhanced top order-estimates
 \eqref{inductiontrKtopmixed.sing} for $K^{m-1}_{22}$ we have the inductive claim: 

\begin{align}\label{inductiontrKtopmixed.sing2}
\begin{split}
\|e_0^{J_0}\partial_\Theta(K_{11}^{m-1}-K_{11}^{S})(r,t,\theta)\cdot 
{\rm cot}\theta \|_{\dot{H}^{s-4}}\leq&\, 8\cdot 
C {\eta}\cdot 
 r^{-\frac{3}{2}-\frac{3}{2}(|J_0|) -c}.
 \end{split}
\end{align}

Let us make a few remarks here about the top order energy estimates on
 these key components to our analysis:

\begin{remark}
\label{top.remarks.orders}
We note that at the top orders for $K^{m-1}$ we include the 
singular weight ${\rm cot}\theta$ in certain of our 
 our top order  estimates. 

 This is in contrast to the estimates for 
$\gamma^{m-1}$ where this singular weight is absent, even at the top 
orders.  The reason we are able to control $K^{m-1}$ with this extra weight 
 is because of the \emph{energy} of $\gamma$ that we control 
 at the top orders involves the spatial direction $e_2$ (which is parallel to $\partial_\Theta$); this will allow 
us to control the singular weight with an application of the Hardy
 inequality further down.

Even beyond this issue, we note the absence of control of \emph{spatial} 
parts of the metric iterates $g^{m-1}_{AB}$, $A,B=T,\Theta$, from our inductive assumptions. 
This ``spatial'' control (\emph{including} a suitable version of the 
 singular weights at the top orders), will  be obtained in the next 
 section, where the control of the spatial components of the metric 
 $g^{m-1}$ will be derived from the bounds we have on $K^{m-1}_{ij}$ and 
 suitable transport equations in an $m$-dependent gauge that we control.

 The derivation  of how 
 control of $K^{m-1}$ and the initial data parameters
  $r^{m-1}_*$, $\tilde{K}^{m-1}_{12}$ yields  spatial regularity of 
  $g^{m-1}$   will be obtained in sections \ref{sec_geom.params} and \ref{sec:met.Christ.bds}.  
\end{remark}  
 
\begin{remark}
\label{top.remarks}
We note that at the very top orders, 
 the 
factors $5\cdot$, $8\cdot$, which are to be compared with the absence of such 
factors in {\eqref{inductiongammalow}-}\eqref{alpham-1assum}, is due to the algebraic structure of the (differentiated) 
Riccati equations, at the middle and top orders. In particular, the  number and coefficients of the 
 most singular terms\footnote{These are related to the notion of ``borderline terms'' we introduce below.}
  are what leads to this extra factor. 
 
 \end{remark}
\medskip

\subsubsection{The asymptotically CMC property of level sets of $r$. }
We remark that by the Sobolev embedding 
$H^2{(\mathbb{S}^2_{\theta,\phi}\times\mathbb{R})}\hookrightarrow L^\infty{(\mathbb{S}^2_{\theta,\phi}\times\mathbb{R})}$ and 
(\ref{inductiontrKopt}) we also have the pointwise bound 
$\|u^{m-1}_{ij}\|_{L^\infty}\le {C_{\rm Sob}}\cdot5Br^{-\frac{3}{2}+\frac{1}{4}}$, since  $s\ge 3+4c+2$. ($C_{\rm Sob}$ is the constant in the
above  Sobolev embedding). 

On the other hand, {by the formula $K_{33}^{m-1}=e_0\gamma^{m-1}$ in polarized axial symmetry and the claim \eqref{gamma1m-1enest}, we
 also deduce a bound on $K^{m-1}_{33}$:\footnote{$K_{31}^{m-1},K_{32}^{m-1}$ are automatically zero.}} 
\begin{align}\label{K33heur}
\begin{split}
K_{33}^{m-1}(r,t,\theta)=&\;e_0\gamma^{m-1}(r,t,\theta)=:
-\frac{\alpha^{m-1}(t,\theta)\sqrt{2M}}{r^\frac{3}{2}}+u_{33}^m(r,t,\theta)\\
\overset{(\ref{gammam-1exp})}{=}&-(\frac{2M}{r}-1)^\frac{1}{2}\frac{\alpha^{m-1}(t,\theta)}{r}+e_0\gamma^{m-1}_1(r,t,\theta),
\end{split}
\end{align}
where
\begin{align}\label{u33m-1enest}
\|u^{m-1}_{33}\|_{H^{s-3-4c}}\leq \|e_0\gamma^{m-1}_1\|_{H^{s-3-4c}}
+\frac{2}{\sqrt{2M}}\|\alpha^{m-1}\|_{H^{s-3-4c}}r^{-\frac{1}{2}}\leq  {5\cdot}B
r^{-\frac{3}{2}+\frac{1}{4}},
\end{align}
and in particular, $\|u^{m-1}_{33}\|_{L^\infty}\leq 
C_{\rm Sob}\cdot Br^{-\frac{3}{2}+\frac{1}{4}}$, for all
 $r\in(0,2\epsilon]$.

Observe, in view of the behavior of $K^{m-1}_{ij}$ at lower orders,  that by taking the 3-dim trace of $K^{m-1}$ we obtain
\begin{align}\label{trKheur}
\text{tr}_{\overline{g}^{m-1}}K^{m-1}=-\frac{3}{2}\frac{\sqrt{2M}}{r^\frac{3}{2}}+\text{tr}_{\overline{g}^{m-1}}u^{m-1}.
\end{align}
Thus, by (\ref{inductiontrKopt}), (\ref{u33m-1enest}), $\text{tr}_{\overline{g}^{m-1}}K^{m-1}$ is \emph{constant} to order $r^{-\frac{3}{2}+\frac{1}{4}}$ in the norm $\|\cdot\|_{H^{s-3-4c}}$, as $r\rightarrow0$. This uniformity of $\text{tr}_{\overline{g}^{m-1}}K^{m-1}$ plays a central role in the derivations of the energy estimates for $\gamma^m$ below and it is one of the key ingredients to deriving its logarithmic blow up, see \S\ref{sec:bott-lowenestgammam}.
\medskip

In particular,  a consequence of our inductive assumptions is that: 

\beq
\label{tru.explicit}
\|{\rm tr}_{\overline{g}^{m-1}} u^{m-1} \|_{H^{s-3-4c}}\le 15
 Br^{-\frac{3}{2}+\frac{1}{4}},
\eeq
for all $r\in(0,2\epsilon]$.
\medskip 

{{\bf Improved behaviour of $K_{12}^{m-1}$}: Up to order ${\rm low}-2$, $K_{12}^{m-1}$ in fact satisfies the stronger bound
\begin{align}\label{K12m-1imp}
\|K_{12}^{m-1}\|_{H^{{\rm low}-2}}\leq 5\cdot Br^{-\frac{1}{2}-\frac{1}{4}}
\end{align}
Although we do not need the improved behaviour of $K_{12}^{m-1}$ to close our estimates below, we find it convenient to put it down here, in order to infer directly the better
 behaviour of the metric in the $r,T,\Theta$ coordinates, see Theorem \ref{thm_strict} and its proof in \S\ref{subsec:REVESNGGimpl}. 
We verify \eqref{K12m-1imp} for the step $m$ in \S\ref{subsec:K12K22asym}.}

\subsubsection{Remark on the regularity spaces}
\label{sec:reg.spaces}

As seen in our inductive statements,  this choice of function spaces comes 
at the cost of more singular 
estimates at the top order (relative to two orders below the top).
 This can be seen, for example, by comparing 
\eqref{inductiongammatopmixed} with \eqref{inductiongammalow}. {The worse 
behaviour in $r$ at the higher orders} is remedied by a descent scheme in 
the $r$-weights, used at many points in 
this paper, which in turn exploits the AVTD behaviour of our solution
in an essential way.

\subsubsection{The passage to the limit $m\to \infty$.}

The above estimates suffice to show the boundedness of the iterates in the 
REVESNGG system. We can then consider \emph{differences} between
 corresponding variables in successive steps in our iteration, establishing 
 that  the iteration defines a contraction mapping for the terms in the 
 REVESNGG system in the corresponding spaces ($\gamma^m-\gamma^{m-1}$ in 
$H^{s-1}$, etc). This then furnishes a solution to the coupled 
REVESNGG system with our prescribed {(smooth)} initial data.
Then, the standard uniqueness result for the EVE implies that this is the 
unique {(smooth)} solution to 
our problem. Moreover, 
 at the lower derivatives $H^{\rm low{-1}}$, the solution will
display the (optimal) asymptotic behaviour that 
was claimed in our theorems.

This contraction mapping argument  is a
 straightforward modification of our argument to derive the claimed bounds;
 we just subtract the corresponding equations for each of the parameters in 
 the REVESNGG. We will not perform this here, since it would be 
 notationally very cumbersome and essentially a straightforward 
 modification of our arguments for boundedness.

\subsubsection{The base case of the inductive step.}\label{subsubsec:0step}

We proceed to prove the above estimates by induction, for all $m\in\mathbb{N}$. In 
particular we assume, that all claims listed above hold for {all steps up to} $m-1$ and we seek to derive 
the same claims for step $m$. We need to check that the claims hold at the zeroth step 
also: 
\medskip

At the zeroth step, $\gamma^0,K^0$ are equal to their Schwarzschild counterparts:
\begin{align}\label{gamma0trK0ind}
\gamma^0=\log r+\log \sin\theta, &&K^0_{11}=\frac{M}{r^2}
(\frac{2M}{r}-1)^{-\frac{1}{2}},\qquad K_{22}^0=K^0_{33}=
-\frac{1}{r}(\frac{2M}{r}-1)^\frac{1}{2}
\end{align}
Also, we have
\begin{align}\label{A0R0i0}
a^0_{\theta 1}=a^0_{t 2}=0, a^0_{\theta 2}= r, a^0_{t1}=(\frac{2M}{r}-1)^{\frac{1}{2}}. &&
\end{align}
The initial hypersurface is $\{r=r_*^0=\epsilon\}$, and $(\rho^0=r,t,\theta)$ are the 
classical coordinates in Schwarzschild. Hence, the above claims
 hold trivially at the zeroth step in our induction.
\medskip

Prior to proceeding with verifying the inductive step $m$, we will note certain consequences
of the inductive assumptions for the step $m-1$. (These consequences will be
 used
in the verification of the inductive step $m$).

\subsection{Basic Analysis tools.}

We  put down some very basic tools on which we rely, such as the Sobolev and generalized Gronwall inequalities 
and certain Fuchsian-type ODE and transport-equation type estimates, which are used 
frequently throughout this paper. 
\medskip

\subsubsection{A generalized Hardy inequality.}

We will frequently use the following Hardy-type inequality, whose proof is in the Appendix of \cite{IonKl}. 
Recall that $\opartial_\theta:=\partial_\theta +\frac{cos\theta}{sin\theta}$. Then: 

\begin{lemma}
\label{gen.Hardy}
For any  function $v(\theta)\in H^1_{\rm loc}(0,\pi)$ the following holds: 
\beq
\int_0^\pi |\opartial_\theta v|^2sin\theta d\theta + \int_0^\pi |v|^2sin\theta d\theta
\ge C_{\rm Hardy} \int_0^\pi | v|^2(sin\theta)^{-1} d\theta
\eeq
\end{lemma}

We also note  a consequence of the above, which follows by Cauchy-Schwarz: 

\beq
\label{gen/Hardy.conseq}
\int_0^\pi |\opartial_\theta v|^2sin\theta d\theta + \int_0^\pi |v|^2sin\theta d\theta
\ge C'_{\rm Hardy}\int_0^\pi |\partial_\theta  v|^2sin\theta d\theta
\eeq

\subsection{A generalized Gronwall inequality.}
 We will frequently use, sometimes without particular mention, the 
 classical Sobolev inequality
\begin{align}\label{Sob}
\|F(t,\theta)\|_{L^\infty_{t,\theta}(\mathbb{S}^2\times\mathbb{R})}
\leq C_{\rm Sob}\|F(t,\theta)\|_{H^2(\mathbb{S}^2\times\mathbb{R})},
\end{align}
where $C_{\rm Sob}>0$ is a  universal constant.

We also recall the following variant of the standard Gronwall inequality: 

\begin{lemma}\label{lem:Gron}
	Let $F,F_0,G,H:{(0,\delta]}\to\mathbb{R}$ be continuous functions, 
	$F_0$ non-increasing, satisfying
	\begin{align}\label{Gronineq}
	F^2(r)\leq F_0^2(r)+\int^{\delta}_r |H(\tau)| F^2(\tau) d\tau
	+\int^{\delta}_r |G(\tau)| |F(\tau)| d\tau,&&r\in (0,\delta).
	\end{align}
	Then $F$ verifies the bound:
	\begin{align}\label{Gronineq2}
	|F(r)|\leq e^{\int^{\delta}_r\frac{1}{2}|H(\tau)|d\tau}
	\bigg(|F_0(r)|+\frac{1}{2}\int^{\delta}_r|G(\tau)|d\tau\bigg),
	\end{align}
	for all $r\in(0,\delta]$.
\end{lemma}
\begin{proof}
	Fix $r_0\in (0,\delta]$ and 
	let $A(r)=F_0^2(r_0)+\int^{\delta}_r |H(\tau)| F^2(\tau) 
	d\tau
	+\int^{\delta}_r |G(\tau)| |F(\tau)| d\tau$, $r\in[r_0,\delta]$. 
	Since $F_0^2(r)$ is non-increasing, we have $F^2(r)\leq A(r)$, 
	for all $r\in[r_0,\delta]$ and hence
	\begin{align*}
	\partial_rA(r)=-|H(r)|F^2(r)-|G(r)||F(r)|\ge -|H(r)|A(r)-|G(r)|\sqrt{A(r)},\qquad r\in[r_0,\delta],
	\end{align*}
	or
	\begin{align*}
	2\partial_r\sqrt{A(r)}\ge -|H(\tau)|\sqrt{A(r)}-|G(r)|,\qquad r
	\in[r_0,\delta].
	\end{align*}
	Hence, integrating in $[r,\delta]$ we obtain
	\begin{align*}
	\sqrt{A(r)}\leq|F_0(r_0)|+\frac{1}{2}\int^\delta_r|G(\tau)|
	d\tau+\frac{1}{2}\int^\delta_r|H(\tau)|\sqrt{A(\tau)}d\tau,\qquad 
	r\in[r_0,\delta).
	\end{align*}
	The standard Gronwall's inequality now implies
	\begin{align*}
	|F(r)|\overset{(\ref{Gronineq})}{\leq}\sqrt{A(r)}\leq 
	e^{\int^\delta_r\frac{1}{2}|H(\tau)|d\tau}\bigg(|F_0(r_0)|
	+\frac{1}{2}\int^\delta_r|G(\tau)|d\tau\bigg),
	\end{align*}
	for all $r\in[r_0,\delta]$. Evaluating the preceding inequality at $r=r_0$, 
	we validate (\ref{Gronineq2}) for $r=r_0$. Since 
	$r_0\in(0,\delta]$ is arbitrary, the conclusion follows.
\end{proof}

 \medskip

Let us put down some standard elliptic estimates on $\mathbb{S}^2$, which help us 
in 
using the operator $\DS$ to obtain our derived estimates in the usual Sobolev 
spaces, 
instead of $\partial_\theta$. 
First we note that for any $\phi$-independent function $v\in H^2(\mathbb{S}^2)$ we have: 

\beq
\label{IBP.easy}
\int_{\mathbb{S}^2}|\partial_\theta v|^2 dV_{\mathbb{S}^2}=-\int_{\mathbb{S}^2}
\DS v
\cdot v dV_{\mathbb{S}^2}\le \delta \int_{\mathbb{S}^2} |\DS v|^2 
dV_{\mathbb{S}^2}+(4\delta)^{-1} \int_{\mathbb{S}^2} |v|^2 dV_{\mathbb{S}^2},
\eeq
for any $\delta>0$. We also recall the standard elliptic estimate (for 
$\phi$-independent functions $v$ over $\mathbb{S}^2$) : 

\beq
\label{basic.elliptic}
\int_{\mathbb{S}^2}|\partial^2_{\theta\theta} v|^2 dV_{\mathbb{S}^2}+\int_{\mathbb{S}^2}|\partial_{\theta} v\cdot {\rm cot}\theta|^2 dV_{\mathbb{S}^2}=
\int_{\mathbb{S}^2}|\DS v|^2 dV_{\mathbb{S}^2}+\int_{\mathbb{S}^2}
|\partial_\theta v|^2 
dV_{\mathbb{S}^2}.
\eeq
Combining this with \eqref{IBP.easy} with $\delta=\frac{1}{2}$ we find that
control of $\int_{\mathbb{S}^2}|\DS v|^2 dV_{\mathbb{S}^2}$ and 
$\int_{\mathbb{S}^2}| v|^2 dV_{\mathbb{S}^2}$ implies control of the 
$H^2(\mathbb{S}^2)$ norm. In particular for $\gamma^m_{\rm rest}$ and for each 
multi-index $I=(2k_1, k_2)$ (where there is an even number of $\phi$-indices), 
 it suffices to derive our claimed bounds for 
 $\opartial^I\gamma^m_{\rm rest}=(\DS)^{k_1}\partial^{k_2}_{t\dots t}\gamma^m_{\rm rest}$:

 \begin{lemma}
\label{lem:interpolation}
On any level set $\Sigma_{r=\delta(t,\theta)}$ let $\oH^k$ be the Sobolev 
spaces built with respect to the operators $\DS,\partial_t$ (with the 
volume form $sin\theta d\theta dt$). Let $H^k$ be the standard Sobolev 
spaces built out of $\partial_t,\partial_\theta$ (with the same volume
 form). 
Consider a function $v(t,\theta)$ and which is bounded in 
$\oH^k[\Sigma_{\rho^m=\tau}]$, where $k$ is even.
Then the same function is bounded in $H^k[\Sigma_{\rho^m=\tau}]$, with the same bounds, up to a universal multiplicative constant. 
\end{lemma} 

\begin{proof}
The proof for all derivatives of order $\le k-1$ 
follows by an iterated application of \eqref{IBP.easy} and \eqref{basic.elliptic}. We also thus obtain 
the desired bound for all derivatives of order $k$, provided an even number of them are $\theta$-derivatives. The missing ones are obtained by the standard interpolation inequality: 

\[
\| \partial^2_{t\theta}v \|^2_{L^2(\mathbb{S}^2\times \mathbb{R})}\le  
\|\partial^2_{\theta\theta}v \|^2_{L^2(\mathbb{S}^2\times \mathbb{R})}+ \|\partial^2_{tt}v \|_{L^2(\mathbb{S}^2\times \mathbb{R})}+ \|v 
\|_{L^2(\mathbb{S}^2\times \mathbb{R})}.
\]
\end{proof}
 
 We will apply the above to $\gamma^m_{\rm rest}$ also: 
 
  For each multi-index $I=(2k_1, k_2)$ Lemma \ref{lem:interpolation} implies that it suffices to derive our claimed bound for 
  $\opartial^I\gamma^m_{\rm rest}$ (and
  $\opartial^I(e_0^{J_0}\gamma^m_{\rm rest})$ at the top orders)
   instead of  $\partial^I\gamma^m_{\rm rest}$ (and
  $\partial^I(e_0^{J_0}\gamma^m_{\rm rest})$ at the top orders):
 
An iterated application of the above Lemma shows that $\|\opartial^{2k_1, k_2}_{\theta\dots\theta t\dots t}\gamma^m_{\rm rest}\|_{L^2(\mathbb{S}^2\times \mathbb{R})}$ controls 
 $\|\partial^{2k_1, k_2}_{\theta\dots\theta t\dots t}\gamma^m_{\rm rest}\|_{L^2(\mathbb{S}^2\times \mathbb{R})}$.

 This still leaves the challenge of deriving our bounds when 
 $I=(2k_1+1, k_2)$. For those, we use the bounds on the  orders 
 $(2k_1+2, k_2), (2k_1, k_2)$; if $2k_1+2\le \rm low$ we use \eqref{IBP.easy} with 
 $\delta=\frac{1}{2}$. In the remaining cases we use $\delta=(\rho^{m-2})^{1/4}$. We easily verify that if we can check our inductive claims for 
 $\opartial^I\gamma^m_{\rm rest} =(\DS)^{k+1}\partial^{k_2}_{t\dots t} \gamma^m_{\rm rest}$ then our full inductive  claim follows.
\medskip

Finally, we put down some useful bounds that generalize the Hardy inequality.
\newcommand{\stwo}{{\mathbb{S}^2}}

We will seek to bound 
$\int_{\mathbb{S}^2}| \frac{\partial_\theta v}{sin\theta}|^2sin\theta d\theta$ for
 certain functions $v(\theta)$ which are \emph{even}  at the poles
 $\theta=0,\pi$.  
 In particular we will derive  bounds on quantities: 
 
 \beq
 \|\partial^I\frac{\partial_\theta v}{sin\theta} \|_{L^2(\stwo)}
 \eeq
 by regular Sobolev norms on $\stwo$. Let us distinguish the two cases $I=(2k_1, k_2)$ 
 and $I=(2k_1+1, k_2)$. Let us consider the first case first, where we derive: 
 \beq
  \|\partial^I\frac{\partial_\theta v}{sin\theta} \|_{L^2(\stwo)}\le \| (\DS)^{k_1}
  \partial_{t\dots t}^{k_2}\frac{\partial_\theta v}{sin\theta} \|_{L^2(\stwo)}+
   \| (\DS)^{k_1-1}\partial_{t\dots t}^{k_2}\frac{\partial_\theta v}{sin\theta} 
   \|_{L^2(\stwo)}.
   \eeq
 Thus it suffices to bound expressions $ \| (\DS)^{k}
  \partial_{t\dots t}^{k_2}\frac{cos\theta\cdot \partial_\theta v}{sin\theta} 
  \|_{L^2(\stwo)}$
  by regular Sobolev norms (the introduction of $cos\theta$ in the 
  numerator above
   instead of the factor $1$ makes no difference, clearly). We do this by 
   merely
    writing: 
\beq
 (\DS)^{k_1}
  \partial_{t\dots t}^{k_2}\frac{cos\theta\cdot \partial_\theta v}{sin\theta} 
 = (\DS)^{k_1+1}
  \partial_{t\dots t}^{k_2}v -(\DS)^{k_1}\partial^2_{\theta\theta}  \partial_{t\dots t}^{k_2}v
\eeq   
 The $L^2(\stwo)$-norm of the RHS is clearly bounded by   
 \[
 \|  (\DS)^{k_1+1}
  \partial_{t\dots t}^{k_2}v\|_{L^2(\stwo)}+\| (\DS)^{k_1}\partial_{t\dots t}^{k_2}
  \partial_\theta v\|_{L^2(\stwo)}\le  \|  
  \opartial_{\theta \dots \theta t\dots t}^{2(k_1+1),k_2}v\|_{L^2(\stwo)}+
  \|  
  \opartial_{\theta \dots \theta t\dots t}^{2k_1+1,k_2}v\|_{L^2(\stwo)}
 \]
  
  For the case where $I=(2k_1+1, k_2)$ above is treated in exactly the same way, except 
  that we keep a $\partial_\theta$ at the left in all formulas and all substitutions 
  above. 
\medskip

\begin{remark}
Below when we apply the Leibnitz rule to terms $\DS [F\cdot G]$ and $\opartial^I [F\cdot G]$ we will denote the terms on the RHS by 
$\sum_{I_1\bigcup I_2=I} \partial^{I_1} F\cdot \partial^{I_2}G$. This is a slight abuse of notation, since the RHS in fact contains derivatives of the form 
$(\DS)^{k_1}\partial^{k_2}_{\theta\dots \theta}\partial^{k_3}_{t\dots t}F $ (involving $\DS$ directly). However the Lemma \ref{lem:interpolation} and the Hardy inequality above 
implies that the $L^2$  norms of those terms are bounded by that of $\opartial^{2k_1+k_2}_{\theta\dots \theta}\partial^{k_3}_{t\dots t}F $. 

Since we bound $L^2$ norms in this paper, this abuse of notation will not cause any confusion. 
\end{remark} 
 
We also frequently use the following classical product  inequality (for $s>2$ always below), often without mention: 

\beq
\label{prodcut.ineq}
\| F_1(t,\theta)\cdot F_2(t,\theta)\|_{H^s(\mathbb{S}^2\times\mathbb{R})}\le C_{\rm product}
\|F_1 \|_{L^\infty (\mathbb{S}^2\times\mathbb{R})}\cdot \|F_2
 \|_{H^s(\mathbb{S}^2\times\mathbb{R})}+C_{\rm product} \|F_1 \|_{H^s (\mathbb{S}^2\times\mathbb{R})}\cdot \|F_2
 \|_{L^\infty (\mathbb{S}^2\times\mathbb{R})}
\eeq 

\subsubsection{Fuchsian ODEs and transport equations: Basic estimates. }

We will be frequently encountering equations of the form: 
\beq
\label{fuchs}
\partial_r u(r,t,\theta)+f(r,t,\theta)\cdot u(r,t,\theta)= G(r,t,\theta),
\eeq
with the coefficient $f(r,t,\theta)$ satisfying an an asymptotic expansion: 
\beq
\label{asympt}
f(r,t,\theta)\sim \zeta(t,\theta)r^{-1},
\eeq
in the sense that: 
\[
|f(r,t,\theta)- \zeta(t,\theta)r^{-1}|\le Br^{-1+\delta},
\]
for some $\delta>0$ and $r\in(0,r_*]$.\footnote{Usually in this paper $\delta=1/4$.}
We then note that the general solution of this equation 
is of the form: 
\[
u(r,t,\theta)= e^{-\int_{r_*(t,\theta)}^r f(s,t,\theta)ds }\int_{r_*(t,\theta)}^r  e^{\int_{r_*(t,\theta)}^s f(y,t,\theta) dy}
G(s,t,\theta) ds+c(t,\theta) e^{-\int_{r_*(t,\theta)}^r f(s)ds },
\]
for any $r_*> 0$ we wish to choose.
The first term arises from  the forcing term in \eqref{fuchs}, 
while the second corresponds to the general solution 
of the corresponding homogeneous equation. In particular, $c(t,\theta)$ is a 
function that we are free to choose. However, specifying an initial condition for the function  
$u(r,t,\theta)$ at some point $t=t_0, \theta=\theta_0$ and $r=r_0$, \emph{or} specifying 
the limit 
\[
\lim_{r\to 0^+} e^{\int_{r_*(t,\theta)}^r f(s)ds } \bigg[u(r,t,\theta)- 
e^{-\int_{r_*(t,\theta)}^r f(s)ds }
\int_{r_*(t,\theta)}^r  e^{\int_{r_*(t,\theta)}^s f(y,t,\theta) dy}
G(s,t,\theta) ds\bigg],
\]
uniquely fixes the value of $c(t_0,\theta_0)$. 
(So the point 
{$r_0$} can be chosen arbitrarily, including {$r_0=0$}. 
In the latter 
case, however, one needs to know that the integral $\int_{0}^{{r_*}} f(s,t,\theta)ds$
 is convergent, for this formula to make sense).  
 \medskip
 
We also note that this formula, 
along with an initial data prescription of the form:
\[
c(t,\theta)=u_{\rm init}(r_*(t,\theta),t,\theta)
\] 
 can be used to derive the following energy estimate
for the solution 

\begin{lemma}
\label{Fuchsian}
Assuming $|\zeta(t,\theta)|< {c_0}$ and $r_*>0$ is small enough so that for all $r\in (0,r_*]$: 
\[
|f(r,t,\theta)|\le {\frac{c_0}{r}},
\] then 
thinking of $u(r):=u(r,t,\theta)$ as a map from $r$ to
 $L^2(\mathbb{S}^2\times\mathbb{R})$, we derive that 
for every $r\in (0,r_*]$:  

\beq
\label{Fuchian.basic}
{\|u(r)\|_{L^2}^2 \le C r^{-2c_0}r_*^{2c_0} 
\|u_{\rm init}\|_{L^2}^2+C
 r^{-2c_0}\sqrt{\int_r^{r_*} s^{2c_0}
\|G(s,t,\theta)\|_{L^2}^2 ds}}.
\eeq
\end{lemma}

The proof follows straightforwardly. 
\medskip

 We are now ready to introduce certain key parameters that capture the
  \emph{spatial} geometry of the metric iterates $h^m, g^m$. The control of these parameters via the inductive assumptions we are making at step $m-1$ 
  will enable us to derive the step $m$ of our inductive claim. 

\subsection{The spatial geometry  parameters and their control by the inductive assumption. }
\label{sec_geom.params}

We will show how the inductive assumption implies certain bounds on 
{secondary} quantities. We start by introducing these quantities: 

	We will frequently need to 
	modify the spatial orthonormal frame ${e^{m-1}_1},{e^{m-1}_2}$ 
	into a new frame, which is tangential to the level sets of 
	$\rho^{m-1}$, see \eqref{rho}-(\ref{partialrho}): 

\begin{definition}
Let: 
	\begin{align}\label{eibar}
	\begin{split}
	\overline{e}_i^{m-1}:=&\,e^{m-1}_i-e^{m-1}_i(\rho^{m-1})
	\partial_{\rho^{m-1}}\\
	=&\,e^{m-1}_i+e^{m-1}_i(\rho^{m-1})
	(\frac{2M}{r}-1)^{-\frac{1}{2}}[1-\partial_r\chi(r)
	(r^{m-1}_*-\epsilon)]^{-1}e_0\in T\Sigma_{\rho^{m-1}},
	\end{split}
	\end{align}
	for $i=1,2$.
\end{definition}

Let us put down some bounds on the coefficients that appear in the above 
equation:

Recall our {\it gauge normalisation assumption} \eqref{e2r.van.it}: 
\begin{align}\label{gaugenorme2it}
\lim_{r\rightarrow0}e_2^{m-1}r=o(r^{-\frac{1}{2}+d^{m-1}_2(t,\theta)}),&&m
\ge 1.
\end{align}
In the next lemma, we will show that in fact \eqref{e2r.van.it},
 together 
with the gauge law \eqref{almpar.trans}, implies that $e_2^{m-1}$ 
annihilates $r$, i.e., it is tangent to the level sets $\Sigma_r$. On the 
other hand, $e^{m-1}_1$ acting on $r$ gives a non-zero, but much less 
singular  term. 

We recall also the inital data bounds \eqref{g.bds}, \eqref{K.bds},  the inductive assumptions 
\eqref{r*.ind.claim.low}, \eqref{r*.ind.claim.high}, as well as 
the expression for $e^{m-1}_1=\te^{m-1}_1$ on the initial data in terms of 
$\tilde{K}^{m-1}_{12}$, \eqref{te1m.tK12}.
 Combining these bounds, 
we control the initial data of $e^{m-1}_1({r^{m-1}_*})$
on $\Sigma_{r^{m-1}_*}$ as follows:
\beq
\label{e1r.init}
\begin{split}
&\|e^{m-1}_1({r^{m-1}_*}) \|_{H^{s-3-4c}}\leq C\eta \e^{-DC \eta}, \\
&\| e^{m-1}_1({r^{m-1}_*}) \|_{\dot{H}^{l}}\leq C\eta 
\e^{-DC\cdot \eta  +(s-3-l)\frac{1}{4}-c},\qquad\text{for all $l\in 
\{s-3-4c+1, \dots s-4\}$.}
\end{split}
\eeq
\begin{lemma}\label{lem:e1re2r}
	The vector field ${e^{m-1}_2}$ annihilates the function $r$, 
	${e^{m-1}_2r}\equiv0$, while ${e^{m-1}_1}r$ satisfies the bounds over 
	level sets of $r$:
	\begin{align}\label{e1rest}
	\begin{split}
	\|e^{m-1}_1r\|_{L^\infty}\leq 
	C\eta r^{-2DC\cdot \eta}
	,\quad
	\|e^{m-1}_1r\|_{H^{s-3-4c}}\leq  C\eta r^{-2DC\cdot \eta  },\\
	\|e_0^{J_0}e^{m-1}_1r\|_{\dot{H}^l}\leq 
	C\eta r^{-2DC\cdot \eta  -\frac{3}{2}|J_0|+(s-3-l)\frac{1}{4}-c},
	\end{split}
	\end{align}
	for all $r\in(0,2\epsilon]$, $|J_0|\leq2$, $s-3-4c<l\leq s-4$.

The same bounds hold on the 
	level sets of $\rho^{m-1}$, with $r$ replaced by $\rho^{m-1}$.
	\medskip
	
	For that parameter, we also have a bound at the top order: 
	\beq\label{e1resr.top}
	\|e_0^{J_0}\partial^I[e^{m-1}_1\rho^{m-1}]\|_{L^2}\leq 
	C\eta (\rho^{m-1})^{-2DC\cdot \eta  -c},
\eeq
		for all $I, |I|=s-3$ \emph{except} the case where $I=(T,T,\dots, T)$. In that case we make no claim. 
\end{lemma}

\begin{proof}
Using the propagation rule \eqref{almpar.trans}, we compute:
	\begin{align}\label{eirODE}
	\begin{split}
	e_0(e^{m-1}_1r)=&-K_{12}^{m-1}e_2^{m-1}r-
	(K_1^{m-1})^b\,e^{m-1}_br+\frac{1}{2}
	(\frac{2M}{r}-1)^{-\frac{1}{2}}\frac{2M}{r^2}e^{m-1}_1r,\\
	=&-2K_{12}^{m-1}e_2^{m-1}r-
	K_{11}^{m-1}e^{m-1}_1r+\frac{1}{2}
	(\frac{2M}{r}-1)^{-\frac{1}{2}}\frac{2M}{r^2}e^{m-1}_1r,\\
	e_0(e^{m-1}_2r)=&\,K_{12}^{m-1}e_1^{m-1}r-
	(K_2^{m-1})^be^{m-1}_br+\frac{1}{2}
	(\frac{2M}{r}-1)^{-\frac{1}{2}}\frac{2M}{r^2}e^{m-1}_2r\\
	=&-K_{22}^{m-1}e^{m-1}_2r+\frac{1}{2}
	(\frac{2M}{r}-1)^{-\frac{1}{2}}\frac{2M}{r^2}e^{m-1}_2r
	\end{split}
	\end{align}
and thus by \eqref{K11heur}, \eqref{K22heur}, \eqref{K12heur}:
	\begin{align}
	\label{e1rODE2}
	\partial_r(e^{m-1}_1r)+\bigg[\frac{\frac{1}{2}-d_1^{m-1}}{r}-(\frac{2M}{r}-1)^{-\frac{1}{2}}u^{m-1}_{11}+\frac{1}{2}\frac{1}{2M-r}\bigg]e^{m-1}_1r\\
	\notag=2(\frac{2M}{r}-1)^{-\frac{1}{2}}K_{12}^{m-1}\,e^{m-1}_2r\\
	\label{e2rODE2}\partial_r(e^{m-1}_2r)+\bigg[\frac{\frac{1}{2}-d_2^{m-1}}{r}-(\frac{2M}{r}-1)^{-\frac{1}{2}}u^{m-1}_{22}\bigg]e^{m-1}_2r=0.
	\end{align}
Recall that 
$\|d_1^{m-1}-\frac{1}{2}\|_{L^\infty}\leq D\cdot C\eta \le\frac{1}{8}$ and 
$\|d_2^{m-1}+1\|_{L^\infty}\leq D\cdot C\eta\le  \frac{1}{8}$. The equation 
\eqref{e2rODE2}, together with the initial assumption 
\eqref{gaugenorme2it} and \eqref{inductiontrKopt}, imply that 
${e^{m-1}_2r}=0$ everywhere. Hence, \eqref{e1rODE2} reduces to a 
homogeneous ODE for ${e^{m-1}_1r}$, whose general solution  has the following behaviour (in $L^\infty)$): 
\begin{align}\label{e1rasym}
e^{m-1}_1r=(e^{m-1}_1r_*^{m-1}) \big[r^{d_1^{m-1}-\frac{1}{2}}
+O(r^{d_1^{m-1}-\frac{1}{2}+\frac{1}{4}})\big].
\end{align}
The lower order energy bounds in the second line of (\ref{e1rest}) follow by directly differentiating (\ref{e1rODE2}), utilising  (\ref{inductiontrKopt}) and the initial data bounds \eqref{e1r.init}, cf. Lemma \ref{Fuchsian}. On the other hand, for the higher order energy estimates, we commute (\ref{eirODE}) instead with $\partial^I$, $|I|=l$, and use the expansion (\ref{K11heur}) only for the coefficients of the top order terms in the resulting equation:
	\begin{equation}\label{inte1rODE}
	\begin{split}
	\partial_r\partial^I(e^{m-1}_1r)+\bigg[\frac{\frac{1}{2}-d_1^{m-1}}{r}-
	(\frac{2M}{r}-1)^{-\frac{1}{2}}u^{m-1}_{11}+\frac{1}{2}\frac{1}{2M-r}
	\bigg]\partial^Ie^{m-1}_1r\\
	=\sum_{I_1\cup I_2=I,\,|I_2|<|I|}(\frac{2M}{r}-1)^{-\frac{1}{2}}
	\partial^{I_1}\big[K_{11}^{m-1}-\frac{1}{2}
	(\frac{2M}{r}-1)^{-\frac{1}{2}}\frac{2M}{r^2}\big]
	\partial^{I_2}e_1^{m-1}r.
	\end{split}
	\end{equation}
	The higher order estimate in \eqref{e1rest}, $|J_0|=0$, follows from 
	Lemma \ref{Fuchsian}, the initial data bounds \eqref{e1r.init} and the 
	inductive assumptions \eqref{inductiontrKopt}-\eqref{inductiontrKlow} 
	for $K^{m-1}_{11}$, by finite induction in  
	$|I|\in\{s-3-4c,\ldots, s-4\}$. The case $|J_0|=1$ follows from
	 applying the $|J_0|=0$ estimate to  (\ref{inte1rODE}), after solving 
	 for $e_0\partial^Ie^{m-1}_1r$. The case $J_0=2$ follows by taking another derivative of that equation and invoking the bounds aready derived.

	 To derive the claims for $e^{m-1}(\rho^{m-1})$ we repeat the same commutation argument (up to adding  inconsequential terms involving 
	 $\chi^\prime(r)$. The top order estimate for that parameter follows, since now $e_1^{m-1}(\rho^{m-1})=0$ on $\Sigma_{\{\rho^{m-1}=\e\}}$, 
	 by invoking the tangency of $e_1^{m-1}$ to that hypersurface.  
	  
\end{proof}

	We will often use the frame 
	$\overline{e}_1^{m-1},\overline{e}^{m-1}_2$ introduced in
	 \eqref{eibar}, 
	instead of $e_1^{m-1},e_2^{m-1}$.
In order to go to-and-fro between coordinate vector fields  and frames
it is also useful to express the vector fields $\overline{e}^{m-1}_i$ in 
terms of coordinate vector fields
for some system of coordinates on the level sets $\Sigma_{\rho^{m-1}}$. 

\newcommand{\Th}{\Theta}

We will in fact be using different coordinate systems on these level sets. All of our coordinate systems $T,\Th$
on the inital data hypersurface  will 
be extended by requiring $e_0(T)=e_0(\Th)$. 
For now let us introduce the transformation formulas for coordinate vector fields to frames, and backwards. 
These formulas are \emph{universal}; we can use them for any system of coordinates $T,\Theta$ propagated according to $e_0(T)=e_0(\Theta)=0$. 
This will follow from our derivation of the relevant evolution equations. 

\begin{definition}
Consider the coordinate vector fields $\partial_t,\partial_\theta$ 
on any level set of $\rho^{m-1}$. Let us define the coordinate-to-frame and frame-to-coordinate coefficients
$a_{ti}^{m-1}$, $a^{m-1}_{\theta i}$,
$(a^{m-1})^{it}$,
$(a^{m-1})^{i\theta}$, $i=1,2$, via the relations
\begin{align}\label{tthetatransebar}
\partial_t=a^{m-1}_{t1}\overline{e}^{m-1}_1+a^{m-1}_{t2}
\overline{e}^{m-1}_2,\qquad\partial_\theta= a^{m-1}_{\theta 1}
\overline{e}^{m-1}_1+a^{m-1}_{\theta 2}\overline{e}^{m-1}_2,\\
\notag \overline{e}^{m-1}_1=(a^{m-1})^{1t}\partial_t+
(a^{m-1})^{1\theta}\partial_\theta,\qquad \overline{e}^{m-1}_2
= (a^{m-1})^{2t}\partial_t+(a^{m-1})^{2\theta}
\partial_\theta.
\end{align}
\end{definition}

Let us note that the values of the coordinate-to-frame coefficients 
also determine the form of the metric $h^{m-1}$ in analogy to 
\eqref{g.from.a.pre}, just adding indices $m-1$ to all the terms there. 
\medskip

We will in fact \emph{not} be using the coordinate-to-frame 
coefficients defined by these background coordinates, for reasons that we 
review after the next formulas. However, we put down the equations on the 
evolution of these parameters and the bounds we can derive on their initial 
data.  This is because our evolution equations are \emph{universal} 
(meaning they hold for all choices of coordinates $T,\Theta$ with
 $e_0(T)=e_0(\Theta)=0$), and to illustrate how the metric $h^{m-1}$ can be 
 reconstructed from these coefficients. 

For future reference, let us 
note here that  the values of 
$({a}^{m-1})^{it},({a}^{m-1})^{i\theta},{a^{m-1}_{ti},a^{m-1}_{\theta i}}$ on 
$\Sigma_{r^{m-1}_*}$
are precisely the coefficients that appear in 
 \eqref{te1m.tK12}{-\eqref{te1m.tK12inv}}, for the step $m-1$.\footnote{{Since $e^{m-1}$ is tangent to $\Sigma_{r_*^{m-1}}=\{\rho^{m-1}=\epsilon\}$ in our gauge, $e_1^{m-1}(\rho^{m-1})=e_1^{m-1}(\epsilon)=0$, and $e_1^{m-1}=\tilde{e}^{m-1}_1$.}}

 \beq
   \begin{split}
   \label{am-1.init}
& a^{m-1}_{t1}(r^{m-1}_*(t,\theta), t,\theta)= \frac{1}{2}\sqrt{1+\sqrt{1-4(\frac{\tilde{K}^{m-1}_{12}}{{\bf K}_{22}-{\bf K}_{11}})^2}}(\g_{tt})^{1/2},
\\&a^{m-1}_{t2}(r^{m-1}_*(t,\theta), t,\theta)=-\frac{{\rm sign}(\frac{\tilde{K}^{m-1}_{12}}{{\bf K}_{22}-{\bf K}_{11}})}{2}\sqrt{1-\sqrt{1-4(\frac{\tilde{K}^{m-1}_{12}}{{\bf K}_{22}-{\bf K}_{11}})^2}}(\g_{tt})^{1/2},
\\&  a^{m-1}_{\theta 2}(r^{m-1}_*(t,\theta), t,\theta)= \frac{1}{2}\sqrt{1+\sqrt{1-4(\frac{\tilde{K}^{m-1}_{12}}{{\bf K}_{22}-{\bf K}_{11}})^2}}(\g_{\theta\theta})^{1/2},
\\& a^{m-1}_{\theta 1}(r^{m-1}_*(t,\theta), t,\theta)=\frac{{\rm sign}(\frac{\tilde{K}^{m-1}_{12}}{{\bf K}_{22}-{\bf K}_{11}})}{2}\sqrt{1-\sqrt{1-4(\frac{\tilde{K}^{m-1}_{12}}{{\bf K}_{22}-{\bf K}_{11}})^2}} (\g_{\theta\theta})^{1/2}. 
 \end{split}
   \eeq
We also note the initial data for the variables $a^{Ai}$, $A=t,\theta$ and $i=1,2$ on the initial data set, as a consequence of 
\eqref{te1m.tK12}:
 \beq
   \begin{split}
   \label{oam-1.init}
& (a^{m-1})^{1t}(r^{m-1}_*(t,\theta), t,\theta )=\frac{1}{2}\sqrt{1+\sqrt{1-4(\tilde{K}^m_{12}({\bf K}_{22}-{\bf K}_{11})^{-1})^2}}
(\g_{tt})^{-1/2},    
\\&(a^{m-1})^{1\theta}(r^{m-1}_*(t,\theta), t,\theta )=
\frac{{\rm sign}((\tilde{K}^m_{12}({\bf K}_{22}-{\bf K}_{11})^{-1}))}{2}\sqrt{1-\sqrt{1-4(\tilde{K}^m_{12}({\bf K}_{22}-{\bf K}_{11})^{-1})^2}}
(\g_{\theta\theta})^{-1/2},
\\& (a^{m-1})^{2t}(r^{m-1}_*(t,\theta), t,\theta )=-\frac{{\rm sign}((\tilde{K}^m_{12}({\bf K}_{22}-{\bf K}_{11})^{-1}))}{2}\sqrt{1-\sqrt{1-4(\tilde{K}^m_{12}({\bf K}_{22}-{\bf K}_{11})^{-1})^2}}
(\g_{tt})^{-1/2} , 
\\&(a^{m-1})^{2\theta }(r^{m-1}_*(t,\theta), t,\theta )=
 \frac{1}{2}\sqrt{1+\sqrt{1-4(\tilde{K}^m_{12}({\bf K}_{22}-{\bf K}_{11})^{-1})^2}}
(\g_{\theta\theta})^{-1/2}.
 \end{split}
   \eeq

In particular, given the bounds {\eqref{tK12.ind.claim.low}, 
\eqref{tK12.ind.claim.top}, \eqref{tK12.ind.claim.high}} for 
$\tilde{K}^{m-1}_{12}(t,\theta)$, and the assumptions \eqref{g.bds}, \eqref{K.bds}  on
$\g_{tt}, \g_{\theta\theta}$, $\K_{22}-\K_{11}$, we derive the following  bounds on these initial data, first at the lower orders $l\le \rm low$:
\begin{align}\label{a.init.low}
\notag\|a^{m-1}_{t1}(r^{m-1}_*(t,\theta),t,\theta)-a^S_{t1}(\e,t,\theta)\|_{\dot{H}^l}
\leq C \eta 
\e^{-\frac{1}{2}-DC\eta},\|a^{m-1}_{\theta 1}(r^{m-1}_*(t,\theta),t,\theta)\|_{\dot{H}^l}
\leq C\eta  \e^{\frac{5}{4}-DC\eta},\\
\notag\|a^{m-1}_{\theta 2}(r^{m-1}_*(t,\theta),t,\theta)-a^S_{\theta 2}(\e, t,\theta)\|_{\dot{H}^l}
\leq C \eta
\e^{1-DC\eta},
\notag\|a^{m-1}_{t 2}(r^{m-1}_*(t,\theta),t,\theta)\|_{\dot{H}^l}
\leq C \eta
\e^{-\frac{1}{4}-DC\eta},
\notag
\end{align}
At the higher orders, the worse behaviour of $\tilde{K}^{m-1}_{12}$ yields 
a more singular behaviour for the above quantities, in terms of the power 
of $\e$. In particular for ${\rm low}+1\le l\le s-4$, $h=l-{\rm low}$:

\begin{align}
\notag\|a^{m-1}_{t1}(r^{m-1}_*(t,\theta),t,\theta)\|_{\dot{H}^l}
\leq&\, C \e^{-\frac{1}{2}-DC\eta-\frac{h}{4}},\qquad
\|a^{m-1}_{\theta 2}(r^{m-1}_*(t,\theta),t,\theta)\|_{\dot{H}^l}
\leq C 
\e^{1-DC\eta-\frac{h}{4}},\qquad\\
\notag\|a^{m-1}_{t 2}(r^{m-1}_*(t,\theta),t,\theta)\|_{\dot{H}^l}
\leq&\, C 
\e^{-\frac{1}{4}-DC\eta-\frac{h}{4}},
 \|a^{m-1}_{\theta 1}(r^{m-1}_*(t,\theta),t,\theta)\|_{\dot{H}^l}
\leq C \e^{\frac{5}{4}-DC\eta-\frac{h}{4}},\\
\notag
\end{align}

Finally, at the top orders we have the most singular behaviour:

\begin{align}
\notag\|a^{m-1}_{t1}
(r^{m-1}_*(t,\theta),t,\theta)\|_{\dot{H}^{s-3}}
\leq&\, C 
\e^{-\frac{1}{2}-DC\eta-c},
\|a^{m-1}_{\theta 2}(r^{m-1}_*(t,\theta),t,\theta)\|_{\dot{H}^{s-3}}
\leq C 
\e^{1-DC\eta-c},\qquad,\\
\notag\|a^{m-1}_{t 2}(r^{m-1}_*(t,\theta),t,\theta)\|_{\dot{H}^{s-3}}
\leq&\, C 
\e^{-\frac{1}{2}-DC\eta-c},\qquad
\|a^{m-1}_{\theta 1}(r^{m-1}_*(t,\theta),t,\theta)\|_{\dot{H}^{s-3}}
\leq C \e^{1-DC\eta-c},
\qquad
\notag
\end{align}

Recall the equation: 
\beq\label{partialrho.again}
\partial_{\rho^m}=[1-\partial_r\chi(r)(r^m_*-\epsilon)]^{-1}\partial_r=
-(\frac{2M}{r}-1)^{-\frac{1}{2}}[1-\partial_r\chi(r)(r^m_*-\epsilon)]^{-1}e_0
\eeq

Let us calculate $\nabla_{e_0}\partial_{\rho^{m-1}}$:

\beq
\label{e0.partialrho}
\nabla_{e_0}\partial_{\rho^m}=-\nabla_{e_0}[(\frac{2M}{r}-1)^{-\frac{1}{2}}[1-\partial_r\chi(r)(r^m_*-\epsilon)]^{-1}e_0]=
(\frac{2M}{r}-1)^{\frac{1}{2}}\partial_r[(\frac{2M}{r}-1)^{-\frac{1}{2}}[1-\partial_r\chi(r)(r^m_*-\epsilon)]^{-1}]e_0
\eeq

Let us use this in evaluating $\nabla_{\overline{e}_i^{m-1}}\partial_{\rho^{m-1}}$:
\beq\begin{split}
&\nabla_{\overline{e}_i^{m-1}}\partial_{\rho^{m-1}}=\nabla_{{e}_i^{m-1}}\partial_{\rho^{m-1}}-[e_i^{m-1}(\rho^{m-1})
(\frac{2M}{r}-1)^{-\frac{1}{2}}[1-\partial_r\chi(r)(r^m_*-\epsilon)]^{-1}
 \nabla_{e_0}\partial_{\rho^{m-1}}
\\&=e_i^{m-1}[-(\frac{2M}{r}-1)^{-\frac{1}{2}}[1-\partial_r\chi(r)(r^m_*-\epsilon)]^{-1}]e_0-(\frac{2M}{r}-1)^{-\frac{1}{2}}[1-\partial_r\chi(r)(r^m_*-\epsilon)]^{-1}\nabla_{e^{m-1}_i}e_0
\\&+[e_i^{m-1}(\rho^{m-1})
(\frac{2M}{r}-1)^{-\frac{1}{2}}[1-\partial_r\chi(r)(r^m_*-\epsilon)]^{-1}[-(\frac{2M}{r}-1)^{-\frac{1}{2}}[1-\partial_r\chi(r)(r^m_*-\epsilon)]^{-1}e_0]
\end{split}
\eeq


For simplicity, in the following derivations, we omit the index $m-1$ from all the relevant variables. 
The commutation relations $[\partial_{\rho},\partial_t]=[\partial_{\rho},\partial_\theta]=0$ yield an ODE for $a_{ti},a_{\theta i}$, $i=1,2$. In particular, we have:

\begin{align}
&0=[\partial_{\rho},\partial_t]=[\partial_\rho,a_{t1}\overline{e}_1+a_{t2}\overline{e}_2],\\
\notag &0=\partial_\rho(a_{t1})\overline{e}_1+\partial_\rho(a_{t2})\overline{e}_2+a_{t1}\nabla_{\partial_\rho}\overline{e}_1+a_{t2}\nabla_{\partial_\rho}\overline{e}_2-a_{t1}\nabla_{\overline{e}_1}\partial_\rho-a_{t2}\nabla_{\overline{e}_2}\partial_\rho,\\
\notag &0=\partial_\rho(a_{t1})\overline{e}_1+\partial_\rho(a_{t2})\overline{e}_2
+a_{t1}\nabla_{\partial_\rho}\overline{e}_1+a_{t2}\nabla_{\partial_\rho}\overline{e}_2-a_{t1}\nabla_{\overline{e}_1}\partial_\rho-a_{t2}\nabla_{\overline{e}_2}\partial_\rho.
\end{align}
Taking the inner product of the previous equation with respect to $e_1$ and using \eqref{partialrho.again}  we obtain:
\begin{align}
e_0a_{t1}=&-a_{t1}h(\nabla_{e_0}\overline{e}_1,e_1)-a_{t2}h(\nabla_{e_0}\overline{e}_2,e_1)+a_{t1}h(\nabla_{\overline{e}_1}e_0,e_1)+a_{t2}h(\nabla_{\overline{e}_2}e_0,e_1)\\
\tag{by \eqref{almpar.trans}}=&-a_{t2}K_{12}+a_{t1}K_{11}+a_{t2}K_{12}
=a_{t1}K_{11}
\end{align}
The analogous computation for $a_{t2}$ (multiplying with $e_2$ the first equation instead) is similar and so are the ones for $a_{\theta i}$, derived from the identity $0=[\partial_\rho,\partial_\theta]$, which yield the following ODE system:
	\begin{align}\label{e0.a}
	\begin{split}
	e_0a_{t1}-K_{11}a_{t1}=0,\qquad e_0a_{t2}-K_{22}a_{t2}=2K_{12}a_{t1},\\
	e_0a_{\theta 1}-K_{11}a_{\theta 1}=0,\qquad e_0a_{\theta 2}-K_{22}a_{\theta 2}=2K_{12}a_{\theta 1}.
	\end{split}
	\end{align}

Then for  the above system of equations, with initial data prescribed on 
$\Sigma_{r^{m-1}_*}$ 	
we can explicitly write out the solutions to the above system:

\beq
\begin{split}
\label{explicit.int}
&a^{m-1}_{t1}(r,t,\theta)= e^{-\int_{r^{m-1}_*(t,\theta)}^{r} K^{m-1}_{11} (1-\frac{2M}{s})^{-\frac{1}{2}} ds}\cdot a^{m-1}_{t1}(r^{m-1}_*(t,\theta),t,\theta),
\\& a^{m-1}_{\theta 1}(r,t,\theta)= e^{-\int_{r^{m-1}_*(t,\theta)}^{r} K^{m-1}_{11} (1-\frac{2M}{s})^{-\frac{1}{2}} ds}\cdot a_{\theta 1}^{m-1}(r^{m-1}_*(t,\theta),t,\theta),
\\&a^{m-1}_{t2}(r,t,\theta)= e^{-\int_{r^{m-1}_*(t,\theta)}^{r} K^{m-1}_{22} (1-\frac{2M}{s})^{-\frac{1}{2}} ds}\cdot a^{m-1}_{t2}(r^{m-1}_*(t,\theta),t,\theta)
\\&+ e^{-\int_{r^{m-1}_*(t,\theta)}^{r} K^{m-1}_{22} 
(1-\frac{2M}{s})^{-\frac{1}{2}} ds}\cdot [-\int_{r^{m-1}_*(t,\theta)}^{r}e^{\int_{r^{m-1}_*(t,\theta)}^{s} K^{m-1}_{22}(\tau,t,\theta) (1-\frac{2M}{\tau})^{-\frac{1}{2}} d\tau}
 2K^{m-1}_{12}(s,t,\theta)a^{m-1}_{t1}(s)\cdot 
 (1-\frac{2M}{s})^{-\frac{1}{2}} ds],
 \\&a^{m-1}_{\theta 2}(r,t,\theta)= e^{-\int_{r^{m-1}_*(t,\theta)}^{r} K^{m-1}_{22} (1-\frac{2M}{r})^{-\frac{1}{2}} ds}
 \cdot a^{m-1}_{t2}(r^{m-1}_*(t,\theta),t,\theta)
+ e^{-\int_{r^{m-1}_*(t,\theta)}^r K^{m-1}_{22} (1-\frac{2M}{s})^{-\frac{1}{2}} ds}
\\&\cdot [-\int_{r^{m-1}_*(t,\theta)}^r e^{\int_{r^{m-1}_*(t,\theta)}^s K^{m-1}_{22}(\tau,t,\theta) (1-\frac{2M}{\tau})^{-\frac{1}{2}} d\tau }2K^{m-1}_{12}(s,t,\theta)a^{m-1}_{\theta 1}(s)\cdot 
 (1-\frac{2M}{s})^{-\frac{1}{2}} ds]
\end{split}
\eeq

	\medskip

A key remark is in order here: The coordinate 
expression on the metric (in terms of $t,\theta,r$) that we can obtain 
from the above will turn out to \emph{not} be adequate to derive our 
desired estimates for $\gamma^m$. The moral reason for this is that these 
coordinates \emph{emanate} from the initial data hypersurface via extension 
along $e_0$. They thus \emph{fail} to capture the principal directions of 
collapse/expansion \emph{at} the singularity $\{r=0\}$. This would manifest 
itself in non-optimal behaviour (in terms of powers of $r$) 
for certain Christoffel symbols in this coordinate system.  

The remedy to this issue is to consider \emph{new coordinates} which are
 \emph{adapted} to the principal collapsing directions \emph{at} 
 the singularity. It is with respect to these new coordinates that the 
 Christoffel symbols will have a suitable behaviour that allows us to close 
 our estimates. 

In fact, there are \emph{two} possible choices of coordinates that we can 
make. The first system is where the optimal behaviour of the spatial part 
of the metric near the singularity becomes apparent: In this system of 
coordinates $(\tilde{\theta},\tilde{t})$, and als $\tilde{r}$ 
 the coordinate vector fields are tangent 
(in an asymptotic sense) to the \emph{principal} collapsing and expanding 
directions at the singularity. 
These are introduced in the last section in the Appendix, in the proof of an optimal Corollary of our main theorem. 
We will not work with these (very rigid) coordinates in our main
 proof however.  Instead, 
for our main proof we use  a ``hybrid'' coordinate  system; one that is in 
between the one that emanates from the initial data set and the one (
$\partial_\Theta$) that 
emanates entirely from the singularity, which captures both principal 
collapsing/expanding directions: 
This 
 ``hybrid'' coordinate system is used  to achieve two goals: First to 
 capture \emph{only} the 
 collapsing 
 direction $e^{m-1}_2$ by one of the coordinate vector fields. Secondly, 
to provide 
  sufficient spatial regularity of the metric when expressed with respect to this 
  coordinate system. 
  \medskip

We will introduce this new coordinate system shortly, after a useful remark 
on the vanishing of certain parameters at the poles $\theta=0,\pi$.

  \subsubsection{Propagation of vanishing  conditions   at the poles.}
\label{poles}

In the analysis we perform below, we will at times invoke the generalized Hardy inequality 
in Lemma \ref{gen.Hardy}. This will apply to functions of $\theta, t$ that vanish at the two poles $0,\pi$. 
We present here how certain key parameters vanish to first order at those two poles at each step in our iteration. 
This ensures that whenever the generalized Hardy inequality in Lemma \ref{gen.Hardy} is invoked, the assumed vanishing of the function 
at the poles will hold. 
\medskip

We have imposed the condition $\te^m_2(r^m_*)=0$ at the poles, which in view of the smoothness of the vector field and function implies: 
 $\te^m_2(r^m_*)=O({\rm sin}\theta)$.

We will show that this vanishing  condition for this and some other 
parameters 
 is propagated off of the initial data hypersurface:

\begin{lemma}
\label{lem:poles}
For each step of our iteration, $\tilde{K}^m_{12}(t,\theta), \tilde{e}^m_2r^m_*(t,\theta)$ both vanish to first order at the poles, in other words 
\beq
\label{tK.van}
\tilde{K}^m_{12}(t,\theta),\tilde{e}^m_2[r^m_*](t,\theta)=O(sin\theta).
\eeq

Moreover the following vanishing conditions hold off the initial data 
hyper-surface:
\beq\label{K.van}
K^m_{12}(r, t,\theta)=O(sin\theta), \forall  r\in (0,2\e], t\in \mathbb{R}.
\eeq

Furthermore 
\[ e^m_1(r,t,\theta=0),  e^m_2(r,t,\theta=\pi), e^m_1(r,t,\theta=\pi),  e^m_2(r,t,\theta=\pi)\]
are parallel to $\partial_t,\partial_\theta$ for all $r\in (0,2\e]$, which is captured by the requirements: 

\beq
\label{van.poles}
(a^m)^{ 1\theta },(a^m)^{2t}, e^m_2(r)=O(sin\theta). 
\eeq
\end{lemma}

\begin{proof}

We prove the above by an induction on $m$: We assume it is true at step $m-1$ and derive the statement at step $m$. 
We have derived that $\gamma^m_{\rm rest}$ is a ${\cal C}^{{\rm low}-2}$ function over $\mathbb{S}^2\times\mathbb{R}$.
As discussed above, this implies that $\partial_\theta  \gamma^m_{\rm rest}=O(sin\theta)$. 

From this we can derive our claim as follows: First we note that in view of 
the regularity of the metric $g^{m-1}$ we have that: 
\beq
\label{pole.note}
(\overline{\nabla}^2)^{g^{m-1}}_{t\theta}\gamma^m_{\rm rest}=O(sin\theta). 
\eeq
We can then invoke the tensorial nature of the LHS and choose normal 
coordinates at each of the poles to  derive that: 
\beq
\label{pole.note2}
(\overline{\nabla}^2)^{g^{m-1}}_{t\theta}\gamma^m=O(sin\theta). 
\eeq
Now, to derive our claim \eqref{van.poles} we express $e^{m-1}_i$ in terms 
of the coordinate vector fields $\partial_t,\partial_\theta,\partial_r$ 
using 
the functions $(a^{m-1})^{iA}$, using \eqref{tthetatransebar}. In view of 
our inductive assumptions on these parameters $(a^{m-1})^{iA}$
we derive:

\beq
\overline{\nabla}_{12}\gamma^m +\frac{1}{2}[\nabla_1\gamma^{m-1}_{\rm rest} 
\nabla_2\gamma^m_{\rm rest}+ 
\nabla_1\gamma^{m}_{\rm rest} \nabla_2\gamma^{m-1}_{\rm rest}]=O(sin
\theta). 
\eeq

Considering the evolution equation \eqref{finredEVERic12it} \eqref{e0.a}, 
we then derive \eqref{K.van}. Invoking equation 
\eqref{K12exp} as well as the expression \eqref{tildee.m} for $\te^m_2$ in 
terms of $\tilde{K}^m_{12}$,  
 this then implies \eqref{tK.van} since the LHS of that equation 
 is of the form $O(sin\theta)$ and the LHS is a smooth
  function of $\tilde{K}^m_{12}$.

 With these conditions verified, we proceed to confirm \eqref{van.poles}. This is initially verified on the initial data hypersurface \eqref{tK.van} 
 and the formula \eqref{varphi.tilK.m}. Then the evolution equations \eqref{e2rODE2} along with \eqref{K.van} confirm \eqref{van.poles}. \end{proof}
 \medskip

 Although not needed, we note that the above proof implies the  vanishing of the even $\theta$-derivatives of $K^m_{12}, \tilde{K}^m_{12}$ as well as $\te^m_2(r^m_*), e^m_2(r^m_*)$
and $(a^m)^{1\theta}, (a^m)^{2t}$ at the poles $\theta=0,\theta=\pi$

  \subsubsection{The new coordinate system: Bounds on metric components 
  and Christoffel symbols, by virtue of our inductive assumptions.}
  \label{sec:met.Christ.bds}
  
  We will consider a new coordinate function $T=T(t,\theta)$ so that the 
  coordinate
   system $\{ T=T(t,\theta),\Theta=\theta\}$ has the coordinate vector 
   field $\partial_\Theta$
   capturing the direction $e^{m-1}_2$ \emph{at} the singularity. (Here, as above $T=T^{m-1}$,
    but we suppress the suffix $m-1$ for notational simplicity).

 We do this as follows: Let $\partial_\Theta, \partial_T$ be the 
 sought-after coordinate vector fields that correspond to the 
 sought-after coordinates. Let us express these sought-after vector fields 
 as linear combinations of the vector fields $\overline{e}_1, \overline{e}_2$ on each level 
 set of $\rho^{m-1}$.
 They will be expressed as linear combinations, given 
 by a formula as follows: 
  
\begin{align}\label{tthetatransebar.again}
\partial_T=a^{m-1}_{T1}\overline{e}^{m-1}_1+a^{m-1}_{T2}
\overline{e}^{m-1}_2,\qquad\partial_\theta= a^{m-1}_{\Th 1}
\overline{e}^{m-1}_1+a^{m-1}_{\Th 2}\overline{e}^{m-1}_2,\\
\notag \overline{e}^{m-1}_1=(a^{m-1})^{1T}\partial_T+
(a^{m-1})^{1\Th}\partial_\Th,\qquad \overline{e}^{m-1}_2
= (a^{m-1})^{2T}\partial_T+(a^{m-1})^{2\Th}
\partial_\theta.
\end{align} 
As in the case for the coordinate vector fields
 $\partial_t,\partial_\theta$ these coefficients 
  are then governed by the evolution equations:  
	\begin{align}\label{e0.a.new}
	e_0a^{m-1}_{T1}-K^{m-1}_{11}a^{m-1}_{T1}=0,\qquad e_0a^{m-1}_{T2}-K^{m-1}_{22}a^{m-1}_{T2}=2K^{m-1}_{12}a^{m-1}_{T1}\\
\notag	e_0a^{m-1}_{\Th 1}-K^{m-1}_{11}a^{m-1}_{\Th 1}=0,\qquad e_0
	a^{m-1}_{\Th 2}-K^{m-1}_{22}a^{m-1}_{\Th 2}=2K^{m-1}_{12}a^{m-1}_{\Th 1},
	\end{align}
where the coefficients $a^{m-1}_{Ai}$ can be thought of as 
 functions of $r,t,\theta$ or of $\rho^{m-1},t,\theta$.

We can solve for $a^{m-1}_{\Th 1}$, 
$a^{m-1}_{\Th 2}$ $a^{m-1}_{T 1}, a^{m-1}_{T 2}$ after we prescribe 
suitable initial conditions \emph{somewhere}.
   We will solve for $a^{m-1}_{\Th 1}$ \emph{backwards} from the singularity 
 setting the free branch equal to zero. This will imply that $a^{m-1}_{\Theta 1}=0$. This condition captures that 
   $\partial_\Th$ is parallel (in an asymptotic sense) to the direction of 
   $e^{m-1}_2$ at the singularity.
   As a consequence of the evolution equation, we derive that  
   $a^{m-1}_{\Th 1}=0$ 
   everywhere.  
    The requirement $\Th=\theta$ is captured by requiring 
    $\partial_\Th\theta=1$. 
Thus recalling \eqref{explicit.int}, we have on $\Sigma_{r^{m-1}_*}$:

\beq
\label{a2Th}
1=\overline{e}^{m-1}_2(\theta)\cdot c^{m-1}_{\Th 2}(t,\theta), 
\eeq    
    where 
    $c^{m-1}_{\Th 2}(t,\theta)=a^{m-1}_{\Th 2}(r^{m-1}_*(t,\theta),t,
    \theta)$. Note that $\te^{m-1}_2(t,\theta)$ on   $\Sigma_{r^{m-1}_*}$
    is given in terms of the already solved-for 
    $\tilde{K}^{m-1}_{12}$, and that:
\[
\overline{e}^{m-1}_2(t,\theta)= \sqrt{1-(\frac{2M}{r}-1)^{-1}
[\te_2^{m-1}(r_*^{m-1})]^2[1-\partial_r\chi(r)(r^{m-1}_*-\e)]^{-2}}
\te^{m-1}_2(t,\theta).
\]    
      This  then implies  bounds on 
    $c^{m-1}_{\Th 2}(t,\theta)$, and thus on 
    $a^{m-1}_{\Th 2}(\rho^{m-1},t,\theta)$ at $\{\rho^{m-1}=\e\}$. 
    (We put these down right  below).

So far we have solved for the vector field $\vec{\Th}$ which is meant to be 
the coordinate vector field $\partial_\Th$, once we specify the coordinate 
function $T=T(t,\theta)$. To obtain this function, we must impose the
necessary  relation: 
  \beq\label{vecTh.T}
  \vec{\Th}(T)=0,
  \eeq
  which on 
 $\Sigma_{r^{m-1}_*}$ translates into: 
 
  \beq\label{vecTh}
  \vec{\Th}(T)=0\implies\sum_{i=1,2} a^{m-1}_{\Th i}
  \overline{e}^{m-1}_i[T(t,
  \theta)]=0.
  \eeq

Moreover the vector field $\partial_T$ in the same coordinate system will 
equal:   

\beq
\partial_T=(\frac{\partial T}{\partial t})^{-1}\partial_t. 
\eeq
  
  Now, the coefficients $a^{m-1}_{\Th i}$ have already been solved for at 
  this
   point, and $a^{m-1}_{\Th 1}=0$ ; we will see that they are $H^{s-3}$ regular (plus allowing an extra singular weight at the poles). 
   Also, recall that the vector fields 
   $\te_i^{m-1}$, $\overline{e}^{m-1}_i$ are
  expressible via
  the values of $\tilde{K}_{12}^{m-1}(t,\theta)$, which has already been 
  solved for, in terms of $\partial_t,\partial_\theta$. 
      This implies that $H^{s-3}$ 
   regularity holds for the function  $\frac{\partial T}{\partial\theta}$ (and $\overline{e}^{m-1}_2(T)$), 
with the some extra singular weights  at the poles.  
   
    We now impose the initial 
  conditions
   $T(t,\theta)=t$ on $\{\theta=0\}$. Equation \eqref{vecTh} then be seen as a 
   1-parameter
   family
   of transport equations on $\{(\theta,t)\in [0,\pi]\times \mathbb{R}\}$. 
   Coupled with the imposed initial condition, we can obtain a unique 
   solution $T(t,\theta)$.   
   
   We can derive  regularity for $\frac{\partial  T}{\partial t}$: Refer to 
   \eqref{vecTh.T} and take another $\partial_t$ derivative. 
We can then take up to another $s-4$ derivatives in the directions $\partial_\Th$ or $\partial_T$.    We recall that $\overline{e}^m_2$ is parallel to 
$\partial_\Th$. Thus the resulting equation yields estimates on up to $s-4$ derivatives of $\partial_t\Th$, provided
at least one of them is in the $\partial_\Th$ direction.

Now, on the initial data surface $\Sigma_{r^{m-1}_*}$ we recall formulas 
\eqref{te1m.tK12} that link 
$\partial_t,\partial_\theta$ to $\te^{m-1}_1, \te^{m-1}_2$
on this surface. 
Combined with the above formula, these give explicit formulas for 
$a^{m-1}_{T1}(r^{m-1}_*(t,\theta), t,\theta)$, 
$a^{m-1}_{T2}(r^{m-1}_*(t,\theta), t,\theta)$.

From this function $T(t,\theta)$ and these formulas, we can obtain the regularity of the coefficients 
$a_{Ai}^{m-1}$, $A=T,\Th$ and $i=1,2$
 on the initial data hypersurface $\Sigma_{r^{m-1}_*}$. These appear at the lower, 
 higher and top 
 orders in the Lemma right below, for $r=r^{m-1}_*(t,\theta)$.



Off of the initial data hypersurface 
the regularity of this solution is    obtainable from the  transport 
equation \eqref{e0.a.new}. We put these down in the next Lemma.

Prior to this, we introduce one piece of notation, which is necessary to 
single out a special case at the top orders:  On any level set of $r$ or $\rho^{m-1}$ (where we will have induced coordinates $T,\Th$) let $\tilde{H}^{s-3}$ 
stand for the homogenous Sobolev space consisting of all iterated 
$\partial_T,\partial_\Theta$ derivatives, \emph{except} for the one where all derivatives are taken in the $\partial_T$-direction: 

\[
\|f \|_{\tilde{H}^{s-3}[\Sigma_r]}=\sqrt{\sum_{I, |I|=s-3, I\ne (T,\dots ,T)} 
\int_{\Sigma_r} |\partial^I f|^2 {\rm sin\theta} d\theta dt}.
\]

\begin{lemma}
\label{a.oa.bds}
	The coefficients  $a^{m-1}_{Ti},a^{m-1}_{\Th i}$, $i=1,2$ in the transformations 
	(\ref{tthetatransebar}) that we just constructed have the following regularity properties in the Sobolev spaces $H^l$, defined with respect to the coordinates $T,\Theta=T^{m-1}, \Theta^{m-1}$:
	
	At the lower orders we claim, for $l\le \rm low$:

\begin{align}\label{a.r.low}
\|a^{m-1}_{T1}(r,t,\theta)\|_{\dot{H}^l}
\leq C 
r^{-\frac{1}{2}-DC\eta},
\|a^{m-1}_{\Th 2}(r,t,\theta)\|_{\dot{H}^l}
\leq C 
r^{1-DC\eta},\\
\notag\|a^{m-1}_{T 2}(r,t,\theta)\|_{\dot{H}^l}
\leq C 
r^{-\frac{1}{4}-DC\eta},
\|a^{m-1}_{\Th 1}(r,t,\theta)\|_{\dot{H}^l}
= 0,\end{align}

	At the higher orders, for ${\rm low}+1\le l\le s-4$, $h=l-{\rm low}$:

	\begin{align}\label{a.r.high}
\|a^{m-1}_{T1}(r,t,\theta)\|_{\dot{H}^l}
\leq&\, C r^{-\frac{1}{2}-DC\eta-\frac{h}{4}},
\|a^{m-1}_{\Th 2}(r,t,\theta)\|_{\dot{H}^l}
\leq C 
r^{1-DC\eta-\frac{h}{4}},\\
\notag\|a^{m-1}_{T 2}(r,t,\theta)\|_{\dot{H}^l}
\leq&\, C 
r^{-\frac{1}{4}-DC\eta-\frac{h}{4}},
 \|a^{m-1}_{\Th 1}(r,t,\theta)\|_{\dot{H}^l}
= 0,
\end{align}

	Finally, at the top orders our claims are as follows:

	\begin{align}\label{a.r.top}
\|a^{m-1}_{T1}
(r,t,\theta)\|_{\tilde{H}^{s-3}}
\leq C 
r^{-\frac{1}{2}-DC\eta-c},
\qquad    \|\partial_\Th a^{m-1}_{T1}
(r,t,\theta)cot\theta \|_{\dot{H}^{s-4}}
\leq C 
r^{-\frac{1}{2}-DC\eta-c},\\
\notag\|a^{m-1}_{\Th 2}(r,t,\theta)\|_{\dot{H}^{s-3}}
\leq C 
r^{1-DC\eta-c},
\|a^{m-1}_{\Th 1}(r,t,\theta)\|_{\tilde{H}^{s-3}}
= 0,\\
\notag\|a^{m-1}_{T 2}(r,t,\theta)\|_{\tilde{H}^{s-3}}
\leq C 
r^{-\frac{1}{2}-DC\eta-c},\qquad
     \|[ a^{m-1}_{T2}
(r,t,\theta)cot\theta] \|_{\dot{H}^{s-4}}
\leq C 
r^{-\frac{1}{2}-DC\eta-c}.\\
\notag
\end{align}
\end{lemma}

(We note the second set of estimates at the top order involves an extra weight 
${\rm cot}\theta$ which is singular at the two poles). 




We note that the above, together with Lemma \ref{lem:e1re2r}  implies the following bounds on the components of the metric $g$ 
with respect to the coordinates $T,\Th$, via the formulas 
\beq
\label{g.from.a.again} 
\begin{split}
& h^{m-1}_{\Theta\Theta}  =  \sum_{i=1,2}   [a^{m-1}_{\Theta i} ]^2 \cdot\bigg{[}1- [e_i(\rho^{m-1}) (\frac{2M}{r}-1)^{-1/2}
[1-\partial_r\chi(r)(r_*-\e)]^{-1}]^2\bigg]
\\&-a^{m-1}_{\Theta 1}a^{m-1}_{\Theta 2} e^{m-1}_1(\rho^{m-1})e^{m-1}_2(\rho^{m-1})(\frac{2M}{r}-1)^{-1}[1-\partial_r\chi(r)(r_*-\e)]^{-2}, 
\\&h^{m-1}_{TT}=\sum_{i=1,2}[a^{m-1}_{ T i}]^2 \cdot\bigg{[}1- [e^{m-1}_i(\rho) (\frac{2M}{r}-1)^{-1/2}
[1-\partial_r\chi(r)(r_*-\e)]^{-1}]^2\bigg{]}
\\&-a^{m-1}_{T 1}a^{m-1}_{T 2} e^{m-1}_1(\rho)e^{m-1}_2(\rho)(\frac{2M}{r}-1)^{-1}[1-\partial_r\chi(r)(r_*-\e)]^{-2}, 
\\&h^{m-1}_{T\Theta}=\sum_{i=1,2}  [a^{m-1}_{ \Theta i}\cdot a^{m-1}_{ T i} ] \cdot\bigg{[} 1- [e^{m-1}_i(\rho) (\frac{2M}{r}-1)^{-1/2}
[1-\partial_r\chi(r)(r_*-\e)]^{-1}]^2\bigg{]}
\\&+[a^{m-1}_{\Theta_1}a^{m-1}_{T_2}+^{m-1}a_{T 1}a^{m-1}_{\Theta 2}] e_1(\rho^{m-1})e_2(\rho^{m-1})
(\frac{2M}{r}-1)^{-1}[1-\partial_r\chi(r)(r_*-\e)]^{-2},
\\&h^{m-1}_{\rho \Theta}=\sum_{i=1,2}(\frac{2M}{r}-1)^{-1}
a^{m-1}_{\Theta i}e_i(\rho^{m-1})\cdot [1-\partial_r\chi(r)(r_*-\e)]^{-2}, 
\\&h^{m-1}_{\rho T}=\sum_{i=1,2}(\frac{2M}{r}-1)^{-1}
a^{m-1}_{T i}e_i(\rho^{m-1})\cdot [1-\partial_r\chi(r)(r^{m-1}_*-\e)]^{-2}.
\end{split}
\eeq

	At the low orders  $l\leq {\rm low}$:
\begin{align}\label{g.r.low}
\|g^{m-1}_{TT}(r,t,\theta)\|_{\dot{H}^l}
\leq  C 
r^{-1-2DC\eta},
\|g^{m-1}_{\Th \Th}(r,t,\theta)\|_{\dot{H}^l}
\leq C 
r^{2-2DC\eta},
\|g^{m-1}_{T \Th}(r,t,\theta)\|_{\dot{H}^l}
\leq C 
r^{\frac{3}{4}-2DC\eta},
\end{align}
At the higher orders, for ${\rm low}+1\le l\le s-4$, $h=l-{\rm low}$:
	
\begin{align}\label{g.r.high}
\|g^{m-1}_{TT}(r,t,\theta)\|_{\dot{H}^l}
\leq&\, C 
r^{-1-2DC\eta-\frac{h}{4}},
\|g^{m-1}_{\Th \Th}(r,t,\theta)\|_{\dot{H}^l}
\leq C 
r^{2-2DC\eta-\frac{h}{4}},
\|g^{m-1}_{T \Th}(r,t,\theta)\|_{\dot{H}^l}
\leq C 
r^{\frac{3}{4}-2DC\eta-\frac{h}{4}}
\end{align}

	\begin{align}
\|g^{m-1}_{TT}
(r,t,\theta)\|_{\tilde{H}^{s-3}}
\leq C 
r^{-1-2DC\eta-c},   \|\partial_\Th g^{m-1}_{TT}
(r,t,\theta)\cdot cot\theta \|_{\dot{H}^{s-4}}
\leq C 
r^{-1-2DC\eta-c},\\
\notag\|g^{m-1}_{\Th \Th}(r,t,\theta)\|_{\dot{H}^{s-3}}
\leq C 
r^{2-2DC\eta-c},\qquad   \|\partial_\Th g^{m-1}_{\Th \Th}(r,t,\theta)\cdot cot\theta \|_{\dot{H}^{s-4}}
\leq  C 
r^{2-2DC\eta-c},\\
\notag \|g^{m-1}_{T \Th}(r,t,\theta){\rm cot}\theta \|_{\tilde{H}^{s-3}}
\leq C 
r^{\frac{3}{4}-DC\eta-c-\frac{1}{4}},
\end{align} 
{\bf Note:} The evolution equations for these parameters, via the evolution equations for $a^{m-1}_{Ai}$ imply that analogous bounds hold for all the $e_0$-derivatives of these quantities, at the cost of an extra power $r^{-\frac{3}{2}}$ on the RHS. 
This follows readily from the analogous bounds on the components $K^{m-1}_{ij}$ (which form part of our inductive assumption), and the evolution equations \eqref{e0.a.new}. 

\emph{Proof of Lemma \ref{a.oa.bds}.}
Let us first derive the claimed bounds on $\Sigma_{r^{m-1}_*}$. 
We start with the de-coupled quantity 
$a^{m-1}_{\Th 2}(t,\theta)$. This is defined by  \eqref{a2Th}; the 
expression \eqref{te1m.tK12}
 for $\te^{m-1}_2$ in terms of the background coordinates 
$\theta,t$, together with the bounds on $\tilde{K}^{m-1}_{12}(t,\theta)$ 
implies our claim  on $\Sigma_{r^{m-1}_*}$ for this parameter. 
For the  parameters $a^{m-1}_{T1}, a^{m-1}_{T2}$ on $\Sigma_{r^{m-1}_*}$ we outlined how the estimates in the 
claimed spaces follow using the estimates we are assuming on $\tilde{K}^{m-1}_{12}$ and $a^{m-1}_{\Th 2}$ on $\Sigma_{r^{m-1}_*}$. 
We derive the claimed bounds by simply applying the inductive assumptions on those parameters 
along with the product inequality. 


Now, we can obtain our bounds for the parameters
 $a^{m-1}_{\Th2}, a^{m-1}_{T2}, a^{m-1}_{T1}$ off
  of $\Sigma_{r^{m-1}_*}$ by using the integral representations

\beq
\begin{split}
\label{explicit.int.again}
&a^{m-1}_{T1}(r,t,\theta)= e^{-\int_{r^{m-1}_*(t,\theta)}^{r} K^{m-1}_{11} (1-\frac{2M}{r})^{-\frac{1}{2}} ds}\cdot a^{m-1}_{T1}(r^{m-1}_*(t,\theta),t,
\theta),
\\&a^{m-1}_{T2}(r,t,\theta)= e^{-\int_{r^{m-1}_*(t,\theta)}^{r} K^{m-1}_{22} (1-\frac{2M}{r})^{-\frac{1}{2}} ds}\cdot a^{m-1}_{T2}(r^{m-1}_*(t,\theta),t,
\theta)
\\&- e^{-\int_{r^{m-1}_*(t,\theta)}^{r} K^{m-1}_{22} (1-\frac{2M}{r})^{-\frac{1}{2}} ds}\cdot \int_{r^{m-1}_*}^r e^{\int_{r^{m-1}_*(t,\theta)}^{s} K^{m-1}_{22}(\tau,t,\theta) (1-\frac{2M}{\tau})^{-\frac{1}{2}} d\tau} 2K^{m-1}_{12}(s,t,\theta)a^{m-1}_{T1}(s)\cdot 
 (1-\frac{2M}{s})^{-\frac{1}{2}} ds,
 \\&a^{m-1}_{\Th 2}(r,t,\theta)= e^{-\int_{r^{m-1}_*(t,\theta)}^{r} K^{m-1}_{22} (1-\frac{2M}{r})^{-\frac{1}{2}} ds}\cdot a^{m-1}_{\Th 2}(r^{m-1}_*(t,\theta),t,
 \theta)
\end{split}
\eeq	 
Then our claim follows straightforwardly by just differentiating the above 
equations, and using our assumed bounds on $a^{m-1}_{Ai}$ on 
$\Sigma_{r^{m-1}_*}$ and those on $K^{m-1}_{ij}(r,t,\theta)$:
The required bounds and regularity for the coefficients 
$a^{m-1}_{Ai}(r^{m-1}_*(t,\theta),t,\theta)$  on the initial data hypersurface $\Sigma_{r^{m-1}_*(t,\theta)}$ have already been established. 
Furthermore the functions $K^{m-1}_{22}(r,t,\theta), K^{m-1}_{12}(r,t,\theta)$, $K^{m-1}_{11}(r,t,\theta)$ have the required regularity, as part of the inductive assumption 
(the latter function lies only in $\tilde{H}^{s-3}$ at the top order, which
places certain of the parameters $a^{m-1}_{Ai}$ in the corresponding space, depending on whether they ``see'' $K^{m-1}_{11}$ on their evolution equations or not); this allows us to derive our claim by differentiating the evolution equations and invoking Lemma \ref{Fuchsian}.

$\Box$
\medskip

We will also need to derive bounds on the metric components $g^{m-1}_{\rho A}$,
 $A=T,\Th$. 
  Let us find an expression for these mixed components 
 $g_{\rho T}, g_{\rho \Th}$. 

Using \eqref{partialrho.again}, we derive: 

\beq\bs
\label{g.cross}
&g^{m-1}_{\rho T}= \sum_{i=1,2} a_{Ti}^{m-1} \cdot e_i^{m-1}(\rho^{m-1})
(\frac{2M}{r}-1)^{-1} [1-\partial_r\chi(r) (r^{m-1}_*-\e)]^{-2}, 
\\&g^{m-1}_{\rho \Th}= \sum_{i=1,2} a_{\Th i}^{m-1} \cdot e_i^{m-1}(\rho^{m-1})
(\frac{2M}{r}-1)^{-1} [1-\partial_r\chi(r) (r^{m-1}_*-\e)]^{-2}.
\end{split}
\eeq

Thus 
 the bounds in Lemmas \ref{lem:e1re2r}, \ref{a.oa.bds} directly imply the following bounds for $l\le\rm low $, $J_1\le 1$: 

\beq
\label{mixed.metr.low}
\|(\partial_\rho)^{J_1}g^{m-1}_{\rho \Th}\|_{\dot{H}^l}\le 
B\rho^{2-DC\eta-J_1}, 
\|(\partial_\rho)^{J_1} g^{m-1}_{\rho T}\|_{\dot{H}^l}\le 
B\rho^{\frac{1}{2}-DC\eta -J_1}.
\eeq
For $l\in \{{\rm low}+1, s-4 \}$ and $h=l-{\rm low}$:
\beq\label{mixed.metr.high}
\|(\partial_\rho)^{J_1}g^{m-1}_{\rho \Th}\|_{\dot{H}^l}\le
 B\rho^{2-DC\eta-\frac{h}{4}-J_1}, 
\|(\partial_\rho)^{J_1} g^{m-1}_{\rho T}\|_{\dot{H}^l}\le 
B\rho^{\frac{1}{2}-DC\eta -\frac{h}{4}-J_1}.
\eeq
We also note that since $\rho^{m-1}=r$ for $r\le \e/2$ and $e^{m-1}_2(r)=0$,  $(g^{m-1})_{\rho \Th}$ vanishes for $\{r\le \e/2\}$.

We next derive analogues of these bounds for the components 
 of the \emph{inverse} of $g^{m-1}$: Recall first  
that the vector fields $\overline{e}^{m-1}_i$ can also be expressed 
in terms of the coordinates $\partial_T, \partial_\Theta$ via formulas: 

\beq
\overline{e}^{m-1}_i=(a^{m-1})^{iT}\partial_T+(a^{m-1})^{i\Th}\partial_\Th 
\eeq

The components of the 2x2 matrix $(a^{m-1})^{iA}, i=1,2, A=T,\Th$ 
are then just the inverse of the matrix $a^{m-1}_{Ai}$. In particular we 
derive the following bounds: 

\begin{lemma}
\label{lem:oa.bds}
\begin{align}\label{oa.r.low}
&\|(a^{m-1})^{1T}(r,t,\theta)\|_{\dot{H}^l}
\leq C 
r^{\frac{1}{2}-DC\eta},
\|(a^{m-1})^{ 2 \Th}(r,t,\theta)\|_{\dot{H}^l}
\leq C 
r^{-1-DC\eta},\\
\notag&\|(a^{m-1})^{ 2 T}(r,t,\theta)\|_{\dot{H}^l}
= 0,  
 \|(a^{m-1})^{ 1\Th}(r,t,\theta)\|_{\dot{H}^l}
\leq    C 
r^{-\frac{3}{4}-DC\eta},\text{ } \text{ }\text{for $l\leq {\rm low}$},
\end{align}

	At the higher orders, for ${\rm low}+1\le l\le s-4$, $h=l-{\rm low}$:

	\begin{align}\label{oa.r.middle}
&\|(a^{m-1})^{1T}(r,t,\theta)\|_{\dot{H}^l}
\leq C r^{\frac{1}{2}-DC\eta-\frac{h}{4}},
\|(a^{m-1})^{2 \Th }(r,t,\theta)\|_{\dot{H}^l}
\leq C 
r^{-1-DC\eta-\frac{h}{4}},\\
\notag&\|(a^{m-1})^{2T }(r,t,\theta)\|_{\dot{H}^l}
= 0,
 \|(a^{m-1})^{1\Th }(r,t,\theta)\|_{\dot{H}^l}
\leq  C 
r^{-\frac{3}{4}-DC\eta-\frac{h}{4}},
\end{align}
	
	Finally, at the top orders our claims are: 
	
	\begin{equation}\label{oa.r.top}
	\begin{split}
&\|(a^{m-1})^{1T}
(r,t,\theta)\|_{\tilde{H}^{s-3}}
\leq
r^{\frac{1}{2}-DC\eta-c}
\notag\|(a^{m-1})^{2\Th }(r,t,\theta)\|_{\dot{H}^{s-3}}
\leq C 
r^{-1-DC\eta-c},
\\&\|\partial_\theta (a^{m-1})^{2\Th }(r,t,\theta)\cdot 
{\rm cot}\theta \|_{\dot{H}^{s-4}}
\leq C 
r^{-1-DC\eta-c}, 
\|(a^{m-1})^{2T }(r,t,\theta)\|_{\dot{H}^{s-3}}
=0, 
\|(a^{m-1})^{ 1\Th}(r,t,\theta)\|_{\dot{H}^{s-3}}
\leq r^{-1-DC\eta},
\end{split}
\eeq  

\end{lemma}

This yields some bounds on the components of $(g^{m-1})^{AB}$ (with raised 
indices) 
in the components with respect to this system of coordinates, for $l\leq {\rm low}$ 
 
\begin{align}\label{a.r.low}
\notag\|(g^{m-1})^{TT}(r,t,\theta)\|_{\dot{H}^l}
\leq C 
r^{1-2DC\eta},
\|(g^{m-1})^{\Th \Th}(r,t,\theta)\|_{\dot{H}^l}
\leq C 
r^{-2-2DC\eta},
\|(g^{m-1})^{\Th T}(r,t,\theta)\|_{\dot{H}^l}
\leq   C 
r^{\frac{1}{4}-2DC\eta},
\notag
\end{align} 
 while at the higher orders for ${\rm low}+1\le l\le s-4$, $h=l-{\rm low}$:
 \begin{equation}\label{a.init.middle}\begin{split}
&\|(g^{m-1})^{TT}(r,t,\theta)\|_{\dot{H}^l}
\leq C r^{1-2DC\eta-\frac{h}{4}},
\|(g^{m-1})^{\Th \Th }(r,t,\theta)\|_{\dot{H}^l}
\leq C 
r^{-2-2DC\eta-\frac{h}{4}},
\\& \|(g^{m-1})^{T\Th }(r,t,\theta)\|_{\dot{H}^l}
\le  C 
r^{-\frac{1}{4}-2DC\eta-\frac{h}{4}},
\end{split}
\end{equation}
and at the top orders when $J_1\le 2$: 
\begin{equation}\label{a.init.top}\begin{split}
&\notag\|(\partial_\rho)^{J_1}(g^{m-1})^{TT}
(r,t,\theta)\|_{\tilde{H}^{s-3}}
\leq C 
r^{1-2DC\eta-c-J_1},
\|(\partial_\rho)^{J_1}(g^{m-1})^{\Th\Th }(r,t,\theta)\|_{\dot{H}^{s-3}}
\leq C 
r^{-2-2DC\eta-c-J_1},
\\&\|(\partial_\rho)^{J_1}(g^{m-1})^{\Th T}(r,t,\theta)\|_{\tilde{H}^{s-3}}
\le  r^{-\frac{1}{2}-DC\eta-c-J_1}.
\end{split}
\end{equation}

Using the bounds on all components of $a^{m-1}_{AB}, A,B=\rho,T,\Theta$, we also obtain bounds on the cross  inverse metric components:

\beq
\label{mixed.metr.inv.low}
\|(\partial_\rho)^{J_1}(g^{m-1})^{\rho \Th}\|_{\dot{H}^l}\le 
B\rho^{-\frac{1}{2}-DC\eta-J_1}, 
\|(\partial_\rho)^{J_1} (g^{m-1})^{\rho T}\|_{\dot{H}^l}\le 
B\e^{\frac{1}{2}-DC\eta -J_1}.
\eeq 
(In fact the first term vanishes for $\rho^{m-1}\le \e/2$, but we do not need that fact). The bounds at higher orders
$l\in \{{\rm low}+1,\dots, s-4\}$ are analogous, letting $h=l-{\rm low}$: 
\beq
\label{mixed.metr.inv.high}
\|(\partial_\rho)^{J_1}(g^{m-1})^{\rho \Th}\|_{\dot{H}^l}\le 
B\rho^{-\frac{1}{2}-DC\eta-J_1-\frac{h}{4}}, 
\|(\partial_\rho)^{J_1} (g^{m-1})^{\rho T}\|_{\dot{H}^l}\le 
B\e^{\frac{1}{2}-DC\eta-J_1-\frac{h}{4}}.
\eeq 

\medskip

 We also put down some estimates for the 
 Christoffel symbols $(\Gamma^{m-1})_{AB}^C$  
 $C=T,\Th$ which will be useful:

\begin{lemma} 
\label{Christ.bounds}

 At the lower and higher orders our claimed bounds for these Christoffel symbols are

  \beq
  \label{Christ.lower}\begin{split}
&  \|(\Gamma^{m-1})^\Th_{\Th\Th}
\|_{H^l[\Sigma_{\rho^{m-1}}]}\le C (\rho^{m-1})^{-2CD\eta}, \|(\Gamma^{m-1})^\Th_{TT}
\|_{H^l[\Sigma_{\rho^{m-1}}]}\le C (\rho^{m-1})^{-3-2CD\eta},
\\& \|(\Gamma^{m-1})^\Th_{T\Th}
\|_{H^l[\Sigma_{\rho^{m-1}}]}\le C (\rho^{m-1})^{-1-\frac{1}{4}-2CD\eta},
    \|(\Gamma^{m-1})^T_{TT}
  \|_{H^l[\Sigma_{\rho^{m-1}}]}\le C (\rho^{m-1})^{-2CD\eta},
\\&\|(\Gamma^{m-1})^T_{\Th\Th}
  \|_{H^l[\Sigma_{\rho^{m-1}}]}\le C (\rho^{m-1})^{2+\frac{1}{4}-2CD\eta},
 \|(\Gamma^{m-1})^T_{T\Th}
  \|_{H^l[\Sigma_{\rho^{m-1}}]}\le C (\rho^{m-1})^{-2CD\eta}    
\end{split} 
  \eeq
  for all $l\le {\rm low}-1$. 
  
  At the higher orders, the corresponding bounds are as follows, where 
  $l\in \{low,\dots ,s-5 \}$. 
  
    \beq
  \label{Christ.higher}\begin{split}
&  \|(\Gamma^{m-1})^\Th_{\Th\Th}
\|_{H^l[\Sigma_{\rho^{m-1}}]}\le C (\rho^{m-1})^{-2CD\eta-\frac{h}{4}}, \|(\Gamma^{m-1})^\Th_{TT}
\|_{H^l[\Sigma_{\rho^{m-1}}]}\le C (\rho^{m-1})^{-3-2CD\eta-\frac{h}{4}},
\\& \|(\Gamma^{m-1})^\Th_{T\Th}
\|_{H^l[\Sigma_{\rho^{m-1}}]}\le C (\rho^{m-1})^{-1-\frac{1}{4}-2CD\eta-\frac{h}{4}},
    \|(\Gamma^{m-1})^T_{TT}
  \|_{H^l[\Sigma_{\rho^{m-1}}]}\le C (\rho^{m-1})^{-2CD\eta-\frac{h}{4}},
\\&\|(\Gamma^{m-1})^T_{\Th\Th}
  \|_{H^l[\Sigma_{\rho^{m-1}}]}\le C (\rho^{m-1})^{2+\frac{1}{4}-2CD\eta-\frac{h}{4}},
 \|(\Gamma^{m-1})^T_{T\Th}
  \|_{H^l[\Sigma_{\rho^{m-1}}]}\le C (\rho^{m-1})^{2-\frac{1}{4}-2CD\eta-\frac{h}{4}}    
\end{split} 
  \eeq
  We also have the following extra bounds at the top orders, where 
  $J_1\le 2$

    \beq
  \label{Christ.top}\begin{split}
&  \|(\partial_\rho)^{J_1} (\Gamma^{m-1})^\Th_{\Th\Th}\cdot cot\Th
\|_{\dot{H}^{s-4}[\Sigma_{\rho^{m-1}}]}\le C (\rho^{m-1})^{-2CD\eta-c-\frac{1}{2}-J_1},\\
\notag& \|
(\partial_\rho)^{J_1}(\Gamma^{m-1})^\Th_{TT}\cdot cot\Th
\|_{\dot{H}^{s-4}[\Sigma_{\rho^{m-1}}]}\le C 
(\rho^{m-1})^{-3-2CD\eta-c-\frac{1}{2}-J_1},
\\& \|(\partial_\rho)^{J_1}(\Gamma^{m-1})^\Th_{T\Th}\cdot cot\Th
\|_{\dot{H}^{s-4}[\Sigma_{\rho^{m-1}}]}\le C 
(\rho^{m-1})^{-1-\frac{1}{4}-2CD\eta-c-\frac{1}{2}-J_1},\\
\notag&    \||(\partial_\rho)^{J_1}(\Gamma^{m-1})^T_{TT}
  \|_{\tilde{H}^{s-4}[\Sigma_{\rho^{m-1}}]}\le C (\rho^{m-1})^{-2CD\eta-c-J_1},
\\&\|(\partial_\rho)^{J_1}(\Gamma^{m-1})^T_{\Th\Th}
  \|_{\dot{H}^{s-4}[\Sigma_{\rho^{m-1}]}}\le C (\rho^{m-1})^{2+\frac{1}{4}-2CD\eta-c-J_1},\\
\notag & \|(\partial_\rho)^{J_1}(\Gamma^{m-1})^T_{T\Th}
  \|_{\tilde{H}^{s-4}[\Sigma_{\rho^{m-1}}]}\le C (\rho^{m-1})^{2-\frac{1}{4}-2CD
  \eta-c-J_1}    
\end{split} 
  \eeq
 \end{lemma}
 
 We will also need to put down some estimates on the Christoffel symbols 
 $(\Gamma^{m-1})_{AB}^C$  where at least one of the indices $A,B$
  (say $A$ wlog)
  equals $\rho$. 
Given our bounds on $\partial(g^{m-1}_{AB})$ and $(g^{m-1})^{CD}$   right 
above, we can 
obtain the following bounds on these Christoffel symbols:\footnote{Note that all these quantities vanish in Schwarzschild.} 

\begin{lemma}
\label{Christof.off.diag}
\beq\begin{split}\label{Christof.off.diag.low}
&\| (\Gamma^{m-1})_{\rho \Th}^\rho \|_{H^{\rm low}}\le BC\eta 
(\rho^{m-1})^{-\frac{1}{8}},  \| (\Gamma^{m-1})_{\rho T}^\rho \|_{H^{\rm low}}\le BC\eta 
(\rho^{m-1})^{-2-\frac{1}{4}-\frac{1}{8}},
\\&\| (\Gamma^{m-1})_{\rho T}^T \|_{H^{\rm low}}\le BC\eta (\rho^{m-1})^{-1-\frac{1}{8}}
\| (\Gamma^{m-1})_{\rho \Th}^\Th \|_{H^{\rm low}}\le BC\eta
 (\rho^{m-1})^{-1-\frac{1}{8}},
\\&\|(\Gamma^{m-1})_{\rho T}^\Th \|_{H^{\rm low}}\le B C\eta (\rho^{m-1})^{-2-\frac{1}{4}-\frac{1}{8}},
\| (\Gamma^{m-1})_{\rho \Th}^T \|_{H^{\rm low}}\le B C\eta
 (\rho^{m-1})^{-\frac{1}{8}},
 \\&\|(\Gamma^{m-1})_{\rho \rho}^\Th \|_{H^{\rm low}}\le BC\eta (\rho^{m-1})^{-\frac{3}{2}-\frac{1}{8}},
\| (\Gamma^{m-1})_{\rho \rho}^T \|_{H^{\rm low}}\le BC\eta(\rho^{m-1})
 (\rho^{m-1})^{-\frac{1}{4}-\frac{1}{8}}.
\end{split}
\eeq
The analogues of these estimates at the higher derivatives are as follows, 
for $h=l-{\rm low}$, $l\in \{{\rm low}+1, \dots, s-4$: 

\beq\begin{split}\label{Christof.off.diag.higher}
&\| (\Gamma^{m-1})_{\rho \Th}^\rho \|_{H^l}\le BC\eta 
(\rho^{m-1})^{-\frac{1}{8}},  \| (\Gamma^{m-1})_{\rho T}^\rho \|_{H^l}\le BC\eta 
(\rho^{m-1})^{-2-\frac{1}{4}-\frac{1}{8}}
\\&\| (\Gamma^{m-1})_{\rho T}^T \|_{H^l}\le BC\eta 
(\rho^{m-1})^{-1-\frac{1}{8}-\frac{h}{4}}
\| (\Gamma^{m-1})_{\rho \Th}^\Th \|_{H^l}\le BC\eta
 (\rho^{m-1})^{-1-\frac{1}{8}},
\\&\|(\Gamma^{m-1})_{\rho T}^\Th \|_{H^l}\le B C
\eta (\rho^{m-1})^{-2-\frac{1}{4}-\frac{1}{8}-\frac{h}{4}},
\| (\Gamma^{m-1})_{\rho \Th}^T \|_{H^l}\le B C\eta
 (\rho^{m-1})^{-\frac{1}{8}-\frac{h}{4}},
 \\&\|(\Gamma^{m-1})_{\rho \rho}^\Th \|_{H^l}\le BC\eta
  (\rho^{m-1})^{-\frac{3}{2}-\frac{1}{8}-\frac{h}{4}},
\| (\Gamma^{m-1})_{\rho \rho}^T \|_{H^l}\le BC\eta(\rho^{m-1})
 (\rho^{m-1})^{-\frac{1}{8}-\frac{h}{4}}.
\end{split}
\eeq

At the top derivatives, with singular weights at the poles,  our bounds are: 
\beq\begin{split}\label{Christof.off.diag.top}
&\|\partial_\theta (\Gamma^{m-1})_{\rho \Th}^\Th\cdot {\rm cot}\Th 
\|_{\dot{H}^{s-4}}\le BC\eta
 (\rho^{m-1})^{-1-\frac{1}{8}-c},
\|(\Gamma^{m-1})_{\rho T}^\Th {\rm cot}\Th\|_{\dot{H}^{s-4}}\le B 
C\eta (\rho^{m-1})^{-2-\frac{1}{4}-\frac{1}{8}-c},
 \\&\|(\Gamma^{m-1})_{\rho \rho}^\Th {\rm cot}\Th\|_{\dot{H}^{s-4}}\le BC\eta
  (\rho^{m-1})^{-\frac{3}{2}-\frac{1}{8}-c}, 
\end{split}
\eeq

\end{lemma}

 \emph{Proof of Lemmas \ref{Christ.bounds}, \ref{Christof.off.diag}:}
 The proof is straightforward, from the definition 
 of of the Christoffel symbols, the bounds obtained for $g^{m-1}_{AB}$ and 
 $(g^{m-1})^{AB}$ directly above, as well as the 
 product inequality. 
 $\Box$

\medskip

For the remaining Christoffel symbols as they will appear (combined) in the wave equation, 
\[
\|\sum_{A,B=T,\Th, \phi}(g^{m-1})^{AB} \Gamma_{AB}^\rho\|_{\dot{H}^l}, 
\]
 we
 make the following connections to the (frame) 
connection coefficients $K^{m-1}_{ij}$, using the definitions of the frame coefficients and the definition of the Levi-Civita connection in 
terms of Christoffel symbols.
 (We use $O(\dots)$ to denote a term which satisfies all the bounds of the quantity 
 $(\dots)$  in parentheses, up to a universal multiplicative constant).

\beq\label{get.trK}\bs
&\sum_{A,B=T,\Th, \phi}(g^{m-1})^{AB} \Gamma_{AB}^\rho\partial_\rho v
=(1+O(\sqrt{r})) {\rm tr}_{\overline{g}^{m-1}} 
K^{m-1}\cdot e_0v
\\&+ [O(e^{m-1}(\rho^{m-1}))\cdot O(\sqrt{r}) +O(\sqrt{r})\cdot [a^{m-1}_{T1}a^{m-1}_{T2}\cdot e^{m-1}_1re^{m-1}_2r]\cdot e_0v
\end{split}
\eeq

The above can be derived as follows, using the conditions $\nabla_{e_0}e_0=0$, $<\nabla_{e_i}e_0, e_0>=0$:
\beq
\label{pf.get.trK}
\begin{split}
&\sum_{i=1,2,3}<\nabla_{e^{m-1}_i}e_0, e^{m-1}_i>= \sum_{i=1,2,3}<\nabla_{\overline{e}^{m-1}_i}e_0, e^{m-1}_i>=\sum_{i=1,2,3}<\nabla_{\overline{e}^{m-1}_i}e_0, \overline{e}^{m-1}_i>
\\&= \sum_{i=1,2,3}<\nabla_{\overline{e}^{m-1}_i}[[1-\partial_r\chi(r)(r)(r^{m-1}_*-\e)]^{-1}(\frac{2M}{r}-1)^{\frac{1}{2}}\partial_\rho], \overline{e}^{m-1}_i>
\\&=
[1-\partial_r\chi(r)(r)(r^{m-1}_*-\e)]^{-1}(\frac{2M}{r}-1)^{\frac{1}{2}}\sum_{i=1,2,3}<\nabla_{\overline{e}^{m-1}_i}\partial_\rho], \overline{e}^{m-1}_i>
\\&+e^{m-1}[[1-\partial_r\chi(r)(r)(r^{m-1}_*-\e)]^{-1}(\frac{2M}{r}-1)^{\frac{1}{2}}]\cdot
(\frac{2M}{r}-1)^{-\frac{1}{2}}e^{m-1}_1(r)
\\&=[1-\partial_r\chi(r)(r)(r^{m-1}_*-\e)]^{-1}(\frac{2M}{r}-1)^{\frac{1}{2}}\sum_{i=1,2,3}\sum_{A,B=T,\Th,\phi}(a^{m-1})^{iA}(a^{m-1})^{iB}<\nabla_{A}\partial_\rho, \partial_B>
\\&+e^{m-1}[[1-\partial_r\chi(r)(r)(r^{m-1}_*-\e)]^{-1}(\frac{2M}{r}-1)^{\frac{1}{2}}]\cdot
(\frac{2M}{r}-1)^{-\frac{1}{2}}e^{m-1}_1(r)
\\&=[1+O(r^{\frac{1}{2}}e^{m-1}_1(\rho)]](\frac{2M}{r}-1)^{\frac{1}{2}}\sum_{i=1,2,3}\sum_{A,B=T,\Th,\phi}(g^{m-1})^{AB}<\nabla_{A}\partial_\rho, \partial_B>
\\&+e^{m-1}[[1-\partial_r\chi(r)(r)(r^{m-1}_*-\e)]^{-1}(\frac{2M}{r}-1)^{\frac{1}{2}}]\cdot
(\frac{2M}{r}-1)^{-\frac{1}{2}}e^{m-1}_1(r)
\\&+[a^{m-1}_{T1}a^{m-1}_{T2}O(r)]\cdot 
e_1^{m-1}(\rho)e^{m-1}_2(\rho).
\end{split}
\eeq
(We used that $a^{m-1}_{\Th 1}=0$ here). Using the defintiion \eqref{e0r} here we derive \eqref{get.trK}. 
\newline

We note here again that $e^{m-1}_2(\rho)=0$ for $\rho\le \e/2$; also 
$e^{m-1}_1(\rho)$ satisfies the bounds in Lemma \ref{lem:e1re2r},
 and notably in the $L^\infty$ 
norm the coefficient $O(e^{m-1}(\rho^{m-1}))$ is bounded by 
$\rho^{-DC\eta}$ (where $DC\eta|\le \frac{1}{8}$); the higher derivatives are bounded by the corresponding bounds for $e^{m-1}(\rho^{m-1})$.
In particular the second term in the RHS is \emph{less singular} near the singularity. (The tangency of $e_2^{m-1}$ to  the level sets of $\rho$ is 
crucial here; had this choice not been made, the second term would have been more singular than the first term). 
\newcommand{\she}{{\sharp\e}}

\subsubsection{Consequences of the energy estimates on the free wave 
$\gamma^m$. }

We will also make frequent use below of certain basic implications of our 
inductive claim. One key such estimate encodes the basic implication of 
the bounds on $(a^{m-1})^{it}(r,t,\theta), (a^{m-1})^{i\theta}(r,t,\theta)$; 
it shows that the behaviour of the parameters 
$e^{m-1}_1\partial^I\gamma^m_{\rm rest}$, 
$e^{m-1}_2\partial^I\gamma^m_{\rm rest}$ is \emph{better} than 
what the energy bounds  \eqref{inductiongammaopt}, 
\eqref{inductiongammalow}, \eqref{inductiongammatopmixed} imply, 
\emph{at all orders below the top order}.

We first  prove the following: 

\begin{lemma}
\label{lem:gammam.bds.prep}
For all $|I|\le s-3$ the quantity 

\[
\| \partial^I\gamma^m_{\rm rest}\|_{H^k(\Sigma_{\rho^{m-1}})}
 \]
satisfies the estimate: 
\begin{align}\label{Hlgammamest}
(\rho^{m-1})^{\frac{1}{4}}\|\gamma^m_{\rm rest} \|_{H^l(\Sigma_{\rho^{m-1}})}\leq&\,\epsilon^{\frac{1}{4}}\|\gamma^m_{\rm rest}
(\epsilon,t,\theta)\|_{H^l(\Sigma_{\rho^{m-1}})}
+\int^\epsilon_{\rho^{m-1}}C\tau^{\frac{1}{4}+\frac{1}{2}}\|
e_0\gamma^m_{\rm rest}\|_{H^l(\Sigma_{\rho^{m-1}})}
d\tau,
\end{align}

\end{lemma}
\emph{Proof:}

\begin{align*}
\tag{by C-S}-\frac{1}{2}\partial_{\rho^{m-1}}\|\gamma^m_{\rm rest}\|^2_{H^l}\leq&\, \|\gamma^m_{\rm rest}\|_{H^l}\|\partial_{\rho^{m-1}}\gamma^m\|_{H^l}\\
\Rightarrow-\frac{1}{2}\partial_{\rho^{m-1}}\big((\rho^{m-1})^{\frac{1}{2}}\|\gamma^m_{\rm rest}\|^2_{H^l}\big)\leq&\,-\frac{1}{4}(\rho^{m-1})^{-\frac{1}{2}}\|\gamma^m_{\rm rest}\|_{H^l}^2
+(\rho^{m-1})^{\frac{1}{2}}\|\gamma^m_{\rm rest}\|_{H^l}\|\partial_{\rho^{m-1}}\gamma^m_{\rm rest}\|_{H^l}\\
\Rightarrow-\partial_{\rho^{m-1}}\big((\rho^{m-1})^{\frac{1}{4}}\|\gamma^m_{\rm rest}\|_{H^l}\big)\leq&\,(\rho^{m-1})^{\frac{1}{4}}\|\partial_{\rho^{m-1}}\gamma^m_{\rm rest}\|_{H^l}
\end{align*}
Hence, integrating in $[\rho^{m-1},\epsilon]$ yields our result 
on each $\Sigma_{\rho^{m-1}}$. $\Box$
\medskip

From this we derive: 

\begin{lemma}\label{lem:Linftyest}
	The lower order energy estimates (\ref{inductiongammaopt}),
	(\ref{inductiontrKopt}) as well as Lemma 
	\ref{lem:oa.bds} imply the bounds:
	\begin{align}\label{Linftybounds}
	\|e_0\partial^I\gamma_{\rm rest}^{m-1}\|_{L^\infty}\leq C\cdot C_S\eta r^{-\frac{3}{2}}, \|^{{m-1}}\overline{\nabla}\partial^I\gamma_{\rm rest}^{m-1}\|_{L^\infty}\leq 2C\cdot C_S\eta r^{-1-DC\eta }|\log r|,
	\end{align}
	for all $r\in(0,2\epsilon]$, provided $s\ge 3+4c+|I|+3$. 
\end{lemma}
\begin{proof}
	The bounds on $e_0\partial^I\gamma^{m-1}_{\rm rest}$ follow by applying (\ref{Sob}) to $e_0\partial^I\gamma_{1}^{m-1}$, see (\ref{gammam-1exp}), and using the assumption (\ref{gamma1m-1enest}), using \eqref{s.bd}. 
	%
	%
	For the spatial gradient of $\partial^I\gamma_{\rm rest}^{m-1}$ we have 
	$|\overline{\nabla}^{^{m-1}h}\partial^I\gamma_{\rm rest}^{m-1}|\leq
	 |e^{m-1}_1\partial^I\gamma_\theta^{m-1}|+|e^{m-1}_2\partial^I
	\gamma_{\rm rest}^{m-1}|$, where 
	\begin{align*}
	|e^{m-1}_i\partial^I\gamma_{\rm rest}^{m-1}|\leq&\, |{^{m-1}\overline{e}}_i\opartial^I\gamma_{\rm rest}^{m-1}|+|(e^{m-1}_i\rho^{m-1})\partial_{\rho^{m-1}}\partial^I\gamma_{\rm rest}^{m-1}|\\
	\tag{by (\ref{e1rest})}\leq &\,|{\overline{a}}^{m-1}_{iT}\partial_T\partial^I\gamma_{\rm rest}^{m-1}|+|{\overline{a}}^{m-1}_{i\Th}\partial_\theta\partial^I\gamma_{\rm rest}^{m-1}|+\frac{r^{-DC\eta}
	(\frac{2M}{r}-1)^{-\frac{1}{2}}}{1-\partial_r\chi(r)({^{m-1}r}_*-\epsilon)}|e_0\partial^I\gamma_{\rm rest}^{m-1}|\\
	\tag{by (\ref{gammam-1exp}),(\ref{gamma1m-1enest}),(\ref{oa.r.low})}\leq&\,C_Sr^{-1-\eta D C }\|\opartial^I\gamma_{\rm rest}^{m-1}\|_{H^{3}}+C_Sr^{-1-\eta DC }+O(r^{\frac{1}{4}})\\
	\tag{$s\ge 3+4c+|I|+3$}\leq&\,Cr^{-1-DC\eta  }|\log r|,
	\end{align*}
	for all $r\in(0,2\epsilon]$.
\end{proof}
An extension of the above estimate on the spatial components of the 
energy of $\gamma^m_{\rm rest}$ is as follows:

\begin{lemma}
\label{lem:gammam.bds}
Assume the  energy estimates 
\eqref{inductiongammaopt}, 
\eqref{inductiongammalow}, \eqref{inductiongammatopmixed}, for the $m^{\rm th}$ step of our induction (i.e. for the function 
$\gamma^m$), up to order $|I'|=l+1$. 
 Then the 
coordinate derivatives $\partial^I$ of order $|I|=l$ will be bounded as follows: 

\beq
\label{lower.der.indn}
\begin{split}
&\| \partial^I\gamma^m_{\rm rest} \|_{L^2[\Sigma_r]}\le  C\eta|log r|, \text{for}
\;\;\; 
|I|\leq s-3-4c,
\\&\| \partial^I\gamma^m_{\rm rest} \|_{L^2[\Sigma_r]}\le  C\eta 
|log r|r^{-\frac{h}{4}}, \text{for}
\;\;\; 
 s-3-4c<|I|\leq s-4, h=|I|-(s-3-4c).
 \end{split}
\eeq

On the other hand, the 
parameters 
$e^{m-1}_1\opartial^I\gamma^m_{\rm rest}$, 
$e^{m-1}_2\opartial^I\gamma^m_{\rm rest}$ satisfy the following bounds,
for $b=1,2$:

\begin{align}
\label{impl.gammaopt} \| e^{m-1}_b
[\partial^I(\gamma^{m}_{\rm rest})]|_{L^2(\Sigma_r)}
\leq C\cdot\eta\cdot r^{-\frac{3}{2}+\frac{1}{4}},\qquad\text{for}\;\;\; 
|I|\leq s-3-4c-1,\\
\label{impl.gammalow}
\|e^{m-1}_b[\partial^I(\gamma^{m}_{\rm rest})]\|_{L^2(\Sigma_r)}
\leq C\cdot \eta\cdot r^{-\frac{3}{2}+\frac{1}{4}+(|I|-{\rm low})\frac{1}{4}},
\qquad
\text{for}\;\;\; s-3-4c\le |I|\leq s-4.
\end{align}
Moreover, 
the same bounds hold on the level sets of $\rho^{m-1}$. 
\end{lemma}

\begin{proof}
For the first claim, 
the proof is immediate, using the inductive assumption 
on the energy of $\gamma^m_{\rm rest}$ (whichever is relevant 
for our order), as well as \eqref{Hlgammamest}. (In fact we can bond these terms by $log r$--the weaker bound claimed here is sufficient for our purposes).

We can then derive \eqref{impl.gammaopt}, \eqref{impl.gammalow}:
We use formulas 
\eqref{tthetatransebar}, \eqref{eibar}  Lemma \ref{lem:e1re2r}
to express the LHSs of \eqref{impl.gammaopt}, \eqref{impl.gammalow}
in terms of $\partial^{I'}\gamma^m_{\rm rest}$, $|I'|=|I|+1$:
\beq\label{re-express}
e^{m-1}_b\partial^I(\gamma^m_{\rm rest})=\sum_{A=T,\Th}(a^{m-1})^{bA}\partial_A
\partial^I\gamma^m_{\rm rest}+e^{m-1}_b(\rho^{m-1})\partial_\rho \partial^I(\gamma^m_{\rm rest})
\eeq
 in view of the bounds 
in  Lemmas \ref{lem:e1re2r} and \ref{lem:oa.bds}, our claim follows, on the level sets of $\rho^{m-1}$. The claim \eqref{lower.der.indn}
 also follows by directly involving the assumed $L^2$ bounds of $e_0\partial^I\gamma^m_{\rm rest}$ in view of the energy bounds on 
$\partial^I\gamma^m_{\rm rest}$ that we are assuming.

The claim on level sets of $r$ also follows by the same integration in $\partial_\rho$ argument, this follows by a general procedure 
we outline in the end of the next section. 
\end{proof}

\begin{remark}
The proof of \eqref{impl.gammaopt}, \eqref{impl.gammalow}
is the realization of the \emph{descent scheme} mentioned in the 
introduction. We note it can be applied to all orders below the top ones. 
In particular (as we will see) it implies that at the orders below top, 
the spatial derivatives $e^{m-1}_b\gamma^{m}_{\rm rest}, b=1,2$ of $\gamma^m_{\rm rest}$ 
that appear in the RHSs of the Riccati equations  
are \emph{less singular} than the derivatives $e_0\gamma^{m-1}$, 
and the main terms that contribute to 
the asymptotics of the connection coefficients $K^m_{ij}$ 
are the $e_0$-derivatives. 
This is a manifestation of the AVTD behaviour of the fields. 
\end{remark}

Following the control on the spatial part of the metric $h^{m-1}$ and 
$g^{m-1}$ and the refined control of $\gamma^m$ using the AVTD behaviour of
 the solution, we are ready to derive the next step of our inductive 
 claims. One final note prior to doing this: Our inductive claims for the derivatives 
 of the key parameters in the REVESNGG system were with respect to the 
 vector fields $\partial_t,\partial_\theta$. However in view of the 
 construction of the new coordinate $T(t,\theta)$ in this section, 
 we see that it suffices to derive our inductive claims on the derivatives 
 $\partial^I_{T\dots T\Th\dots\Th}$ of our parameters instead of
 the derivatives  $\partial^I_{t\dots t\theta\dots\theta}$.

\section{The estimates for the next iterate: The free wave $\gamma^m$. }\label{sec:itergamma}
Here we study  the free wave equation:
\beq
\label{waveq.it}
\Box_{g^{m-1}}(\gamma^m)=0, 
\eeq
which  holds on the entire region 
$[r\in (0,2\e)]\times \mathbb{S}^2\times\mathbb{R}$. 
\medskip

We recall that the prescribed initial data for this equation are required 
to live 
on a hypersurface $\Sigma_{r^{m-1}_*}$ given graphically by:
\[
\Sigma_{r_*^{m-1}}:=\{r=r^{m-1}_*(t,\theta)\}. 
\]
The function $r^{m-1}_*(t,\theta)$ is here assumed to satisfy the 
inductive assumption on regularity given by \eqref{r*.ind.claim.low}, \eqref{r*.ind.claim.high}, \eqref{te2.r*.low}, \eqref{te2.r*.high}.  

We then proceed to solve this equation in the entire region 
$r\in (0,2\e)\times \mathbb{S}^2\times\mathbb{R}$ using energy 
estimates. As noted, we will not be controlling $\gamma^m$ directly but 
instead $\gamma^m_{\rm rest}=(\gamma^m-\gamma^{\rm S})$. 

To do this, we recall the function  $\rho^{m-1}$; our 
energy estimates will live over level sets of this function. 
 We also recall that it suffices to use the   operators $\opartial^I$ instead of 
$\partial^{J}_{T\dots T\Th\dots \Th}$ to derive our claimed estimates on the next iterate $\gamma^m$. We will be employing these operators, since their commutation 
properties with the wave operator are more favorable. 

So, 
we will be considering the wave equation commuted with coordinate and 
frame vector fields, considering the equations:
\beq
\label{commuted.wav}
\Box_{g^{m-1}}\opartial^{k_1+k_2}_{T\dots T\Theta\dots\Theta} 
e_0^{J_0}(\gamma^m-\gamma^{\rm S})=[\Box_{g^{m-1}}, 
\opartial^{k_1+k_2}_{T\dots T\Theta\dots\Theta} e_0^{J_0}]
(\gamma^m-\gamma^S)-\opartial^{k_1+k_2}_{T\dots T\Theta\dots\Theta} 
e_0^{J_0}\Box_{g^{m-1}}(\gamma^S).
\eeq
here $k_1+k_2\le s-3$, $J_0\le 2$; in fact $J_0=0$ unless
 $k_1+k_2=s-3$.

\medskip

In particular, we will be deriving energy estimates for the equation: 
\beq
\label{wav.v}
\Box_{g^{m-1}}v=G(r,t,\theta),
\eeq
and then replacing the RHS from \eqref{commuted.wav}. 
We distinguish three  cases which are treated separately: 

The first case is where $k=k_1+k_2\le {\rm low}$. In this case, the 
inductive claim \eqref{inductiongammaopt} 
we must confirm (for the value $m$ of the index) 
asserts the optimal behaviour for $\gamma^m_{\rm rest}$; in fact the claim 
\eqref{gammam-1exp} is yet more refined (also optimal) information. The second 
case is when ${\rm low}+1\le k_1+k_2\le s-4$. The last one is $k_1+k_2=s-3$ and
 $J_0=0,1,2$. 
At orders higher than $\rm low$ the bounds claimed are non-optimal.

\medskip

These claims are asserted on level sets of the function $\rho^{m-1}$ and 
also on level sets of $r$. We derive the claims on the level sets of
 $\rho^{m-1}$ and then discuss at the end of this section how to extend these to the level sets of 
 $r$. 
 \medskip
 
  It is useful to put down some facts about the intrinsic and extrinsic 
  geometry of these level sets in the metric $g^{m-1}$. We do this right 
  below: 

\subsubsection{Some estimates on the geometry of the level sets of
 $\rho^{m-1}$.}

The lapse $\Phi^{m-1}$ (\ref{Phi}) behaves like

\begin{align}\label{lapselike}
\Phi^{m-1}=
( 1+B r^{\frac{1}{4}}).
\end{align}

We determine the behaviour of the volume form 
$\mathrm{vol}_{\Sigma_{\rho^{m-1}}}$ by integrating the mean curvature
 (\ref{trKheur}) (using the first variation of area formula, along with the fact that $e_0$ is \emph{asymptotically normal} to the level set 
$\Sigma_{\rho^{m-1}}$, captured via $e^{m-1}_1(r)\cdot [e_0(r)]=O(r^{3/8})$):

\begin{align}\label{volSigmarho}
\mathrm{vol}_{\Sigma_{\rho^{m-1}}}\sim e^{\gamma^{m-1}}
e^{\int^{r_*}_r(\frac{2M}{\tau}-1)^{-\frac{1}{2}}
(K_{11}^{m-1}+K_{22}^{m-1})d\tau}\mathrm{vol}_{Euc}\sim(2M)^{-\frac{1}{2}} r^\frac{3}{2}
\mathrm{vol}_{Euc},\qquad\mathrm{vol}_{Euc}= \sin\theta dtd\theta d\phi.
\end{align}


\subsubsection{General framework for the energy estimates: The 
weighted multiplier.}

We will use $r$-weighted $e_0$ multipliers to obtain energy estimates for the \emph{free} wave $\gamma^m$ (which solves \eqref{waveq.it}). 

For now, for any function $v(r,t,\theta,\phi)$, with $\partial_\phi v=0$, define the weighted $e_0$-current:
\begin{align}\label{J}
J_a=Q_{ab}[v](e_0)^bf(r)=f(r)(e_0)^b(\partial_av\partial_bv-\frac{1}{2}{g^{m-1}_{ab}}\partial^dv\partial_dv).
\end{align}
Note that $Q_{00}[v]=\frac{1}{2}[(e_0v)^2+|\overline{\nabla}^{m-1}v|^2]$.
The  divergence of (\ref{J}) (with respect to the $(3+1)$-dimensional metric $g^{3+1}$) reads: 
\begin{align}\label{divJ}
^{m-1}\nabla^aJ_a=&\,\notag^{m-1}\nabla^a\big[Q_{ab}[v](e_0)^bf(r)\big]\\
\notag=&\,-e_0[f(r)]Q_{00}[v]+f(r)(K^{m-1})^{ab}(\partial_av\partial_bv-\frac{1}{2}{^{m-1}g_{ab}}\partial^dv\partial_dv)+f(r)e_0v\square_{^{m-1}g} v\\
\notag=&\,-e_0[f(r)]\frac{1}{2}\big[(e_0v)^2+|\overline{\nabla}v|^2\big]-\frac{1}{2}\text{tr}_{^{m-1}\overline{g}}K^{m-1}f(r)|\nabla v|^2_{^{m-1}g}\\
\notag&+f(r)\big[K^{m-1}_{11}(e_1v)^2+K^{m-1}_{22}(e_2v)^2+K^{m-1}_{12}e_1ve_2v\big]+f(r)e_0v\square_{^{m-1}g} v\\
=&-\frac{1}{2}\text{tr}_{^{m-1}\overline{g}}u^{m-1}f(r)|\nabla v|^2_{^{m-1}g}-\big[\frac{3}{2}\frac{\sqrt{2M}}{r^\frac{3}{2}}f(r)+e_0f(r)\big]\frac{1}{2}(e_0v)^2\\
\notag&+\big[\frac{3}{2}\frac{\sqrt{2M}}{r^\frac{3}{2}}f(r)-e_0f(r)\big]\frac{1}{2}|\overline{\nabla}v|^2_{^{m-1}\overline{g}}
+(\frac{d_1^{m-1}\sqrt{2M}}{\tau^\frac{3}{2}}+u_{11}^{m-1})f(r)(e_1v)^2\\
\notag&+(\frac{d_2^{m-1}\sqrt{2M}}{\tau^\frac{3}{2}}+u_{22}^{m-1})f(r)(e_2v)^2
+f(r)u_{12}^{m-1}e_1ve_2v+f(r)e_0v\square_{^{m-1}g} v.
\end{align}
Note that we have made use of the asymptotically CMC property; indeed ${\rm tr}_{\overline{g}^{m-1}}K^{m-1}$
to leading order contributes the (constant in $t,\theta$) factor 
$\frac{3}{2}\frac{\sqrt{2M}}{\tau^{\frac{3}{2}}}$ above. 
\medskip

We integrate the above over the domain defined by 
$\{ t\in (-\infty,+\infty), \theta\in (0,\pi), \phi\in [0,2\pi)\}$ 
and two level sets of $\rho^{m-1}$; 
one is denoted merely by $\Sigma_{\rho^{m-1}}$ 
and the value of  $\rho^{m-1}$  can be arbitrarily
 small,  and the other is $\{\rho^{m-1}=\epsilon\}$. (Recall that 
 $\{\rho^{m-1}=\epsilon\}$ is the hypersurface 
 $\{r=r^{m-1}_*\}=\Sigma_{r^{m-1}_*}$ 
 on which the initial data live).

Integrating (\ref{divJ}) over this domain, 
 and employing Stokes' theorem, we write the LHS of (\ref{divJ}) as a boundary integral to obtain the energy identity:
\begin{align}\label{Stokesv}
&\int_{\Sigma_{\rho^{m-1}}}f(r(^{m-1}\rho))Q_{ab}[v](e_0)^bn^a\mathrm{vol}_{\Sigma_{\rho^{m-1}}}-\int_{\Sigma_{{r^{m-1}_*}}}f({^{m-1}r}_*)Q_{ab}[v](e_0)^bn^a\mathrm{vol}_{\Sigma_{{{r^{m-1}_*}}}}\\
\notag=&\int_{\rho^{m-1}}^{\epsilon}\int_{\Sigma_\tau}\Phi^{m-1} \bigg[\frac{1}{2}\text{tr}_{^{m-1}\overline{g}}u^{m-1}f(\tau)|\nabla v|^2_{^{m-1}g}+\big[\frac{3}{2}\frac{\sqrt{2M}}{\tau^\frac{3}{2}}f(\tau)+e_0f(\tau)\big]\frac{1}{2}(e_0v)^2\\
\notag&-\big[\frac{3}{2}\frac{\sqrt{2M}}{\tau^\frac{3}{2}}f(\tau)-e_0f(\tau)\big]\frac{1}{2}|\overline{\nabla}v|^2_{^{m-1}\overline{g}}
-(\frac{d_1^{m-1}\sqrt{2M}}{\tau^\frac{3}{2}}+u_{11}^{m-1})f(\tau)(e_1v)^2\\
\notag&-(\frac{d_2^{m-1}\sqrt{2M}}{\tau^\frac{3}{2}}+u_{22}^{m-1})f(\tau)(e_2v)^2
-f(\tau)u_{12}^{m-1}e_1ve_2v-f(\tau)e_0v\square_{^{m-1}g} v
\bigg]
\mathrm{vol}_{\Sigma_\tau}ds,
\end{align}
where we recall that by virtue of our inductive assumption $\|u^{m-1}_{ij}\|_{L^\infty}\leq B
r^{-1-\frac{1}{4}}$ and $e_0\rho^{m-1}=-(\frac{2M}{r}-1)^\frac{1}{2}[1-\partial_r\chi(r)({^{m-1}r}_*-\epsilon)]$. On the other hand, by (\ref{nSigmarho}) and (\ref{e1rest}) we find:
\begin{align}\label{Q0n}
Q_{ab}[v](e_0)^bn^a=&\,\frac{\frac{1}{2}[(e_0v)^2+|\overline{\nabla}v|^2_{^{m-1}\overline{h}}]-\frac{e^{m-1}_1\rho^{m-1}}{e_0\rho^{m-1}}e_0ve_1v-\frac{e^{m-1}_2\rho^{m-1}}{e_0\rho^{m-1}}e_0ve_2v}{2\sqrt{1-\frac{(e^{m-1}_1\rho^{m-1})^2}{(e_0\rho^{m-1})^2}-\frac{(e^{m-1}_2\rho^{m-1})^2}{(e_0\rho^{m-1})^2}}}\\
=&\,
\frac{1}{4}(1+O([\rho^{m-1}]^{3/8}))[(e_0v)^2+|\overline{\nabla}v|^2_{\overline{h}^{m-1}}],
\end{align}
where $|\overline{\nabla}v|^2_{\overline{h}^{m-1}}=|e^{m-1}_1v|^2+|e^{m-1}_2v|^2$.
Here the term $O([\rho^{m-1}]^{3/8})$ satisfies the bound: 

\beq
\label{Oxpded}
|O([\rho^{m-1}]^{3/8})|\le C\eta |\rho^{m-1}|^{3/8}
\eeq

We here  recall the bound \eqref{e.bd} which implies

\[
|\frac{1+O([\rho^{m-1}]^{3/8})}{1-O(\e^{3/8})}-1|\le \frac{\eta}{10}.
\]
Generally, if we are able to control the bulk in the RHS of 
(\ref{Stokesv}) in terms of $f(r)Q_{00}[v]$, \emph{times a function 
$W(r)$ for which $\int_0^\e|W(r)|\sqrt{r}dr$ is finite}, then 
 we can utilize the Gronwall lemma to obtain a uniform energy 
 estimate for $v$ over the hypersurfaces  $\Sigma_{\rho^{m-1}}$, 
 $\rho^{m-1}\in (0,2\epsilon]$.\footnote{The reason for the factor
  $\sqrt{r}$
  in $\int_0^\e|W(r)\sqrt{r}dr$ is the form \eqref{lapselike} of the
   lapse function.}

With this strategy in mind, we consider the coefficients of the terms in 
the RHS of \eqref{Stokesv}, and observe that the $W(r)$ that naturally 
arises 
is of the from $r^{-\frac{3}{2}}$. This is precisely at the borderline 
where we do \emph{not} obtain a finite integral 
$\int_0^\e|W(r)|\sqrt{r}dr$.
 In fact, the most dangerous integrand in the RHS of (\ref{Stokesv}) is 
 $\Phi^{m-1}\frac{3}{2}\frac{\sqrt{2M}}{\tau^\frac{3}{2}}f(\tau)\frac{1}{2}(e_0v)^2$, which corresponds to the leading order behaviour of $\text{tr}_{^{m-1}\overline{g}}K^{m-1}$, see (\ref{trKheur}). It is the asymptotically CMC property of our geodesic 
 parameter that yields the \emph{constant} multiple of $r^{-\frac{3}{2}}$.  
Thus $f(\tau)$ must be chosen to  cancel this particular term out; so 
 we must   choose $f(\tau)=\tau^\frac{3}{2}$.
As we shall see below, the latter weight choice is also exactly consistent with the logarithmic behaviour of $v=\partial^I\gamma^m$ at $r=0$ that we will derive for the lower derivatives.

Having chosen $f(\tau)=\tau^{\frac{3}{2}}$ to cancel out the most dangerous term 
in \eqref{Stokesv}, we must consider the 
 terms 
 \beq
 \begin{split}
& -\big[\frac{3}{2}\frac{\sqrt{2M}}{\tau^\frac{3}{2}}f(\tau)-
e_0f(\tau)\big]\frac{1}{2}|\overline{\nabla}v|^2_{^{m-1}\overline{g}}
-(\frac{d_1^{m-1}(t,\theta)\sqrt{2M}}{\tau^\frac{3}{2}})f(\tau)(e_1v)^2
-(\frac{d_2^{m-1}(t,\theta)\sqrt{2M}}{\tau^\frac{3}{2}})f(\tau)(e_2v)^2
\end{split}
\eeq
  in the bulk estimate. We note that in principle these are bounded by 
  an expression 
  \[
  \int_{\rho}^\e \frac{Const}{s}[f(r)Q_{ab}[v](e_0)^bn^a ]
  {\rm vol}_{\Sigma_\rho} ds.
  \]
In principle this term would again  spell trouble, since the integrating factor we 
would obtain in our estimate would be 
$e^{\int_\rho^\e \frac{1}{s}ds}$which is not uniformly bounded  as $\rho\to 0$.
 \emph{However},  in view of the expressions for 
 $d_1^{m-1}(t,\theta)\sim \frac{1}{2}, d_2^{m-1}(t,\theta)\sim -1$ 
 and the bounds on $|d_1^{m-1}(t,\theta)-\frac{1}{2}|$, 
 $|d_2^{m-1}(t,\theta)+1|<\frac{1}{8}$ the sign of 
 $Const$ will in fact be 
 \emph{negative}, thus 
 \emph{favorable} for us.\footnote{  In \cite{RS1,RS2} certain analogous borderline 
 terms appeared with a favourable sign, in fact at the level of the entire system of the full Einstein equation; the definite, favourable sign of all such terms was 
 called ``approximate monotonicity''. The favorable signs here concern only the free wave, and only the spatial directions in the energy.}
As we will find, the remaining bulk terms from \eqref{Stokesv} \emph{except} for $\Box_{g^{m-1}}v$
 can be bounded by the energy of $v$ we are controlling times a coefficient that 
 is \emph{integrable} in $r$, and are thus \emph{not dangerous}
  in terms of deriving 
 our desired estimates. This will follow by virtue of our inductive assumptions on the terms $u^{m-1}_{ij}$.
\medskip

 Thus throughout our analysis of equation 
 $\Box_g^{m-1}v=F$, $v=\opartial^I(e_0)^{J_1}\gamma^m_{\rm rest}$ for all orders 
 $|I|$ the multiplier 
 that we choose in forming the energy current $J$ 
  will be $\tau^{3/2}e_0$. Recalling the pointwise bounds $|u^{m-1}_{ij}(r,t,\theta)|\le Br^{-1-\frac{1}{4}}$, the resulting identity is then:
  
  \begin{align}\label{Stokes.appl}
  &\int_{\Sigma_{\rho^{m-1}}}[\rho^{m-1}]^{\frac{3}{2}}Q_{ab}[v](e_0)^bn^a\mathrm{vol}_{\Sigma_{\rho^{m-1}}}-\int_{\Sigma_{{r^{m-1}_*}}}\e^{\frac{3}{2}}Q_{ab}[v](e_0)^bn^a\mathrm{vol}_{\Sigma_{{{r^{m-1}_*}}}}\\
\notag\le &\int_{\rho^{m-1}}^{\epsilon}\int_{\Sigma_\tau}\bigg[\frac{1}{2}[O(\tau^{-1-\frac{1}{4}})]\tau^{\frac{3}{2}}
[(e_0v)^2+(e^{m-1}_1v)^2+(e^{m-1}_2 v)^2]^2-\tau^{\frac{3}{2}} e_0v\square_{^{m-1}g} v
\bigg]
\mathrm{vol}_{\Sigma_\tau}ds,
  \end{align}
  
 This analysis of the free wave equation will guide us in deriving 
 the claimed estimates 
 \eqref{inductiongammaopt} at the lowest orders. We also need to 
 understand the commutation of our equations with derivatives 
 $\opartial^{(2k_1,k_2)}_{T\dots T\Theta\dots\Theta}$, $2k_1+k_2=k\le s-3$. 
 In preparation of performing this computation, we introduce some 
language and  notational conventions  
 to describe the terms generated by these commutations:

 \subsubsection{Language conventions.}
 \label{sec:lang_conv}
 Motivated  by the $r$-weighted multiplier estimates we discussed  for 
 the free wave equation, we describe the broad class of  estimates that we will be 
 deriving below for the various derivatives of $\gamma^m_{\rm rest}$ 
 \medskip
 
 Consider any inequality of the form \eqref{bas.ener}, for all $r\le 2\e$, where $w$ is a constant (which will vary depending on the context below):

 \beq
\begin{split} \label{bas.ener}
&\sum_{|I|\le b} r^w E[\partial^I\gamma^m_{\rm rest}](r)\le 
\sum_{|I|\le b} \e^w E[\partial^I\gamma^m_{\rm rest}](\e)
+\dots +
\int_r^\e Q_1(\tau) \cdot 
\sqrt{\tau^w E[\partial^I\gamma^m_{\rm rest}](\tau)}d\tau 
\\&+
\int_r^\e Q_2(\tau) \cdot 
\tau^w E[\partial^I\gamma^m_{\rm rest}](\tau)d\tau. 
\end{split} 
 \eeq

Here $Q_1(\tau), Q_2(\tau)$ will be fixed functions, which depend on the 
context below. We recall the estimates in Lemma \ref{lem:Gron}. 
We then call terms of the form 

\[
\int_r^\e Q_1(\tau) \cdot 
\sqrt{t^w E[\partial^I\gamma^m_{\rm rest}](\tau)} d\tau, 
\int_r^\e Q_2(\tau) \cdot 
\tau^w E[\partial^I\gamma^m_{\rm rest}](\tau)d\tau
\]
\emph{below borderline} if we have apriori bounds on 
$Q_1(\tau), Q_2(\tau)$ which yield: 
\[
\int_0^\e |Q_1(\tau)|d\tau, \int_0^\e |Q_2(\tau)|d\tau<\infty. 
\]
(In fact we will always have uniform bounds on the RHSs, depending on $B,\e$--this is not important for this discussion now).
\medskip

If we only  have bounds  $Q_1(\tau)\lesssim \tau^{-1}$ or $Q_2(\tau)\lesssim \tau^{-1}$ we call 
such terms \emph{borderline}. In particular for such terms 
the Gronwall inequality in Lemma \ref{lem:Gron} does \emph{not} 
yield  finite energy bounds as $r\to 0^+$. 
\medskip

We extend this notion to estimates of the form: 
 \beq
\begin{split} 
&\sum_{|I|\le b} r^w E[\partial^I\gamma^m_{\rm rest}](r)\le 
\sum_{|I|\le b} \e^w E[\partial^I\gamma^m_{\rm rest}](\e)
\\&+\dots+
\int_r^\e |\int_{\Sigma_\tau} F^I_{\rm bdd}(\tau,t,\theta) \cdot 
\tau^{w/2} {\bf e}(\partial^I\gamma^m_{\rm rest})(\tau){\rm vol}_{\rm Euc}d\tau| +
|\int_r^\e F^0(\tau) \cdot 
\tau^w E[\partial^I\gamma^m_{\rm rest}](\tau)d\tau|,
\end{split} 
 \eeq
 where ${\bf e}$ stands for one of the vector fields $e_0, e^{m-1}_1, e^{m-1}_2$. 
 
In this setting, we will be assuming that 
the functions $F^I_{\rm bdd}$ satisfy bounds in 
$L^2_{t,\theta}[\tau]$ for all $\tau\in (0,2\e)$; ``assuming'' here will 
mean either because of the inductive assumptions we are making on the 
parameters from step $m-1$, \emph{or} from bounds 
on derivatives of $\gamma^m_{\rm rest}$ at orders below 
$|I|$, as well as \emph{straightforward} combinations of such estimates 
for products of such factors. (By straightforward here we mean estimates 
that can be obtained by applying Cauchy-Schwarz, the product  and Sobolev inequalities).   If from either 
direct invocations of our inductive
 bounds of straightforward combinations of such bounds 
 we obtain bounds of the form: 
 
  \[
  \sqrt{\int_{\Sigma_\tau} |F^I_{\rm bdd}(\tau,t,\theta)|^2 {\rm vol}_{\rm Euc}}\le B^2 C^2 \tau^{-1+\delta}
  \]
for some $\delta>0$ we call the expression 
\[
\int_r^\e |\int_{\Sigma_\tau} F^I_{\rm bdd}(\tau,t,\theta) \cdot 
t^{w/2} {\bf e}\partial^I\gamma^m_{\rm rest}(\tau)|{\rm vol}_{\rm Euc}d\tau
\]
\emph{below borderline}. Observe that such functions  can be bounded 
by a term of the form 
\[\int_r^\e Q_1(\tau) \cdot 
\sqrt{\tau^w E(\partial^I\gamma^m_{\rm rest}](\tau)}d\tau
\]
which are below borderline 
in the language of the previous definition. If we can \emph{only} obtain bounds 
of the form 
 
  \[
  \sqrt{\int_{\Sigma_\tau} |F^I_{\rm bdd}(\tau,t,\theta)|^2 
  {\rm vol}_{\rm Eucl}}\le D \tau^{-1}
  \]
for some constant $D>0$   then we call the expression 
 \[
  \int_r^\e |\int_{\Sigma_\tau} F^I_{\rm bdd}(\tau,t,\theta) \cdot 
\tau^{w/2} {\bf e}[\partial^I\gamma^m_{\rm rest}](\tau) {\rm vol}_{\rm Eucl}|d\tau
\]
\emph{borderline}. We use similar language conventions for the terms: 

\[
|\int_r^\e F^0(\tau,t,\theta) \cdot 
\tau^w\cdot 
E[\partial^I\gamma^m_{\rm rest}](\tau) d\tau|
\]

In this case, we will be assuming that $F^0(\tau,t,\theta)$
depends on parameters of the step $m-1$ or on lower derivatives of 
$\gamma^m_{\rm rest}$. Moreover we will be assuming that we have 
$L^\infty_{t,\theta}(\tau)$ bounds on these parameters on each 
$\Sigma_\tau$. If these bounds then yield bounds on 
$||F^0(\tau,t,\theta)||_{L^\infty_{t,\theta}}$ of the form: 

\[
|F^0(\tau,t,\theta)|\le  B^2 \tau^{-1+\delta}
\]
for all $\tau\in (0,2\e)$ and for some fixed $\delta>0$ then the term 

\[
|\int_r^\e F^0(\tau,t,\theta) \cdot 
\tau^w E[\partial^I\gamma^m_{\rm rest}](\tau) d\tau|
\]
 is called below borderline. If these bounds yield 
 bounds on 
$||F^0(\tau,t,\theta)||_{L^\infty_{t,\theta}}$ of the form: 

\[
|F^0(\tau,t,\theta)|\le D \tau^{-1}
\]
for all $\tau\in (0,2\e)$, for some $D>0$, 
then the term is called borderline. 
\medskip

These notions will be useful in deriving our energy estimates for 
$\gamma^m_{\rm rest}$ in the rest of this section.  

\begin{remark}
We note that the bulk of our analysis will be in deriving the desired 
estimates from the initial data set 
$\Sigma_{r^{m-1}_*}=\{\rho^{m-1}=\e\}$
 \emph{towards} the singularity at $\{r=0\}=\{\rho^{m-1}=0\}$. 
 The estimates for the remaining region $\rho^{m-1}\in (\e,2\e]$ are just easier versions of these estimates and we do not write them out explicitly. 
\end{remark}

\subsection{The wave equation expanded: An inhomogenous equation for 
$\gamma^m_{\rm rest}$.}

Our claims concern the function $\gamma^m_{\rm rest}=\gamma^m-\gamma^S$ 
instead of the free wave $\gamma^m$ which satisfies \eqref{waveq.it}. 
Thus we study the inhomogeneous wave equation: 

\beq
\label{inhomog.wave.it}
\Box_{g^{m-1}}\gamma^m_{\rm rest}=-\Box_{g^{m-1}}\gamma^S. 
\eeq

It is useful to  express the wave operator with respect to the coordinates 
$\{\rho^{m-1}, T^{m-1},\Theta^{m-1}, \phi\}$.  (We suppress the suffix $m-1$ from $T,\Theta$ for simplicty below). 
There are two ``main parts'' of the wave 
operator--the most important involve 
the more singular derivatives in the $e_0$ (parallel to  $\partial_{\rho^{m-1}}$) direction. 
There are also the terms in the spatial directions $\partial_T,\partial_\Th$. These appear in the 
last two lines of the next equation. There are some cross terms, due to the fact  that $\partial_{\rho^{m-1}}$ 
is not normal to $\partial_\theta$ and $\partial_t$ in the region 
$\rho^{m-1}\in [\e/2,3\e/2]$. The origin of these terms  is manifested in the transition between the vector fields $e^{m-1}_i$, $i=1,2$ and the 
vector fields $\overline{e}^{m-1}_i$ via formulas \eqref{tthetatransebar}.
 (The latter vector fields are tangent to the level sets of $\rho^{m-1}$, recall). 
We then have, for any function $v(t,\theta, \rho^{m-1})$:

\begin{align}\label{wave.coords}
\notag&\Box_{g^{m-1}}v=-\bigg[(\frac{2M}{r}-1)
[1-\partial_r\chi(r)({^{m-1}r}_*-\epsilon)]^2\bigg]\partial_{\rho^{m-1}}^2v\\
&\,+\bigg[\mathrm{tr}_{^{m-1}\overline{g}}K^{m-1}\cdot (\frac{2M}{r}-1)^{\frac{1}{2}}
+\frac{M}{r^2}
(\frac{2M}{r}-1)^{-\frac{1}{2}}-(\frac{2M}{r}-1)^{\frac{1}{2}}\partial_r^2\chi(r)
({^{m-1}r}_*-\epsilon)[1-\partial_r\chi(r)({^{m-1}r}_*-\epsilon)]\bigg]
\partial_{\rho^{m-1}}v\\
\notag&+ [O(e^{m-1}(\rho^{m-1}))\cdot O(\sqrt{r}) +O(\sqrt{r})\cdot [a^{m-1}_{T1}a^{m-1}_{T2}\cdot e^{m-1}_1re^{m-1}_2r]\cdot  (\frac{2M}{r})^{\frac{1}{2}}\partial_{\rho^{m-1}}v\\
\notag& +2\sum_{A=T,\Th} (g^{m-1})^{A\rho }\partial_{A\rho} v
-2\sum_{A=\phi,T,\Th}({g}^{m-1})^{\rho B}(^{m-1}\Gamma)_{\rho B}^\rho
\partial_\rho v\\
\notag&- 2
\sum_{B=\phi,T,\Th}(^{m-1}\overline{g})^{\rho B}(^{m-1}\Gamma)_{\rho B}^\Th
\partial_\Th v-2
\sum_{B=\phi,T,\Th}(^{m-1}\overline{g})^{\rho B}(^{m-1}\Gamma)_{\rho B}^T
\partial_T v\\
\notag&+\sum_{A,B=\phi,T,\Th}(^{m-1}\overline{g})^{AB}\partial^2_{AB}v- 
\sum_{A,B=\phi,T,\Th}(^{m-1}\overline{g})^{AB}(^{m-1}\Gamma)_{AB}^\Th
\partial_\Th v-
\sum_{A,B=\phi,T,\Th}(^{m-1}\overline{g})^{AB}(^{m-1}\Gamma)_{AB}^T
\partial_T v.
\end{align}


Here $(^{m-1}\Gamma)_{AB}^C$ are the Christoffel symbols of ${g}^{m-1}$ 
with respect to 
the system of coordinates $\{\rho^{m-1}, T,\Th, \phi\}$. 
\medskip

Let us make a comment here on the importance of choosing $e^{m-1}_2$ being ``tangent'' 
to the singularity, which implies that $e^{m-1}_2(r)$ and $e^{m-1}_2(\rho^{m-1})$ 
vanish for $r\le \e/2$: 
\begin{remark}
Had we not made the choice 
$e^{m-1}_2(r)=o(r^{-\frac{1}{2}+d^{m-1}_2(t,\theta)})$, we would have had 
$e^{m-1}_2(r)=O(r^{-\frac{1}{2}+d^{m-1}_2(t,\theta)})$. In that case, the terms in the 
third line of the above would have been \emph{more singular} than the terms in the 
second line. In particular we would \emph{not} have been able to derive our claimed 
bounds for $\gamma^m$ that we are claiming. Thus this (gauge) choice for the frame element $e^{m-1}_2$ is essential for our argument here. 
\end{remark}

Note that for the Schwarzschild metric all derivatives, metric components 
and Christoffel symbols which involve both the ``spatial'' directions 
$T,\Theta$ 
and the ``time-like'' direction $\rho$ all vanish. In the wave equation 
above there are such terms, and we will call them ``mixed'' terms:

\begin{definition}
\label{wave.mixed}
Consider the terms in \eqref{wave.coords} which involve either a ``mixed'' 
metric component $g_{A\rho}$, $A=T,\Th$ and/or a ``mixed'' Christoffel 
symbol $\Gamma_{AB}^C$ where at least one of $A,B,C$ is of the form $T,\Th$ 
and at least one other in the form $\rho$. 
We denote the sum of such terms in the operator by $\Box_{g^{m-1}}^{\rm mixed}$.  

We also consider the sum of all terms involving \emph{only} derivatives in the 
directions $\Th,T$ and metric and Christoffel symbols $g_{AB}, g^{AB}, \Gamma_{AB}^C$ 
have all indices taking values among $\Th,T$; the sum of those terms is denoted by 
$\Box_{g^{m-1}}^{\rm spatial}$.
\end{definition}

\subsubsection{Bounds on the inhomogenous term in the wave equation \eqref{inhomog.wave.it}.}
\label{sec:Inhomog}

Our first step will be to derive some bounds on the inhomogeneous term of 
\eqref{inhomog.wave.it}. Prior to stating our claim, we single out one exceptional 
case: At the top order estimates $|I|=s-3,  |J_0|=2$ on $\gamma^m_{\rm rest}$  
we observe that since $e_0(\theta)=0$, it suffices to derive our claimed bounds 
\eqref{inductiongammatopmixed} on 
$\gamma^m$ instead of on $\gamma^m_{\rm rest}$. Thus we will not need to subtract 
$\gamma^S$ at the most top order; so in the Lemma below we will be making no claim on
 the terms that would be generated had we made that subtraction. 

 \begin{lemma}
 \label{lem:inhomog}
 There exists a universal constant $B=B_{|I|}>0$ so that for all $I$, with $|I|\le s-3-4c$:
 \begin{align}
 \label{BoxgammaS.bds}
 &\| \opartial^I(\Box_{^{m-1}g}\gamma^S)\|_{L^2_{t,\theta}[\Sigma_{\rho^{m-1}=\tau}]}
\le 
  B^2  \tau^{-3+\frac{1}{4}}.
 \end{align}

 For $s-3-4c<|I|\le s-4$ the estimate we have is: 
  \begin{align}
 \label{BoxgammaS.bds2}
&\sum_{s-3-4c<|I|\le s-4} || \opartial^I(\Box_{^{m-1}g}\gamma^S)
||_{L^2_{t,\theta}
 [\Sigma_{\rho^{m-1}=\tau}]}\le 
 6\sqrt{2M}C\eta \tau^{-3-\frac{|I|-(s-3-4c)}{4}}
+  B^2  \tau^{-3+\frac{1}{4}+\frac{|I|-(s-3-4c)}{4}}.
 \end{align}
 Moreover, the only terms in $\opartial^I(\Box_{g^{m-1}}\gamma^S)$ that contribute to the first (more singular in $\tau$) 
 term in the RHS of \eqref{BoxgammaS.bds2} is the term 
 $\opartial^I(tr_{\overline{g}^{m-1}}K^{m-1})\cdot e_0\gamma^S$ in 
 \eqref{wave.coords} (with $\gamma^S$ being the function $v$ that is being acted on). 

 \end{lemma}

\begin{proof}

We commence with the bounds on 
\[
\Box_{g^{m-1}}\gamma^S = \Box_{g^{m-1}} log\rho^{m-1} +\Box_{g^{m-1}}
logsin\theta. 
\]

The two terms on the RHS will also be treated and bounded separately. 
We commence with the term $\Box_{^{m-1}g} log\rho^{m-1}$. We write $\rho=\rho^{m-1}$ for brevity. 
\medskip

We subtract from this the quantity $\Box_{g^S} log\rho=0$ (where the wave operator 
is expressed in terms of the coordinates $\rho=\rho^{m-1}, T,\Theta, \phi$--here 
$\rho=r$); thus we are reduced to 
bounding the following terms: 
\beq
 \begin{split}\label{boxgammamexp.minS}
 &(\text{tr}_{^{m-1}\overline{g}}K^{m-1}-\text{tr}_{g^S}K^S)e_0log\rho, \text{ }\text{ }
\Box^{\rm mixed}_{g^{m-1}}{\rm log}\rho,  
\end{split}
\eeq
The remaining terms involve derivatives $\partial_T,\partial_\Theta$ that 
annihilate $log\rho$. 
\medskip
 
 Let us commence with the first  term in \eqref{boxgammamexp.minS}:

The term 
$(\text{tr}_{^{m-1}\overline{g}}K^{m-1}-\text{tr}_{g^S}K^S)e_0(log\rho)$
 is first expanded using
 the expression
 \eqref{partialrho}
 to first  calculate  
 
 \beq\label{e0.rho}
 e_0(\rho^{m-1})=-(1-\frac{2M}{r})^{\frac{1}{2}}[1-\partial_r\chi(r)
 (r^{m-1}_*-\e)]^{-1}. 
 \eeq

Therefore we will derive our bounds as follows:
\beq
\begin{split}
\label{1st.expanded}
&\|\opartial^I(\text{tr}_{^{m-1}\overline{g}}(K^{m-1}-K^S)e_0log \rho)
 \|_{L^2[\Sigma_{\rho^{m-1}}]}
 \\&=\|\opartial^I\bigg(\text{tr}_{^{m-1}\overline{g}}(K^{m-1}-K^S)(1-\frac{2M}{r})^{\frac{1}{2}}[1-\partial_r\chi(r)
 (r^{m-1}_*(t,\theta)-\e)]^{-1}(\rho^{m-1})^{-1}\bigg)
 \|_{L^2[\Sigma_{\rho^{m-1}}]}
 \\&\le \sum_{I_1\bigcup I_2=I}\|\partial^{I_1}[(-\text{tr}_{^{m-1}\overline{g}}(K^{m-1}-K^S)]\partial^{I_2}\bigg[(1-\frac{2M}{r})^{\frac{1}{2}}[1-\partial_r\chi(r)
 (r^{m-1}_*(t,\theta)-\e)]^{-1}(\rho^{m-1})^{-1}\bigg]
 \|_{L^2[\Sigma_{\rho^{m-1}}]}
 \\&\le \|\partial^I (\text{tr}_{^{m-1}\overline{g}}(K^{m-1}-K^S)\|_{L^2}\cdot \|
 (1-\frac{2M}{r})^{\frac{1}{2}}[1-\partial_r\chi(r)
 (r^{m-1}_*(t,\theta)-\e)]^{-1}(\rho^{m-1})^{-1}
 \|_{L^\infty}
 \\&+ \|(\text{tr}_{^{m-1}\overline{g}}(K^{m-1}-K^S)\|_{L^\infty}\cdot 
\\& \|\partial^I\bigg[(1-\frac{2M}{r})^{\frac{1}{2}}[1-\partial_r\chi(r)
 (r^{m-1}_*(t,\theta)-\e)]^{-1}(\rho^{m-1})^{-1}\bigg]
 \|_{L^2[\Sigma_{\rho^{m-1}}]}.
 \end{split}
 \eeq
 Now, we consider the case where $|I|\le \rm low$ first. In that setting in both summands in the RHS of the above, 
 the second factor is bounded by 
  $C((\rho^{m-1})^{-3/2})$. 


Note that by our inductive assumption \eqref{tru.explicit} at the low  
 orders the first factor in those summands is bounded by 
$4C\eta(\rho^{m-1})^{-3/2+\frac{1}{4}}$.\footnote{The extra power 
$r^{\frac{1}{4}}$ captures the asymptotically CMC property at the lower
 orders.} Thus this term satisfies our claimed bounds at the lower orders. 
At the higher orders, we have the same bound for the second factor. 
However, for the first summand, 
by our inductive assumption \eqref{inductiontrKlow} the  RHS of 
\eqref{1st.expanded}
 is bounded by a more singular power or $\rho^{m-1}$; it has an additional 
 factor 
$(\rho^{m-1})^{-\frac{|I|-(s-3-4c)}{4}}$, \emph{and} it does 
not have the extra power of $(\rho^{m-1})^{\frac{1}{4}}$.
 In either case our claims are satisfied for 
this term at the higher orders, since no extra power $(\rho^{m-1})^{1/4}$ 
is claimed there.  
\medskip

We now consider the second  term in \eqref{boxgammamexp.minS}; this terms equals:
 \beq
\begin{split}& [O(e^{m-1}(\rho^{m-1}))\cdot O(\sqrt{r}) +O(\sqrt{r})\cdot [a^{m-1}_{T1}a^{m-1}_{T2}\cdot e^{m-1}_1re^{m-1}_2r]
\cdot (\frac{2M}{r})^{\frac{1}{2}}
 \partial_{\rho^{m-1}}log\rho^{m-1}
\\&= 
 [O(e^{m-1}(\rho^{m-1}))\cdot O(\sqrt{r}) +O(\sqrt{r})\cdot [a^{m-1}_{T1}a^{m-1}_{T2}\cdot e^{m-1}_1re^{m-1}_2r]
\cdot (\frac{2M}{r})^{\frac{1}{2}}
\rho^{m-1})^{-1}.
 \end{split}
\eeq

We then invoke the bounds we have on $e^{m-1}_1(\rho^{m-1})$ in Lemma 
\ref{lem:e1re2r}; using Cauhcy-schwarz and the product inequality, as well as the bounds on $a^{m-1}_{Ti}, i=1,2$ we 
easily see that all terms involving $e^{m-1}_i(\rho^{m-1})$ are bounded a claimed in our Lemma.

  The remaining terms
  in the RHS of 
\eqref{boxgammamexp.minS} are handled in a similar manner, 
using the 
inductive assumptions on the various parameters and the product inequality.  
 \medskip

Next we derive the claimed bounds for $\Box_{^{m-1}g}logsin\theta$. 
\newcommand{\tg}{{\tilde{\gamma}}}

 For this term we proceed  by expanding the wave operator  as in \eqref{wave.coords}; among coordinate derivatives, the 
 only non-zero terms are then 
those involving derivatives in the $\Th$-direction, since 
$v=logsin\theta$.  
These are: 

\beq
\label{first.Gamma}
(^{m-1}g)^{\Th\Th} \partial_{\Th\Th}(logsin\theta)
-[(^{m-1}g)^{\phi\phi}
\cdot (\Gamma^{m-1}_{\phi\phi,\Th})] \cdot
(^{m-1}g)^{\Th\Th}\partial_\Th (logsin\theta),
\eeq

\beq
\label{2nd.Gamma}
\sum_{a,b=T,\Th,r}(^{m-1}g)^{ab} (\Gamma^{m-1})_{ab}^\Th\partial_\Th(logsin\theta). 
\eeq
(Note that $\sum_{a,b=T,\Th,r}(g^S)^{ab}[^S\Gamma^\Th_{ab}]=0$).

 The first sum of terms is \emph{more} singular (in terms of powers 
$\theta^{-1}, (\pi-\theta)^{-1}$) than the second, so we commence with the sum of 
those two terms.

\begin{equation}\label{dangexp}
\begin{split}
& (^{m-1}g)^{\Th\Th} \partial_{\Th\Th}(logsin\theta)-[(^{m-1}g)^{\phi\phi}
\cdot (\Gamma^{m-1}_{\phi\phi,\Th}) \cdot
(^{m-1}g)^{\Th\Th}\partial_\Th (logsin\theta))=
\\& (^{m-1}g)^{\Th\Th} \partial_{\Th\Th}(logsin\theta)
-[(^{m-1}g)^{\phi\phi}\cdot (\Gamma^{m-1}_{\phi\phi,\Th})-(g^S)^{\phi\phi}(
\Gamma^S_{\phi\phi,\Th})] \cdot
(^{m-1}g)^{\Th\Th}\partial_\Th (logsin\theta)
\\&-
[(g^S)^{\phi\phi}\cdot (\Gamma^S_{\phi\phi,\theta})] \cdot
(^{m-1}g)^{\Th\Th}\partial_\Th (logsin\theta).
\end{split}
\end{equation}

 A key thing to observe here is a certain cancellation of two singular terms in the right hand
  side (the first and the last, taken together), using that
   $\partial_\Th\theta=1$, since $\Th=\theta$ in these coordinates:
 
 \begin{align}
& (^{m-1}g)^{\Th\Th} \partial_{\Th\Th}(logsin\theta)-
[(g^S)^{\phi\phi}(\Gamma^{m-1}_{\phi\phi,\Theta})] \cdot
(^{m-1}g)^{\Th\Th}\partial_\Th (logsin\theta)\\
\notag &= (^{m-1}g)^{\Th
\Th}
[-\frac{1}{(sin\theta)^2}+(cot\theta)^2](\partial_\Th\theta)^2+ (^{m-1}g)^{\Th\Th}\partial_{\Th\Th}\theta
\frac{cos\theta}{sin\theta}=-(^{m-1}g)^{\Th\Th}.
 \end{align}

Recall the bounds in previous subsection on $(g^{m-1})^{\Th\Th}$; 
these imply that  the RHS term above satisfies the bounds
  claimed for the LHS of our  Lemma \ref{lem:inhomog}.

 The remaining terms in 
\eqref{dangexp}
 can be controlled as follows: 
 
 \begin{equation}
 \begin{split}
 \label{Gphiphith}
&- [(^{m-1}g)^{\phi\phi}\cdot (\Gamma^{m-1}_{\phi\phi,\Theta})-(g^S)^{\phi\phi}(
\Gamma^S_{\phi\phi,\Theta})] \cdot
(^{m-1}g)^{\Theta\Theta}\partial_\theta (logsin\theta)
\\&=
-\frac{1}{2}[(^{m-1}g)^{\phi\phi}\cdot 
(-\partial_\Th (e^{2\gamma^{m-1}}))-(g^S)^{\phi\phi}
(-\partial_\Th e^{2\gamma^S})] \cdot
(^{m-1}g)^{\Th\Th}\partial_\Th (logsin\theta))
\\&= \partial_\Th(\gamma^{m-1}-\gamma^S)\cdot 
(^{m-1}g)^{\Th\Th}\partial_\Th (logsin\theta)
\\&= \partial_\Th (\gamma^{m-1}_{\rm rest})\cdot (g^{m-1})^{\Th\Th}\cdot 
cot\theta.
 \end{split}
 \end{equation}
 
 Now consider $\opartial^I=(\DS)^{k_1}\partial_t^{k_2}$ acting on  the above.
 We derive that: 
 
 \beq
 \begin{split}
& \|\opartial^{I} [\partial_\Th (\gamma^{m-1}_{\rm rest})\cdot (g^{m-1})^{\Th\Th}
\cdot 
cot\theta]\|_{L^2[\Sigma_{\rho^{m-1}}]}
\le \| \partial^{I}\bigg{(}[\partial_\Th\gamma^{m-1}_{\rm rest}\cot\theta]\cdot 
(g^{m-1})^{\Th\Th}\bigg{)}\|_{L^2[\Sigma_{\rho^{m-1}}]} \\&\le 
\sum_{I_1\bigcup I_2=I}\| [(\partial^{I_1}[\partial_\Th
\gamma^{m-1}_{\rm rest}\cdot cot\theta]\cdot 
(\partial^{I_2}g^{m-1})^{\Th\Th}\|_{L^2[\Sigma_{\rho^{m-1}}]}.
\end{split}
 \eeq
 Let us see how the above can be bounded, since the same argument will be used 
 frequently in the rest of the paper. We recall first that by the product inequality the RHS is bounded by: 
 
 \[
 B\cdot \|[\partial_\Theta
\gamma^{m-1}_{\rm rest}\cdot cot\theta]\|_{H^{|I|}[\Sigma_{\rho^{m-1}}]}\cdot 
\|(g^{m-1})^{\Theta\Theta}\|_{L^\infty[\Sigma_{\rho^{m-1}}]}+
 B\cdot \|\partial_\theta
\gamma^{m-1}_{\rm rest}\cdot cot\theta\|_{L^\infty[\Sigma_{\rho^{m-1}}]}\cdot 
\|(g^{m-1})^{\Theta\Theta}\|_{H^{|I|}[\Sigma_{\rho^{m-1}}]}
 \]

The factors:
  \[
  \|(g^{m-1})^{\Theta\Theta}\|_{L^\infty[\Sigma_{\rho^{m-1}}]}, 
  \|[\partial_\Theta
(\gamma^{m-1}_{\rm rest})\cdot cot\theta\|_{L^\infty[\Sigma_{\rho^{m-1}}]}
  \]
  are bounded respectively by $B(\rho^{m-1})^{-2-\frac{1}{4}}$ and $2\|\partial^2 
  (\gamma^{m-1}_{\rm rest})\|_{L^\infty[\Sigma_{\rho^{m-1}}]}\le C\eta (\rho^{m-1})^{-1/8}$.
  The other factors are bounded by our inductive assumptions, after also using the 
  Hardy inequality for the first one; in particular we invoke the inequality:
  
  \beq
  \label{1}
 \begin{split}
 & \|(\partial^{I}[\partial_\theta
\gamma^{m-1}_{\rm rest}\cdot cot\theta]\|_{L^2[\Sigma_{\rho^{m-1}}]}\le 
 2 \|(\partial^{I}[\DS
\gamma^{m-1}_{\rm rest}]\|_{L^2[\Sigma_{\rho^{m-1}}]}+  
\|\partial^{I}[\partial_\theta
\gamma^{m-1}_{\rm rest})]\|_{L^2[\Sigma_{\rho^{m-1}}]}\le C\eta |log\rho^{m-1}|, \\&
\text{for} \text{ } |I|\le {\rm low}-2,
 \\&\|(\partial^{I}[\partial_\theta
\gamma^{m-1}_{\rm rest}\cdot cot\theta]\|_{L^2[\Sigma_{\rho^{m-1}}]}\le 
2  \|\partial^{I}[\DS
\gamma^{m-1}_{\rm rest}]\|_{L^2[\Sigma_{\rho^{m-1}}]}+  
\|\partial^{I}[\partial_\theta
\gamma^{m-1}_{\rm rest}]\|_{L^2[\Sigma_{\rho^{m-1}}]}
\\&\le C\eta (\rho^{m-1})^{-\frac{1}{8}-\frac{|I|-({\rm low}-2)}{4}}, 
\text{for} \text{ } 
 {\rm low}-1\le |I|\le s-4.
\end{split} 
  \eeq
  (A remark on the last inequality at the order $|I|=s-4$: We note that the desired
   estimate holds after we first re-express  the leftmost derivatives in $\partial^I$
   (be it $\partial_\Th$ or $\partial_T$) in the term $\partial^I\DS\gamma^{m-1}$
   in terms of $e^{m-1}_i$, using $a^{m-1}_{\A i}$; the desired bound then follows from 
   the inductive estimate at the top order for the function  $\gamma^{m-1}$ (this 
   assumed
    bound  comes from the previous step in the induction). 
  
  On the other factors we have already derived the bounds:
  \beq
  \label{2}
  \begin{split}
 & \|\partial^I(g^{m-1})^{\theta\theta}\|_{L^2[\Sigma_{\rho^{m-1}}]}\le B^2 (\rho^{m-1})^{-2-\frac{1}{4}}, \text{if} \text { } |I|\le {\rm low}
 \\& \|\partial^I(g^{m-1})^{\theta\theta}\|_{L^2[\Sigma_{\rho^{m-1}}]}\le B^2 (\rho^{m-1})^{-2-\frac{1}{4}-\frac{|I|-{\rm low}}{4}}, \text{if} \text { } 
 {\rm low} \le |I|\le  s-4.
  \end{split}
  \eeq

Combining these bounds  proves that the term \eqref{Gphiphith}
    satisfies the desired bound. 
  All remaining terms can be bounded by our assumptions on Christoffel symbols and the argument above. 
  This completes our proof of the  bounds in \eqref{BoxgammaS.bds}, \eqref{BoxgammaS.bds2}.
  \end{proof}

\subsubsection{Formulas and bounds for commutations of $\Box_{g^{m-1}}$ with $\opartial^I$. }
\label{sec:Comm}

Having bounded the RHS of \eqref{inhomog.wave.it} in suitable spaces, we next need to 
act on the LHS of that equation by $\opartial^I$,
 and commute the derivatives past $\Box_{g^{m-1}}$, 
$|I|=l\leq s-4$. (Note that since $\opartial^I$ consists of compositions of 
$\DS, \partial_T$ and $s-4$ is odd, there will be at least one $\partial_T$
 derivative). The 
case of all derivatives being in the $\partial_\Th$ direction is covered at the top
 order where $|I|=s-3$
\medskip

We will put down 
some formulas that will help us in calculating the required commutation. 

We proceed in two steps: First we consider the terms in the wave operator that have coefficients that are singular at the two poles 
$\theta=0,\pi$ and proceed to 
calculate and bound the commutation terms that arise from those with some care. 
Next, we study the remaining terms that arise in the commutation, and bound those also; the latter terms are almost straightforward applications of our inductive assumptions. The second step will be performed in the subsequent subsections. 
\medskip

Initially, let us consider the sum of terms in  $\Box_{g^{m-1}}$ that corresponds to the 
Laplacian on the 2-spheres 
$\Th\in (0,\pi), \phi\in [0,2\pi)$, $\rho^{m-1}={\rm fixed}$, $T=\rm fixed $. This operator
 $\Delta_{g^{m-1}}^{\rho^{m-1}, T}$ is defined via:

\[
\Delta_{g^{m-1}}^{\rho^{m-1}, T}v= (^{m-1}g)^{\Th\Th} [\partial^2_{\Th\Th}v-
(^{m-1}\Gamma_{\Th\Th}^\Th) \partial_\Th v+\partial_\Th \gamma^{m-1}\partial_\Th v]. 
\]
(See also the RHS of the above as it appears in \eqref{wave.coords}). 

Note that this operator corresponds to a special sum of terms 
\[
\sum_{A,B=\Th,\phi}(g^{m-1})^{AB} [\partial_{AB}-\Gamma_{AB}^\Th\partial_\Th]
\]
in the last line of \eqref{wave.coords}.
We can then re-express the above operator using $\gamma^{m-1}_{\rm rest}$ instead of 
$\gamma^{m-1}$: 

\beq
\label{Delta.preuse}\begin{split}
&\Delta_{g^{m-1}}^{\rho^{m-1}, T}v= (^{m-1}g^{\Th\Th}) [\partial^2_{\Th\Th}v-
\frac{cos\Th}{sin\Th}\partial_\Th v]+
\frac{1}{2}[(^{m-1}g^{\Th\Th})]^2\partial_\Th (^{m-1}g_{\Th\Th})
 \partial_\Th v+(^{m-1}g^{\Th\Th})\cdot \partial_\Th \gamma^{m-1}_{\rm rest} \cdot 
 \partial_\Th v
 \\&+\sum_{A=T,\rho}[(^{m-1}g^{\Th})]^2\partial_\Th (^{m-1}g_{A\Th})
 \partial_\Th v.
 \end{split}
\eeq
Note that the first term in the RHS is precisely 
$(^{m-1}g)^{\Th\Th}\Delta_{\mathbb{S}^2}v$, where we recall that 
$\Delta_{\mathbb{S}^2}$ is the standard Laplacian 
on the 2-sphere, acting on $\phi$-independent functions. 
We denote the first line in the RHS of the above by
 $\tilde{\Delta} _{g^{m-1}}^{\rho^{m-1}, T}v$.
\newcommand{\tDelta}{\tilde{\Delta}} 
 
We will first  calculate the 
commutation of  
$\tDelta_{g^{m-1}}^{\rho^{m-1}, T}$ with $\DS$. 

We note that the factors $ \partial_\Th \gamma^{m-1}_{\rm rest}$, 
$\partial_\Th (^{m-1}g_{\Th\Th})$ vanish at the two poles $\Th=0,\pi$. To 
take advantage of this, we re-write the terms involving those two factors as: 

\beq
\label{Delta.use}
\begin{split}
&\tDelta_{g^{m-1}}^{\rho^{m-1}, T}v= (^{m-1}g^{\Th\Th}) [\partial^2_{\Th\Th}
v-\frac{cos\Th}{sin\Th}\partial_\Th v]
\\&-
\frac{1}{2}[(^{m-1}g^{\Th\Th})^2 \cdot \frac{\partial_\Th (^{m-1}
g_{\Th\Th})}{sin\Th}]
\cdot  [sin\Th\cdot \partial_\Th v]+[(^{m-1}g^{\Th\Th})\cdot 
 \frac{\partial_\Th \gamma^{m-1}_{\rm rest}}{sin\Th}]
  \cdot [sin\Th \partial_\Th v].
 \end{split}
\eeq

Using this formula, the commutation of $\Delta_{g^{m-1}}^{\rho^{m-1}, T}$
with $\DS$ is easily calculated, in \eqref{Comm.Deltas}  right below. 

\beq
\begin{split}
\label{Comm.Deltas}
&[ \Delta_{\mathbb{S}^2},\tDelta_{g^{m-1}}^{\rho^{m-1}, t}](v)=\Delta_{\mathbb{S}^2}(^{m-1}g)^{\Th\Th}\cdot \DS v+2\partial_\Th(^{m-1}g^{\Th\Th})\cdot\partial_\Th \DS v
\\&-\DS [\frac{1}{2}(^{m-1}g^{\Th\Th})^2 \frac{\partial_\Th (^{m-1}g_{\Th\Th})}{sin\Th}]
 \cdot (sin\Th \cdot \partial_\Th )v
 \\&-\partial_\Th [(^{m-1}g^{\Th\Th})^2 \frac{\partial_\Th 
 (^{m-1}g_{\Th\Th})}{sin\Th}]
 \cdot \partial_\Th(sin\Th \cdot \partial_\Th )v
-[\frac{1}{2}(^{m-1}g^{\Th\Th})^2 \frac{\partial_\Th (^{m-1}g_{\Th\Th})}{sin\Th}]
 \cdot [\DS,(sin\Th \cdot \partial_\Th ](v)
\\&+
\DS[(^{m-1}g^{\Th\Th})\cdot 
 \frac{\partial_\Th \gamma^{m-1}_{\rm rest}}{sin\Th}]
  \cdot [sin\Th \partial_\Th v]+2\partial_\Th[(^{m-1}g^{\Th\Th})\cdot 
 \frac{\partial_\Th \gamma^{m-1}_{\rm rest}}{sin\Th}]
  \cdot \partial_\Th [sin\Th \partial_\Th v]
  \\&+
  [(^{m-1}g^{\Th\Th})\cdot 
 \frac{\partial_\Th \gamma^{m-1}_{\rm rest}}{sin\Th}]
  \cdot [\DS,sin\Th \partial_\Th]( v).
\end{split}
\eeq
We note that $[\DS,sin\Th \partial_\Th]=2cos\Th\cdot \DS$. Thus the last factors in the 
last two lines can be replaced by $2cos\Th\cdot\DS v$. We can iterate the above formula 
and then act by $\partial_T$ repeatedly on the resulting formulas. We  use this 
formula repeatedly  to calculate and bound:
\[
[\opartial^I, \tDelta^{\rho^{m-1},T}_{g^{m-1}}](\gamma^m_{\rm rest})
\]
in $L^2({\rm sin}\Theta d\Theta dT)$. 

We next wish to see how $\DS$ commutes with the remaining part of $\Box_{g^{m-1}}$.
Again recall formula \eqref{wave.coords} (denoting $\gamma^m$ by $v$); 
 the terms that require special treatment due to their singular behaviour at the poles are precisely:

$\Delta_{g^{m-1}}^{\rho^{m-1}, t}v$ (whose ``main part'' 
$\tDelta_{g^{m-1}}^{\rho^{m-1}, t}v$ was already calculated above) 
\emph{and} the terms $(^{m-1}\Gamma_{TT}^\Th)\partial_\Th v$, 
$(^{m-1}\Gamma_{T\Th}^\Th)\partial_\Th v$, and the terms in the third line of 
\eqref{wave.coords}. 
These will be treated directly below. For the \emph{rest} 
of the terms in \eqref{wave.coords} the commutation formula 
is straightforward: We use the commutation of $\partial_{\rho^{m-1}}$, $\partial_T$
 with $\partial_\Th$. Thus the commutation terms generated can be calculated 
from the Leibnitz rule, and in particular are \emph{not} singular at the poles. 
These commutations will give rise to terms that are written out (in generic notation)
in \eqref{boxpartialIgammam} directly below. 
\medskip

Let us now calculate: 
\beq
\label{Delta-comm}
[\DS, (^{m-1}g^{TT})(^{m-1}\Gamma_{TT}^\Th)\partial_\Th], 
[\DS, (^{m-1}g^{T\Th})(^{m-1}\Gamma_{T\Th}^\Th)\partial_\Th].
\eeq
We will explain how the terms in the third line of \eqref{wave.coords}, as well as the terms in the second line of \eqref{Delta.use} can be computed
 and bounded by the same argument we present for the ones here.

For both these terms, we observe that the coefficient of $\partial_\Th$ 
of the second operator in $[\cdot, \cdot]$ \emph{vanishes} at the two poles.

To make use of this, we re-write:

\beq
\label{mult.div}
\begin{split}
&(^{m-1}g^{TT})(^{m-1}\Gamma_{TT}^\Th)\partial_\Th=
[(^{m-1}g^{TT})(^{m-1}\Gamma_{TT}^\Th)(sin\Th)^{-1}]\cdot
 (sin\Th \partial_\Th), 
 \\&(^{m-1}g^{T\Th})(^{m-1}\Gamma_{T\Th}^\Th)\partial_\Th
= [(^{m-1}g^{T\Th})(^{m-1}\Gamma_{T\Th}^\Th)\cdot (sin\Th)^{-1}]\cdot (sin\Th \partial_\Th)
 \end{split}
\eeq 

In particular given our derived bounds on the Christoffel symbols (as well as the Hardy inequality),
  the RHS is uniformly bounded at orders $|I'|\le {\rm low}-2$ by 

\[
B(\rho^{m-1})^{-2-\frac{1}{8}}.
\]
At each order beyond that, the power of $\rho^{m-1}$ becomes more singular by
 $-1/4$.

With these formulas in hand, we bound the first term in \eqref{Delta-comm} as follows: 

\beq
\begin{split}
\label{Delta.other.comm}
&\|[\DS, (^{m-1}g^{TT})(^{m-1}\Gamma_{TT}^\Th)\partial_\Th v]\|_{L^2(sin\Th d
\Th dt)}\le 
\\&\|\DS[(^{m-1}g)^{TT}(^{m-1}\Gamma_{TT}^\Th)\cdot (sin\Th)^{-1}]
\cdot (sin\Th\cdot\partial_\Th v) \|_{L^2(sin\Th d\Th dt)}
\\&+2\|\partial_\Th[(^{m-1}g)^{TT}(^{m-1}\Gamma_{TT}^\Th)\cdot (sin\Th)^{-1}]
\cdot \partial_\Th (sin\Th\cdot\partial_\Th v) \|_{L^2(sin\Th d\Th dt)}
\\& +2\| [(^{m-1}g)^{TT}(^{m-1}\Gamma_{TT}^\Th)\cdot (sin\Th)^{-1}]
\cdot \DS v \|_{L^2(sin\Th d\Th dt)}.
\end{split}
\eeq 
(In the very last term we applied the commutation formula: 
\[
[\DS, sin\Th\partial_\Th]=2cos\Th\cdot \DS). 
\]

We note that in view of Lemma \ref{a.oa.bds}:
\beq
\label{Del.by.E}
\|\DS v\|_{L^2(sin\Th d
\Th dt)[\Sigma_{\rho^{m-1}}]}\le (\rho^{m-1})^{1-\frac{1}{8}} E[\partial v].
\eeq

 This calculation can be iteratively applied to calculate 
 $[\opartial^I, (^{m-1}g^{TT})(^{m-1}\Gamma_{TT}^\Th)\partial_\Th]
 \gamma^m_{\rm rest}$ and to bound this term in $L^2(sin\Th d\Th dt)$ by the RHSs of 
 the inequalities in Proposition \ref{opt.RHSs}. 
An entirely analogous calculation (and derivation of resulting bounds) 
can be performed on 
\[
\|\partial^{I'}[(sin\Th)^{-1}\cdot (^{m-1}g^{T\Th})(^{m-1}\Gamma_{T\Th}^\Th)]
\|_{L^2(\mathbb{S}^2)}.
\]
 
For the two terms in the third line of \eqref{wave.coords} we use the same argument, ``creating'' the vector field $sin\Th \partial_\Th$ by multiplying each expression by $1=\frac{{\rm sin}\Th}{\sin\Th}$, and using the above argument for commutations. We also use the vanishing of $e^{m-1}_2(\rho^{m-1})$ at the poles 
to bound expressions  
\[
\| \partial^I[\frac{e^{m-1}_2\rho^{m-1}}{sin\Th}]\|_{L^2}
\]
 by 
 \[
 \| \partial^I[\partial_\Th e^{m-1}_2\rho^{m-1}]\|_{L^2}.
 \]
 These terms are in fact easier to handle, since they vanish identically for $\rho^{m-1}\le \e/2$.


 What remains is to find all the \emph{other} commutation terms in 
$[\opartial^I, \Box_{g^{m-1}}]$. These we write out (schematically, using the 
notational conventions introduced earlier in this section) in the RHS of the next
 equation, from the third line of the RHS onwards:\footnote{We write $|I|=l$ below, for short. }

\begin{align}\label{boxpartialIgammam}
&\square_{^{m-1}g}\opartial^I\gamma^m_{\rm rest}=
- \sum_{I_1\cup I_2=I,\,|I_2|<l}\bigg[\partial^{I_1}\bigg(\text{tr}_{^{m-1}\overline{g}}
K^{m-1}\bigg)e_0\partial^{I_2}\gamma^m_{\rm rest}\\
\notag&+\opartial^{I_1}
[[O(e^{m-1}(\rho^{m-1}))\cdot O(\sqrt{r}) +O(\sqrt{r})\cdot [a^{m-1}_{T1}a^{m-1}_{T2}\cdot e^{m-1}_1re^{m-1}_2r]]\partial^{I_2}
\partial_{\rho^{m-1}}](\gamma^m_{\rm rest})\bigg]\\
\notag& - [\overline{\partial}^I,\sum_{A,B=T,\Th, (A,B)\ne (\Th,\Th)}(g^{m-1})^{AB}\partial_{AB}]\gamma^m_{\rm rest}+
\sum_{A,B=T,\Th,\phi}[\opartial^I,(^{m-1}g^{AB})(^{m-1}\Gamma_{AB}^T) \partial_T] \gamma^m_{\rm rest}\\
\notag& -[\opartial^I, \Delta_{g^{m-1}}^{\rho^{m-1}, t} ]
\gamma^m_{\rm rest} -[\opartial^I,\Box^{\rm mixed}_{g^{m-1}}]\gamma^m_{\rm rest} -\opartial^I(\Box_{^{m-1}g}\gamma^S).
\end{align}

We note that the  last 
term in the last line have been controlled already by virtue of Lemma 
\ref{lem:inhomog}.  
Let us denote by $\widetilde{RHS}[\eqref{boxpartialIgammam}$] the RHS of the equation 
\eqref{boxpartialIgammam}, except for the last term, and the 
sum of terms $\tDelta_{g^{m-1}}^{\rho^{m-1}, t}$ in
 $\Delta_{g^{m-1}}^{\rho^{m-1}, t}$. 
We claim: 

\begin{proposition}
\label{opt.RHSs}
At step $m$ assume all inductive claims concerning the REVESNGG parameters of step
 $m-1$ hold true. Also, 
for each order $|I|\le s-4$ assume that the inductive claims on $\gamma^m$ for orders
 $|I'|<|I|$ hold true.

Then on each level set $\Sigma_\tau$ of $\rho^{m-1}$ the $L^2_{t,\theta}$
 norm of $\widetilde{RHS}[\eqref{boxpartialIgammam}]$ (with respect 
 to the volume form $sin\Theta d\Theta d T d\phi$) is bounded as follows: 
 For $|I|\le \rm low$:
  
 \beq
 \| \widetilde{RHS}[\eqref{boxpartialIgammam}]\|_{L^2_{t,\theta}[\tau]}\le 
  B^2 C \tau^{-3+\frac{1}{4}}+ B\tau^{-1+\frac{1}{2}-DC\eta }\cdot 
\sqrt{ E[\opartial^I\gamma^m_{\rm rest}]}.
 \eeq
While for ${\rm low}+1\le |I|\le s-4$, for $h={|I|-{\rm low}}$: 
 \beq
 \| \widetilde{RHS}[\eqref{boxpartialIgammam}]\|_{L^2_{t,\theta}[\tau]}\le 
 [8(C\eta)^2] \tau^{-\frac{3}{2}-\frac{h}{4}}+
  B^2 C \tau^{-1-\frac{1}{4}-\frac{h}{4}}+ B\tau^{-1+\frac{1}{2}-DC\eta }\cdot 
\sqrt{ E[\opartial^I\gamma^m_{\rm rest}]}.
 \eeq 
 \end{proposition}

\emph{Proof of Proposition \ref{opt.RHSs}:} 
 We first show our claim for the lower 
orders $|I|=l\le \rm low$: 

The most important  term in the RHS of \eqref{boxpartialIgammam}
is: 

\beq
\begin{split}
\label{3.terms}
&\sum_{I_1\cup I_2=I,\,|I_2|<l}\partial^{I_1}\text{tr}_{^{m-1}\overline{g}}K^{m-1}e_0\partial^{I_2}\gamma^m_{\rm rest}.
\end{split}
\eeq

To control this term  we use the crucial fact that 
$\text{tr}_{^{m-1}\overline{g}}K^{m-1}$ is asymptotically constant along 
$\Sigma_\tau$ to leading order, as $r\rightarrow0$, as captured in 
(\ref{trKheur}),(\ref{inductiontrKopt});
this yields: 
\begin{align}\label{partialItrKm-1}
\partial^{I_1}\text{tr}_{^{m-1}\overline{g}}K^{m-1}=O\big(\chi_{[\frac{\epsilon}{2},\frac{3\epsilon}{2}]}\|\partial^{I_1}{^{m-1}r}_*\|\big)+\partial^{I_1}\text{tr}_{^{m-1}\overline{g}}u^{m-1},\qquad 1\leq|I_1|\leq s-3-4c,
\end{align}
where $\chi_{[\frac{\epsilon}{2},\frac{3\epsilon}{2}]}(\rho^{m-1})$ is the 
characteristic function of the interval $[\frac{\epsilon}{2},\frac{3\epsilon}{2}]$. 
This implies that
\begin{align}\label{partialIgammambulkest3}
\sum_{1\leq |I_1|\leq l}\|\partial^{I_1}\text{tr}_{^{m-1}\overline{g}}
K^{m-1}\|_{L^2_{t,\theta}[\Sigma_\tau]}\le 
\frac{3B}{\tau^{1+\frac{1}{4}}}. 
\end{align}
for every $l\leq s-3-4c$. We thus observe that this term 
 will contribute a 
below-borderline bulk  term in the final energy estimate, since using our inductive
 assumption and the product inequality: 

\beq
\begin{split}
&\| \sum_{I_1\cup I_2=I,\,|I_2|<l}\partial^{I_1}\text{tr}_{^{m-1}\overline{g}}K^{m-1}
e_0\partial^{I_2}\gamma^m_{\rm rest}
 \|_{L^2_{t,\theta}[\tau]}
 \\&\le C_{|I|}\bigg[ \|\partial{\rm tr}_{\overline{g}^{m-1}}K^{m-1}\|_{L^\infty[\tau]}
 \cdot  \|e_0\gamma^m_{\rm rest}  \|_{\dot{H}^{l-1}[\tau]} +
\|\partial{\rm tr}_{\overline{g}^{m-1}}K^{m-1}\|_{\dot{H}^{l-1}[\tau]}
 \cdot  \|e_0\gamma^m_{\rm rest}  \|_{L^\infty[\tau]} 
 \bigg]\le 
 \frac{4B}{\tau^{2+\frac{3}{4}}}.
 \end{split}
\eeq
We again stress  here the importance of the asymptotically CMC property of 
the level sets of $r$: 
Had this cancellation in $\partial {\rm tr} K^{m-1} $ not been present we would have not been able to 
obtain the desired energy estimates, since we would have obtained a 
borderline term 
already at the lower order energies. 
\medskip

We next consider the commutation of $\opartial^I$ with 
$\Box^{\rm spatial}_{g^{m-1}}$ and with $\Box^{\rm mixed}_{g^{m-1}}$, and 
bound the resulting expressions. We distinguish the part
 $\Delta_{g^{m-1}}^{\rho^{m-1},T}$ in $\Box^{\rm spatial}_{g^{m-1}}$ which involves 
 derivatives in the $\Theta$-direction--thee terms can be bounded 
  using the formulas \eqref{Comm.Deltas} we already derived, and satisfy our required bounds, so we can consider the rest of the commutation terms.

We first consider the following terms  in the RHS of \eqref{boxpartialIgammam}:
\beq
\begin{split}
\label{comm.terms.Lapl}
&\sum_{I_1\bigcup I_2=I, |I_2|<|I|} \sum_{A,B=T,\Th,\rho, C=T,\rho }\partial^{I_1} 
(g^{AB}\Gamma_{AB}^C) 
\partial^{I_2}\partial_C \gamma^m_{\rm rest}, 
\sum_{I_1\bigcup I_2=I, |I_2|<|I|-1} \sum_{A,B=T,\rho}\partial^{I_1} g^{AB} 
\partial^{I_2}\partial_{AB} \gamma^m_{\rm rest}, 
\\&\sum_{I_1\bigcup I_2=I, |I_2|=|I|-1} \sum_{A,B=T,\rho}\partial^{I_1} g^{AB} 
\partial^{I_2}\partial_{AB} \gamma^m_{\rm rest}.
\end{split}
\eeq

We claim that the first two terms in the above are bounded in 
$L^2_{{\rm sin}\Th d\Th dT}[\Sigma_{\rho^{m-1}=\tau}]$ by virtue of our inductive assumptions by 
$B\tau^{-1-\frac{1}{4}}$. We claim that the second term is bounded in the same space by 
$B\tau^{-1-\frac{1}{4}} \sqrt{E[\opartial^I\gamma^m_{\rm rest}]}$.

These two claims follow readily from the inequality: 

\beq
\label{cvxty.ineq}
\begin{split}
&\|\sum_{I_1\bigcup I_2=I, |I_2|<|I|} \sum_{A,B=T,\Th,\rho, C=T,\rho}\partial^{I_1} 
(g^{AB}\Gamma_{AB}^C) 
\partial^{I_2}\partial_C \gamma^m_{\rm rest}\|_{L^2[\Sigma_{[\rho^{m-1}=\tau]}]}
\\&\le 
\sum_{A,B=T,\Th,\rho, C=T,\rho}\|\partial^{I} 
(g^{AB}\Gamma_{AB}^C) 
\|_{L^2[\Sigma_{\rho^{m-1}=\tau}]}\cdot 
\|\partial_C \gamma^m_{\rm rest}\|_{L^\infty[\Sigma_{[\rho^{m-1}=\tau]}]}
\\&+ \sum_{|I_1|=1, |I_2|=|I|-1}\sum_{A,B=T,\Th,\rho, C=T,\rho}\|\partial^{I_1}
(g^{AB}\Gamma_{AB}^C) 
\|_{L^\infty[\Sigma_{\rho^{m-1}=\tau}]}\cdot 
\|\partial^{I_2}\partial_C \gamma^m_{\rm rest}
\|_{L^2_{[\Sigma_{\rho^{m-1}=\tau}]}}.
\end{split}\eeq

Now, recall the bounds on the Christoffel symbols that we have derived. 
Recall also 
the Sobolev embedding, which bounds the $L^\infty$ norms above by the 
corresponding 
$H^2$ bound on those quantities; thus combining these estimates we derive 
that the 
above quantity is bounded by $B\tau^{-1-\frac{1}{4}}$. 

Let us now consider the third term in \eqref{comm.terms.Lapl}. In this 
term, the 
total number of derivatives on $\gamma^m_{\rm rest}$ \emph{equals} the 
number of 
derivatives that we are trying to bound (in our claimed energy bound). On 
the other 
hand, using that $|I_2|=l-1$ this term can be re-expressed and bounded as 
follows: 

\beq
\begin{split}
&\|\sum_{A,C=T,\rho}\partial g^{AC} \cdot \partial_{AC}\partial^{I_2}
\gamma^m_{\rm rest}\|_{L^2}
\\&\le 
\sum_{A,C=T,\rho}\|\partial g^{AC} \|_{L^\infty [\Sigma_{\rho^{m-1}=\tau}]}\cdot 
\sum_{i=1,2}\| (a^{m-1})^{iA}\|_{L^\infty [\Sigma_{\rho^{m-1}=\tau}]}\cdot \|
e^{m-1}_i\partial_{C}\partial^{I_2}
\gamma^m_{\rm rest}\|_{L^2 [\Sigma_{\rho^{m-1}=\tau}]}
\\&\le
B\tau^{-1-\frac{1}{4}} 
E[\partial^{I'}\gamma^m_{\rm rest}],
\end{split}
\eeq
where $|I'|=l$. 
The RHS bound is achieved when $C=A=\Th, i=2$. When either of 
$A,C$ takes the value $T$ and/or $i$ takes the value $1$ the bound is much 
less singular (in terms of powers of $\tau$). 
\medskip

Now, the commutation terms in $[\opartial^I,\big{(}\Delta^{\rho^{m-1},T}_{g^{m-1}}-\tilde{\Delta}^{\rho^{m-1},T}_{g^{m-1}}\big{)}]$, 
$[\opartial^I,g^{TT} \Gamma^{\Th}_{TT}\partial_\Th]$, 
$[\opartial^I, g^{T\Th}\Gamma^\Th_{T\Th}\partial_\Th]$ satisfy the bounds in our 
Proposition by the same simple application of the 
product inequality, by invoking the formulas we derived for
these terms using the formulas \eqref{Comm.Deltas}, \eqref{Delta.other.comm}, 
our inductive assumptions on the components of $g^{m-1}$ and the inductive assumptions 
on the lower-order derivatives of $\gamma^m$.

We can now derive $L^2$ bounds to the rest of the terms in the RHS of 
\eqref{boxpartialIgammam}. The remaining terms that were not treated
 are those commutation terms that arise from 
$\opartial^I$ acting on $\Box_{g^{m-1}}\gamma^m_{\rm rest}$ (see 
\eqref{wave.coords}), when some of 
the derivatives hit terms in the first two lines, (except for 
${\rm tr}_{\overline{g}^{m-1}}K^{m-1}$ which was the first to be considered 
above). Among the remaining terms, all terms involving \emph{one}
 derivative $\partial_{\rho^{m-1}}$ or $e^{m-1}_i(\rho^{m-1})$ 
are straightforwardly bounded as claimed, invoking the bounds 
in Lemma \ref{lem:e1re2r}
on $e^{m-1}_i(\rho^{m-1})$. 

We are left with one term: 
The  term on the RHS, 
 involving \emph{two} derivatives 
$\partial^2_{\rho^{m-1}\rho^{m-1}}\gamma^m_{\rm rest}$. 
\medskip

The first of these terms yields terms: 

\[
\partial^{I_1}\partial^2_{\rho^{m-1}\rho^{m-1}}\gamma^m_{\rm rest}\partial^{I_2}\bigg[(\frac{2M}{r}-1)
[1-\partial_r\chi(r)(r^{m-1}_*-\e)]^2\bigg],
\]
where $I_1\bigcup I+2=I$, $|I_1|\le |I|-1$. 
These terms do not immediately fall under our inductive assumptions since they involve two derivatives $\partial_\rho$.  	
For those  terms 
we invoke the inhomogenous wave equation \eqref{inhomog.wave.it} and the expression \eqref{wave.coords} for the wave operator
 to express these derivatives in terms of terms involving at most one $\partial_{\rho^{m-1}}$ derivatives. Our desired bound then follows from the already-derived bounds on the RHS of the resulting equation.  
In particular we obtain the following bounds, for $|I|\le s-3-4c$: 

\beq
\|\opartial^{I_1} \partial_{\rho^{m-1}}^2\gamma^m_{\rm rest}
\|_{L^2[\Sigma_{\rho^{m-1}=\tau}]}\le B C\eta \tau^{-2}+C\eta 
\tau^{-\frac{1}{8}}\sqrt{E[\partial^I\gamma^m_{\rm rest}]}, |I_1|=|I|-1.
\eeq
While for $|I|=k\in \{s-3-4c+1, \dots, s-4\}$:

\beq
\|\opartial^{I_1} \partial_{\rho^{m-1}}^2\gamma^m_{\rm rest}
\|_{L^2[\Sigma_{\rho^{m-1}=\tau}]}\le B C\eta \tau^{-2+\frac{k-low}{2}}
+C\eta 
\tau^{-\frac{1}{8}}\sqrt{E[\opartial^I\gamma^m_{\rm rest}]}, |I_1|=|I|-1.
\eeq
Thus these commutation terms also satisfy the required bounds. 
\medskip

Having controlled the commutation terms in the norm $L^2[{\rm sin\Th}d\Th dT]$ 
we can now derive our inductive claims on $\opartial^I\gamma^m_{\rm rest}$, at lower, and higher, and top orders:

\subsection{Lower order energy estimates: $E[\partial^I\gamma^m_{\rm rest}]$, $|I|\leq s-3-4c$}\label{sec:bott-lowenestgammam}
%



 
%
We summarize the energy estimate for $\partial^I\gamma^m_{\rm rest}$ (which proves our inductive step at the lower orders), 
$|I|\leq s-3-4c$ in the next proposition.
\begin{proposition}\label{prop:gammambotlowenineq}
If the inductive assumptions in \S\ref{Indhyp} on the metric $g^{m-1}$ 
hold true, then the following energy inequality on level sets 
$\{\rho^{m-1}=\rho\}$ is valid, for some universal constant $B>0$, and for
 the function $O(\dots)$ satisfying the bounds in 
\eqref{Oxpded}:
\beq
\label{gammambotlowenineq}
\begin{split}
&\sum_{|I|\leq s-3-4c}\{\rho^3(1+O((\rho^{m-1})^{1/4}))E[\partial^I\gamma^m_{\rm rest}]
\}
(\Sigma_{\rho^{m-1}=\rho})\leq\sum_{|I|\leq s-3-4c}(1+O(\e^{1/4}))
\epsilon^3E[\partial^I\gamma^m_{\rm rest}(\epsilon,t,\theta)]
\\&+\int^\epsilon_\rho
\frac{B^2}{\tau^{1-\frac{1}{4}}}
\sum_{|I|\leq s-3-4c}\tau^3E[\partial^I\gamma^m_{\rm rest}]d\tau
+\int^\epsilon_\rho \frac{5B\cdot  
C\eta}{\tau^{1-\frac{1}{4}}}
\sqrt{\sum_{|I|\leq s-3-4c}\tau^3 E[\partial^I\gamma^m_{\rm rest}]}
d\tau.
\end{split}
\eeq
\end{proposition}
Note that the inductive bound (\ref{inductiongammaopt}) for 
$\gamma^m_{\rm rest} $ follows readily from (\ref{gammambotlowenineq}) and Gronwall's 
inequality in Lemma \ref{lem:Gron}. In particular for 
$F^2(r)=\sum_{|I|\le s-3-4c} r^3 E[\partial^I\gamma^m_{\rm rest}]$, as well as 
$H(\tau)=10B^2\cdot \frac{1}{\tau^{1-\frac{1}{4}}}$
and 
$G(\tau)=  10BC\eta \tau^{-1+\frac{1}{4}}$
we find that \eqref{Gronineq2} implies:

\beq
\begin{split}
&\sqrt{(1+O(\tau^{1/4}))\sum_{|I|\leq s-3-4c}\tau^3 E[\partial^I\gamma^m_{\rm rest}][\Sigma_{\rho^{m-1}}]}
\\&\le e^{10B^2|
\int^\e_{\rho^{m-1}}
\tau^{-1+\frac{1}{4}}d\tau|}\big(\sqrt{(1+O(\e^{1/4}))
\sum_{|I|\le s-3-4c} \e^3 
E(\partial^I\gamma^m_{\rm rest} (\e,t,\theta)}
+\frac{1}{2}\int_{\rho^{m-1}}^\e \frac{B\cdot C 
\eta}{\tau^{1-\frac{1}{4}}}d\tau \big)
\\&\le e^{10B^2\int^\e_{\rho^{m-1}}
\tau^{-1+\frac{1}{4}}d\tau}\big( (1+O(\e^{1/4})) \eta 
+BC \e^{\frac{1}{4}}\eta  \big).
\end{split}
\eeq

Here the terms $O(\tau^{1/4})$, $O(\e^{1/4})$ satisfy the bounds \eqref{Oxpded}. 
In particular the above inequality implies:

\beq
\begin{split}
&\sqrt{\sum_{|I|\leq s-3-4c}\tau^3 E[\partial^I
\gamma^m_{\rm rest}][\Sigma_{\rho^{m-1}}]}\le (1+\e^{1/4}10B^2
)(\frac{C\eta}{4}+BC\e^{1/4}\eta).
\end{split}
\eeq

Thus invoking the bounds  \eqref{e.bd}, \eqref{Oxpded}
we derive that: 
\[
\sqrt{\sum_{|I|\leq s-3-4c}\tau^3 E[\partial^I
\gamma^m_{\rm rest}][\Sigma_{\rho^{m-1}}]}< C\eta, 
\]
as desired. 
So our claim in the lower orders follows, provided we can show Proposition 
 \ref{prop:gammambotlowenineq}. We do this next:

\begin{proof}[Proof of Proposition \ref{prop:gammambotlowenineq}] The Proposition will be proven by 
finite induction on $|I|=l$. So in particular the claimed bounds
\eqref{inductiongammaopt} are assumed to hold for all $\partial^I\gamma^m_{\rm rest}$ with $|I|\le l-1$. 
Recall the identity (\ref{Stokesv}) and set $v=\partial^I\gamma^m_{\rm rest}$, 
$|I|=l\leq s-3-4c$, $f(r)=r^\frac{3}{2}$ to obtain the energy inequality: 
[using also $\|u^{m-1}_{ij}\|_{L^\infty}\leq 5Br^{-1-\frac{1}{4}}$, $\|d_1^{m-1}(t,\theta)-\frac{1}{2}\|_{L^\infty}\leq DC\eta\le \frac{1}{8} $, 
$\|d_2^{m-1}(t,\theta)+1\|_{L^\infty}\leq DC\eta \le \frac{1}{8}$]
\begin{align}\label{partialIgammamenineq}
&\int_{\Sigma_{\rho^{m-1}}}r^\frac{3}{2}Q_{ab}[\partial^I\gamma^m_{\rm rest}](e_0)^bn^a\mathrm{vol}_{\Sigma_{\rho^{m-1}}}-\int_{\Sigma_{{r^{m-1}_*}}}{(^{m-1}r}_*)^\frac{3}{2}Q_{ab}[\partial^I\gamma^m_{\rm rest}](e_0)^bn^a\mathrm{vol}_{\Sigma_{{{r^{m-1}_*}}}}\\
\notag\leq&\int_{\rho^{m-1}}^{\epsilon}\int_{\Sigma_\tau}\Phi^{m-1} \bigg[\frac{\sqrt{2M}}{\tau^\frac{3}{2}}\big[
(DC\eta -2)\tau^\frac{3}{2}(e_1\partial^I\gamma^m_{\rm rest})^2
+(DC\eta -\frac{1}{2})\tau^\frac{3}{2}(e_2\partial^I\gamma^m_{\rm rest})^2\big]\\
\notag&+C\frac{1}{\tau^{1+\frac{1}{4}}}\tau^\frac{3}{2}\big[(e_0\partial^I \gamma^m_{\rm rest})^2+|\overline{\nabla}\partial^I\gamma^m_{\rm rest}|^2_{^{m-1}\overline{g}}\big] -\tau^\frac{3}{2}e_0\partial^I\gamma^m_{\rm rest}\square_{^{m-1}g} \partial^I\gamma^m_{\rm rest}
\bigg]
\mathrm{vol}_{\Sigma_\tau}d\tau
\end{align}
Now, invoke  (\ref{Q0n}),(\ref{lapselike}),(\ref{volSigmarho}) 
as well as the estimates in Lemma \ref{lem:e1re2r}; together with  the bounds 
$DC\eta <\frac{1}{8}$ just above, 
 we derive that the terms in the second line of the above 
are in fact \emph{negative} and  can thus be dropped to yield: 
\begin{align}\label{partialIgammamenineq2}
&(1+O((\rho^{m-1})^{1/4})){(\rho^{m-1})^3}E[\partial^I\gamma^m_{\rm rest}]-
(1+O(\e^{1/4}))\epsilon^3E[\partial^I\gamma^m_{\rm rest}(\epsilon,t,\theta)]\\
\notag\leq&\int_{\rho^{m-1}}^{\epsilon}\frac{5B}{\tau^{1-\frac{1}{4}}}\tau^3E[\partial^I\gamma^m_{\rm rest}]d\tau -\int_{\rho^{m-1}}^{\epsilon}\int_{\Sigma_\tau}\Phi^{m-1} \tau^\frac{3}{2}e_0\partial^I\gamma^m\square_{^{m-1}g} \partial^I\gamma^m_{\rm rest}
\mathrm{vol}_{\Sigma_\tau}d\tau
\end{align}

It remains to estimate the last term in (\ref{partialIgammamenineq2}). 
For this we
invoke Proposition \ref{opt.RHSs}, which we have already proven.
Our claim follows. 
\end{proof}

\subsection{Middle order energy estimates: $E[\partial^I\gamma^m_{\rm rest}]$, $s-3-4c<|I|\leq s-4$}

Next we derive the middle order energy estimates for 
$\gamma^m_{\rm rest}$,  (\ref{inductiongammalow}). We begin 
with proving a `summed
up'  estimate which involves the \emph{sum} of all energies for
 $\partial^I\gamma_{\rm rest}^m$ for $|I|$ between orders $s-3-4c+1$ and 
 $s-4$. In  particular this will confirm (\ref{inductiongammalow}) is 
valid for $\gamma^m_{\rm rest}$ in the case $|I|=s-4$. 

We will then derive the (stronger) 
inductive step claims for the lower derivatives 
$\partial^I\gamma^m_{\rm rest}$, $|I|\in \{s-3-4c+1,\dots, s-5\}$.

\subsubsection{Estimate on the sum of all middle-order energies. }
Our claim for the \emph{sum} of all middle-order energies is as follows: 

\begin{proposition}\label{prop:gammamlowenineq}
Assuming the inductive bounds in \S\ref{Indhyp} and the
lower-order  energy estimate (\ref{inductiongammaopt}), then the following 
energy inequality on level sets $\{\rho^{m-1}=\rho\}$  is
 valid
for some universal constant $B>0$:
\begin{align}\label{gammamlowenineq}
&\sum_{s-3-4c<|I|\leq s-4}(1+O(\rho))
\rho^3E[\partial^I\gamma^m_{\rm rest}]\leq
\sum_{s-3-4c<|I|\leq s-4}(1+O(\e))\epsilon^3E[\partial^I\gamma^m_{\rm rest}
(\epsilon,t,\theta)]\\
\notag &+\epsilon^{\frac{1}{2}-DC\eta}B^2C^2\eta^2r^{-2c+\frac{1}{2}}
+\int^\epsilon_\rho \frac{10B^2}{\tau^{1-\frac{1}{4}}}
\sum_{s-3-4c<|I|\leq s-4}\tau^3E[\partial^I\gamma^m_{\rm rest}]d\tau\\
\notag&+\int^\epsilon_\rho (\frac{5 C\eta+5(C\eta)^2}{\tau}
+\frac{BCC_{\rm Sob}\eta }{\tau^{1-\frac{1}{4}}})\tau^{-c+\frac{1}{4}}\cdot 
\sqrt{\sum_{s-3-4c<|I|\leq s-4}\tau^3E[\partial^I\gamma^m_{\rm rest}]}d\tau
\end{align}
for all $\rho\in(0,2\epsilon]$.
\end{proposition}
\begin{remark}
 We note that our claim concerns the energies of orders between $s-3-4c+1$
  and $s-4$, and only such energies appear also in the RHS. 
The contribution of the lower order energies
(for which the inductive claim has already been proven) is contained  in the term 
$\epsilon^{\frac{1}{2}-DC\eta}BC^3\eta^2r^{-2c+\frac{1}{2}}$. 
\end{remark}

\begin{remark}
The main difference of (\ref{gammamlowenineq}) from 
(\ref{gammambotlowenineq}) is the additional factor of $\tau^{-c}$ in the 
last term in the RHS of (\ref{gammamlowenineq}), which 
is responsible for the (weaker) bounds that we can derive at the middle 
orders;  the borderline term is 
\[
(\frac{(5 C\eta+5(C\eta)^2)}{\tau}\tau^{-c+\frac{1}{4}}
\sqrt{\sum_{s-3-4c<|I|\leq s-4}\tau^3E[\partial^I\gamma^m_{\rm rest}]}.
\]
We highlight in the proof the terms that 
contribute to the key borderline coefficient 
$(\frac{(5 C\eta+(5C\eta)^2)}{\tau}$.\footnote{This same 
 coefficient 
  also forces the choice of $c>0$, 
which when chosen large enough allows us to close our estimates--this is discussed a little furtherdown.} 
\end{remark}
Let us check how this Proposition implies our claim: Given (\ref{gammamlowenineq}), we employ Lemma \ref{lem:Gron} to derive:

\beq
\begin{split}
\label{gammamlowGron}
&\sqrt{(1+O(\rho^{1/4}))\sum_{s-3-4c<|I|= s-4}r^3E[\partial^I\gamma^m][\Sigma_\rho]}\leq
 e^{\int^\epsilon_r 5B\tau^{-1+\frac{1}{4}}d\tau} 
 \bigg[\sqrt{\sum_{s-3-4c<|I|= s-4}(1+O(\e^{1/4}))\epsilon^3E[\partial^I
 \gamma^m(\epsilon,t,\theta)][\Sigma_\rho]}
\\&    +\int^\epsilon_\rho(\frac{(5 C\eta+5(C\eta)^2}{\tau}+
\frac{C^{3/2}B^{1/2} }{\tau^{1-\frac{1}{4}}})\tau^{-c+\frac{1}{4}}d\tau\bigg]
\le (1+ 3BC \epsilon^{\frac{1}{4}} )[ (1+O(\e^{1/4}))\eta +C\eta
(\frac{5+5\cdot C\eta }{c}
+\frac{\sqrt{C}r^\frac{1}{4}\cdot\eta}{c+\frac{1}{4}})\rho^{-c+\frac{1}{4}}].
\end{split}
\eeq

Thus, (recalling that $C\eta<\frac{1}{8}$ and our choice of $c>20$) \eqref{c.bd} 
and 
since $\epsilon$ is appropriately small 
so that $2\e^{1/4} BC_S<\frac{C}{4} $,
we deduce the bound
\begin{align}\label{gammamlowGron2}
\sqrt{\sum_{s-3-4c<|I|= s-4}r^3E[\partial^I\gamma^m_{\rm rest}][\Sigma_r]}\leq
\frac{9}{10} C\eta r^{-c+\frac{1}{4}},&&\text{for all $r\in(0,2\epsilon]$}.
\end{align}
Now, directly below we will derive the improved bounds for all 
terms at lower orders $|I|<s-4$: letting $h=|I|-(s-3-4c)$ the energies of 
those  terms will be bounded by $10c r^{-\frac{h}{4}}$. In view of the 
bounds we have imposed on $\e$, \eqref{gammamlowGron2} then 
 confirms the inductive assumption (\ref{inductiongammalow}) for 
$\gamma^m_{\rm rest}$ in the case $|I|=s-4$.

\begin{proof}[Proof of Proposition \ref{prop:gammamlowenineq}]

We repeat the argument from the lower order case on invoking 
\eqref{partialIgammamenineq}, and 
 discarding the borderline bulk  terms 
in the second line  using their favorable sign. 

Thus  (\ref{partialIgammamenineq2}) is 
still valid for 
$\partial^I\gamma^m_{\rm rest}$, $|I|=l\leq s-4$. 
Our claim then follows by invoking Proposition \ref{opt.RHSs} for these middle orders.

\end{proof}

\subsubsection{Improved estimates at orders $s-4-k$, $k\in \{2,\dots, 4c-1\}$.}

Now that we have (\ref{gammamlowGron2}) at our disposal, we proceed to show that $\gamma^m_{\rm rest}$ satisfies the stronger claims in
 (\ref{inductiongammalow}), for all the lower derivatives $I$,  
 $s-3-4c< |I|=l<s-4$. 
\begin{proposition}\label{prop:gammamlowenineq2}
Assuming the inductive assumptions in \S\ref{Indhyp} hold true and the energy estimates \eqref{inductiongammalow},  (\ref{gammamlowGron2}) hold, then the following stronger energy inequality is valid for the (sum of the) first $4c-k$ of the higher
 derivatives:
\begin{align}\label{gammamlowenineq2}
&\sum_{s-3-4c<|I|\leq s-4-k}(1+O((\rho^{m-1})^{1/4})) r^3E[\partial^I\gamma^m_{\rm rest}][\Sigma_{\rho^{m-1}}]\leq \sum_{s-3-4c<|I|\leq s-4-k}(1+O(\e^{1/4}))\epsilon^3E[\partial^I\gamma^m_{\rm rest}(\epsilon,t,\theta)]\\
\notag&+\epsilon^{\frac{1}{2}-DC\eta}BC^2\eta \tau^{-2c+\frac{k}{2}+\frac{1}{2}}
+\int^\epsilon_\rho \frac{CB}{\tau^{1-\frac{1}{4}}}\sum_{s-3-4c<|I|\leq s-4-k}\tau^3E[\partial^I\gamma^m_{\rm rest}]d\tau\\
\notag&+\int^\epsilon_{\rho^{m-1}} \tau^{-c+\frac{k}{4}+\frac{1}{4}}\frac{BC^2\eta }{\tau^{1-\frac{1}{4}}}\sqrt{\sum_{s-3-4c<|I|\leq s-4-k}\tau^3E[\partial^I\gamma^m_{\rm rest}]}d\tau+\int^\epsilon_\rho
\frac{L_p}{\tau^p}\tau^{-2c+\frac{k}{2}+\frac{1}{2}}d\tau
\end{align}
for all $r\in(0,2\epsilon]$ and every $0<k<4c$. The exponent $p$ in 
the last term equals $p=1$, for $k\leq 2c$,\footnote{This corresponds to
 the higher half derivatives among the higher orders.}
 and $p=1-\frac{1}{4}$, for $k> 2c$.\footnote{This corresponds to
 the lower half derivatives among the higher orders.} The coefficient 
 $L_p$ equals $5C^2\eta ^2$ in the case $k\le 2c$; it equals $B^2C\eta$
 when $k>2c$. 
\end{proposition}
Let us check how the above Proposition implies our desired inductive step at the lower middle-derivatives. 
Lemma \ref{lem:Gron} applied to the energy inequality (\ref{gammamlowenineq2}) yields: 
\begin{align}\label{gammamlowGron3}
&\sqrt{(1+O((\rho^{m-1})^{1/4}))\sum_{s-3-4c<|I|\leq s-4-k}r^3E[\partial^I
\gamma^m_{\rm rest}](\rho^{m-1},t,\theta)}\\
\notag &\leq\,
e^{5\int^\epsilon_\rho B^2\tau^{-1+\frac{1}{4}}d\tau}
\bigg[\sqrt{(1+O(\e^{1/4}))\sum_{s-3-4c<|I|\leq s-4-k}\epsilon^3E
[\partial^I\gamma^m_{\rm rest}(\epsilon,t,\theta)]}\\
\notag &+\epsilon^\frac{1}{4}\sqrt{B}C\eta r^{-c+\frac{k}{4}+\frac{1}{4}}
+\int^\epsilon_\rho \tau^{-c+\frac{k}{4}}\frac{CB\eta }{\tau^{1-\frac{1}{4}}}d\tau+\sqrt{\int^\epsilon_r
\frac{L_p}{\tau^p}\tau^{-2c+\frac{k}{2}+\frac{1}{2}}d\tau}\bigg]
\end{align}
Then in the case of lower middle-order derivatives 
 $(4c-1\ge k> 2c$; recall that in that case $p=1-\frac{1}{4})$ we observe that
 $c-\frac{k}{4}+\frac{p}{2}-\frac{1}{2}=c-\frac{k}{4}-\frac{1}{8}\geq \frac{1}{8}$, and that the power of 
 $r$ in the last term of the above  is $-c+\frac{k}{4}+\frac{1}{8}$; in particular the power is by $\frac{1}{8}$ \emph{larger} than we need. Using that $r\in (0,2\e)$, 
 we derive that: 
 
 \beq\label{mid.bd.final}
\sqrt{\sum_{s-3-4c<|I|\leq s-4-k}r^3E[\partial^I
\gamma^m_{\rm rest}](r,t,\theta)}\le 
\frac{C}{2}(\eta+\epsilon^\frac{1}{8}
CBr^{-c+\frac{k}{4}+\frac{1}{4}}
+2BC\eta 
r^{-c+\frac{k}{4}+\frac{1}{8}+\frac{1}{4}}\big{)}.
 \eeq
 In view of the bounds  \eqref{e.bd} 
 we have have imposed on $\e$ in terms of $C, B$ 
 we derive: 
\begin{align}\label{gammamlowGron4}
\sqrt{\sum_{s-3-4c<|I|\leq s-4-k}r^3E[\partial^I\gamma^m_{\rm rest}]}\leq
\frac{3}{4}C\eta r^{-c+\frac{k}{4}}.
\end{align}
 So  by a finite induction on $k$ we derive:

\begin{align}\label{gammamlowGron5}
\sqrt{r^3E[\partial^I\gamma^m_{\rm rest}]}\leq C\eta r^{-c+\frac{k}{4}+\frac{1}{4}},&&\text{for every $|I|=s-4-k$, $2c<k<4c$.}
\end{align}

This confirms our inductive claim in this case.

On the other hand, for the case $(k\leq 2c;p=1)$ we
\emph{do not} have the gain of a power of $r^{\frac{1}{8}}$ more than we 
need. 
In particular our bound \eqref{gammamlowGron3} implies: 

\beq
\sqrt{\sum_{s-3-4c<|I|\leq s-3-k}r^3E[\partial^I
\gamma^m_{\rm rest}](r,t,\theta)}\le 
\frac{C}{2}(\eta+\epsilon^\frac{1}{8}
CBr^{-c+\frac{k}{4}+\frac{1}{4}}
+\frac{5C+(5C)^2\eta}{c}\eta 
r^{-c+\frac{k}{4}}\big{)}
\eeq
So in this case to derive our inductive claim we
invoke the lower bounds \eqref{c.bd} we imposed on $c, c>20$
coupled with \eqref{e.bd} for the second term to show that the RHS is 
$\lesssim C\eta r^{-c+\frac{k}{4}+\frac{1}{4}}$. 
\medskip

Thus matters are reduced to proving Proposition 
\ref{prop:gammamlowenineq2}. We do this next: 
\medskip

\begin{proof}[Proof of Proposition \ref{prop:gammamlowenineq2}]
We argue by finite induction. Starting from the estimate (\ref{gammamlowGron2}), $|I|=s-4$, that we proved above, assume (\ref{gammamlowenineq2}) 
and hence (\ref{gammamlowGron4}) are valid for $|I|=s-3-1,\ldots,s-3-k+1$. We will derive (\ref{gammamlowenineq2}) for the fixed $0<k<4c$.
Recall that (\ref{partialIgammamenineq2}) is valid for any multi-index $I$. We proceed to estimate the last term in the RHS of 
(\ref{partialIgammamenineq2}) for $s-3-4c<|I|\leq s-4-k$ by plugging in the wave equation (\ref{boxpartialIgammam}) and arguing similarly to 
(\ref{gammamlowenineq}) to obtain:

\begin{align}\label{middlegammambulkest}
&\bigg|\int_{\rho^{m-1}}^{\epsilon}\int_{\Sigma_\tau}\Phi^{m-1} \tau^\frac{3}{2}e_0\partial^I\gamma^m_{\rm rest}\square_{^{m-1}g} \partial^I\gamma^m_{\rm rest}
\mathrm{vol}_{\Sigma_\tau}d\tau\bigg|\\
\notag\leq&\bigg|\int^\epsilon_{\rho^{m-1}}\int_{\Sigma_\tau}\Phi^{m-1}\tau^\frac{3}{2}e_0\partial^I\gamma^m_{\rm rest}\sum_{I_1\cup I_2=I,\,|I_2|<|I|}\partial^{I_1}\text{tr}_{^{m-1}\overline{g}}K^{m-1}e_0\partial^{I_2}\gamma^m_{\rm rest}\mathrm{vol}_{\Sigma_\tau}d\tau\bigg|\\
\notag& +\bigg|\int^\epsilon_{\rho^{m-1}}\int_{\Sigma_\tau}\Phi^{m-1}\tau^\frac{3}{2}e_0\partial^I\gamma^m_{\rm rest}\sum_{I_1\cup I_2=I,\,|I_2|<|I|}\partial^{I_1}(\text{tr}_{^{m-1}\overline{g}}K^{m-1}-\text{tr}_{\overline{g}^S}K^S)e_0\partial^{I_2}\gamma^S\mathrm{vol}_{\Sigma_\tau}d\tau\bigg|\\
\notag& |\int_{\rho^{m-1}}^\e \int_{\Sigma_\tau}
\Phi^{m-1} \tau^{\frac{3}{2}}e_0\partial^I\gamma^m_{\rm rest} \widetilde{{\rm RHS}[\eqref{boxpartialIgammam}]} {\rm vol}_{\Sigma_\tau}d\tau|. 
\end{align}
Here $\widetilde{{\rm RHS}[\eqref{boxpartialIgammam}]}$ stands for all the \emph{other} 
terms in the RHS of \eqref{boxpartialIgammam} except for 
the two we wrote out explicitly (involving  ${\rm tr}_{\overline{g}^{m-1}}K^{m-1}$). 
Given the estimates we have derived for all these terms, we find their contribution 
to be below-borderline; thus invoking Cauchy-Schwartz  
we find they contribute to all the  
terms in the RHS of \eqref{gammamlowenineq2}. Thus matters are reduced to 
controlling the first two lines in the RHS of 
\eqref{middlegammambulkest}.

At this point our method of proof deviates from that of Proposition \ref{prop:gammamlowenineq}, due to the more careful handling needed for the first two  terms in the RHS of (\ref{middlegammambulkest}) in order to derive the desired stronger conclusion. These are the only borderline terms in the middle order energy estimates and combined, these give the last term in the RHS of the energy inequality (\ref{gammamlowenineq2}), instead of the borderline coefficient $\sim\eta \tau^{-1}$ in the third line of (\ref{gammamlowenineq}). In fact the borderline terms correspond to $I_1=I,I_2=\emptyset$, $|I|=s-4-k$, while the rest of the summands can be easily seen by the inductive assumption (\ref{inductiontrKlow}),(\ref{trKheur}) on $K^{m-1}$ to be below borderline. We control these other, below-borderline terms as follows: 
\begin{align}\label{middlegammambulkest3}
&\bigg|\int^\epsilon_{\rho^{m-1}}\int_{\Sigma_\tau}\Phi^{m-1}\tau^\frac{3}{2}e_0\partial^I\gamma^m_{\rm rest}\cdot 
\sum_{I_1\cup I_2,|I_1|<|I|,|I_2|<|I|}\partial^{I_1}\text{tr}_{^{m-1}\overline{g}}K^{m-1}e_0\partial^{I_2}\gamma^m_{\rm rest}\mathrm{vol}_{\Sigma_\tau}d\tau\bigg|\\
\notag\leq&\,C_{\rm Sob}\int^\epsilon_{\rho^{m-1}}\sqrt{\tau}\|\partial\text{tr}_{^{m-1}\overline{g}}K^{m-1}\|_{H^{s-4-4c}}\sum_{s-3-4c<|I|\leq s-3-k} 
\tau^3E[\partial^I\gamma^m_{\rm rest}] d\tau\\
\notag&+C_{\rm Sob}\int^\epsilon_{\rho^{m-1}}\sqrt{\tau}\tau^\frac{3}{2}\|e_0\gamma^m_{\rm rest}\|_{H^{s-4-4c}}\|\text{tr}_{^{m-1}\overline{g}}K^{m-1}\|_{H^{s-3-k}}
\tau^\frac{3}{2}\|e_0\partial^I\gamma^m_{\rm rest}\|_{L^2}d\tau\\
\notag\leq&\int^\epsilon_{\rho^{m-1}}\frac{C\eta}{\tau^{1-\frac{1}{4}}}\sum_{s-3-4c<|I|\leq s-4-k}\tau^3E[\partial^I\gamma^m_{\rm rest}] 
d\tau+\int^\epsilon_{\rho^{m-1}}\frac{C_{\rm Sob}B\eta }{\tau}\tau^{-c+\frac{k}{4}+\frac{1}{4}+\frac{1}{4}}
(\tau^\frac{3}{2}\|e_0\partial^I\gamma^m_{\rm rest}\|_{L^2})d\tau
\end{align}
The same argument can be used to bound 
\[
\bigg|\int^\epsilon_{\rho^{m-1}}\int_{\Sigma_\tau}\Phi^{m-1}\tau^\frac{3}{2}e_0\partial^I\gamma^m_{\rm rest}\cdot 
\sum_{I_1\cup I_2,|I_1|<|I|,|I_2|<|I|}\partial^{I_1}\text{tr}_{^{m-1}\overline{g}}K^{m-1}e_0\partial^{I_2}\gamma^S\mathrm{vol}_{\Sigma_\tau}d\tau\bigg|
\]
by the same bounds. 
\medskip

We proceed now to the borderline term with  $I_1=I,I_2=\emptyset$, 
$|I|=s-4-k$. The key observation that allows us to handle this term in a 
more 
refined manner is the following: having at our disposal the already 
derived  energy bound (\ref{gammamlowGron2}), we may integrate by parts 
and view 
the resulting term as an inhomogeneous term that \emph{we have already 
controlled}.  Moreover, in the case $k> 2c$, we may also exploit the 
splitting of the 
mean curvature (\ref{trKheur}) at the lower derivatives, after performing 
$4c-k$ consecutive integrations by parts to offload derivatives from 
$\partial^{I_1}[{\rm tr}_{g^{m-1}} K^{m-1}-{\rm tr}_{g^S} K^S]$.
 More precisely, we treat this 
term as follows: 

{\it Case $k\leq 2c$}: Integrate by parts once the $I_1=I,I_2=\emptyset$ term in (\ref{middlegammambulkest}), where $|I|=s-4-k$, and use the induction step together with the assumption (\ref{inductiontrKlow}) to derive
\begin{align}\label{bordtermlowest}
&|\int^\epsilon_{\rho^{m-1}}\int_{\Sigma_\tau}\Phi^{m-1}\tau^\frac{3}{2}e_0\partial^I\gamma^m_{\rm rest}\partial^I\text{tr}_{^{m-1}\overline{g}}K^{m-1}e_0\gamma^m_{\rm rest}\mathrm{vol}_{\Sigma_\tau}d\tau|\\
\tag{$I_1\cup I_2=I,|I_2|=1$}=&|\int^\epsilon_{\rho^{m-1}}
\int_{\Sigma_\tau}\partial^{I_1}(\text{tr}_{^{m-1}\overline{g}}K^{m-1}-\text{tr}_{g^S} K^S)\partial^{I_2}\bigg(\Phi^{m-1}\tau^\frac{3}{2}e_0\partial^I\gamma^m_{\rm rest}\cdot e_0\gamma^m_{\rm rest}\frac{\mathrm{vol}_{\Sigma_\tau}}{\mathrm{vol}_{Euc}}\bigg)\mathrm{vol}_{Euc}d\tau|\\
\tag{$\Phi\sim\sqrt{\tau},\,\mathrm{vol}_{\Sigma_\tau}
\sim\tau^\frac{3}{2}\mathrm{vol}_{Euc},\,e_0\gamma^m_{\rm rest}\sim 
\tau^{-\frac{3}{2}}$}\leq&\int^\epsilon_{\rho^{m-1}}
\frac{3C^3\eta^3}{\tau}\tau^{-2c+\frac{k}{2}+\frac{1}{2}}
d\tau.
\end{align}

We analogously treat the borderline terms 
from the second line RHS of \eqref{middlegammambulkest}, which arose from 
 from 
  $\Box_{^{m-1}g}\gamma^S$. 

\begin{align}\label{bordtermRHSlowest}
&|\int^\epsilon_{\rho^{m-1}}\int_{\Sigma_\tau}\Phi^{m-1}
\tau^\frac{3}{2}e_0\partial^I\gamma^m_{\rm rest}\partial^I
(\text{tr}_{^{m-1}\overline{g}}K^{m-1}-
\text{tr}_{\overline{g}^S}K^S)e_0\gamma^S
\mathrm{vol}_{\Sigma_\tau}d\tau|\\
\tag{$I_1\cup I_2=I,|I_2|=1$}=&|\int^\epsilon_{\rho^{m-1}}
\int_{\Sigma_\tau}\partial^{I_1}
(\text{tr}_{^{m-1}\overline{g}}
K^{m-1}-\text{tr}_{\overline{g}^S}
K^S)\partial^{I_2}\bigg(\Phi^{m-1}\tau^\frac{3}{2}e_0\partial^I\gamma^m_{\rm rest}e_0\gamma^S\frac{\mathrm{vol}_{\Sigma_\tau}}{\mathrm{vol}_{Euc}}\bigg)\mathrm{vol}_{Euc}d\tau|\\
\tag{$\Phi\sim \sqrt{\tau},\,\mathrm{vol}_{\Sigma_\tau}
\sim\tau^\frac{3}{2}\mathrm{vol}_{Euc},\,e_0\gamma^m_{\rm rest}\sim 
\tau^{-\frac{3}{2}}$, with the factor of $2M$ cancelling out}\leq&\int^\epsilon_{\rho^{m-1}}
\frac{4C^2\eta^2}{\tau}\tau^{-2c+\frac{k}{2}+\frac{1}{2}}
d\tau
\end{align}

{\it Case $k> 2c$}: We need to bound the same two terms as in the
 previous case. In this setting, we integrate by parts $4c-k$ times from the the
  $I_1=I,I_2=\emptyset$ written-out terms
 in (\ref{middlegammambulkest}), where $|I|=s-4-k$:
\begin{align}\label{bordtermlowest2}
&|\int^\epsilon_{\rho^{m-1}}\int_{\Sigma_\tau}\Phi^{m-1}\tau^\frac{3}{2}e_0\partial^I\gamma^m_{\rm rest}\partial^I\text{tr}_{^{m-1}\overline{g}}K^{m-1}e_0\gamma^m_{\rm rest}\mathrm{vol}_{\Sigma_\tau}d\tau|\\
\tag{$I_1\cup I_2=I, |I_2|=4c-k$}=&\,|(-1)^{4c-k}\int^\epsilon_{\rho^{m-1}}\int_{\Sigma_\tau}\partial^{I_1}\text{tr}_{^{m-1}\overline{g}}K^{m-1}\partial^{I_2}\bigg(\Phi^{m-1}\tau^\frac{3}{2}e_0\partial^I\gamma^m_{\rm rest}e_0\gamma^m_{\rm rest}\frac{\mathrm{vol}_{\Sigma_\tau}}{\mathrm{vol}_{Euc}}\bigg)\bigg|_{I_1\cup I_2=I}\mathrm{vol}_{Euc}d\tau|\\
\tag{$|I_1|=s-4-4c$, using $s>8c+3$}\leq&\int^\epsilon_{\rho^{m-1}}\|\partial^{I_1}\text{tr}_{^{m-1}\overline{g}}K^{m-1}\|_{L^2}C\eta \sqrt{\tau} \|\tau^\frac{3}{2}e_0\partial^I\gamma^m_{\rm rest}\|_{H^{4c-k}}\mathrm{vol}_{Euc}d\tau\\
\tag{by (\ref{trKheur}) and the induction step for $\partial^{I_2}\partial^I\gamma^m$, $|I_2|+|I|=s-5-(2k-4c)$}\leq&\int^\epsilon_{\rho^{m-1}}\frac{BC^2\eta^3}{\tau^{1-\frac{1}{4}}} \tau^{-c+\frac{2k-4c}{4}+\frac{1}{2}}d\tau\\
\notag=&\int^\epsilon_{\rho^{m-1}}\frac{BC^2\eta^2}{\tau^{1-\frac{1}{4}}} \tau^{-2c+\frac{k}{2}+\frac{1}{2}}d\tau
\end{align}
We also apply the same argument to the second
 borderline term with the factor 
$\gamma^S$ again to find: 

 \begin{align}\label{bordtermRHSlowest3}
&\int^\epsilon_{\rho^{m-1}}\int_{\Sigma_\tau}\Phi^{m-1}\tau^\frac{3}{2}e_0\partial^I(\gamma^m_{\rm rest})\partial^I\text{tr}_{^{m-1}\overline{g}}(K^{m-1}- \text{tr}_{^{m-1}\overline{g}}K^S)  e_0\gamma^S
\mathrm{vol}_{\Sigma_\tau}d\tau\\
\tag{$I_1\cup I_2=I, |I_2|=4c-k$}=&\,(-1)^{4c-k}\int^\epsilon_{\rho^{m-1}}\int_{\Sigma_\tau}\partial^{I_1}(\text{tr}_{^{m-1}\overline{g}}K^{m-1}-\text{tr}_{^{m-1}\overline{g}}K^S)\partial^{I_2}\bigg(\Phi^{m-1}\tau^\frac{3}{2}e_0\partial^I\gamma^Se_0\gamma^m_{\rm rest}\frac{\mathrm{vol}_{\Sigma_\tau}}{\mathrm{vol}_{Euc}}\bigg)\bigg|_{I_1\cup I_2=I}\mathrm{vol}_{Euc}d\tau\\
\tag{$|I_1|=s-4-4c$, using $s>8c+3$}\leq&C_S\int^\epsilon_{\rho^{m-1}}\|\partial^{I_1}(\text{tr}_{^{m-1}\overline{g}}K^{m-1}-\text{tr}_{^{m-1}\overline{g}}K^S)\|_{L^2}C \sqrt{\tau} \|\tau^\frac{3}{2}e_0\partial^I\gamma^m_{\rm rest}\|_{H^{4c-k}}d\tau\\
\tag{by (\ref{trKheur}) and the induction step for $\partial^{I_2}\partial^I\gamma^m$, $|I_2|+|I|=s-5-(2k-4c)$}\leq&\int^\epsilon_{\rho^{m-1}}\frac{C_SB(C\eta)^2}{\tau^{1-\frac{1}{4}}} \tau^{-c+\frac{2k-4c}{4}+\frac{1}{2}}d\tau\\
\notag=&\int^\epsilon_{\rho^{m-1}}\frac{B^2(C^2\eta)}{\tau^{1-\frac{1}{4}}} \tau^{-2c+\frac{k}{2}+\frac{1}{2}}d\tau.
\end{align}

The RHSs of (\ref{bordtermlowest}), (\ref{bordtermRHSlowest}) (\ref{bordtermlowest2}), (\ref{bordtermRHSlowest3}) correspond exactly to the last term in the 
claimed energy inequality (\ref{gammamlowenineq2}). Thus, combining (\ref{partialIgammamenineq2}), (\ref{bordtermlowest})-(\ref{bordtermRHSlowest3}) and summing in $s-3-4c<|I|\leq s-4-k$, we arrive at (\ref{gammamlowenineq2}). This completes the proof of the proposition and hence the higher order energy estimates (\ref{inductiongammalow}) for $\gamma^m_{\rm rest}$.
\end{proof}
\subsection{Top order estimates for $\gamma^m$}\label{subsec:topgammam}
\label{sec:top_order}

The top order inductive assumptions (\ref{inductiongammatopmixed}) that we 
wish to derive for $\gamma^m_{\rm rest}$ are divided into cases based on $J_0$, 
($J_0=0,1,2$). We only study the top-order case $J_0=2$, since the  remaining cases 
follow from it by taking integrals of the top order estimates in the $e_0$ direction.

We also note that at this top order it suffices to derive our claim for 
$\gamma^m$ 
directly instead of $\gamma^m_{\rm rest}$. This is just because for $|I|=s-3$:\footnote{(using $e_0$ instead of $\partial_\rho$ for convenience--the two vector fields are parallel amd the extra 
terms generated by replacing $e_0$ by $(\rho)^{-1/2}\partial_\rho$ are easily seen 
to be allowed by our claim). }

\[
\partial^I e_0 e_0\gamma^m_{\rm rest}= \partial^I e_0 e_0\gamma^m+\partial^I e_0 \big(\frac{2M}{r}-1\big)^{\frac{1}{2}}\cdot \frac{1}{r}.
\]
The second term in the RHS can be straightforwardly  bounded invoking Lemma
 \ref{lem:e1re2r}, and the bounds are better that those claimed on the LHS. 
 So it suffices to derive the claimed bounds on $\gamma^m$. We also recall that our claim \eqref{inductiongammatopmixed} is for all derivatives $\partial^I$ but where 
 $I\ne (T,\dots, T)$.
 \medskip

We will prove our claimed estimate for 
$\opartial^I\partial_{\rho\rho}^2\gamma^m_{\rm rest}$, with the RHS in 
our claim  having an extra factor $(\frac{2M}{r}-1)^{-1}$. In view of the relation \eqref{e0r}, this
clearly implies our claim with two $e_0$-derivatives 
 in place of the two $\partial_\rho$-derivatives.

So we use $\partial_\rho$ instead of $e_0$ as a commutator field. We derive the equation:

\begin{align}\label{wave.coords2}
&\Box_{g^{m-1}}[\opartial^I\partial_{\rho\rho}\gamma^m]=\\
\notag&\sum_{I_1\bigcup I_2=I, J_1+ J_2=2, |I_2|+J_2<s-1}
\opartial^{I_1}\partial^{J_1}_\rho \bigg[(\frac{2M}{r}-1)
[1-\partial_r\chi(r)({^{m-1}r}_*-\epsilon)]^2\bigg]\cdot \opartial^{I_2}
\partial_\rho^{J_2}
\partial_{\rho^{m-1}}^2\gamma^m \\
-\notag&\,\sum_{I_1\bigcup I_2=I, J_1+ J_2=2, |I_2|+J_2<s-1} \bigg{\{} \opartial^{I_1}\partial^{J_1} \bigg[+\bigg[\mathrm{tr}_{^{m-1}\overline{g}}K^{m-1}\cdot (\frac{2M}{r}-1)^{\frac{1}{2}}
+\frac{M}{r^2}
(\frac{2M}{r}-1)^{-\frac{1}{2}}\\
\notag &-(\frac{2M}{r}-1)^{\frac{1}{2}}\partial_r^2\chi(r)
({^{m-1}r}_*-\epsilon)[1-\partial_r\chi(r)({^{m-1}r}_*-\epsilon)]\bigg]
\\
\notag&+ [O(e^{m-1}(\rho^{m-1}))\cdot O(\sqrt{r}) +O(\sqrt{r})\cdot [a^{m-1}_{T1}a^{m-1}_{T2}\cdot e^{m-1}_1re^{m-1}_2r]\cdot  (\frac{2M}{r})^{\frac{1}{2}}\partial_{\rho^{m-1}}\bigg]
\opartial^{I_2}\partial^{J_2}_\rho \partial_{\rho^{m-1}}\gamma^m \\
\notag &+2\opartial^{I_1}\partial^{J_1}_\rho[\sum_{A=T,\Th} (g^{m-1})^{A\rho }\partial^2_{A\rho}] \opartial^{I_2}\partial^{J_2}_\rho\gamma^m 
-\opartial^{I_1}\partial^{J_1}_\rho[\sum_{A=\phi,T,\Th}(^{m-1}\overline{g})^{\rho B}(^{m-1}\Gamma)_{\rho B}^\rho]
\opartial^{I_2}\partial^{J_2}_\rho[\partial_\rho \gamma^m]\\
\notag&- 2
\opartial^{I_1}\partial^{J_1}_\rho[\sum_{B=\phi,T,\Th}(^{m-1}\overline{g})^{\rho B}(^{m-1}\Gamma)_{\rho B}^\Th\cdot ({\rm sin}\Th)^{-1}] \opartial^{I_2}\partial^{J_2}_\rho[
{\rm sin}\Th\cdot  \partial_\Th \gamma^m]\\
\notag&-2
\opartial^{I_1}\partial^{J_1}_\rho[\sum_{B=\phi,T,\Th}(^{m-1}\overline{g})^{\rho B}(^{m-1}\Gamma)_{\rho B}^T] \opartial^{I_2}\partial^{J_2}_\rho[\partial_T \gamma^m]
+\opartial^{I_1}\partial^{J_1}_\rho[\sum_{A,B=\phi,T,\Th}(^{m-1}\overline{g})^{AB}]\opartial^{I_2}\partial^{J_2}_\rho[\partial^2_{AB}\gamma^m]\\
\notag&- 
\opartial^{I_1}\partial^{J_1}_\rho[\sum_{A,B=\phi,T,\Th}(^{m-1}{g})^{AB}\frac{(^{m-1}\Gamma)_{AB}^\Th}{{\rm sin}\Th}] \opartial^{I_2}\partial^{J_2}_\rho[{\rm sin}\Th \partial_\Th \gamma^m]\\
\notag&-
\opartial^{I_1}\partial^{J_1}_\rho[\sum_{A,B=\phi,T,\Th}(^{m-1}{g})^{AB}(^{m-1}\Gamma)_{AB}^T] \opartial^{I_2}\partial^{J_2}_\rho[\partial_T \gamma^m]\bigg{\}}
\end{align}

Our analysis proceeds by the energy estimate method we used in the 
lower and higher orders.
Similarly to (\ref{partialIgammamenineq})-(\ref{partialIgammamenineq2}), setting 
$v=\opartial^I\partial_{\rho\rho}\gamma^m, f(r)=r^{\frac{3}{2}+2}$ in (\ref{Stokesv}) we
 derive:\footnote{The larger exponent of the weight, $r^{\frac{3}{2}+2}$, compared to that of $r^\frac{3}{2}$ in (\ref{partialIgammamenineq2}) generates in fact additional favourable terms in the RHS, but we do not need them to close our argument.}
\beq
\begin{split}\label{StokespartialIeie0gammam}
&(1+O((\rho^{m-1})^{1/4})){(^{m-1}\rho)^{5}}E[\opartial^I\partial_{\rho\rho}^2\gamma^m]
(\Sigma_{\rho^{m-1}})
\leq (1+O(\e^{1/4}))\epsilon^{5}E[\opartial^I
\partial_{\rho\rho}^2\gamma^m(\epsilon,t,\theta)]
+\int_{\rho^{m-1}}^{\epsilon}\frac{B}{\tau^{1-\frac{1}{4}}}\tau^{5}E[\opartial^I
\partial_{\rho\rho}^2e_0\gamma^m]d\tau 
\\&-\int_{\rho^{m-1}}^{\epsilon}\int_{\Sigma_\tau}\Phi^{m-1} 
\tau^{\frac{3}{2}+2}(e_0\opartial^I\partial_{\rho\rho}^2\gamma^m)
\square_{^{m-1}g} \opartial^I\partial_{\rho\rho}^2\gamma^m
\mathrm{vol}_{\Sigma_\tau}d\tau
\end{split}
\eeq

We will divide the various terms that are generated from plugging (\ref{wave.coords2}) into (\ref{StokespartialIeie0gammam}) into three categories: 
\begin{align}\label{boxtermsplit}
\Phi^{m-1} \tau^{\frac{3}{2}+2}e_0\partial^I\partial_{\rho\rho}^2\gamma^m\square_{^{m-1}g} \partial^I\partial_{\rho\rho}^2\gamma^m= 
\mathfrak{F}_{Gron}+\mathfrak{F}_{bord}+\mathfrak{F}_{IBP},
\end{align}
where the terms in $\mathfrak{F}_{Gron}$ will be  below borderline in the
 sense that they satisfy:

\begin{align}\label{belbordterms}
\bigg|\int_{\Sigma_\tau}\mathfrak{F}_{Gron}\mathrm{vol}_{\Sigma_\tau}\bigg|\leq\frac{B}{\tau^{1-\frac{1}{4}}}\tau^{5}
E[\partial^I\partial^2_{\rho\rho}\gamma^m]+\frac{ B}{\tau^{1-\frac{1}{4}}}\tau^{-c}\sqrt{\tau^5E[\partial^I\partial^2_{\rho\rho}\gamma^m]}
+ B^2C^2 \eta^2\tau^{-2c-2+\frac{1}{4}}.
\end{align} 

They are innocuous in our RHS since they  can be treated directly 
by the Gronwall inequality in Lemma \ref{lem:Gron}.  
On the other hand, $\mathfrak{F}_{bord}$ includes all borderline terms and satisfies the estimate 
\begin{align}\label{bordterms}
\bigg|\int_{\Sigma_\tau}\mathfrak{F}_{bord}\mathrm{vol}_{\Sigma_\tau}\bigg|\leq\frac{C\eta \cdot 4\cdot 5}{\tau}\tau^{-c}\sqrt{\tau^5E[\partial^I\partial^2_{\rho\rho}\gamma^m]}.
\end{align}
%


Lastly, the third category $\mathfrak{F}_{IBP}$ consists of terms with 
factors that have an excessive number of spatial derivatives (relative to 
the bounds 
on various quantities that we are inductively assuming) and on which we 
need to perform integrations by parts twice, once in a spatial direction 
and once 
in $\partial_\rho$. As we shall see below, the generated terms from this procedure 
will then 
all fall in the category $\mathfrak{F}_{Gron}$; thus we claim:
\begin{align}
\label{IBPterms}
\bigg|\int_0^\e\int_{\Sigma_\tau}\mathfrak{F}_{\rm IBP}\mathrm{vol}_{\Sigma_\tau}d\tau 
\bigg|\leq\int_0^\e\frac{B}{\tau^{1-\frac{1}{4}}}\tau^{-c}\sqrt{\tau^5E[\partial^I\partial^2_{\rho\rho}\gamma^m]}d\tau.
\end{align}

   \medskip

\begin{remark}
Bulk terms that required integrations by parts 
appeared also at the middle orders, however the reason there was different--it was to derive less singular 
behaviour in $r$ at the orders strictly below $s-4$.

In the setting here, 
we are morally able to 
 re-express the terms in $\mathfrak{F}_{IBP}$ 
 as terms of the form $\mathfrak{F}_{Gron}$
 \emph{because} of the choice of $e_0$ (the multiplier vector field) 
 \emph{also} 
 as a commutator at this  top order we are considering here. 
\end{remark}

\begin{proposition}\label{prop:gammamtopenineq}
Assuming the inductive estimates in \S\ref{Indhyp},  and the lower order estimates (\ref{gammamlowGron2}),(\ref{gammamlowGron5}), then the following energy inequality is valid at the top order:
\begin{align}\label{gammamtopenineq}
&\sum_{|I|=s-3,\,i=1,2}(^{m-1}\rho)^5(1+O((\rho^{m-1})^{1/4}))E[\opartial^I\partial^2_{\rho\rho}\gamma^m][\Sigma_{\rho^{m-1}}]\\
\notag &\leq\sum_{|I|= s-3,\,i=1,2}\bigg[\epsilon^{5}(1+O(\e^{1/4}))E[\opartial^I\partial^2_{\rho\rho}\gamma^m(\epsilon,t,\theta)][\Sigma_{\e}]\\
\notag&+\e^{1/4} BC^3 \eta^2\tau^{-2c-2+\frac{1}{4}}
+\int^\epsilon_{\rho^{m-1}} \frac{BC}{\tau^{1-\frac{1}{4}}}\tau^5E[\opartial^I\partial^2_{\rho\rho}\gamma^m]d\tau
+\int^\epsilon_{\rho^{m-1}} (\frac{C\eta \cdot 20}{\tau}+\frac{B^2}{\tau^{1-\frac{1}{4}}})\tau^{-c}\sqrt{\tau^5E[\opartial^I
\partial^2_{\rho\rho}\gamma^m]}d\tau
\end{align}
for all $\rho^{m-1}\in(0,2\epsilon]$.
\end{proposition}
Applying Lemma \ref{lem:Gron} to (\ref{gammamtopenineq}) and arguing as in (\ref{gammamlowGron}), we obtain the desired top order estimate:
\begin{align}\label{gammamtopenest}
\sqrt{(\rho^{m-1})^{5}E[\partial^I\partial^2_{\rho\rho}\gamma^m]}\leq C\eta {(^{m-1}\rho)^{-c}},&&\text{ for all $\rho^{m-1}\in(0,2\epsilon]$},
\end{align}
in view of the bounds \eqref{c.bd}, \eqref{e.bd} 
that we have imposed. So we proceed to prove the Proposition:

\begin{proof}[Proof of Proposition \ref{prop:gammamtopenineq}] It suffices to prove the validity of the splitting (\ref{boxtermsplit}):
\begin{align}\label{bulktermsplit}
\notag\int_{\rho^{m-1}}^{\epsilon}\int_{\Sigma_\tau}\Phi^{m-1} \tau^{\frac{3}{2}+2}e_0\partial^I\partial^2_{\rho\rho}
\gamma^m\square_{^{m-1}g} \partial^I\partial^2_{\rho\rho}\gamma^m
\mathrm{vol}_{\Sigma_\tau}d\tau\\
=\int^\epsilon_{\rho^{m-1}}\int_{\Sigma_\tau}(\mathfrak{F}_{Gron}+\mathfrak{F}_{bord}+\mathfrak{F}_{IBP})\mathrm{vol}_{\Sigma_\tau}d\tau
\end{align}
where the $\mathfrak{F}_{Gron},\mathfrak{F}_{IBP}$ terms satisfy (\ref{belbordterms}), \eqref{IBPterms}, while the borderline terms included in $\mathfrak{F}_{bord}$ satisfy (\ref{bordterms}).

We start by grouping together all the terms which are strictly below 
borderline. 

%

\begin{lemma}\label{lem:FGron}

Consider the RHS of  \eqref{wave.coords2}. All terms with 
 $|I_1|<s-3$
are placed in $\mathfrak{F}_{\rm Gron}$.

The further terms 
from the RHS of \eqref{wave.coords2}
that fall under ${\cal F}_{Gron}$ are all terms with $|I_1|=s-3$ that do \emph{not} 
involve a top-order differentiated Christoffel symbol, nor $K^{m-1}$.

The remaining terms in the RHS of \eqref{wave.coords2} are divided as follows: Those with $I_1=s-3$ involving a top-order 
differentiated  Christoffel symbol 
are placed in $\mathfrak{F}_{IBP}$; so are  all terms involving $K$ \emph{except} the terms involving $\partial^{I_1} [{\rm tr}K^{m-1}]e_0\gamma^m$. 
These  latter term are placed in ${\cal F}_{\rm Bord}$. 

Then with this grouping of terms, the bounds  (\ref{belbordterms}), \eqref{IBPterms},  (\ref{bordterms}) hold. 
\end{lemma}

\begin{proof}
We first single out some special terms which would be potentially problematic at the
 poles $\Th=0,\pi$; we will note that these will exhibit cancellation, entirely analogous to the one we encountered earlier. 
 
In particular, for the terms which contain factors 
$\partial^{I_1}\partial^{J_1}_\rho[(g^{m-1})^{AB}\frac{(\Gamma^{m-1})_{AB}^\Th}{{\rm sin}\Th}]\cdot \partial^{I_2}\partial^{J_2}_\rho ({\rm sin}\Th\cdot \partial_\Th \gamma^m)$
as well as $\partial^{I_1}\partial^{J_1}_\rho (g^{m-1})^{AB} \partial^{I_2}\partial^{J_2}_\rho [\partial_{AB}\gamma^m]$ there are  a few summands that require special care, which we single out here: 

The special cases are when $J_2=0$ (and thus $J_1=2$),  \emph{and} 
both $\partial_\rho$ 
hit the factor $(g^{m-1})^{AB}$. In that case, 
there is a very special pair of summands; first the  summand 
$[\partial^{2}_\rho(g^{m-1})^{\phi\phi}\partial^{I_1}[\frac{(\Gamma^{m-1})_{\phi\phi}^\Th}{{\rm sin}\Th}]\cdot \partial^{I_2}({\rm sin}\Th\cdot \partial_\Th \gamma^m)$, and 
secondly the summand $\partial^{I_1}\partial^{2}_\rho (g^{m-1})^{\Th\Th} \partial^{I_2}
[\partial_{\Th\Th}\gamma^m]$. These terms are special once the functions $\gamma^m$
 (without $e_0$ or $\partial_\rho$-derivatives) can be singular at the two poles. 

In both terms, we note that if we express:

\[
\gamma^m=(\gamma^m-{\rm log}{\rm sin}\theta)+{\rm log}{\rm sin}\theta, 
\gamma^{m-1}=(\gamma^{m-1}-{\rm log}{\rm sin}\theta)+{\rm log}{\rm sin}\theta
\]
then replacing  $\gamma^m$ in the second factor by ${\rm log}{\rm sin}\theta$, and 
the terms $(\gamma^{m-1})^{\phi\phi}(\Gamma^{m-1})^\Th_{\phi\phi}=\partial_\Th 
\gamma^{m-1}$, $\gamma^m$ by $\partial_\Th {\rm log}{\rm sin}\Th$,
 ${\rm log}{\rm sin}\Th$ we obtain terms which \emph{cancel out} (modulo terms that are
  not singular at the poles). From this point onwards, we consider the terms that are 
  \emph{left over} after this cancellation, notably:
  
  \beq
  \begin{split}
  \label{involved}
& \partial^{2}_\rho(\partial^{I_1}[\frac{(\partial_\Th\gamma^{m-1})}{{\rm sin}\Th}]\cdot 
\partial^{I_2}({\rm sin}\Th\cdot \partial_\Th [\gamma^m-{\rm log}{\rm sin}\theta])), 
\\&  \partial^{2}_\rho(\partial^{I_1}\frac{(\partial_\Th[\gamma^{m-1}-
{\rm log}{\rm sin}\theta)}{{\rm sin}\Th}]\cdot \partial^{I_2}({\rm sin}\Th\cdot
 \partial_\Th [\gamma^m]),  
  \partial^{I_1}\partial^{2}_\rho (g^{m-1})^{\Th\Th} \partial^{I_2}
\partial_{\Th\Th}[\gamma^m-{\rm log}{\rm sin}\theta]. 
\end{split}
\eeq

We proceed with these ``new, replaced'' terms, as well as all the others we have not
 considered. 
\medskip

We first show that all terms that we placed in $\mathfrak{F}_{\rm Gron}$ (along with those in \eqref{involved} just derived) 
 satisfy the bound \eqref{belbordterms}.
 
This follows by directly invoking all the bounds on the background geometry that are collected in the previous section, as well as the Cauchy-Schwarz and the 
product inequality straightforwardly applied. 
For all terms  other than these two singled out ones, recall that 
$|I_1|<s-3$; to control the terms as claimed,  we may need to apply the Hardy inequality to the first factor, as was done for the lower-order terms (this is because
of the singular behaviour of $\partial_\Th \gamma^m, \partial_\Th \gamma^{m-1}$
 at the poles). 
The restrictions we have imposed on $I_1$ for the terms that we placed in 
$\mathfrak{F}_{\rm Gron}$
imply that the resulting terms we obtain lie in spaces that have been already bounded, 
and satisfy the bounds in \eqref{belbordterms}. The setting where we do not obtain 
terms that have already been bounded are the two that we just singled out. 
In those settings,  the application of the regular Hardy inequality then suffices to obtain the desired bounds. 
\medskip

Let us consider the terms that we placed in  ${\cal F}_{\rm IBP}$. 
We commence with the most involved such terms, 
which will be treated using integrations by parts. The remaining terms can be treated by a similar integration 
by parts argument. The most involved terms are the first two in \eqref{involved}, as well as: 
\[
\int_\e^{\rho^{m-1}}\int_{\Sigma_\tau} \Phi^{m-1}\tau^{\frac{3}{2}+2}\opartial^I \partial_{\rho\rho}^2 [(g^{m-1})^{\Th\Th})\frac{(\Gamma^{m-1})_{\Th\Th}^\Th}{{\rm sin}\Th}]\cdot {\rm sin}\Th\partial_\Th\gamma^m \cdot e_0\opartial^I\partial_{\rho\rho}^2\gamma^m {\rm vol}_{\Sigma_\tau} d\tau
\]
The argument is essentially the same in all three cases, so we just consider the term right above. 

This summand 
requires an integration by parts,
since the Christoffel  terms are bounded with up to $|I|=s-4$ derivatives, as opposed to the 
$|I|=s-3$ we have here. 
We integrate by parts only in the derivatives $\opartial^I$ in the first factor, to derive, up to below-borderline terms: 

\beq
\begin{split}
&\int_\e^{\rho^{m-1}}\int_{\Sigma_\tau}\Phi^{m-1}\tau^{\frac{3}{2}+2}\opartial^{I}\partial_{\rho\rho}^2 [(g^{m-1})^{\Th\Th}\frac{(\Gamma^{m-1})_{\Th\Th}^\Th}{{\rm sin}\Th}] {\rm sin} \Th
\partial_\Th \gamma^m  
 \cdot e_0\opartial^I\partial_{\rho\rho}^2\gamma^m {\rm vol}_{\Sigma_\tau} d\tau
\\&=-\int_\e^{\rho^{m-1}}\int_{\Sigma_\tau} \Phi^{m-1}\tau^{\frac{3}{2}+2}\opartial^{I'} \partial_{\rho\rho}^2[(g^{m-1})^{\Th\Th}\frac{(\Gamma^{m-1})_{\Th\Th}^\Th}{{\rm sin}\Th} ]{\rm sin} \Th
\partial_\Th \gamma^m  \cdot \partial [e_0\opartial^I\partial_{\rho\rho}^2\gamma^m]  {\rm vol}_{\Sigma_\tau}d\tau
\\&-\int_\e^{\rho^{m-1}}\int_{\Sigma_\tau} \Phi^{m-1}\tau^{\frac{3}{2}+2}\opartial^{I'} \partial_{\rho\rho}^2[(g^{m-1})^{\Th\Th}\frac{(\Gamma^{m-1})_{\Th\Th}^\Th}{{\rm sin}\Th}]  \partial [{\rm sin} \Th
\partial_\Th \gamma^m ] \cdot[e_0\opartial^I\partial_{\rho\rho}^2\gamma^m]  {\rm vol}_{\Sigma_\tau}d\tau
\end{split}
\eeq
where $|I'|=s-4$. Note that the second term in the RHS can be bounded as required in \eqref{IBPterms}, 
so we may consider only the first term in the RHS. In that term, 
 the last factor in the RHS has an extra derivative $\partial$, from the integration by parts. 
Now we integrate by parts the derivative  $e_0$ in the second factor. We derive, up to below-borderline terms:\footnote{In particular, the flux terms across 
$\Sigma_\e, \Sigma_{\rho^{m-1}}$ which arise in the last integration by parts below
 will be below-borderline. } 

\beq\label{post.IBP}\bs
&\int_\e^{\rho^{m-1}}\Phi^{m-1}\int_{\Sigma_\tau} \Phi^{m-1}\tau^{\frac{3}{2}+2}\opartial^{I'} \partial_{\rho\rho}^2\frac{(g^{m-1})^{\Th\Th}(\Gamma^{m-1})_{\Th\Th}^\Th}{{\rm sin}\Th} [{\rm sin} \Th
\partial_\Th \gamma^m ]  \cdot \partial [e_0\opartial^I\partial_{\rho\rho}^2\gamma^m] 
{\rm vol}_{\Sigma_\tau} d\tau
\\&=\int_\e^{\rho^{m-1}}\int_{\Sigma_\tau}\Phi^{m-1}\tau^{\frac{3}{2}+2} e_0 \opartial^{I'} \partial_{\rho\rho}^2\frac{(g^{m-1})^{\Th\Th}
(\Gamma^{m-1})_{\Th\Th}^\Th}{{\rm sin}\Th} [{\rm sin} \Th
\partial_\Th \gamma^m ] 
\cdot  \opartial^{I''}\partial^2_{\rho\rho}\gamma^m d{\rm vol}_{\Sigma_\tau}d\tau 
\\&-\int_{\Sigma_{\rho^{m-1}}}\Phi^{m-1} \tau^{\frac{3}{2}+2}\opartial^{I'} \partial_{\rho\rho}^2\frac{(g^{m-1})^{\Th\Th}(\Gamma^{m-1})_{\Th\Th}^\Th}{{\rm sin}\Th} [{\rm sin} \Th
\partial_\Th \gamma^m ] \cdot  \opartial^{I''}\partial_{\rho\rho}^2\gamma^md{\rm vol}_{\Sigma_\tau}. 
\end{split}
\eeq
Here $|I'|=s-4$ and $|I''|=s-2$. Note in particular that the integrals 
$
\int_{\Sigma_\tau}|\opartial^{I''}\partial_{\rho\rho}^2\gamma^m| {\rm vol}_{\Sigma_\tau}$
 can be bounded by the top-order energy of $E[\opartial^I\partial_{\rho\rho}\gamma^m]$
 (where $|I=s-3$),
 times $\tau^{-\frac{1}{2}-\frac{1}{8}}$:
 
 \[
 \int_{\Sigma_\tau}|\opartial^{I''}_{A\dots}\partial_{\rho\rho}^2\gamma^m| 
 {\rm vol}_{\Sigma_\tau}\le  \int_{\Sigma_\tau}|(a^{m-1})_{Ai}\overline{e}^{m-1}_i\opartial^{I}\partial_{\rho\rho}^2\gamma^m| {\rm vol}_{\Sigma_\tau}\le 
 \tau^{\frac{3}{2}-\frac{1}{2}-\frac{1}{8}}\sqrt{E[\opartial^I\partial_{\rho\rho}\gamma^m]}.
 \] 
(We have used the pointwise bounds in Lemma \ref{a.oa.bds}). 
By this argument  we derive that 
 all 
 the terms in the RHS of \eqref{post.IBP} are below borderline and thus they can be 
 bounded 
  as in \eqref{belbordterms}. All the remaining terms in ${\cal F}_{\rm IBP}$ can
  also 
   be 
  treated in this way--integrating by parts 
one spatial derivative onto $e_0\partial^I\partial^2_{\rho\rho}\gamma^m_{\rm rest}$, 
followed by another integration by parts  in $e_0$ from the resulting factor. 
The resulting terms can be bounded as claimed in \eqref{IBPterms}.

\medskip

Next, we treat the borderline terms, as defined in Lemma \ref{lem:FGron}. 
These are when $|I|=|I_1|$, and all the derivatives in 
$\opartial^I$  hit 
the term ${\rm tr} K^{m-1}$. Thus these borderline  terms
as they appear in the bulk integral are: 
\[
\opartial^I \partial_\rho^{J_1}{\rm tr} K^{m-1}\cdot\partial^{J_2}_\rho e_0\gamma^m \cdot e_0\opartial^I
\partial_{\rho\rho}^2\gamma^m, 
\]
$J_1+J_2=2$, $|I|=s-3$.  Then all terms with $J_1=0,1,2$ all give rise to borderline terms.

We derive the estimates (using the standard volume element ${\rm sin}\theta d\theta dt$ on $\{\rho^{m-1}=\tau\}$: 

\beq
\begin{split}
&\|\opartial^I {{\rm tr}_{\overline{g}} K}^{m-1} \partial^2_{\rho\rho}
e_0\gamma^m\|_{L^2[\Sigma_{\rho^{m-1}=\tau}]}\le 
\|\partial^I[ {{\rm tr}_{\overline{g}} K}^{m-1} -{\rm tr}_{g^S}K^S]\|_{L^2[\Sigma_{\rho^{m-1}=\tau}]}
\cdot \| \partial^2_{\rho\rho}e_0\gamma^m\|_{L^\infty[\Sigma_{\rho^{m-1}=\tau}]}
\\&\le 5C\eta \tau^{-5-c}. 
\end{split}
\eeq
This then directly implies that those terms are bounded as claimed in 
\eqref{bordterms}. The exact same bounds hold for the other borderline terms: 

\beq
\begin{split}
&\|\opartial^I \partial_\rho {{\rm tr}_{\overline{g}} K}^{m-1} \partial_{\rho}
e_0\gamma^m\|_{L^2[\Sigma_{\rho^{m-1}=\tau}]}\le 
\|\partial_\rho \partial^I[ {\rm tr}_{\overline{g}} K^{m-1} -{\rm tr}_{g^S}K^S]\|_{L^2[\Sigma_{\rho^{m-1}=\tau}]}
\cdot \| \partial_{\rho}e_0\gamma^m\|_{L^\infty[\Sigma_{\rho^{m-1}=\tau}]}
\\&\le 5C\eta \tau^{-5-c}. 
\end{split}
\eeq

\beq
\begin{split}
&\|\opartial^I \partial^2_{\rho\rho} {{\rm tr}_{\overline{g}} K}^{m-1} 
e_0\gamma^m\|_{L^2[\Sigma_{\rho^{m-1}=\tau}]}\le 
\|\partial_\rho \partial^I[ {\rm tr}_{\overline{g}} K^{m-1} -{\rm tr}_{g^S}K^S]\|_{L^2[\Sigma_{\rho^{m-1}=\tau}]}
\cdot \| \partial_{\rho}e_0\gamma^m\|_{L^\infty[\Sigma_{\rho^{m-1}=\tau}]}
\\&\le 5C\eta \tau^{-5-c}. 
\end{split}
\eeq

We have then shown that all terms in the RHS of \eqref{wave.coords2}
fall under the categories 
$\mathfrak{F}_{\rm IBP}, \mathfrak{F}_{\rm Gron}, \mathfrak{F}_{\rm Bord}$, and derived
 the bounds we claimed on those. We thus derive Proposition
  \ref{gammamtopenineq}.
\end{proof}
\end{proof}

\subsection{Renormalized energy estimates at the low orders: Proof of \eqref{gammam-1exp}, \eqref{alpham-1assum},  \eqref{gamma1m-1enest}.}
 \label{sec:asympt_low}

The optimal energy bound {}{(\ref{inductiongammaopt})} for the lower derivatives of $\gamma^m_{\rm rest}$ yields a logarithmic upper {}{bound 
 for $\gamma^m_{\rm rest}$ itself at the low orders}. Indeed, integrating ${}{\partial_r\partial^I\gamma^m}$ over {}{$[\rho^{m-1},\epsilon]$, we have
\begin{align}\label{inte0gammam}
\partial^I\gamma^m_{\rm rest}(\rho^{m-1},t,\theta)=\partial^I\gamma^m_{\rm rest}(\epsilon,t,\theta)+\int^{\rho^{m-1}}_{\epsilon}\partial_\tau\partial^I\gamma^m_{\rm rest} d\tau.
\end{align}
Hence, we obtain the $L^\infty$ bounds (for $|I|\leq \mathrm{low}-2$)
\begin{align}\label{Linftygammam}
\notag&|\partial^I\gamma^m_{\rm rest}(\rho^{m-1},t,\theta)-\partial^I\gamma^m_{\rm rest}(\epsilon,t,\theta)|\\
\leq&\,\bigg|\int^{\epsilon}_{\rho^{m-1}}\partial_\tau\partial^I\gamma^m_{\rm rest} d\tau\bigg|
\leq+\int^{\epsilon}_{\rho^{m-1}}C\sqrt{\tau}(2M)^{-\frac{1}{2}}\|e_0\gamma^m_{\rm rest}\|_{H^{|I|+2}} d\tau\\
\notag\leq&\,C(2M)^{-\frac{1}{2}}\eta
\int^{\epsilon}_{\rho^{m-1}}\frac{1}{\tau}d\tau = 
C(2M)^{-\frac{1}{2}}\eta\log\frac{\epsilon}{\rho^{m-1}}
\end{align}
and the $L^2$ bounds
\begin{align}\label{L2gammam}
\|\gamma^m_{\rm rest}(\rho^{m-1},t,\theta)\|_{H^l}\leq 
\|\gamma^m_{\rm rest}(\epsilon,t,\theta)\|_{H^l}
+C(2M)^{-\frac{1}{2}}\eta\log\frac{\epsilon}{\rho^{m-1}}
\end{align}
for all $|I|=l\leq \mathrm{low}$.} 

However, in order to prove the leading order behaviour (\ref{gammam-1exp}){}{-\eqref{gamma1m-1enest}} for 
$\gamma^m_{\rm rest}$, we need to derive renormalised energy estimates for 
the variable $\frac{\gamma^m_{\rm rest}}{\log \rho^{m-1}}$. (Note that since $r=\rho^{m-1}$ for $r\le \e/2$ it follows readily that it suffices 
to prove \eqref{gammam-1exp}  with $r$ replaced by $\rho^{m-1}$). 
For the rest of this subsection we write $\rho$ instead of $\rho^{m-1}$, for 
brevity.  We derive a wave equation for this parameter; in calculating the RHS of this equation, we will be 
 using the already derived estimates for $\gamma^m_{\rm rest}$ and the inductive estimates for the metric, Christoffels, $K^{m-1}$, $r_*^{m-1}$, keeping only the leading order terms in explicit form. The other terms will be incorporated in $O_I(\rho^{-3+\frac{1}{4}})$, satisfying the relevant bound for up to $|I|$ spatial derivatives in $L^2$, since their exact form does not matter in the estimates below:
\begin{align}\label{boxgammam/logr}
\square_{^{m-1}g}\frac{\gamma^m_{\rm rest}}{\log {\rho}}=&\,2\nabla^a
\gamma^m_{\rm rest}\nabla_a\frac{1}{\log {\rho}}+\gamma^m_{\rm rest}\square_{^{m-1}g}\frac{1}{\log {\rho}}-\frac{\Box_{^{m-1}g}\gamma^S}{\log {\rho}}\\
\notag=& -2e_0\gamma^m_{\rm rest} e_0\frac{1}{\log {\rho}}-\gamma^m_{\rm rest}(e_0^2\frac{1}{\log {\rho}}+\mathrm{tr}_{^{m-1}\overline{g}}K^{m-1}e_0\frac{1}{\log {\rho}}\big)\\
\notag&{}{+O_I(\rho^{-3+1/4}) + \chi_{[\frac{\epsilon}{2},\frac{3\epsilon}{2}]}O_I(\epsilon^{-1})}\\
\notag=&-2e_0(\frac{\gamma^m_{\rm rest}}{\log {\rho}})\log {\rho}\, e_0\frac{1}{\log {\rho}}+2\frac{\gamma^m_{\rm rest}}{\log {\rho}}(\frac{2M}{\rho}-1)\frac{1}{|\log {\rho}|^2{\rho}^2}\\
\notag&-\frac{\gamma^m_{\rm rest}}{\log {\rho}}\log {\rho}\bigg[(\frac{2M}{r}-1)(\frac{2}{|\log {\rho}|^3{\rho}^2}+\frac{1}{|\log {\rho}|^2{\rho}^2})+\frac{M}{|\log {\rho}|^2{\rho}^3}\\
\notag&-\frac{3}{2}\frac{\sqrt{2M}}{\rho^\frac{3}{2}}(\frac{2M}{\rho}-1)^\frac{1}{2}\frac{1}{|\log {\rho}|^2{\rho}}
+\mathrm{tr}_{^{m-1}\overline{g}}u^{m-1}(\frac{2M}{\rho}-1)^\frac{1}{2}\frac{1}{|\log {\rho}|^2{\rho}}\bigg]\\
\notag&{}{+O_I(\rho^{-3+1/4}) }
\end{align}
{}{where in absorbing certain terms into $O_I(\rho^{-3+1/4})$
 we used Lemma \ref{lem:inhomog}. We also note that the term $\chi_{[\frac{\epsilon}{2},\frac{3\epsilon}{2}]}O_I(\epsilon^{-1})$ comes from replacing $r$ with $\rho$ and is way below borderline, cf. \eqref{rho} and the inductive assumptions on $r_*^{m-1}$ \eqref{r*.ind.claim.low}, \eqref{r*.ind.claim.high}. Hence, it can be incorporated in $O_I(\rho^{-3+\frac{1}{4}})$ as well.}

Notice that the most singular zeroth order terms in the previous RHS cancel. {}{We may then apply \eqref{Linftygammam}-\eqref{Linftybounds} to any $\gamma^m_{\rm rest}$ having less singular coefficients to incorporate them in the $O_I(\rho^{-3+1/4})$ terms: }
\begin{align}\label{boxgammam/logr2}
\square_{^{m-1}g}\frac{\gamma^m_{\rm rest}}{\log {\rho}}=&-2(\frac{2M}{{r}}-1)^\frac{1}{2}\frac{1}{{\rho}\log {\rho}}e_0(\frac{\gamma^m_{\rm rest}}{\log {\rho}}){}{+O_I(\rho^{-3+1/4})}
\end{align}
{}{keeping only the leading order first order term in explicit form. This term in fact yields a crucial cancellation in the energy estimates for the renormalised variable $\frac{\gamma^m_{\rm rest}}{\log\rho}$ below.}

Next we will commute the above equation with ${}{\partial^I}$, $|I|\leq s-3-4c$.
We obtain an equation on $\square_{^{m-1}g}\partial^I\frac{\gamma^m_{\rm rest}}{\log r}$; the  terms in the RHS are ${}{\partial^I}$ acting on the RHS of \eqref{boxgammam/logr2}
as well as the commutation terms generated by commuting ${}{\partial^I}$ with the wave 
operator $\square_{^{m-1}g}$. The latter have already been computed in section 
\ref{sec:Comm}. 
Combining with the lower-order estimates that we have already obtained on 
$\gamma^m_{\rm rest}$ we derive that those terms are bounded in 
$L^2_{[sin\Th d\Th dT]}[\Sigma_\rho]$
 by 
 \beq\label{bd,sim}
{}{ B \rho^{-3+\frac{1}{4}}}.
 \eeq
{}{In short,} 
letting $O_{L^2}({}{\rho^{-3+\frac{1}{4}}})$ stand for  a
general sum of terms that are bounded in $L^2_{[sin\Th d\Th dT]}[\Sigma_\rho]$ by 
\eqref{bd,sim}, we obtain: 
\begin{align}\label{boxpartiallgammam/logr}
 &\square_{^{m-1}g}{}{\partial^I}\frac{\gamma^m_{\rm rest}}{\log \rho}=
{}{-2(\frac{2M}{{\rho}}-1)^\frac{1}{2}\frac{1}{{\rho}\log {\rho}}e_0\partial^I(\frac{\gamma^m_{\rm rest}}{\log {\rho}})+ O_{L^2}(\rho^{-3+\frac{1}{4}})}.
\end{align}
In the next proposition we derive improved estimates for $\gamma^m_{\rm rest}$ that also confirm
 (\ref{gammam-1exp}), (\ref{alpham-1assum}), \eqref{gamma1m-1enest}.
\begin{proposition}\label{prop:gammam/logr}
The following renormalized estimate for $\frac{\gamma^m_{\rm rest}}{\log {}{\rho}}$ is valid:
\begin{align}\label{gammamrenest}
{}{\rho}^3|\log {}{\rho}|^4\sum_{|I|\leq s-3-4c}E[\partial^I\frac{\gamma^m_{\rm rest}}{\log {}{\rho}}][\Sigma_{{}{\rho}}]\leq C\eta,&&{}{\rho}\in (0,2\epsilon].
\end{align}
Moreover, $\gamma^m$ has the expansion:
\begin{align}\label{gammamexp}
\gamma^m=\alpha^m(t,\theta){}{\log \rho+\gamma^m_1(\rho},t,\theta),
\end{align}
where 
\begin{align}\label{alpham}
\alpha^m(t,\theta)-1\in H^{s-3-4c},&&\|\alpha^m-1\|_{H^{s-3-4c}}\leq C\eta ,
\end{align}
and $e_0\gamma^m_1$ satisfies the estimate
\begin{align}\label{gammam1enest}
\|e_0^{J_0}(\gamma^m_1)\|_{H^{s-3-4c}}\leq B{}{\rho}^{-\frac{3}{2}|J_0|+\frac{1}{4}},&&|J_0|=1,2,
\end{align}
for all ${}{\rho}\in(0,2\epsilon]$. 
\end{proposition}
\emph{Proof:}
Putting $v=\partial^I\frac{\gamma^m_{\rm rest}}{\log {}{\rho}},f(r({}{\rho}))={}{\rho}^\frac{3}{2}|\log {}{\rho}|^4$ in (\ref{Stokesv}) and utilising (\ref{Q0n}),(\ref{lapselike}),(\ref{volSigmarho}) we deduce that
\begin{align}\label{Stokespartiallgammam/logr}
&\int_{\Sigma_{{}{\rho}}}(1+O({}{\rho}^{1/4})){}{\rho}^3|\log {}{\rho}|^4\big[(e_0{}{\partial}^I\frac{\gamma^m_{\rm rest}}{\log {}{\rho}})^2+|\overline{\nabla}{}{\partial}^I\frac{\gamma^m_{\rm{}{rest}}}{\log {}{\rho}}\big]\mathrm{vol}_{Euc}\\
\notag&-\int_{\Sigma_{{}{\epsilon}}}(1+O(\e^{1/4}){}{\epsilon^3|\log\epsilon|^4}\big[(e_0{}{\partial}^I\frac{\gamma^m_{\rm rest}}{\log {}{\rho}})^2+|\overline{\nabla}{}{\partial}^I\frac{\gamma^m_{\rm rest}}{\log {}{\rho}}\big]\mathrm{vol}_{Euc}\\
\notag\leq&\int_{{}{\rho}}^{\epsilon}\frac{C}{\tau^{1-\frac{1}{4}}}\tau^3|\log\tau|^4E[{}{\partial}^I\frac{\gamma^m_{\rm rest}}{\log \tau}]d\tau-\int^\epsilon_{{}{\rho}}\int_{\Sigma_\tau}\Phi^{m-1}(\frac{2M}{\tau}-1)^\frac{1}{2}\frac{2}{\tau\log \tau}\tau^\frac{3}{2}|\log\tau|^4(e_0{}{\partial}^I\frac{\gamma^m_{\rm rest}}{\log\tau})^2 \mathrm{vol}_{\Sigma_\tau}d\tau\\
\notag&-\int_{{}{\rho}}^{\epsilon}\int_{\Sigma_\tau}\Phi^{m-1} \tau^\frac{3}{2}|\log\tau|^4e_0{}{\partial}^I\frac{\gamma^m_{\rm rest}}{\log \tau}\square_{^{m-1}g} {}{\partial}^I\frac{\gamma^m_{\rm rest}}{\log \tau}
\mathrm{vol}_{\Sigma_\tau}d\tau
\end{align}
Note that the second term in the RHS of (\ref{Stokespartiallgammam/logr}) has an unfavourable sign for an upper bound ({}{since $\log\tau<0$ when $\tau<1$}) and according to (\ref{lapselike}),(\ref{volSigmarho}), the coefficient of the 
$\tau^\frac{3}{2}|\log\tau|^4(e_0{}{\partial}^I\frac{\gamma^m_{\rm rest}}{\log\tau})^2$ is of the order $\tau^{-1}|\log\tau|^{-1}$ which fails to be integrable in $[0,\epsilon]$.
 Normally, this would prevent us from deriving a uniform Gronwall type energy estimate. However,  there is a crucial cancellation coming from 
  the RHS of (\ref{boxpartiallgammam/logr}) that will allow us to apply the Gronwall inequality and derive the claimed estimate:
\begin{align}\label{bulkrenest}
&\int_{\Sigma_{{}{\rho}}}(1+O({}{\rho}^{1/4})){}{\rho}^3|\log {}{\rho}|^4\big[(e_0{}{\partial}^I\frac{\gamma^m_{\rm rest}}{\log {}{\rho}})^2+|\overline{\nabla}{}{\partial}^I\frac{\gamma^m_{\rm{}{rest}}}{\log {}{\rho}}\big]\mathrm{vol}_{Euc}\\
\notag&-\int_{\Sigma_{{}{\epsilon}}}(1+O(\e^{1/4}){}{\epsilon^3|\log\epsilon|^4}\big[(e_0{}{\partial}^I\frac{\gamma^m_{\rm rest}}{\log {}{\rho}})^2+|\overline{\nabla}{}{\partial}^I\frac{\gamma^m_{\rm rest}}{\log {}{\rho}}\big]\mathrm{vol}_{Euc}\\
\notag\leq&\int_{{}{\rho}}^{\epsilon}\frac{B}{\tau^{1-\frac{1}{4}}}\tau^3|\log\tau|^4E[\partial^I\frac{\gamma^m_{\rm rest}}{\log\tau}]d\tau+{}{\epsilon^\frac{1}{8}}C^2\eta^2\\
\notag&-\int^\epsilon_{{}{\rho}}\int_{\Sigma_\tau}\Phi^{m-1}(\frac{2M}{\tau}-1)^\frac{1}{2}\frac{2}{\tau\log \tau}\tau^\frac{3}{2}|\log\tau|^4(e_0\partial^I\frac{\gamma^m_{\rm rest}}{\log\tau})^2 \mathrm{vol}_{\Sigma_\tau}d\tau\\
\notag&+\int_{{}{\rho}}^{\epsilon}\int_{\Sigma_\tau}\Phi^{m-1} \tau^\frac{3}{2}|\log\tau|^4e_0\partial^I\frac{\gamma^m_{\rm rest}}{\log \tau}
\bigg[2(\frac{2M}{\tau}-1)^\frac{1}{2}\frac{1}{\tau\log \tau }e_0\partial^I(\frac{\gamma^m_{\rm rest}}{\log r })
+O_{L^2}(\rho^{-3+\frac{1}{4}})\bigg]\mathrm{vol}_{\Sigma_\tau}d\tau\\
\notag\leq&\int_{{}{\rho}}^{\epsilon}\frac{B}{\tau^{1-\frac{1}{4}}}\sum_{|I_1|\leq |I|}\tau^3|\log\tau|^4E[\partial^{I_1}\frac{\gamma^m_{\rm rest}}{\log \tau}]d\tau+{}{\epsilon^\frac{1}{8}}C^2\eta^2
\end{align}
Thus, employing Lemma \ref{lem:Gron} and invoking our  smallness assumptions on   
$\epsilon>0$ relative to the other parameters, we arrive at (\ref{gammamrenest}).

The renormalised estimate (\ref{gammamrenest}) implies that the map 
$\frac{\gamma^{m}_{\rm rest}}{\log {}{\rho}}: \{0<{}{\rho}<2\epsilon\}\to H^{s-3-4c}(\Sigma_{{}{\rho}})$ is uniformly continuous:
\begin{align*}
&\big\|\frac{\gamma^{m}_{\rm rest}({}{\rho}_2,t,\theta)}{\log {}{\rho}_2}-\frac{\gamma^{m}_{\rm rest}({}{\rho}_1,t,\theta)}{\log {}{\rho}_1}\big\|_{H^{s-3-4c}}\\
=&\,\big\|\int^{{}{\rho}_2}_{{}{\rho}_1}\partial_\tau\frac{\gamma^{m}_{\rm rest}}{\log\tau}d\tau\big\|_{H^{s-3-4c}}
\tag{by Minkowski's integral inequality}\leq\bigg|\int^{{}{\rho}_2}_{{}{\rho}_1}\|\partial_\tau\frac{\gamma^{m}_{\rm rest}}{\log\tau}\|_{H^{s-3-4c}}d\tau\bigg|\\
\leq&\,\bigg|\int^{{}{\rho}_2}_{{}{\rho}_1}\frac{C\eta}{\tau(\log \tau)^2}d\tau\bigg|=C\eta\big|\frac{1}{\log {}{\rho}_2}-\frac{1}{\log {}{\rho}_1}\big|
\end{align*}
Hence, $\frac{\gamma^{m}_{\rm rest}({}{\rho},t,\theta)}{\log {}{\rho}}$ has a limit in $H^{s-3-4c}$, as ${}{\rho}\rightarrow0$, which we denote by $\alpha^m-1:=\alpha^m(t,\theta)-1\in H^{s-3-4c}$.
{}{From} the previous computations it also follows that:
\begin{align*}
&\|[\alpha^m-1](t,\theta)\|_{H^{s-3-4c}}\\
\leq&\, \|\frac{\gamma^{m}_{{}{\rm rest}}({}{\epsilon},t,\theta)}{\log {}{\epsilon}}\|_{H^{s-3-4c}}+\int^\epsilon_0B\sqrt{\tau}\|e_0\frac{\gamma^{m}_{\rm rest}}{\log\tau}\|_{H^{s-3-4c}}d\tau \\
\tag{by \eqref{gammamrenest}}\leq&\,\|\frac{\gamma^{m}_{\rm rest}({}{\epsilon},t,\theta)}{\log{}{\epsilon} }\|_{H^{s-3-4c}}+{}{\int^{{}{\epsilon}}_0\frac{C\eta}{\tau|\log\tau|^2}d\tau}
\leq\|\frac{\gamma^{m}_{\rm rest}({}{\epsilon},t,\theta)}{\log {}{\epsilon}}\|_{H^{s-3-4c}}+{}{\frac{C\eta}{|\log\epsilon|}}
\end{align*}
 {}{Thus, by virtue} of the initial data assumption on 
$\gamma^m_{\rm rest}$ 
we derive that: 
\beq
\label{alpha.ind.step}
\|\alpha^m(t,\theta)-1\|_{H^{s-3-4c}}\le C\eta. 
\eeq
as claimed in our inductive step \eqref{alpham-1assum}.  
\medskip

We next prove the inductive claim \eqref{gammam-1exp}, \eqref{gammam1enest} on the remainer term 
$\gamma^m_1$, which was defined via:
\begin{align*}\label{gamma1.def}
\gamma^m_1:=\gamma^m-\alpha^m\log r.
\end{align*}
Consider
\begin{align*}
&{}{\rho^3}\|e_0{\gamma}^m_1\|^2_{L^2}={}{\rho}^3\int_{\Sigma_{{}{\rho}}}[e_0(\gamma^m_{\rm rest}-(\alpha^m-1)\log r)]^2\mathrm{vol}_{Euc}\\
\notag =&\,{}{\rho}^3\int_{\Sigma_{{}{\rho}}}\big[e_0\big(\log {}{\rho}\cdot(\frac{\gamma^m_{\rm rest}}{\log {}{\rho}})-(\alpha^m-1)\log r)\big]^2\mathrm{vol}_{Euc}\\
\notag=&\,{}{\rho}^3\int_{\Sigma_{{}{\rho}}}|\log {}{\rho}|^2[e_0(\frac{\gamma^m_{\rm rest}}{\log {}{\rho}})]^2+(\frac{2M}{{}{\rho}^3}+O({}{\rho}^{-2}))
 \big[|\frac{\gamma^m_{\rm rest}}{\log {}{\rho}}|^2 +|\alpha^m-1|^2 -2(\alpha^m-1)\frac{\gamma^m_{\rm rest}}{\log {}{\rho}} \big]
\mathrm{vol}_{Euc}.
\end{align*}  
Therefore, by the limit
\[
\lim_{{}{\rho}\rightarrow 0} \|\frac{\gamma^m_{\rm rest}}{\log {}{\rho}}-( \alpha^m-1) \|_{L^2[\Sigma_{{}{\rho}}]}=0
\]
derived just above, we conclude that 
${}{\rho}^3\|e_0{\gamma}^m_1\|_{L^2}\to 0$ as ${}{\rho}\to 0$. 
\medskip

The above computation can be obviously iterated for 
$\partial^I\gamma^m_1$, yielding
\begin{align}\label{gammam1enlim2}
\lim_{{}{\rho}\rightarrow0}{}{\rho}^3\|
e_0\partial^I\gamma^m_1\|^2_{L^2}=0,&&|I|\leq s-3-4c.
\end{align}
The latter limit can be improved to a quantitative rate of decay with the 
use of the wave equation for $\gamma^m_{\rm rest}$ \eqref{boxpartialIgammam}, which we 
rewrite plugging in (\ref{gammamexp}) in the terms involving derivatives in the $\partial_r$ direction. We recall that by the estimates on $\gamma^{m-1}_{\rm rest}$ that we have derived, as well as the bounds on 
 the metric $g^{m-1}$ and its Christoffel symbols, 
 which imply (in bounding the lower derivatives of $\gamma^m_{\rm rest}$): 
 \[
 \| \Box^{\rm mixed}_{g^{m-1}}
\gamma^m_{\rm rest}\|_{H^{\rm low}}+\|\Box^{\rm spatial}_{^{m-1}\overline{g}}\gamma^m_{\rm rest}\|_{H^{\rm low}}+
 \|\Box_{g^{m-1}}\gamma^S\|_{H^{\rm low}}\le B {}{\tau^{-3+\frac{1}{4}}}.
 \]
Thus, plugging in \eqref{gammamexp} into the wave equation for $\gamma^m_{\rm rest}$ (see \eqref{wave.coords})
we find: 
\begin{align*}
&-e_0^2\gamma^m_1-\text{tr}_{^{m-1}\overline{g}}K^{m-1}
e_0\gamma^m_1={}{O_I(\rho^{-3+\frac{1}{4}})}\\
&\notag \implies \partial_{{}{\rho}}e_0\gamma^m_1+\big(\frac{3}{2}\frac{1}{{}{\rho}}+{}{O_I(\rho^{-1+\frac{1}{4}})}\big)e_0\gamma^m_1={}{O_I(\rho^{-3+3/4})}\\
&\notag\implies \partial_{{}{\rho}}({}{\rho}^\frac{3}{2}e_0\gamma^m_1)={}{O_I(\rho^{-\frac{3}{2}+\frac{3}{4}})}.
\end{align*}
{}{Integrating in $[\rho,\epsilon]$ and taking $H^{s-3-4c}$ norms, we obtain the desired estimate \eqref{gammam1enest} for $J_0=1$. The case $J_0=2$, follows from applying the $J_0=1$ estimate to the above equation.}

\subsection{The AVTD behaviour 
of $\gamma^m_{\rm rest}$ in the lower orders, via a descent scheme: The inductive step  \eqref{inductionHess}.}

We put down some consequences of the energy estimates we have derived on
 $\gamma^m_{\rm rest}$, proving  \eqref{inductionHess}. 
 The estimates have been claimed as part of the inductive step, and  capture the AVTD behaviour of the solution, 
 \emph{at the lower orders}; notably we show that the kinetic part of the 
 energy 
  of $\gamma^m_{\rm rest}$, 
  $\int_{t,\theta} |e_0(\gamma^m_{\rm rest})|^2$  (and the kinetic energy
   of the below-top order  derivatives of $\gamma^m$) \emph{dominates} 
 the potential part of the energy 
 $\int_{t,\theta} |e_1(\gamma^m)|^2+|e_2(\gamma^m)|^2$; their ratio is in 
 fact bounded by a strictly positive power of $r$. 
\medskip

This behaviour is used in an essential way in deriving the claimed 
(optimal) bounds for the behaviour of $K^m_{ij}(r,t,\theta)$ further down. 
We already derived in Lemma \ref{lem:Linftyest}
 the improved behaviour of the derivatives 
$e^{m-1}_i\partial^I\gamma^m_{\rm rest}$ relative to 
$e_0\partial^I\gamma^m_{\rm rest}$, at the lower orders and in the 
$L^\infty$ norm. The challenge now is to 
show the same improved estimates for the suitable combinations of 
$e^{m-1}_i\partial^I\gamma^m$, $i=1,2$ at all norms below the top.

This challenge is imperative in order 
 to derive the claimed inductive bounds in section \ref{sec:Kassns}: 
 To obtain the inductive step $m$ of those claims,  we 
will need to control the RHSs of equations \eqref{finredEVERic11it}
\eqref{finredEVERic12it}, \eqref{finredEVERic22it}
 in the suitable norms. 
These RHSs depend on quantities that have already been bounded at this 
point, notably the function $\gamma^m$ (that was just solved for) and the 
previous $(2+1)$-metric $h^{m-1}$.\footnote{Note that these derivatives appear in the RHSs of 
equations \eqref{finredEVERic11it}
\eqref{finredEVERic12it}, \eqref{finredEVERic22it}--it is essential to prove that the \emph{most singular 
terms} in those RHSs are the ones involving the time-derivatives $e_0$.} 
\medskip

However we do \emph{not} just use the energy estimates we have just 
obtained on $\gamma^m_{\rm rest}=\gamma^m-\gamma^S$ and its derivatives;
such an approach would \emph{not} see the claimed AVTD behaviour described above. 
And it would moreover \emph{not} allow us to derive the 
claimed bounds on the connection coefficients of $h^m$.
Instead, to capture the AVTD behaviour for $\gamma^m_{\rm rest}$ and its 
lower derivatives,  
we utilize a \emph{descent scheme} for the \emph{spatial} 
derivatives of the function $\gamma^m_{\rm rest}$. 
The descent scheme relies on  the following idea:

Given an order $k\in\mathbb{N}$ and the energy bounds 
we have derived on the energy of $\partial^I(\gamma^m-\gamma^S)$, the bounds 
\emph{do not} distinguish between the directions $e^{m-1}_0, e^{m-1}_1, e^{m-1}_2$ in the energy \eqref{Epsi}. However, at the same time, the bound on the energy of \emph{one higher} 
derivatives $\partial_t\partial^I(\gamma^m_{\rm rest})$, 
$\partial_\theta \partial^I(\gamma^m_{\rm rest})$ yields a much-improved bound 
for the $L^2$ norm of  
$\partial_t\partial^I(\gamma^m_{\rm rest}), \partial_\theta \partial^I(\gamma^m_{\rm rest})$, and in fact on the $L^2$ norm of 
$e^{m-1}_1\partial^I(\gamma^m_{\rm rest}), e^{m-1}_2 \partial^I(\gamma^m_{\rm rest})$. 

Then, if we use the expression \eqref{tthetatransebar} for the vector 
fields 
$\overline{e}^{m-1}_1, \overline{e}^{m-1}_2$ in terms of
 $\partial_t,\partial_\theta$ 
we can derive \emph{better} bounds for the $L^2$ norms 
of the quantities $\partial^I({e^{m-1}_i}\gamma^m_{\rm rest})$ than the ones implied 
by the energy estimates for $\partial^I(^m\gamma_{\rm rest})$. 

We utilize this strategy at all orders below the top  in the remainder of this subsection; the 
estimates derived 
there will be put to use in the next section  where we control the 
geometry of the metric $h^m$ via the Riccati system.

We note that  this descent scheme clearly does \emph{not} work at the top order. 
(Since there is no higher order from which we can descend). 
The estimates on the RHS of the Riccati equations are derived separately 
at the top order  later in the next  section, by a different argument which 
utilizes the specific algebraic structure of the RHSs of those 
equations in an essential way. 
 
\subsubsection{Control of forcing terms in the Riccati equations: The 
lower and higher  orders}
\label{sec:FT_lm}

Consider the RHSs of the Riccati equations \eqref{finredEVERic11it}
\eqref{finredEVERic12it}, \eqref{finredEVERic22it}
Let us consider the terms there that depend exclusively on 
$e_0$-derivatives of the function $\gamma^m$.  
In view of the inductive step  \eqref{gammam-1exp}, 
that  we have now derived, those terms satisfy the following estimates: 

\beq 
\label{e0gammalow}
||e_0(\gamma^m)+ (2M)^{1/2}r^{-3/2}
  \alpha^m(t,\theta)||_{H^{\rm low}}\le B r^{-3/2+1/4} 
\eeq
Moreover we note that \eqref{gammam1enest} can also be re-cast as: 
\beq
||e_0e_0(\gamma^m)+\frac{3}{2}(2M) r^{-3}
 \alpha^m(t,\theta)||_{H^{\rm low}}\le B r^{-3+1/4}.
\eeq

At the intermediate orders $k\in \{ low+1, \dots, s-3\}$, the energy estimate \eqref{inductiongammaopt}
yields: 
\beq
\label{e0gammamid}
||e_0(\gamma^m_{\rm rest})||_{H^{k}}\le C\eta r^{-3/2-\frac{k-low}{4}} 
\eeq

We also note for future reference that the energy estimates, in conjunction with the wave equation imply: 
\beq
\label{e0e0gammamid}
||e_0e_0(\gamma^m_{\rm rest})||_{H^{k}}\le C\eta r^{-3-\frac{k-low}{4}} 
\eeq. 
\medskip

While for the top order terms we recall that the inductive step on the top order energy of 
$\gamma^m_{\rm rest}$ that we have already derived 
 implies the 
estimates: 
\beq
\label{top.en.e0e0}
\sqrt{\int_{t,\theta} |\partial^{I} e^{m-1}_b e_0e_0(\gamma^m_{\rm rest})|^2}[\Sigma_{\rho^{m-1}=\tau}] \le C\eta \tau^{-\frac{3}{2}-3-c}, |I|=s-3, b=0,1,2.
\eeq
\medskip

We next consider the rest of the terms in the RHSs of \eqref{finredEVERic11it}
\eqref{finredEVERic12it}, \eqref{finredEVERic22it};
 these all depend on $e^{m-1}_1, e^{m-1}_2$-derivatives of $\gamma^m$.
 As explained, 
our goal  is to derive that these terms (below the top order) satisfy
\emph{better} bounds than the energy estimates we obtained for $\gamma^m$ would suggest:

In particular  the terms we seek to bound  are: 

\beq
\begin{split}
\label{Riccati.RHSs.again}
&\onabla^{m-1}_{22}(\gamma^m)+({e}_2^{m-1})(\gamma^{m-1})({e}_2^{m-1}
(\gamma^m)),\onabla^{m-1}_{11}(\gamma^m)+({e}_1^{m-1})(\gamma^{m-1})
({e}_1^{m-1})(\gamma^m))
\\&\onabla^{m-1}_{12}(\gamma^m)+\frac{1}{2}[(e^{m-1}_1\gamma^{m-1})\cdot 
(e^{m-1}_2\gamma^{m})+(e^{m-1}_2\gamma^{m-1})\cdot 
(e^{m-1}_1\gamma^{m})].
\end{split}
\eeq
We seek to bounds these expressions in the spaces $H^k, k\le s-4$ in this 
subsection, in particular proving the inductive step of \eqref{inductionHess}.
 The top order estimates (when $k=s-3$) are dealt with in the subsequent subsections. 
\medskip

We claim bounds for these quantities as follows: 

\begin{lemma}
\label{hessian.bounds.again}
At the optimal orders $H^k$ $k\le {\rm low}$, the terms in 
\eqref{Riccati.RHSs.again}
 are all bounded by $B^2 r^{-2-\frac{1}{2}+\frac{1}{8}}$. 
At the higher orders $k\in \{{\rm low}+1,\dots, s-5\}$  their $H^k$ norm is
bounded by $B^2 r^{-2-1/2-\frac{k-low}{4}+\frac{1}{8}}$; for $k=s-4$ the bound in $B^2r^{-3+\frac{1}{4}}$.  
\end{lemma}

\begin{remark}\label{improv.B}
We note that the first claim implies directly that the terms in \eqref{Riccati.RHSs.again}
 are all bounded by $C\eta r^{-2-\frac{1}{2}}$. This follows directly from \eqref{e.bd}. 
\end{remark}

\begin{proof}

We commence by re-casting the RHSs of the Riccati equations  by using 
derivatives with respect to frame elements $^{m-1}\overline{e}_i$, via formulas \eqref{eibar}. 
The reason for this is that we have formulas \eqref{tthetatransebar}
 to express 
these 
vector fields in terms of the coordinates $\partial_t,\partial_\theta$. 

In particular, we will derive the claimed bounds for the quantities 
\beq
\begin{split}
\label{Riccati.RHSs.alter}
&\onabla^{m-1}_{\overline{e}_2\overline{e}_2}(\gamma^m)+(\overline{e}^{m-1}_2)
(\gamma^{m-1})(\overline{e}^{m-1}_2
(\gamma^m)),\onabla^{m-1}_{\overline{e}_1\overline{e}_1}(\gamma^m)
+(\overline{e}^{m-1}_1)(\gamma^{m-1})
(\overline{e}^{m-1}_1)(\gamma^m))
\\&\onabla^{m-1}_{\overline{e}_1\overline{e}_2}(\gamma^m)+\frac{1}{2}
[(\overline{e}^{m-1}_1\gamma^m)\cdot (\overline{e}^{m-1}_2(\gamma^{m-1})+
(\overline{e}^{m-1}_1\gamma^{m-1})\cdot
 (\overline{e}^{m-1}_2(\gamma^{m})]
\end{split}
\eeq
(Here $\onabla$ is the connection intrinsic to the level sets of $\rho^{m-1}$).

Recall also the bounds on $e^{m-1}_1(r)$ 
\eqref{e1rasym}. Combining with \eqref{e0gammalow}, \eqref{e0gammamid}
we see that if we can prove the bounds for \eqref{Riccati.RHSs.alter} then the claimed 
bounds for 
\eqref{Riccati.RHSs.again} follow. So we next derive these bounds on the terms in  \eqref{Riccati.RHSs.alter}.

Using: 

\beq
\label{break-up}
\begin{split}
&||^{m-1}\onabla_{\overline{2}\overline{2}}(^m\gamma)+(^{m-1}\overline{e}_2)(^{m-1}\gamma)
(^{m-1}\overline{e}_2
(^m\gamma))||_{H^l}\le 
\|^{m-1}\onabla_{\overline{2}\overline{2}}(\gamma^S)+(^{m-1}\overline{e}_2)(^{m-1}\gamma)
(^{m-1}\overline{e}_2
(\gamma^S))||_{H^l}
\\&+
||^{m-1}\onabla_{\overline{2}\overline{2}}(^m\gamma-\gamma^S)+(^{m-1}\overline{e}_2)(^{m-1}\gamma)
(^{m-1}\overline{e}_2
(^m\gamma-\gamma^S))||_{H^l},
\end{split}
\eeq 
 it is clear that it suffices to bound the two terms on the RHS of the 
 above separately by the claim in our Lemma.  
 \medskip
 
 Let us commence by bounding the terms in the second line.
  We will invoke the bounds 
on $\gamma^m-\gamma^S=\gamma^m_{\rm rest}$ at the lower derivatives:

\beq
\begin{split}
&||^{m-1}\onabla_{22}(\gamma^m-\gamma^S)+(^{m-1}\overline{e}_2)
(^{m-1}\gamma)
(^{m-1}\overline{e}_2
(^m\gamma-\gamma^S))||_{H^l}
\\&\le \|^{m-1}\onabla_{22}(\gamma^m-\gamma^S)\|_{H^l}
+\|(^{m-1}\overline{e}_2)
(^{m-1}\gamma)
(^{m-1}\overline{e}_2
(^m\gamma-\gamma^S))\|_{H^l}
\end{split}
\eeq



We commence with the first term: We use formulas \eqref{tthetatransebar} 
to express $\overline{e}^{m-1}_2$  in terms of the coordinate vector fields 
$\partial_t,\partial_\theta$.
 We are thus reduced to bounding: 

\beq
\begin{split}
&\sum_{A,B=T,\Th}||[(a^{m-1})^{2A}\cdot (a^{m-1})^{2B}]\partial^2_{AB}
(^m\gamma-\gamma^S)||_{H^l}+
\\&\sum_{A,B=T,\Th; C=T,\Th}
\|[(a^{m-1})^{2A}\cdot (a^{m-1})^{2B}](\Gamma^{m-1})_{AB}^C\partial_C (\gamma^m-\gamma^S)\|_{H^l}
\end{split}
\eeq

Now, invoking  the bounds in Lemma \ref{lem:oa.bds}
 on 
$(a^{m-1})^{iT}, (a^{m-1})^{i\Theta}$ in $H^{\rm low}$, Lemma  \ref{Christ.bounds}
on the Christoffel symbols, 
as well as the estimates in Lemma  \ref{lem:gammam.bds} on the terms
 $\partial^I(\gamma^m-\gamma^S)$, as well as the product inequality, our desired bounds for this term follow. 
 \medskip
 
For the second term, we use: 

\beq
\begin{split}
\label{brk1}
&||(^{m-1}\overline{e}_2)(^{m-1}\gamma)
(^{m-1}\overline{e}_2
(^m\gamma-\gamma^S))||_{H^l}\le ||(^{m-1}\overline{e}_2)\cdot (^{m-1}\gamma-\gamma^S)
(^{m-1}\overline{e}_2
(^m\gamma-\gamma^S))||_{H^l}
\\&+||(^{m-1}\overline{e}_2)(\gamma^S)
(^{m-1}\overline{e}_2
(^m\gamma-\gamma^S))||_{H^l}
\le ||(^{m-1}\overline{e}_2)(^{m-1}\gamma-\gamma^S)
(^{m-1}\overline{e}_2
(^m\gamma-\gamma^S))||_{H^l}
\\&+||^{m-1}a_{2\Theta}cot(\theta)\cdot 
(^{m-1}\overline{e}_2
(^m\gamma-\gamma^S))||_{H^l}
\end{split}
\eeq 
The first term in the RHS of the above is controlled by the  product inequality, 
and by recalling the expression \eqref{tthetatransebar} 
for $\overline{e}^{m-1}_2$ in terms of the 
derivatives $\partial_t,\partial_\theta$ as above, to find:

\beq
\begin{split}\label{simple.decompo}
& ||(\overline{e}^{m-1}_2)(\gamma^{m-1}-\gamma^S)
(\overline{e}^{m-1}_2
(^m\gamma-\gamma^S))||_{H^l}
\\&\le 
 ||(\overline{e}^{m-1}_2)(^{m-1}\gamma-\gamma^S)||_{L^\infty}\cdot 
 ||\overline{e}^{m-1}_2
(^m\gamma-\gamma^S))||_{H^l}\le C\eta r^{-1-1/4} \cdot C\eta r^{-1-1/4}. 
\end{split}
\eeq

The last term in \eqref{brk1} 
can be controlled by invoking the  Hardy  and product inequalities  to derive: 
\beq
\begin{split}
\label{a.use}
&||(a^{m-1})^{2\theta}cot(\theta)\cdot 
(\overline{e}^{m-1}_2
(\gamma^m-\gamma^S))||_{H^l}
\\&\le C||(a^{m-1})^{2\theta}||_{H^{l}}\cdot 
||cot\theta\cdot (^{m-1}\overline{e}_2
(\gamma^m-\gamma^S))||_{L^\infty}+C||^{m-1}a_{2\theta}||_{L^\infty}
\cdot 
||(\overline{e}^{m-1}_2
(\gamma^m-\gamma^S))||_{H^{l+1}}
\\&\le 2CB\eta r^{-1-1/8}\cdot Cr^{-1-1/8},
\end{split}
\eeq
or if $l={\rm low}$ the power in the RHS has an extra power $-\frac{1}{4}$. 
 And for each order beyond $\rm low$ the bound worsens   by $r^{-1/4}$. 
Thus combining the two previous estimates, we derive our claimed bound 
for the terms in the second line   of \eqref{break-up}. 
\medskip

We now bound the first term in the RHS of \eqref{break-up} in a similar 
 manner, again expressing the vector fields $^{m-1}\overline{e}_i$ 
 in terms of the 
 coordinate vector fields $\partial_t,\partial_\theta$ and using the 
 inductive bounds on
  $(a^{m-1})^{i t},(a^{m-1})^{i \theta}$. 
  
  Again, noting that the $log \rho$ 
  term  is cancelled by 
 the derivatives $\overline{e}_i$ we are taking, we find that
   we can expand out the covariant derivative to find: 
  \beq
  \label{expand.cov.sec}
  \begin{split}
 & ^{m-1}\overline{\nabla}_{22} \gamma^S- (\overline{e}^{m-1}_2
 ( \gamma^S))^2
  \\&= ((a^{m-1})^{2\theta})^2 
  [(csc\theta)^2-(cot\theta)^2]+
  ((a^{m-1})^{2\zeta})^2\partial_\theta(a^{m-1})^{2\theta})cot\theta-\sum_{A,B=t,\theta}(a^{m-1})^{2\vartheta})
  (\Gamma^{m-1})_{AB}^\Theta cot\theta).
\end{split} 
  \eeq
Note that $(a^{m-1})^{2\theta}\partial_\theta(a^{m-1})^{2\theta}{\rm cot}\theta$ is bounded by 
$\overline{e}^{m-1}_2( a^{m-1})^{2\theta}{\rm cot}\theta$ 
  in all energy norms $H^k, k\le s-4$.  Thus, invoking the 
inductive assumptions and the product inequality we derive our desired bounds.

\medskip

The terms $||^{m-1}\onabla_{11}(^m\gamma)+(^{m-1}e_1)(^{m-1}\gamma)
(^{m-1}
e_1(^m\gamma)))||_{H^l}$, 
$ ||^{m-1}\onabla_{12}(^m\gamma)+(^{m-1}e_1)(^{m-1}\gamma)
(^{m-1}
e_2(^m\gamma)))||_{H^l}$  are controlled in an analogous (in fact simpler) manner
in all cases \emph{except} at the order $|I|=s-4$ where the bounds on the factor 
\[
\partial^{s-4}_{TT\dots T}[\nabla^{m-1}_{11}\gamma^m-e^{m-1}_1\gamma^{m-1}e^{m-1}_1\gamma^m]
\]
require a special note, due to the \emph{lack} of a bound on the 
norm of $E[\partial^{s-3}_{TT\dots T}\gamma^m]$ at the very top order. 

In this case, we instead we use the wave equation on $\gamma^m$ to \emph{re-express}
the terms above in terms of other derivatives which we \emph{can} control; in
 particular we replace the RHS by 
 
 \[\partial^{s-4}_{TT\dots T}[\overline{\nabla}^{m-1}_{22}\gamma^m
 +\overline{\nabla}_2\gamma^{m-1}\onabla_2\gamma^m +{\rm tr}K^{m-1}e_0\gamma^m]
\] 
The RHS of the above is them bounded as claimed, in view of the bounds we have already
 derived on the first (spatial) derivatives in the $e^{m-1}_2$-directions, as well as 
 our 
the inductive assumptions  on the $K^{m-1}_{ii}$ coefficients of the previous step.  
\medskip

Other than this special case, all other lower-order derivatives follow by the 
argument we presented for $K^m_{22}$. 

This concludes the proof of our Lemma. 
\end{proof}







\subsection{The estimates for $\gamma^m$ re-cast on level sets of $r$.}
We make a small extension of our heretofore derived results, in preparation for our analysis 
of the Riccati system in the next section. 

The inductive claim we have verified proves estimates for $\gamma^m_{\rm rest}$ on 
level sets of $\rho^{m-1}$. We note also that at the top order estimates, 
the vector fields $e^{m-1}_b, b=0,1,2$ are also involved. 
\medskip

Our aim is to prove:

\begin{proposition}
\label{thm_r}
The inductive steps  that we have derived also hold verbatim on 
level sets of $r$, with respect to the coordinate vector fields 
$\partial_t,\partial_\theta$ defined with respect to the coordinates
 $\{r,t,\theta\}$. 
  The only difference is that the constant on the RHS 
 will be multiplied by a factor of $\frac{9}{8}$. 
\end{proposition}
\begin{proof}

\newcommand{\opa}{\overline{\partial}}
For the purposes of this proof let us denote by 
$\partial_t,\partial_\theta$ the \emph{previous} vector fields defined
 with respect to 
the coordinate system $\{\rho^{m-1},t,\theta\}$ and by
 $\opa_t,\opa_\theta$ the ones defined with respect to the coordinate
  system $\{r,t,\theta\}$. 

 We recall that $\rho^{m-1}=r$ for $r\le \frac{\e}{2}$, and thus 
 $\opa_t=\partial_t$ and $\opa_\theta=\partial_\theta$ for 
 $r\in (0,\e/2)$. 
  Thus our 
 argument will be to \emph{commence} our estimates on $\{r=\e/2\}$ 
 and solve \emph{backwards}, until $r=\frac{3\e}{2}$. 
 
 We then just need to use the already-derived estimate with respect to 
 the vector fields $\partial_t,\partial_\theta$. Coupled with the 
 expressions \eqref{eibar}, \eqref{tthetatransebar}  we can express  
 $\opa_t,\opa_\theta$ in terms of 
 $e_0, \partial_t,\partial_\theta$ to derive the \emph{same} estimates 
 qualitatively, with the vector fields $\partial$ replaced by vector 
 fields $\opa$. We then note that via the already-derived energy estimates for $\gamma^m_{\rm rest}$ at the different orders,  we also have \emph{bulk estimates}
 for quantities: 
 
 \[
 \int^{\rho^{m-1}=\e_2}_{\rho^{m-1}=\e_1} 
 \tau^{\frac{3}{2}|J_0|}|  e^{m-1}_i[\partial^I
(e_0)^{J_0}\gamma^m_{\rm rest}](\tau)|^2sin\theta d\theta dtd\tau\le
  \frac{C^2\eta^2}{3}[(\e_1)^{-2-h(|I|)}-(\e_2)^{-2-h(|I|)}]. 
 \]
 Here $h(|I|)=0$ for $|I|\le s-3-4c$, and 
 $h(|I|)= \frac{|I|-(s-3-4c)}{4}$. 
 These also imply the same qualitative estimates: 
 
 \[
 \int^{r=\e_2}_{r=\e_1} 
 \tau^{\frac{3}{2}|J_0|}| e_i^{m-1}[\opa^I
 (e_0)^{J_0}\gamma^m_{\rm rest}](\tau)|^2 sin\theta d\theta dtd\tau\le\frac{9}{8}
 \cdot 
  \frac{C^2\eta^2}{3}[(\e_1)^{-2-h(|I|)}-(\e_2)^{-2-h(|I|)}]. 
 \]
 Therefore utilising the energy estimates we derived above at all orders,  
 across the surfaces $\Sigma_r$ our claim follows in a straightforward 
 (simpler) way, via the Gronwall inequality, where all commutation terms 
 have now already been controlled. 
 \end{proof}

\section{The estimates for the next metric iterate $^mh$.}\label{sec:iterh}

As explained in the introduction, the next step in the induction is
 to construct the next iterate 
of the metric 
$^mh$; this involves all the relevant connection coefficients 
$K^m_{ij}(r,t,\theta)$, and coordinate-to-frame coefficients $a_{Ai}^m(r,t,\theta)$.
 It also involves determining the 
initial data 
hypersurface $\Sigma_{{r^m_*}}$  
 on which the prescribed initial data are to be induced, via the function ${r^m_*}(t,\theta)$, 
 and also the component $\tilde{K}^m_{12}(t,\theta)$ of the second 
 fundamental form
 on that hypersurface, 
 which captures the rotation angle between the fixed background canonical 
 frame $\E_1, \E_2$, and the
  frame $\te^m_1,\te^m_2$ on $\Sigma_{{r^m_*}}$. 

This section is split into three parts. In the first we solve for the variables 
$K^m_{22}(r,t,\theta), K^m_{12}(r,t,\theta)$. These  
 solve the system of equations \eqref{finredEVERic22it}, 
 \eqref{finredEVERic12it}, with the prescribed initial conditions \eqref{init.cond.m}
 \emph{at the singularity} $\{r=0\}$.
At a second step we use  the above solution (for all $r\in (0,2\e)$ to solve for the
two functions ${r^m_*}(t,\theta), 
\tilde{K}^m_{12}(t,\theta)$. These 
 are
 chosen so as to solve 
 the system \eqref{key.link2}, \eqref{key.link2'} and we 
  derive the claimed inductive  estimates for these parameters. In the third part we 
solve for the remaining variables $K^m_{11}(r,t,\theta)$, $a^m_{Ai}(r,t,\theta)$ 
with initial data suitably defined on $\Sigma_{r^m_*}$.
\newline

The next subsection commences the first part:

\subsection{Energy estimates for $K^m$: Proof of  \eqref{K22heur}, \eqref{K12heur}, 
\eqref{inductiontrKtopmixed}, \eqref{inductiontrKlow},
\eqref{inductiontrKopt} for $K^m_{12}(r,t,\theta), K^m_{22}(r,t,\theta)$}\label{backforwestKm}
As stated in our inductive claim, our desired estimates for all the parameters are to hold on both level sets of $r$ and level sets of 
$\rho^{m-1}$. 
In fact these two 
parameters agree for $r\le \e/2, r\ge 3\e/2$ and are comparable in the in-between region. Since we are dealing with transport equations the transition 
from one estimate to the other is straightforward. For completeness, we prove the claim for the $r$-level sets for $K^m_{12}, K^m_{22}$ and for the 
$\rho^{m-1}$-parameters for all the other quantities. A straightforward adaptation of the equations in either of the two situations yields the claim for the 
other level sets. (The only difference in the equations is the introduction of a multiplicative factor involving $\chi(r), r^{m-1}_*$ which satisfies uniform bounds
in all the relevant spaces). 

\subsubsection{The functions $K^m_{22}(r,t,\theta), K^m_{12}(r,t,\theta)$
and their low derivatives as integrals from the singularity.}

We now prove the existence of solutions $K^m_{22}(r,t,\theta), K^m_{12}(r,t,\theta)$ to the equations 
(\ref{finredEVERic22it}), (\ref{finredEVERic12it}), deriving  that they 
verify the 
 inductive assumptions \eqref{K22heur}, \eqref{K12heur}
(\ref{inductiontrKtopmixed}), \eqref{inductiontrKlow},
(\ref{inductiontrKopt}), \eqref{inductiontrKtopmixed.sing} for the  $m$-th step in the iteration. 
The benefit of having closed the energy estimates for $\gamma^m$ in the previous 
subsection is that we may treat (\ref{finredEVERic11it}), (\ref{finredEVERic12it}), (\ref{finredEVERic22it}) as 
ODEs for $K^m_{ij}$, decoupled from the rest of the variables; each of these ODEs  we can 
solve either forwards or backwards.
\medskip

 We recall that the parameters $K^m_{22}, K^m_{12}(r,t,\theta)$ satisfy: 
 (All $e_i$ below are short for $e^{m-1}_i$). 
 
 \begin{align}
 \label{K22m.again}
 e_0K^m_{22}+(K^m_{22})^2-(K^{m-1}_{12})^2+e_0\gamma K^m_{22}=&\,
 \overline{\nabla}^{m-1}_{22}(\gamma^m)+(e_2(\gamma^m))(e_2(\gamma^{m-1})-
e^2_0(\gamma^m)-(e_0(\gamma^m))^2\\
\label{K12m.again}
 e_0K^m_{12}+(2K^m_{22}+e_0(^m\gamma))K^m_{12}=&\,
 \overline{\nabla}^{m-1}_{12}(\gamma^m)+\frac{1}{2}
 [e_1(\gamma^{m-1}) \cdot  e_2(\gamma^m)+
 e_1(\gamma^m) \cdot e_2(\gamma^{m-1})]
 \end{align}

We recall that these equations are to be  solved backwards from the singularity
for $r\in (0,2\epsilon)$; in the case of $K^m_{22}$ the requirement is that the solution should be smooth and negative (at least 
initially close to $r=0$).  
  For $K^m_{12}$  the requirement is that the solution should vanish to order 
 $o(r^{2d^m_2(t,\theta)-\alpha^m(t,\theta)})$.

We will now prove that $K^m_{22}(r,t,\theta)$ and $K^m_{12}(r,t,\theta)$ satisfy the 
following expansions in $r$  (both in the $L^\infty_{t,\theta}$ and $H^{\rm low}_{t,\theta}$ norms): 
\begin{align}\label{Kmijexp}
& K_{22}^m(r,t,\theta):=\frac{d_2^m(t,\theta)\sqrt{2M}}{r^\frac{3}{2}}+u_{22}^m(r,t,\theta),\qquad
 K_{12}^m(r,t,\theta)=:u_{12}^m(r,t,\theta),
 \\ \notag\text{where}\text{ }\space\space  d_2^m(t,\theta):=&\,\frac{\alpha^m(t,\theta)-\frac{3}{2}
-\sqrt{(\alpha^m(t,\theta)-\frac{3}{2})^2+6\alpha^m(t,\theta)-4|\alpha^m(t,\theta)|^2}}{2}.
\end{align}
  The 
functions $u^m_{22}(r,t,\theta), u^m_{12}(r,t,\theta)$ are
claimed to be (in the inductive step \eqref{K22heur}, \eqref{K12heur}) lower-order corrections (in terms of behaviour in $r$),
 as $r\rightarrow0$. We arrived at this  coefficient $d^m_2(t,\theta)$ by solving for the 
unique leading-order formal solution of the ODE (\ref{finredEVERic22it}).
Here invoking Proposition \ref{prop:gammam/logr} (where we derived control on $\alpha^m$), we find: 
 $\|d_2^m(t,\theta)+1\|_{L^\infty}\leq DC\eta\le \frac{1}{8}$.
 
Our goal is to derive the claimed estimates $Br^{-1-1/4}$  on 
$u^m_{12}(r,t,\theta), u^m_{22}(r,t,\theta)$ at the lower orders ($H^l, l\le \rm low$)
 as well as 
in $L^\infty$. This will verify the claims \eqref{inductiontrKopt}, 
 for 
$K^m_{12}(r,t,\theta), K^m_{22}(r,t,\theta)$.

 We obtain the higher order estimates on $K^m_{22}(r,t,\theta)$ and
  $K^m_{12}(r,t,\theta)$ 
 in the next subsection. 
\medskip

{\it Notation}: As in the previous subsection, unless otherwise stated, we will use the symbols $\nabla,\overline{\nabla}$ to denote the covariant 
derivatives intrinsic to ${h^{m-1}}$ and the normal space to $e_0$ in ${h^{m-1}}$
 respectively. Moreover the notation $O(r^b), O_\eta(r^b)$ will be used to denote a term bounded 
by $Br^b, B\eta r^b$ (where $B$ is the universal fixed constant we use  throughout). The norms in which these bounds will be assumed to 
hold will be clear from the equation where they appear--unless stated otherwise they will be in the same norm as the LHS of the relevant equation. 

\begin{remark}
 A note is in order on the level sets where we will be deriving our estimates:
  In this subsection, we will be deriving our claims on the parameters $K^m_{22}, K^m_{12}$  on level sets of the function 
 $r$. Once the function $r_*^m(t,\theta)$ has been solved for further down, we will remark how the exact same estimates hold on level sets of the
  function $\rho^m$ which is built out of $r^m_*$. The latter
step will complete  our inductive claim for these two parameters.  
\end{remark}

\subsubsection{ Asymptotic expansion of $K^m_{12}(r,t,\theta)$, $K^m_{22}(r,t,\theta)$
at the lower orders.}\label{subsec:K12K22asym}
Recall that (\ref{K33heur}) can be re-expressed as: 
\[
e_0\gamma^m=-\alpha^m(t,\theta) \frac{\sqrt{2M}}{r^{3/2}}+u^m_{33}(r,t,\theta). 
\]

Using this notation, and 
substituting the expression (\ref{Kmijexp}) in the LHS of (\ref{finredEVERic12it})-(\ref{finredEVERic22it}), as well as 
\eqref{gammamexp}, \eqref{gammam1enest}, (\ref{u33m-1enest}) 
(the latter is valid for $\gamma^m$ thanks to Proposition \ref{prop:gammam/logr}) 
 in the 
$e_0^2\gamma^m,e_0\gamma^m$ terms in the RHS we derive the equivalent system:
\begin{align}
\label{u22mRicit}&e_0u_{22}^m+(2d_2^m(t,\theta)-\alpha^m(t,\theta))\frac{\sqrt{2M}}{r^\frac{3}{2}}u_{22}^m+(u_{22}^m)^2-
(u_{12}^{m-1})^2+u_{33}^mu_{22}^m\\
\notag=&\,\overline{\nabla}_{22}\gamma^m+(\overline{\nabla}_2\gamma^m)
(\overline{\nabla}_2\gamma^{m-1})-
e_0u_{33}^m-(u_{33}^m)^2+(2\alpha^m-d_2^m)\frac{\sqrt{2M}}{r^\frac{3}{2}}u_{33}^m+O(r^{-2}),\\
\label{u12mRicit}&e_0u_{12}^m+(2d_2^m-
\alpha^m)\frac{\sqrt{2M}}{r^\frac{3}{2}}u_{12}^m+(2u_{22}^m+u_{33}^m)\,u_{12}^m
=\overline{\nabla}_{12}\gamma^m+\frac{1}{2}[\overline{\nabla}_1\gamma^{m-1} \overline{\nabla}_2\gamma^m+\overline{\nabla}_2\gamma^{m-1} \overline{\nabla}_1\gamma^m].
\end{align}
(The $O(r^{-2})$ in the first equation depends on $\alpha^m(t,\theta)$).  

We recall also that by the now-derived inductive step  \ref{hessian.bounds.again}
we have that the expressions in the RHSs of the above are  all
bounded in $H^k$, $k\le \rm low$ 
and in $L^\infty$ 
by $Br^{-1-\frac{1}{4}}$, $C_{\rm Sob}Br^{-1-\frac{1}{4}}$. 
\medskip

This then allows us to solve these two  equations as a 
de-coupled system of a non-linear and a linear ODE.

\begin{proposition}
\label{prop:umenest}
There exists a unique smooth solution 
$u^m_{22}(r,t,\theta), u^m_{12}(r,t,\theta)$ to \eqref{u22mRicit},
 \eqref{u12mRicit}, with the additional requirement for 
 $u^m_{12}(r,t,\theta)$ that as $r\to 0$ we have the
 bound 
$ u^m_{12}(r,t,\theta)=o(r^{2d_2^m-\alpha^m})$. 
These unique solutions satisfy the bounds for all $r\in(0,2\epsilon]$:

\begin{align}\label{umenest}
\sum_{(i,j)=(1,2), (2,2)}[\| u_{ij}^m\|_{L^\infty}\leq  Br^{-1-\frac{1}{4}}.
\end{align}

\end{proposition}
The conclusion of the previous proposition validates the inductive assumption (\ref{inductiontrKopt}) for $u^m_{ij}$, for $(i,j)=(1,2)$ and $(i,j)=(2,2)$.  
\begin{proof}
Rewrite the system (\ref{u22mRicit})-(\ref{u12mRicit}) in the form
\begin{align}
\label{partial_ru22m}\partial_r(r^{\alpha^m-2d_2^m}u^m_{22})=&-(\frac{2M}{r}-1)^{-\frac{1}{2}}r^{\alpha^m-2d_2^m}\bigg[\overline{\nabla}_{22}\gamma^m+(\overline{\nabla}_2\gamma^m)^2
-e_0u^m_{33}-(u_{33}^m)^2+O(1)u^m_{22},\\
\notag&+(2\alpha^m-d_2^m)\frac{\sqrt{2M}}{r^\frac{3}{2}}u^m_{33}
+\alpha^mO(\frac{1}{r^2})-(u_{22}^m)^2+(u_{12}^{m-1})^2-u_{33}^mu_{22}^m\bigg],\\
\label{partial_ru12m}\partial_r(r^{\alpha^m-2d_2^m}u_{12}^m)=&
-(\frac{2M}{r}-1)^{-\frac{1}{2}}r^{\alpha^m-2d_2^m}\bigg[\overline{\nabla}_{12}\gamma^m+\overline{\nabla}_1
\gamma^m \overline{\nabla}_2\gamma^m-(2u_{22}^m+u^m_{33})\,u_{12}^m\\
\notag &-\alpha^mO(\frac{1}{r^\frac{1}{2}}) u_{12}^m\bigg].
\end{align}
We proceed by integrating  (\ref{partial_ru22m}),(\ref{partial_ru12m}) over $[0,r]$ for any $r\in [0,2\e]$, 
imposing the conditions\footnote{Note that these conditions at $r=0$ are verified by a function satisfying (\ref{umenest}).}
\begin{align}\label{u22mu12minit}
\lim_{r\rightarrow0}r^{\alpha^m-2d_2^m}u^m_{22}=\lim_{r\rightarrow0}
r^{\alpha^m-2d_2^m}u_{12}^m=0,
\end{align}
to obtain:
\begin{align}
\label{partial_ru22m.int}r^{\alpha^m-2d_2^m}u^m_{22}=&-\int^r_0(\frac{2M}{\tau}-1)^{-\frac{1}{2}}\tau^{\alpha^m-2d_2^m}\bigg[\overline{\nabla}_{22}\gamma^m+(\overline{\nabla}_2\gamma^m)
(\overline{\nabla}_2\gamma^{m-1})
-e_0u^m_{33}-(u_{33}^m)^2+O(1)u^m_{22},\\
\notag&+(2\alpha^m-d_2^m)\frac{\sqrt{2M}}{\tau^\frac{3}{2}}u^m_{33}
+\alpha^mO(\frac{1}{\tau^2})-(u_{22}^m)^2-(u_{12}^{m-1})^2-u_{33}^mu_{22}^m\bigg]d\tau\\
\label{partial_ru12m.int}r^{\alpha^m-2d_2^m}u_{12}^m=&-\int^r_0(\frac{2M}{\tau}-1)^{-\frac{1}{2}}\tau^{\alpha^m-2d_2^m}\bigg[\overline{\nabla}_{12}\gamma^m+\frac{1}{2}(\overline{\nabla}_1\gamma^{m-1} \overline{\nabla}_2\gamma^m+\overline{\nabla}_1\gamma^{m} \overline{\nabla}_2\gamma^{m-1})-\mathrm{tr}_{^m\overline{g}}u^m\,u_{12}^m\\
\notag&-\alpha^mO(\frac{1}{\tau^\frac{1}{2}})u_{12}^m\bigg]d\tau
\end{align}
Utilising the already derived  inductive step of \eqref{inductionHess}  for 
the RHS we infer that
\begin{align}
\label{u22mLinftyineq} u^m_{22}=&\,\frac{1}{r^{\alpha^m-2d_2^m}}
\int^r_0\tau^{\alpha^m-2d_2^m} O(\frac{1}{\tau^{1-\frac{1}{4}}})u_{22}^m + 
\tau^{\alpha^m-2d_2^m}  \tau^{-3+\frac{1}{4}}d\tau\\
\notag&+\frac{1}{r^{\alpha^m-2d_2^m}}\int^r_0(\frac{2M}{\tau}-1)^{-\frac{1}{2}}
\tau^{\alpha^m-2d_2^m}\big[(u^m_{22})^2+(u^{m-1}_{12})^2\big]d\tau\\
\label{u12mLinftyineq} u^m_{12}=&\,
\frac{1}{r^{\alpha^m-2d_2^m}}
\int^r_0\tau^{\alpha^m-2d_2^m} O(\frac{1}{\tau^{1-\frac{1}{4}}})u_{12}^{m-1}
 + 
\tau^{\alpha^m-2d_2^m}  \tau^{-3+\frac{1}{4}} d\tau\\
\notag&+\frac{1}{r^{\alpha^m-2d_2^m}}\int^r_0(\frac{2M}{\tau}-1)^{-\frac{1}{2}}\tau^{\alpha^m-2d_2^m}
(2u^m_{33}+u^m_{33})u_{12}^m d\tau.
\end{align}

A standard Picard iteration argument of iterating linear equations with the prescribed 
behaviour at 
at $r=0$ (\ref{u22mu12minit}) for $u^m_{12},u^m_{22}$ furnishes\footnote{Note that since $\alpha^m-2d_2^m\ge 3-\frac{1}{4}$, in view of the bounds we 
are inductively assuming or deriving  for the terms inside the integrals, 
these integrals are manifestly convergent. } a continuous solution to 
(\ref{u22mRicit})-(\ref{u12mRicit}) satisfying the pointwise bounds: 
\begin{align}\label{umLinftyest}
|u^m_{12}(r,t,\theta)|, |u^m_{22}(r,t,\theta)|\le Br^{-1-\frac{1}{4}}
\end{align}
for $r\in(0,2\epsilon]$.

\end{proof}

\begin{remark}
For the function $u^m_{12}(r,t,\theta)$ the condition imposed in \eqref{u22mu12minit} is a standard initial 
condition to be imposed on a linear first order  ODE. However for $u^m_{22}(r,t,\theta)$ the vanishing imposed in \eqref{u22mu12minit} 
is a \emph{necessary} requirement to produce a solution  of \eqref{K22m.again} (re-cast as \eqref{u22mRicit}) 
that remains smooth until $r=0$. 
\end{remark}

Next, we derive $H^{\rm low}_{t,\theta}$ estimates 
for $u^m_{12}(r,t,\theta), u^m_{22}(r,t,\theta)$, 
proving (\ref{inductiontrKopt}) for $(i,j)=(1,2), (2,2)$. 
We argue by finite induction, assuming the estimate\footnote{for 
some $l$-dependent constant $B_{l-1}$.  }
\begin{align}\label{umenest2}
||\partial^Iu^m_{12}(r,t,\theta)||_{L^2_{t,\theta}}, ||\partial^I 
u^m_{22}(r,t,\theta)||_{L^2_{t,\theta}}\leq B_{l-1}r^{-1-\frac{1}{4}},\forall |I|\le l-1\le {\rm low}-1.
\end{align}
is valid for $|I|\leq l-1< s-4-4c$ and proceed to show that the analogous estimate holds for $\partial^Iu^m_{ij}$, where $|I|=l\leq s-3-4c$. Choosing $B$ suitably large establishes our claim at the lower orders. 


Note that for $|I|=0$, (\ref{inductiontrKopt}) for $(i,j)=(1,2), (2,2)$ holds true by virtue of (\ref{umLinftyest}).
\medskip

To prove this inductive step consider the variables $\partial^I u^m_{22}(r,t,\theta), \partial^Iu^m_{12}(r,t,\theta)$ (recall that
    $\partial^I$ means we differentiate $|I|=l$ times in either of the directions 
    $\partial_\theta,\partial_t$); the evolution equation for these parameters 
    arises 
{by differentiating \eqref{u22mRicit}, \eqref{u12mRicit}: The resulting \emph{linear} ODE equation for  
    $\partial^I u^m_{22}(r,t,\theta)$ is of the form: 
      \beq
    \label{K22m.diff}\begin{split}
& \partial_r(\partial^Iu^m_{22})-(\frac{2M}{r}-1)^{-\frac{1}{2}}(2K^m_{22}+e_0\gamma)\cdot \partial^Iu^m_{22}=
 -(\frac{2M}{r}-1)^{-\frac{1}{2}}\partial^I\bigg[\overline{\nabla}^{m-1}_{22}
 \gamma^m\\
 &+e_2(\gamma^{m-1})\cdot e_2(\gamma^m)
-e_0u_{33}^m-(u_{33}^m)^2+(2\alpha^m-d_2^m)\frac{\sqrt{2M}}{r^\frac{3}{2}}u_{33}^m+(u^{m-1}_{12})^2+O(r^{-2})\bigg]\\
&+(\frac{2M}{r}-1)^{-\frac{1}{2}}\sum_{y=0}^{l-1} \big[O(r^{-\frac{3}{2}})\partial^yu_{22}^m+C_y\partial^y u_{22}^m(\partial^{l-y}u_{33}^m+\partial^{l-1-y}u^m_{22})\big]
\end{split}
    \eeq
where $O(r^{-3/2})$ depends on up to $l$ derivatives of $\alpha^m(t,\theta)$.
The terms in 
$\sum_{y=0}^{l-1} C_y\cdot \partial^y u^m_{22}\cdot \partial^{l-1-y} 
u^m_{22}$ 
contain only \emph{lower that $l$} derivatives of $u^m_{22}$.}

   


Thus 
 we may inductively derive $L^2$ bounds on $\partial^I u^m_{22}$ over the 
 hypersurfaces $\Sigma_\rho$. We denote the RHS of 
 \eqref{K22m.diff} by ${\rm RHS}[\partial^Iu^m_{22}](r,t,\theta)$. 
 Moreover, all such quadratic expressions 
involve terms that have been previously controlled as part of the 
inductive step. In particular we have estimates on the $L^2$ norms of 
these terms 
over the hypersurfaces $\Sigma_{\rho^{m-1}},\Sigma_r$ by $B_{l-1}r^{-\frac{3}{2}-1-\frac{1}{8}}$.  At this point we make a key
 observation:

 The ODE equation \eqref{K22m.diff} is linear in $\partial^Iu^m_{22}$ and  admits
  a \emph{free branch} solution which  corresponds to the homogenous equation. Let:

\beq
\label{int.factor}
w(r,t,\theta)=w^m(r,t,\theta)= \int_\e^r (\frac{2M}{s}-1)^{-\frac{1}{2}}(2K^m_{22}(s,t,\theta)+e_0\gamma^m(s,t,\theta))ds= 
[2d_2^m(t,\theta)-\alpha^m(t,\theta)]\cdot {\rm log}r+\e^{\frac{1}{4}}. 
\eeq
(Note that $[2d_2^m(t,\theta)-\alpha^m(t,\theta)]+3|\le 1/4$), in view of 
the bounds on $d_2^m(t,\theta)-d_2^S, \alpha^m(t,\theta)-\alpha^S$). The latter
 expression follows from and 
the inductive assumptions on $K^m(r,t,\theta), K^m_{33}(r,t,\theta)$ in $L^\infty$ which we just verified. 
Then the equation \eqref{K22m.diff} admits a general solution of the form:

 \[
(\partial^Iu^m_{22})(r,t,\theta)=  c(t,\theta) e^{w(r,t,\theta)} +e^{w(r,t,\theta)} \int_0^r e^{-w(s,t,\theta)} [{\rm RHS}[\partial^Iu^m_{22}](s,t,\theta)]ds.
 \]
We note that the free branch
\[
c(t,\theta) e^{w(r,t,\theta)}
\]
 of the solution is \emph{more singular} in $r$ than the solution of the 
 undifferentiated equation obtained above.
 In particular we recall that 
$e^{w(r,t,\theta)}\sim r^{-3+\epsilon^m(t,\theta)}$,
where $|\e^m(t,\theta)|<\frac{1}{8}$. 
 The presence of such a free 
 branch would completely \emph{invalidate} the inductive claims 
 \eqref{inductiontrKopt} and also \eqref{inductiontrKlow}.  
However, since we solve this equation \emph{backwards},
 we are free to set this singular free 
branch to zero and we do so. 
Thus the solution that we consider for $l\le s-3-4c$ is: 

\beq
\label{K22m.int}
(\partial^Iu^m_{22})(r,t,\theta)=  e^{w(r,t,\theta)} \int_0^r e^{-w(s,t,\theta)} 
[{\rm RHS}[\partial^Iu^m_{22}](s,t,\theta)]ds.
\eeq

Our claimed bound follows directly by the already-derived bounds on 
the $L^2$-norm of the integrand in the RHS. 

We next derive the analogue of this integral expression for
 $\partial^I u^m_{12}(r,t,\theta)=\partial^I K^m_{12}(r,t,\theta)$:

The function $\partial^I K_{12}$ is solved-for backwards from $r=0$ by considering the ODE \eqref{u12mRicit}, differentiated by 
 $\partial^I$. The resulting equation is of the form: 
 
  {
 \beq
  \begin{split}
&\partial_r\partial^Iu_{12}^m-(\frac{2M}{r}-1)^{-\frac{1}{2}}(2K_{22}^m+e_0\gamma^m)\partial^Iu_{12}^m
=-(\frac{2M}{r}-1)^{-\frac{1}{2}}\partial^I\bigg[\overline{\nabla}_{12}\gamma^m\\
&+\frac{1}{2}[\overline{\nabla}_1\gamma^{m-1} \overline{\nabla}_2\gamma^m+\overline{\nabla}_2\gamma^{m-1} \overline{\nabla}_1\gamma^m]\bigg]
+(\frac{2M}{r}-1)^{-\frac{1}{2}}\sum_{y=0}^{l-1}C_y\partial^{l-y}(2K_{22}^m+e_0\gamma^m)\partial^yu_{12}^m
\end{split}
  \eeq
} 
 

In analogy with the case of $u^m_{22}$, the solution of this that we consider is: 

 \beq
 \label{K12m.int}
(\partial^Iu^m_{12})(r,t,\theta)= e^{w(r,t,\theta)} \int_0^r e^{-w(s,t,\theta)} [{\rm RHS}
[\partial^Iu^m_{12}](s,t,\theta)]ds.
 \eeq
 I.e.~again the free branch of the solution is set to zero. 
 (Note that this free branch was already set to zero for the undifferentiated equation \eqref{partial_ru12m}). 

The desired estimates for $\partial^I u^m_{12}$ 
 then follow by a finite induction, just as for 
$u^m_{22}$. 
\medskip

We also note that given the
definitions 
\[
\partial^I K^m_{22}(r,t,\theta)=r^{-3/2} \sqrt{2M}\partial^I d_2^m(t,\theta)+\partial^I 
u^m_{22}(r,t,\theta), \partial^I K^m_{12}(r,t,\theta)=\partial^Iu^m_{12}(r,t,\theta),  
\]
the bounds \eqref{alpham}
on $\|\partial^I(\alpha^m(t,\theta)-1)\|_{L^2}\le C\eta$, for all $|I|\le \rm low$,  the corresponding bound 
\[
\|\partial^I(d^m_2(t,\theta)-d^S_2(t,\theta))\|_{L^2_{t,\theta}}\le DC\eta,
\]
the assumed closeness of $\K_{22}$ to $K^S_{22}$ and the bounds
just derived   on $u^m_{12}(\e,t,\theta)$, $u^m_{22}(\e,t,\theta)$ 
as well as the bound  fixing the smallness \eqref{e.bd} of $\e$ relative to $C\eta$  imply the following bounds:

\begin{lemma}
\label{ground.est}
The functions $K^m_{22}(\e, t,\theta)$, $K^m_{12}(\e,t,\theta)$
at $\Sigma_{r=\e}:=\{r=\e\}$  thought of as a functions in $t,\theta$ 
satisfy the following bounds for all $|I|\le {\rm low}$: 

\beq
\label{1st.bds}\begin{split}
&\int_{\Sigma_{r=\e}}  |\partial^I[K^m_{22}(\e,t,\theta)- \K_{22}(t,\theta) 
 ]|^2 sin\theta d\theta dt
 \le 2C^2 \eta^2 \e^{-3}, 
 \\& \int_{\Sigma_{r=\e}}  |\partial^IK^m_{12}(\e,t,\theta) |^2 sin\theta d\theta dt \le B^2 C^2 \e^{-2-\frac{1}{2}}\le 
 C^2 \eta^2 \e^{-3+\frac{1}{2}}.
 \end{split}
\eeq
\end{lemma}

This estimate will play a key role in solving for $r^m_*(t,\theta)$ and 
$\tilde{K}^m_{12}(t,\theta)$ via the inverse function theorem in section \ref{IDsol}. 
For now, however, let us also derive bounds on $K^m_{22}(r,t,\theta), K^m_{12}(r,t,\theta)$ at the higher orders:

\subsection{The bounds on $K^m_{22}(r,t,\theta), K^m_{12}(r,t,\theta)$ at the higher orders.}

In order to derive estimates at the higher orders, we subtract the Riccati equation
 satisfied by $K^S_{22}(r,t,\theta)$ from our equation, to derive: 

 \beq
 \begin{split}
 \label{K22m-S}
&\ e_0[K^m_{22}-K^S_{22}]+ 2K^S_{22}\cdot (K^m_{22}-K^S_{22})+
(K^m_{22}-K^S_{22})^2
 -(K^{m-1}_{12})^2+e_0\gamma^m (K^m_{22}-K^S_{22})=
 \\& {}^{m-1}\overline{\nabla}_{22}(^m\gamma)+
e_2(\gamma^{m-1})\cdot e_2(\gamma^m)
 -
e^2_0(^m\gamma-^S\gamma) -
[e_0(^m\gamma-^S\gamma)]^2-e_0(\gamma^m_{\rm rest})[e_0\gamma^S+K^S_{22}],
\\& e_0K^m_{12}+(2K^m_{22}+e_0(\gamma^m))K^m_{12}=
 \overline{\nabla}^{m-1}_{12}(^m\gamma)+\frac{1}{2}[e_1(^{m-1}\gamma) e_2(^m\gamma)+e_1(^{m}\gamma) e_2(^{m-1}\gamma)].
 \end{split}
 \eeq

We   take the derivatives $\partial^I, |I|\in \{{\rm low}+1,\dots, s-4\}$ 
 of this equation  to derive 

\begin{align}
\label{highRicm22} &e_0\partial^I{}[ K_{22}^m-K^S_{22}] +\frac{(2d_2^m-\alpha^m)\sqrt{2M}}{r^\frac{3}{2}}\partial^I {}[K^m_{22}-K^S_{22}]
+O(\frac{1}{r^{1+\frac{1}{4}}})\partial^I( K_{22}^m-K^S_{22})\\
\notag=&\,\partial^I{} \big[
\overline{\nabla}_{22}\gamma^m+(\overline{\nabla}_2\gamma^{m-1})(\overline{\nabla}_2\gamma^m)
-e_0^2\gamma^m_{\rm rest}-(e_0\gamma^m_{\rm rest})^2+(K_{12}^{m-1})^2
-2(e_0\gamma^m_{\rm rest})e_0\gamma^S -(e_0\gamma^m_{\rm rest})K^S_22 \big]\\
\notag&-\sum_{I_1\cup I_2,\,|I_1|< |I|}\partial^{I_1}K_{22}^m \partial^{I_2}{}e_0\gamma^m_{\rm rest}
-\sum_{I_1\cup I_2=I,\,|I_1|<|I|}\partial^{I_1}{}K_{22}^m\partial^{I_2}e_0 \gamma^m_{\rm rest}\\
\notag&-\sum_{I_1\cup I_2=I,\,|I_1|<|I|,\,|I_2|<|I|}
\partial^{I_1}(K_{22}^m-K^S_{22})\partial^{I_2}(K_{22}^m-K^S_{22}),\\
\label{highRicm12}& e_0\partial^I K_{12}^m+(2d_2^m(t,\theta)-\alpha^m(t,\theta))\frac{\sqrt{2M}}{r^\frac{3}{2}}\partial^IK_{12}^m
+O(\frac{1}{r^{1+\frac{1}{4}}})\partial^IK_{12}^m\\
\notag=&\,\partial^I \big[\overline{\nabla}_{12}\gamma^m+\frac{1}{2}
(e_1\gamma^{m-1}e_2\gamma^m
+e_1\gamma^me_2\gamma^{m-1})
\big]
-\sum_{I_1\cup I_2=I,\,|I_1|<|I|}\partial^{I_1}K_{12}^m
\partial^{I_2}(2K_{22}^m+e_0\gamma^m)
\end{align}

Here, 
we will bound the functions 
$\partial^IK^m_{22}(\delta,t,\theta), \partial^IK^m_{12}(t,\theta,t,\theta)$ 
$|I|\le s-4$ on level sets
$\Sigma_{\delta}$ of $r$, but also on all hypersurfaces 
$\Sigma_{\delta(t,\theta)}$, for functions $\delta(t,\theta)$ that 
are close (in suitable norms) to $\e$.

We derive our bounds for 
$\partial^I [K^m_{22}-K^S_{22}](\delta(t,\theta),t,\theta)$. The bounds for 
$\partial^I K^m_{12}(\delta(t,\theta), t,\theta)$ 
follow by essentially the same argument.

\begin{remark}\label{K11.rem}
 We also note here that \emph{if} we were to solve $K^m_{11}$ \emph{backwards} 
 from $r=0$ by setting the free branch of that solution to zero, we would derive the 
 same estimate for that parameter as for $K^m_{22}$, by the same proof at all low and 
 high orders $|I|\le s-4$. 
\end{remark}
\medskip

 We will first bound a suitably weighted $L^2$ 
norm $\partial^I (K^m_{22}-K^S_{22})(s,t,\theta)$ in the bulk region 
$\{r\le 2\epsilon\}$. 

In particular for all $|I|=k\ge s-3-4c$, $k\le s-4$ we will derive
the bounds: 
\beq
\label{bulk1}
\int_{0}^{2\epsilon}\int_{\Sigma_{r}}[\partial^I[K^m_{22}-K^S_{22}]]^2
r^{3+\frac{1}{2}(k-(s-3-4c))}\cdot 
sin\theta dt  d\theta dr\le 4 C^2\eta^2 \epsilon^{}.
\eeq

\beq
\label{bulk2}
\int_{0}^{2\epsilon}\int_{\Sigma_{r}}[\partial^IK^m_{12}]^2
r^{3+\frac{1}{2}(k-(s-3-4c))}\cdot 
sin\theta dt  d\theta dr\le 4 C^2\eta^2 \epsilon^{1+\frac{1}{2}}.
\eeq


For brevity of notation, for each $k>low$, we denote by $p(k)$ 
the term: 
\[
p(k):=\frac{1}{2}(k-(s-3-4c)). 
\]

Then, after \eqref{bulk1}, \eqref{bulk2} 
have been established, we will derive the following energy estimate for 
$\int_{\Sigma_{\delta(t,\theta)}}[|\partial^I[K^m_{22}-K^S_{22}]|^2r^{p(k)}](\delta(t,\theta),t,\theta))sin\theta dtd\theta$, where $\Sigma_{\delta(t,\theta)}$
 is any graphical hypersurface expressed in terms of  $\{\rho^{m-1}=\delta(t,\theta)\}$ or $\{r=\delta(t,\theta)\}$, $\delta(t,\theta)\sim \e$:

\beq
\label{K22m.r*.higher}
\int_{\Sigma_{\delta(t,\theta)}}|\partial^I 
[K^m_{22}(r(t,\theta),t,\theta)-
K^S_{22}(t,\theta)]|^2
 sin\theta d\theta dt\le C\eta^2 \delta^{-3-p(k)}.
\eeq
(After these bounds have been derived, we will also explain how the same 
bounds hold on level sets of the new coordinate function $\rho^m$).

Moreover the same proof  applies to $K^m_{12}$ to yield:

\beq
\label{Km12.high}
\int_{\Sigma_{r=\delta(t,\theta)}}|\partial^I 
K^m_{12}(\delta(t,\theta),
t,\theta)|^2dt sin\theta d\theta
\le \e^{-p(k)} 3 C^2  \eta^2 \delta^{-3+\frac{1}{2}}. 
\eeq
\medskip

(The reason for the smaller power of $r$  in the RHS of the evolution equation 
 \eqref{finredEVERic12it} for 
 $K^m_{12}$), where no \emph{singular} $e_0$-derivatives are present; the only 
 derivatives of $\gamma^m$ are in the directions $e^{m-1}_1, e^{m-1}_2$, 
 which are less singular in $r$, by virtue of the AVTD behaviour of our solution. 
 We also note that stronger bounds for $K^m_{12}$ can be derived at the orders below the 
 top, in view 
 of Lemma 
\ref{hessian.bounds.again}, but these are not needed and 
so we do not put them down.  
 \medskip

\emph{Proof of \eqref{bulk1}, \eqref{bulk2}:} We focus on \eqref{bulk1} and 
explain at the end the modification needed in deriving \eqref{bulk2}. 
We will prove this claim by a finite induction. So we assume our claim has been proven 
to orders $ |I|=l-1\le s-5$ and we will derive it for order 
$|I|=l\le s-4$. 

Consider the equation \eqref{highRicm22}. 
Recall this becomes a linear 1st order ODE equation in 
$\partial^IK^m_{22}$. 
Recall the integrating factor for that equation is  
$F(r,t,\theta)=e^{-\int_{\e}^r(\frac{2M}{s}-1)^{-1/2} [2K_{22}^m+ e_0(^m\gamma)](s,t,\theta)ds}$.

 Notably, using the expressions we 
have for the asymptotics of $K^m_{22}(\rho,t,\theta), e_0(^m\gamma)$ we 
see that in $L^\infty$ 
 integrating factor is asymptotic to 
$r^{3+\e^m(t,\theta)}$ as $r\to 0$ in the sense that for some function $c(t,\theta)\sim 1$: 

\[
\|F(r,t,\theta)-c(t,\theta)r^{3+\e^m(t,\theta)} \|_{L^\infty}\le Cr^{-2-3/4}, 
\] 
moreover  the function $\e^m(t,\theta)$ satisfies 
$|\e^m(t,\theta)|\le DC\eta\le 1/8$. 
\medskip

Now, consider the equation \eqref{highRicm22};
using the integrating factor $F(r)$ and shorthand notation this can be re-expressed as:

\beq
\label{int.fact}
r^{p(k)+5}[F^{-1}(r)\partial_r(F(r)\partial^I[K^m_{22}-K_{22}^S](r,t,\theta))]^2=  (\frac{2M}{r}-1)^{-1}
[{\rm RHS}[\eqref{highRicm22}]]^2\cdot r^{p(k)+5}.
\eeq

In fact, let us recall the bounds on ${\rm RHS}[\eqref{highRicm22}]]$ that we derived 
in the inductive step of \eqref{inductionHess}.

We will integrate this equation over $t,\theta,r$ (with the volume form
 $ sin\theta dt  d\theta dr$)
and then apply the standard  Hardy inequality on the RHS:  The first integration gives:

\beq
\begin{split}
&\int_0^{2\e}\int_{t,\theta}r^{p(k)+5}[F^{-1}(r)\partial_r\bigg( F(r)
\partial^I
[K^m_{22}-K_{22}^S](r,t,\theta)\bigg)]^2drdtsin\theta d\theta
\\& = \int_0^{2\e}\int_{t,\theta}r^{p(k)+5}
(\frac{2M}{r}-1)^{-1}
[{\rm RHS}[\eqref{highRicm22}](s,t,\theta)]]^2 drdt sin\theta d\theta.
\end{split}
\eeq

The standard weighted 1-d Hardy inequality 
 then implies a lower bound for the LHS of the above: 

\beq
\begin{split}
&\int_0^{2\e}\int_{t,\theta}r^{p(k)+5}
[F^{-1}(r)\partial_r(F(r)(\partial^I[K^m_{22}-K_{22}^S](r,t,
\theta))]^2dt sin\theta dr
d\theta
\\& \ge C\int_0^{2\e}\int_{t,\theta}r^{p(k)+5}
[F^{-1}(r)\partial_r(
r^{3+\e^m(t,\theta)}(\partial^I[K^m_{22}-K_{22}^S](r,t,
\theta)))]^2dt sin\theta d\theta dr
\\&\ge C_{p(k)} \int_0^{2\e}\int_{t,\theta}r^{p(k)+5}
[r^{-(3+\e^m(t,\theta))-1}
r^{3+\e^m(t,\theta)}[\partial^I 
[K^m_{22}-K_{22}^S](r,t,\theta))]]^2dtsin\theta dr
d\theta
\end{split}
\eeq
(The constant $C(p(k))$ has come from the classical Hardy inequality. Note that 
$C(p(k))\sim [p(k)]^2$).
Thus we derive:

\beq
\begin{split}
& \int_0^{2\e}\int_{t,\theta}r^{p(k)+5-2}[\partial^I[K^m_{22}-K_{22}^S](r,t,
\theta))]^2
 drdt sin
 \theta 
d\theta
\\&\le C_{p(k)}^{-1}\int_0^{\e}\int_{t,\theta}(\frac{2M}{r}-1)^{-1}
r^{5}r^{p(k)}
[{\rm RHS}[\eqref{highRicm22}]]^2\cdot 
  drdt sin\theta d\theta.
\end{split}
\eeq
 
Let us now derive a bound on the RHS of the above: 
We note that the terms involving spatial derivatives $\onabla\gamma^m$ have 
all been bounded in  
Lemmas \ref{hessian.bounds.again}.
 Lemma \ref{hessian.bounds.again} 
directly implies that the contribution of those terms to 
the total norm is bounded by $\frac{B^2\e^{1/4}}{C(p(k))} \e^3$.
 The terms involving $e_0e_0(\gamma^m), e_0(\gamma^m)$ have 
been bounded in \eqref{e0gammamid}, 
\eqref{e0e0gammamid}. The contribution of those terms  (after the integrations in the RHS of the above) is thus  bounded by 
$\frac{2}{C(p(k)}C^2\eta^2\e $. Finally, there are all the terms that involve 
lower derivatives of $K^m_{22}, K^m_{12}$. 
Since we are assuming that those terms are already bounded as in 
\eqref{bulk1}, we find that the contributions of those terms is bounded by 
$\frac{B^2\e^{1/4}}{C(p(k))} \e$.

In sum, using \eqref{e.bd},  we conclude that the RHS of the above is
 bounded by $\frac{6}{C(p(k))}C^2\eta^2\e $ for all $|I|\le s-4$. $\qed$


Thus we derive an upper  bound for the LHS of the above by 
$\frac{7}{C(p(k)}C^2\eta^2\delta^3$.
\medskip

Having bounded this bulk term, we can now bound 

\[
\int_{\Sigma_{r=\e}}r^{p(k)+4}
|\partial^I[K^m_{22}-K_{22}^S](r,t,\theta)|^2sin\theta  dt
 d\theta.
\]

By the fundamental theorem of calculus we find: 

\beq
\begin{split}
\label{expand}
&\int_{\Sigma_{\delta}}r^{p(k)}r^{4} (
\partial^I[K^m_{22}-K_{22}^S] )^2   dt sin\theta d\theta=
\int_{t,\theta}\int_0^{\delta}\partial_r \{r^{p(k)+4}
[\partial^I[K^m_{22}-K_{22}^S]]^2 \}
drdt sin\theta d\theta
\\&\le 
(p(k)+6)\int_{t,\theta}\int_0^{\delta} r^{p(k)} 
(r^{2}\partial^I[K^m_{22}-K_{22}^S])^2   
drdtsin\theta d\theta
\\&+
\int_{t,\theta}\int_0^{\delta}(\frac{2M}{r}-1)^{-1/2} 
r^{p(k)}r^{4}{{\rm RHS}[\eqref{highRicm22}}]\cdot 
\partial^I[K^m_{22}-K_{22}^S]  drdt  sin\theta d\theta.
\end{split}
\eeq
 (The coefficient $+6$ in the first term in the RHS has come from incorporating the second and third terms in the LHSs of \eqref{highRicm22} into that term).

Then the first term in the RHS of \eqref{expand}
 has already been bounded by 
$\frac{10(p(k)+4))}{C(p(k))}C^2\eta^2\delta$. 
The second term can be controlled by Cauchy-Schwarz:

\beq
\begin{split}
&\int_{t,\theta}\int_0^{\delta} (\frac{2M}{r}-1)^{-1/2} r^{p(k)}(r^4
{\rm RHS}[\partial^I[K^m_{22}-K_{22}^S]]\cdot 
\partial^I[K^m_{22}-K_{22}^S]] dr dt sin\theta 
d\theta
\\&\le 
\kappa\int_{t,\theta}\int_0^{\delta} (\frac{2M}{r}-1)^{-1} 
r^{p(k)} r^{5}
|[\widetilde{{\rm RHS}[\eqref{highRicm22}]}|^2 drdt sin\theta d
\theta
\\&+(4\kappa)^{-1}\int_{t,\theta}\int_0^{\delta}  r^{p(k)} 
r^3|\partial^I[K^m_{22}-K_{22}^S]|^2
  drdt sin\theta d\theta
\end{split}
\eeq
Choose $\kappa=1$. Then given the bounds 
we have obtained on the second term in the RHS (the bulk bound derived 
above) 
and the bounds we have derived  on 
the bulk integral of ${\rm RHS}[\eqref{highRicm22}$, we 
derive the bound, for every fixed $\delta$ (factoring out the $r^3$ from 
the LHS, and recalling the bound on the constant $C(p(k))$):

\beq
\label{Km22.high}
\int_{\Sigma_{\delta}}|\partial^I[K^m_{22}-K_{22}^S]|^2
(\delta,t,\theta)dt sin\theta d\theta\le \delta^{-p(k)} 10  C^2  \eta^2 
\delta^{-3}. 
\eeq

If the hypersurface $\{r=\delta\}$ is replaced by a function 
$\{r=\delta(t,\theta)\}, \delta(t,\theta) \sim \e$ the same argument 
applies to derive: 

\beq
\label{Km22.high.var}
\int_{\Sigma_{\delta(t,\theta)}}|(\partial^I[K^m_{22}-K_{22}^S])
(\delta(t,
\theta),t,\theta)|^2dt sin\theta d\theta\le 2 \e^{-p(k)}10 C^2 
 \eta^2 \e^{-3}. 
\eeq

In particular as we will see below (once we have defined the function 
$\rho^m$ via $r^m_*(t,\theta)$) these bounds imply: 

\beq
\label{Km22.high.var.rhom}
\int_{\Sigma_{\rho^m}}|\partial^I
[K^m_{22}-K_{22}^S]
(r=\rho^m(t,\theta))|^2dt sin\theta d\theta\le 5\e^{-p(k)} 
 C^2  \eta^2 \e^{-3}. 
\eeq

Again, once we have derived bounds for 
$\rho^m$ via $r^m_*(t,\theta)$ these bounds will  imply: 

\beq
\label{Km12.high.var.rhom}
\int_{\Sigma_{\rho^m}}|\partial^I(K^m_{12}(r=\rho^m(t,\theta))|^2dt
 sin\theta d\theta\le 3\e^{-p(k)} 
 C^2 \eta^2 \e^{-3+\frac{1}{2}}. 
\eeq

\medskip

\subsection{The inductive step for $K^m_{22}, K^m_{12}$ at the top order.}
\label{sec:K22K11.top}

\newcommand{\oK}{\underline{K}}
We here prove the inductive claim for $K^m_{22}, K^m_{12}$ at the top 
orders. We note that to derive our claim, we will use that the variables 
$K_{22}, K_{12}$ are \emph{de-coupled} from the remaining variable $K_{11}$ 
in the
 equations \eqref{finredEVERic11pre}, 
 \eqref{finredEVERic22pre}, \eqref{finredEVERic12pre}. This remains true in 
 the iterative variables
  $K^m_{22}(r,t,\theta), K^m_{12}(r,t,\theta)$, via equations
   \eqref{finredEVERic12it}, \eqref{finredEVERic22it}.

We note that the proof of the top order claims would have been  
simpler, had we been treating directly the \emph{coupled} system in section \ref{sec:system}. The setting of the iteration that we use here requires us to use some tricks; these make use of this
 de-coupling of $K^m_{22}, K^m_{12}$ from $K^m_{11}$.   
   \medskip

We just treat the top order case in the case \eqref{inductiontrKtopmixed.sing}
of extra singular weights, with the 
unknowns 
 $\partial^J\partial_\theta [K^m_{22}-K^S_{22}]{\rm cot}\theta$, $|J|=s-4$,
 $\partial^I [K^m_{12}\cdot {\rm cot}\theta]$; the cases 
\eqref{inductiontrKtopmixed}, $|I|=s-3$,
are an easier version of this case. 
\medskip

\newcommand{\oh}{\overline{h}}

Our use of the de-coupling of the two variables that we are interested in here, is to 
construct  a metric $\oh^m$ which is partly {artificial}, in that its 
components $\oK^m_{22}(r,t,\theta), \oK^m_{12}(r,t,\theta)$ agree with 
$K^m_{22}(r,t,\theta), K^m_{12}(r,t,\theta)$; but its connection coefficients where 
 $e_1$ appears twice do not agree with that of the metric iterate $h^m$.\footnote{The 
 latter is for convenience only, and inconsequential since we are only proving claims on 
$K^m_{22}(r,t,\theta), K^m_{12}(r,t,\theta)$ here. }

We define a new tensor  $\oK^m_{ij}$ in place of  
$K^m_{ij}$ as follows:
\beq
\oK_{22}=K^m_{22}, 
\oK^m_{12}=K^m_{12}, 
\oK^m_{11}=
(K^m_{11})^{\rm reg},
\eeq
where the latter is defined to be the \emph{unique} solution of \eqref{finredEVERic11it}
with the ``free branch'' (with behaviour $O(r^{2d_1^m(t,\theta)-\alpha^m(t,\theta)})$ 
of the solution set to zero.

In particular proving our claim for $\oK^m_{22}, \oK^m_{12}$ 
will imply our claim for $K^m_{22}, K^m_{12}$. 
We also note that having solved for the variables $\oK^m_{ij}(r,t,\theta)$, we can the 
define a metric $\oh^m_{ij}$ once we specify functions $e_2(r), e_1(r)$, as well as 
(asymptotic) initial data for  the corresponding coordinate-to-frame coefficients 
$a^m_{Ai}(r,t,\theta)$. To define this new (artificial) $(2+1)$-metric 
$\underline{h}^m$ we 
set $e_2(r)=o(r^{-\frac{1}{2}+d_2^m(t,\theta)})$, but also 
$e_1(r)=o(r^{\frac{1}{2}-d_1^m(t,\theta)})$; equations \eqref{eirODE} then imply that 
$e_1(r)=0,e_2(r)=0$. We also 
prescribe the solutions of the system in \eqref{e0.a} by requiring the free branches 
of $a^m_{t2}, a^m_{\theta 1}$ to zero, and also the 
coefficients of the free branches of $a^m_{\theta 2}, a^m_{t1}$ as follows:
$a^m_{\theta 2}$ is prescribed asymptotically as $r\to 0$ as before. (In particular 
$\partial_\theta$ captures the direction of $e_2^m$ at the singularity). 
$a^m_{t 1}$ is also chosen asymptotically as $r\to 0$ to ensure the (asymptotic) 
commutation of the (asymptotic) vector fields $a^m_{t1} e^m_1, a^m_{\theta 2} e^m_2$.\footnote{A fuller discussion of this relation appears in section 
\ref{subsec:REVESNGGimpl}--here we just employ this aspect of the construction
 in that subsection  for technical convenience. }
 We note in particular that 
we 
have the formulas \eqref{explicit.int}, by replacing the initial data factors, 
and we commence the integrals in the exponentials by $\e$. 
\medskip

We note (see remark \ref{K11.rem}) that we can derive the same bounds for 
$\oK^m_{11}$ as those 
claimed for $K^m_{11}$ at all orders below the top. 
We also note that we have $a^m_{t2}=a^m_{\theta 1}=0$, and for the variables 
$a^m_{\theta 2}, a^m_{t1}$ we have the same bounds as for the ``real'' metric $h^m$. 

Now, let us derive the claimed bounds on $K^m_{22}=\oK^m_{22}, K^m_{12}=\oK^m_{12}$
at the top orders.

We do this via the Lemma: 
\begin{lemma}
\label{trK.em}
Under the inductive assumptions on $\gamma^m$, $g^{m-1}$, we also claim for all 
$|I|=s-3$, 
and on any level set $\Sigma_r$, $\Sigma_{\rho^m}$ we have the 
estimate:

\beq
\label{K.em.flux}
\int_{\Sigma_r} r^{2c+3} 
|\partial^I\overline{K}^m|^2
{\rm sin}\theta d\theta dt d\tau
\le (C\eta)^2 r^3. 
\eeq

We also have analogues of these estimates at the other top order terms: 
In particular for each multi-index $J, |J|=s-4$ and for $i=1,2$ we  claim:

\beq
\label{K.em.flux.sing}
\int_{\Sigma_r} r^{2c+3 } 
|\partial^J\partial_\theta \overline{K}^m_{ii}\cdot {\rm cot}\theta|^2
{\rm sin}\theta d\theta dt d\tau
\le (C\eta)^2 r^3, \int_{\Sigma_r} r^{2c+3 } 
|\partial^I \overline{K}^m_{12}\cdot {\rm cot}\theta|^2
{\rm sin}\theta d\theta dt d\tau
\le (C\eta)^2 r^{3}.
\eeq
(In the second equation the multi-index $I$ satisfies $|I|=s-3$)
\end{lemma}

We note that in view of the definition of $\oK^m$, the above implies 
\eqref{inductiontrKtopmixed} for $K^m_{22}, K^m_{12}$.

\begin{proof}

We will prove the slightly harder  case of 
\eqref{K.em.flux.sing}. The  claim \eqref{K.em.flux} follows by an easier adaptation of this
 argument. 
We will first derive our bounds for 
$\partial^J \partial_\theta {\rm tr}_{\overline{h}^{m-1}} \oK^m\cdot {\rm cot}\theta$;\footnote{Here $\overline{h}^{m}$ is the 2-dimensional metric on the space 
$e_0^\perp$, spanned by $e^{m-1}_1, e^{m-1}_2$.} we will see further down how this implies our claim for the full 
second fundamental form $\oK^m$ at the top orders, as claimed. 
\medskip

We recall the Riccati equations \eqref{finredEVERic11it}, \eqref{finredEVERic12it}, 
\eqref{finredEVERic22it}; we consider these Ricatti equations for $\overline{h}^m$ whose deformation 
tensor (for the vector field $e_0$) in $\underline{K}^m$--by our construction, the Ricatti 
equations   hold verbatim for the components of $\oK^m$, 
with all  occurrences of $K^m_{ij}$ replaced by $\oK^m_{ij}$; 
Note that if we add the evolution equations with  $e_0\oK^m_{11}, e_0\oK^m_{22}$ and invoke the wave equation \eqref{waveq.it} on $\gamma^m$, then after
we multiply the resulting equation  by $r^{c+\frac{3}{2}}$ we derive an equation of the form: 
 
 \beq
 \label{Ric00.appl}\begin{split}
& e_0[r^{c+\frac{3}{2}} {\rm tr}_{\overline{h}^{m}}\oK^m]=-(\frac{2M}{r}-1)^{\frac{1}{2}}(c+\frac{3}{2})
 r^{c+\frac{1}{2}} {\rm tr}_{\overline{h}^{m}}\oK^m+r^{c+\frac{3}{2}}\{
 -\frac{1}{2}[{\rm tr}_{\overline{h}^{m}}\oK^m]^2-|\oK^m|^2_{\overline{h}^{m}}
\\&-[e_0^2\gamma^m+(e_0\gamma^m)^2+\sum_{i=1,2}(\oK_{ii}^m-\oK_{ii}^{m-1})e_0\gamma^m+(e_0\gamma^m-e_0\gamma^{m-1})e_0\gamma^m\\
&+(\oK^m_{12})^2-(\oK^{m-1}_{12})^2]\}\end{split}
 \eeq

In particular we note that all forcing terms involving $\gamma^m$ in the RHS now satisfy bounds in our desired energy spaces; in particular the terms with second spatial derivatives of $\gamma^m$ do not appear. 
This is the upshot of invoking  the wave equation on $\gamma^m$ to re-express spatial derivatives of $\gamma^m$ in terms
of time
 $e_0$-derivatives.

We then take the $\partial^I$ derivative of the above equation,  with $|I|=s-3$ (for the first claim in our Lemma), 
or act on it by the operator 
 ${\rm cot}\theta\cdot \partial^J\partial_\theta$
with $|J|=s-4$, for the second case. We consider only the second case, since the first is an easier adaptation of this.

 We obtain commutation terms which will not be top-order in the resulting equation; these we denote by ${\rm l.o.t.'s}$. 
Our differentiated equation then yields: 
\beq
\label{e0trK}\begin{split}
&e_0[{\rm cot}\theta\cdot \partial^J\partial_\theta {\rm tr}_{\overline{h}^{m}}\oK^m]+ \sum_{a,b=1,2,3}[{\rm cot}\theta\cdot \partial^J\partial_\theta  \oK^m_{ab}]\cdot [(\oK^m)^{ab}+g^m_{ab}e_0\gamma^m]-{\rm cot}\theta\cdot \partial^J\partial_\theta [(\oK^{m-1}_{12})^2-3(\oK^m_{12})^2]
\\&= -2{\rm cot}\theta\cdot \partial^J\partial_\theta [e_0e_0\gamma^m] 
+\partial^I[\nabla\gamma^{m-1},\nabla\gamma^m]+{\rm l.o.t.'s},\end{split}
\eeq
where in the last term we have used generic notation for products of first derivatives. 

Using  the usual Hardy inequality in $r$ from $r=0$,
we derive the bound: 

\beq
\int_{0}^r\int_{\Sigma_\tau} \tau^{2c+3} [\partial^J\partial_\theta {\rm tr}\oK^m\cdot {\rm cot}
\theta]^2 dt
{\rm sin\theta d\theta }d\tau\le \frac{4}{(2c+4)^2}\int_{0}^r\int_{\Sigma_\tau}
 \tau^{2c+3+2} 
[\partial_r
\partial^J\partial_\theta {\rm tr}\oK^m\cdot {\rm cot}\theta]^2 dt
{\rm sin\theta d\theta }d\tau 
\eeq
Now combine this with the equation \eqref{e0trK}, to substitute the RHS.

Let us note a bound in $L^2[\Sigma_r]$ on 
$\partial^J\partial_\theta  e_0\gamma^m {\rm cot}\theta$.
We express $\partial^J\partial_\theta$ by $\partial^I$ with $|I|=s-3$. 
We then use the Hardy inequality to bound: 

\[
\| \partial^J\partial_\theta  e_0e_0\gamma^m {\rm cot}\theta\|_{L^2[\Sigma_r]}\le 
\|a^{m-1}_{\Theta 2} \|_{L^\infty}\cdot \|e^{m-1}_2
\partial^Ie_0e_0\gamma^m  \|_{L^2[\Sigma_r]} \le \|a^{m-1}_{\Theta 2} \|_{L^\infty}
\cdot \sqrt{E[\partial^Ie_0e_0\gamma^m]}.
\]

Let us recall the bound on the top order term 
$\partial^Ie^{m-1}_ie_0e_0(\gamma^m)$, $i=0,1,2$ in $L^2_{({\rm sin}\theta d\theta dt)}[\Sigma_\tau]$
by 
$C\eta \tau^{-3-c-\frac{3}{2}}$. 

Using this  we will be able to 
prove our claim (using \eqref{c.bd} as before)
 \emph{provided} we can prove the bound: 

\beq
\label{to.IBP}\begin{split}
&\int_0^r\int_{\Sigma_{\tau}}
 r^{2c+6}\sum_{a,b=1,2,3}[(\oK^m)^{ab}\partial^J\partial_\theta\oK^m_{ab}]
\cdot 
[\partial^J\partial_\theta(\oK^m)^{a'b'}\cdot \oK^m_{a'b'}]({\rm cot}\theta)^2
dt{\rm sin\theta d\theta }d\tau
\\&\le \int_0^r\int_{\Sigma_\tau} 4M\tau^{2c+3}
 |\partial^J\partial_\theta {\rm tr}\oK^m\cdot {\rm cot}\theta|^2{\rm sin}\theta d\theta dtdr+O(r^{\frac{1}{4}}).
 \end{split}
\eeq

This inequality can be proven by
first invoking the pointwise bounds on $\|\oK^m_{ij}\|_{L^\infty[\Sigma_r]}$ 
(as noted, these are the 
same as those for $K^m_{ij}$), 
which reduce matters to bounding 

\beq\label{bd.this}
\int_0^r\int_{\Sigma_{\tau}} r^{2c+3}\sum_{a,b=1,2}[\partial^J\partial_\theta\oK^m_{ab}]
\cdot 
[\partial^J\partial_\theta(\oK^m)^{ab}]({\rm cot}\theta)^2
dt{\rm sin\theta d\theta }d\tau
\eeq
by the RHS of \eqref{to.IBP}. In particular we will show that: 

\beq
\begin{split}\label{get.bd}
&\bigg{|}\sum_{a,b=1,2}\int_0^r(1+O(\tau^{\frac{3}{8}}))\int_{\Sigma_{\tau}} r^{2c+3}[\partial^J\partial_\theta\oK^m_{ab}]
\cdot 
[\partial^J\partial_\theta(\oK^m)^{ab}]({\rm cot}\theta)^2
dt{\rm sin\theta s\theta }d\tau
\\&-\int_0^r\tau^{2c+3} \int_{\Sigma_\tau}
 |\partial^J\partial_\theta {\rm tr}\oK^m\cdot {\rm cot}\theta|^2{\rm sin}\theta d\theta dtdr\bigg{|}=O(r^{\frac{1}{4}}).
\end{split}
\eeq
This is done by 
 integration by parts: We commence with the term 
 $\partial^J\partial_\theta {\rm tr}\oK^m$ to which we apply the Codazzi equation to
  write: 
 \beq\label{Codazz.appl}
 \partial^J\partial_\theta {\rm tr}\oK^m= \partial^J\onabla^bK^m_{\theta b}-
 \partial^JRic^{\oh^m}_{0\theta}.
 \eeq
 
 We then invoke  the standard coordinate expression for the Ricci curvature components 
 $Ric_{0\theta}$,
  $Ric_{0T}$; in particular we express this in terms of coordinate derivatives 
  of the metric components and Christoffel symbols.
   We note that the terms with most derivatives will have a 
  derivative $\partial_\rho$; this is since \emph{for this ``artificial'' metric} 
  $\underline{h}^m$ 
  $e_0$ is normal to $\partial_T,\partial_\Theta$ everywhere. [Recall that 
 in this proof  we choose $e_1$ so that $e_1(r)=0$] . 
  Thus, the bounds on $\partial^JRic^{\oh^m}_{0\theta}$ follow just from the estimates 
  we have on the metric components and Christoffel coefficients, 
  as part of the implications of the inductive assumptions for the artificial 
  metric $\underline{h}^m$--the estimates we get for all Christoffel symbols and 
  metric components of $\underline{h}^m$ are the same as in the implications of our 
  inductive assumptions for $h^m$. 
  
 In particular, we will have at most $s-3$ spatial derivatives hitting any 
 metric component. These metric components are expressible in terms of the coordinate-
 to-frame component, the latter having been expressed as integrals from $r=0$ involving 
 the components $\oK^m_{ij}$. Now, some of these integrals involve fewer than 
 $s-3$ spatial derivatives of $\oK^m$; these have already been bounded and contribute to 
 the term $O(r^{\frac{1}{4}})$ in \eqref{get.bd}. The terms that depend on $s-3$ 
 derivatives of $\oK^m$ are expressible as integrals of these quantities, and by 
 Cauchy-Schwarz can be absorbed into the main terms we are seeking to bound. 
In particular we can show: 
\newcommand{\uh}{\underline{h}} 
 
 \beq\begin{split}\label{curv.bd.1}
& \sum_{b,a=1,2,3}\int_0^r\int_{\Sigma_\tau} \tau^{2c+3} \partial^J 
|(R^{\uh^m})_{\theta 0}|^2({\rm cot}\theta)^2{\rm sin}\theta d\theta dtd\tau
\\&\le \int_0^rO(\tau^{\frac{3}{8}})\int_{\Sigma_\tau} \tau^{2c+3}|\partial^J\partial_
\theta K^m_{ab}|^2({\rm cot}\theta)^2{\rm sin}\theta d\theta dtd\tau.
\end{split} \eeq
Thus the term generated by the second term in \eqref{Codazz.appl}
 can be absorbed (in \eqref{get.bd}) into the first term, 
 which is the one we are seeking to bound in \eqref{bd.this}.  
 \medskip

The ``main term'' we obtain from \eqref{Codazz.appl} 
is the one with $\onabla_a\oK^m_{\theta b}$. 
So we are reduced to controlling the term: 

\[
\int_0^r\int_{\Sigma_\tau} \tau^{2c+3}\sum_{a,b=1,2}[\partial^J  \onabla^a\oK^m_{\Theta  a}{\rm cot}\Th][\partial^J\onabla^{b} (\oK^m_{\Theta b}{\rm cot}\Th]
{\rm sin}\theta d\theta dt d\tau
\]

We then integrate by parts the derivatives $\onabla^a,\onabla^b$.
  The resulting main term (after less singular and already-bounded terms that are generated by commutations) is: 
\beq
\label{to.IBP.2}
\int_0^r\int_{\Sigma_\tau} \tau^{2c+3}\sum_{a,b=1,2}[\partial^J \onabla^a 
\oK^m_{\theta b}{\rm cot}\Th][\partial^J
\onabla^a (\oK^m)_{\theta b}{\rm cot}\Th]
{\rm sin}\theta d\theta dt d\tau.
\eeq
 
After this integration by parts, we invoke the Codazzi equations again, 
to obtain: 

\beq
\label{to.IBP.3}\begin{split}
&\int_0^r\int_{\Sigma_\tau} \tau^{2c+3}\sum_{a,b=1,2}[\partial^J  \nabla^a\oK^m_{\theta  b}{\rm cot}\Th][\partial^J\nabla^{b} (\oK^m_{\theta a}{\rm cot}\Th]
{\rm sin}\theta d\theta dt d\tau
\\&= \int_0^r\int_{\Sigma_\tau} \tau^{2c+3}\sum_{a,b=1,2}
[\partial^J \partial_\Th \oK^m_{a  b}{\rm cot}\Th][\partial^I\partial_\Th (\oK^m)^{ ab}{\rm cot}\Th]
{\rm sin}\theta d\theta dt d\tau
\\&+ \int_0^r\int_{\Sigma_\tau} \tau^{2c+3}\sum_{a,b=1,2}
[\partial^J R_{ab\Th 0}]^2[{\rm cot}\Th]^2
{\rm sin}\theta d\theta dt d\tau.
\end{split}
\eeq
The first term on the RHS are the term we wanted to obtain for our claim. 
The curvature term is bounded as in \eqref{curv.bd.1}, by first expressing that curvature term in terms of components of 
$Ric^{\uh^m}_{0i}$,\footnote{Using the fact that $\uh^m$ is a $2+1$ dimensional metric.}
in terms of derivatives of the metric components, which ultimately are controlled 
by the components $\oK^m$ that we are bounding now; the lower derivatives have already been bounded, while the higher derivatives 
are again expressibe in terms of \emph{integrals in $\tau$ of $\partial^I\oK^m$} thus these terms 
can be absorbed into the main terms.  
This completes our proof of \eqref{K.em.flux.sing}; for $i=2$ we thus derive our 
desired estimates for $K^m_{22}=\oK^m_{22}$ at the top orders. 
\medskip

The desired top-order estimates for $K^m_{12}=\oK^m_{12}$ are obtained from \eqref{get.bd}
by virtue of  the already desired bonds on ${\rm tr}\oK^m$.

\end{proof}

\subsection{Capturing the hypersurface that carries the initial data: \\
Determining the functions $r^m_*(t,\theta), \tilde{K}^m_{12}(t,\theta)$.}\label{IDsol}

We here seek to identify the hypersurface $\Sigma_{{r^m_*}}$ on which our prescribed initial 
data will be induced.

As noted in the introduction, in the gauge we have chosen two key parameters related to 
the initial data are not fixed 
apriori by us, but must be solved for at each stage of the iteration: 
These are:

a.  the function ${r^m_*}(t,\theta)$ which defines the graphical 
hypersurface 
$\Sigma_{r^m_*(t,\theta)}:=\{r=r^m_*(t,\theta)\}$ on which the prescribed  
initial data will live. 

b. The direction $e^m_2$ on the initial data hypersurface (expressed as a 
rotation of the 
background fixed frame $\E_0,\E_1, \E_2$)
 which upon transport along $e_0$ according to
  equation \eqref{gaugeit}
  yields the collapsing direction at the singularity; 
 moreover $e^m_2$ and 
 the normal direction $e^m_1$ to $e^m_2$ 
 is asymptotically a principal direction, in the sense that 
 $|^mK_{12}(r,t,
 \theta)|\cdot|K^m_{22}(r,t,
 \theta)|^{-1}=o(1)$, $|K^m_{12}(r,t,
 \theta)|\cdot|K^m_{11}(r,t,
 \theta)|^{-1}=o(1)$ as $r\to 0$. This 
 direction $^me_2$ is in 1-1 correspondence with 
 a special direction $\tilde{e}^m_2$ on $\Sigma_{r^m_*}$; the latter is captured by $\tilde{K}^m_{12}$

\medskip

As we have  seen in formula \eqref{varphi.tilK.m}, the rotation angle
 $\varphi^m(t,\theta)$
to transition from the background frame 
$(\E_1(t,\theta), \E_2(t,\theta))$ to $(\tilde{e}^m_1(t,\theta), \tilde{e}^m_2(t,\theta))$ 
is in 1-1 correspondence with the value of the 
initial data tensor $\tilde{K}$ \emph{evaluated against the frame} 
$(\tilde{e}^m_1, \tilde{e}^m_2)$
at $(t,\theta)$. 
Thus $\tilde{K}(\tilde{e}^m_1,\tilde{e}^m_2)=\tilde{K}^m_{12}(t,\theta)$ 
is the key parameter that ``sees'' the rotation of the background frame.

\subsection{The inductive step on  the functions $r^m_*(t,\theta)$, 
$\tilde{K}^m_{12}(t,\theta)$, re-formulated.}

We will be proving  the inductive claims for  
$r^m_*(t,\theta)$ and $\tilde{K}^m_{12}(t,\theta)$. We find it technically more 
convenient (for technical reasons related to the poles at $\theta=0,\pi$) to obtain our 
estimates in the spaces $L^2(d\theta dt)$ instead of $L^2(sin\theta d\theta dt)$, for 
all $\partial_\theta, \partial_t$ derivatives up to order ${\rm low-1}$.
In fact in this space the boundary condition $\te^m_2(r^m_*)=0$ at the two poles 
is more readily imposed.  At the orders 
higher than that we will obtain our estimates in the spaces defined by 
$L^2(sin\theta d\theta dt)$.
\medskip
To obtain our strengthened bound at the lower orders, we note that by virtue of the 
standard Hardy inequality in $\theta$, the bounds \eqref{1st.bds} (which hold for $k\le \rm low$ 
imply the bounds: 

\beq
\label{1st.bds.str}\begin{split}
&\int  |\partial^I[K^m_{22}(\e,t,\theta)- \K_{22}(t,\theta)  ]|^2  d\theta dt
 \le 2C^2 \eta^2 \e^{-3}, 
 \\& \int  |\partial^IK^m_{12}(\e,t,\theta) |^2  d\theta dt \le B^2 C^2 
 \e^{-2-\frac{1}{2}}\,
 \end{split}
\eeq
for all multi-indices $I, |I|\le {\rm low}-1$. 
\medskip

To distinguish estimates obtained with respect to this volume form, we will denote this
 space by $L^2( d\theta dt)$. When we consider hypersurfaces $\Sigma$ below on
  which coordinates $\theta, t$ naturally live, 
we will denote this volume form by $L^2(\Sigma, d\theta dt)$. 
The above notation will also extend to the standard Sobolev spaces $H^k(d\theta dt)$.
\emph{Outside this subsection}, when we write $L^2$ or $H^k$ 
 we will mean that $L^2$ or $H^k$  is with respect 
 to the usual volume form 
$sin\theta d\theta dt$.
We then claim:

\begin{proposition}
\label{prop.claim}
Consider the functions $K^m_{22}(r,t,\theta), K^m_{12}(r,t,\theta)$ 
defined by the formulas \eqref{K22m.int} (using \eqref{Kmijexp}), \eqref{K12m.int}. 
Then there exists a unique pair of functions 
$r^m_*(t,\theta), \tilde{K}^m_{12}(t,\theta)$ with $\e^{-1}r^m_*(t,\theta)-1, \e^{-3}\tilde{K}^m_{12}(t,\theta)$ both close to zero 
 in the 
norms below so that  the functions 
$\tilde{K}^m_{22}[\Sigma_{{r^m_*}}]$, $\tilde{K}^m_{12}[\Sigma_{{r^m_*}}]$
 defined  via the relations  \eqref{Korthexp}, \eqref{K12exp} 
 (and invoking   \eqref{tAm.def} to express the second term in \eqref{K12exp}
 in terms of the unknown $\tilde{K}^m_{12}$) in terms of $r^m_*(t,\theta)$, 
 ($\tilde{K}^m_{12}(t,\theta)=\tilde{K}_{12}^m(r^m_*(t,\theta),t,\theta)$)
  satisfy the requirement:

\beq
\label{to.solve}
\tilde{K}^m_{22}(t,\theta)= F^{22}_{t,\theta}[\tilde{K}^m_{12}(t,\theta)], 
{\rm subject}\text{ }{\rm to} \text{ } \space \space \partial_\theta r^m_*=0\space\space \text{at}
 \text{ }
\theta=0,\pi. 
\eeq

 Moreover the functions $r^m_*(t,\theta)$ 
  $\tilde{K}^m_{12}(t,\theta)$ then satisfy the following estimates:

 At the lower derivatives $|I|=k\le {\rm low}-1:=s-4-4c$: 
\beq
\label{bd1}
\sum_{i=0}^k||\partial^i(r^m_*-\e)||_{L^2(\Sigma_{r^m_*}, d\theta dt)}\le (D+1) C \eta 
\epsilon,\sum_{i=0}^{k}||
\te_2\partial ^i(r^m_*-\e)||_{L^2(\Sigma_{r^m_*}, d\theta dt)}\le 2(D+1)C \eta 
\epsilon^{-1/2}. 
\eeq

 \beq
 \label{bd2}
\sum_{i=0}^l|| \partial^{i}\tilde{K}^m_{12}||_{L^2(\Sigma_{r^m_*}, d\theta dt)}\le DC \eta \epsilon^{-3/2+\frac{1}{4}},
 \eeq
 
 At the higher derivatives $|I|=k\in \{ {\rm low},\dots, s-3\}$ we have the bounds: 
 
\beq
\label{highbd1}
\int_{-\infty}^\infty\int_0^\pi  |\te_2\partial^I 
({r^m_*})|^2sin\theta d\theta dt\le C^2 \eta^2
\epsilon^{-1 -\frac{k-{\rm low}}{2}}, 
\eeq

\beq
\label{highbd2}
 \int_{-\infty}^\infty\int_0^\pi |\partial^I ({r^m_*}-\e)|^2sin\theta d\theta dt\le 4
C^2 \eta^2 \epsilon^{2-\frac{k-{\rm low}}{2}},
\eeq

\beq
\label{highbd3}
\begin{split}
&\int_{-\infty}^\infty\int_0^\pi  
|\partial^I (\tilde{K}^m_{12})|^2sin\theta d\theta dt\le 
2C^2 \eta^2 \epsilon^{-3+\frac{1}{2}-\frac{k-{\rm low}}{2}} \text{ }\text{ }(|I|\le s-4), 
\\&\int_{-\infty}^\infty\int_0^\pi  
|\partial^I (\tilde{K}^m_{12})|^2sin\theta d\theta dt\le 
2C^2 \eta^2 \epsilon^{-3-2c}\text{ }\text{ } (|I|=s-3). \end{split}
\eeq
 
\beq
\label{extra.bound.te2K12}\begin{split}
&\int_{-\infty}^\infty\int_0^\pi  
|\partial^I (\te^m_2(r^m_*) \cdot \te^m_2\tilde{K}^m_{12})|^2sin\theta d\theta dt\le 
4C^2 \eta^2 \epsilon^{-3+\frac{1}{2}-\frac{k-{\rm low}}{2}} (|I|\le s-4),
\\&\int_{-\infty}^\infty\int_0^\pi  
|\partial^I (\te^m_2(r^m_*) \cdot \te^m_2\tilde{K}^m_{12})|^2sin\theta d\theta dt\le 
4C^2 \eta^2 \epsilon^{-3-\frac{k-{\rm low}}{2}} (|I|= s-3).
\end{split}
\eeq
At the top orders, the same bounds hold for $\partial^I$ replaced by $\partial^J\partial_\Th [(\dots){\rm cot}\Th]$.  
 
\end{proposition}

Once we have proven this proposition, the inductive claim for 
$r^m_*(t,\theta)$, $\tilde{K}^m_{12}(t,\theta)$ will be verified, and we 
may 
define $\tilde{K}^m_{11}$, \emph{on} $\Sigma_{{r^m_*}}$  as an 
   explicit function of $\tilde{K}_{12}^m$. 
This will be done in the later subsections.    
    \medskip

\subsection{Solving for $r_*^m(t,\theta), \tilde{K}^m_{12}(t,\theta)$  at the
 lower orders: weak formulation, and  solutions via a perturbation.}
 
Here we prove  Proposition \ref{prop.claim}, producing the desired bounds on 
${r^m_*}(t,\theta), \tilde{K}^m_{12}(t,\theta)$ 
only at the lower $\le {\rm low}-1$ orders. 
The higher order bounds will be obtained in the next subsection.

We will prove our result, once we suitably 
express the requirement \eqref{to.solve} (with the boundary condition at the poles)
in terms of the sought-after functions ${r^m_*}(t,\theta)$ 
and $\tilde{K}^m_{12}(t,\theta)$. 
We  first  present the key ideas in our argument:
\medskip

The equation  \eqref{to.solve} involves the function $\tilde{K}^m_{12}$, as well as the function $\tilde{K}^m_{22}(t,\theta)$. The function 
$\tilde{K}^m_{22}(t,\theta)$
will be expressed in terms of the function $K^m_{22}(r^m_*(t,\theta),t,\theta)$ via the equation \eqref{rstarlowODE.opt}
recalling that $K^m_{22}(r^m_*(t,\theta),t,\theta)$ is a function of $r^m_*$ alone, we see that the left hand side of \eqref{to.solve}
is a second-order equation in the unknown $r^m_*(t,\theta)$. (We note that the derivatives are exclusively in the direction $\te^m_2$--and this vector
field in fact depends on $\tilde{K}^m_{12}$, in view of formula \eqref{te1m.tK12}). 

Now, the RHS of \eqref{to.solve} depends on $\tilde{K}^m_{12}(t,\theta)$; this function also depends on $K^m_{12}(r^m(t,\theta),t,\theta)$, and thus in 
particular on $r^m_*(t,\theta)$. The relation is given by \eqref{key.link2'}
. In short, the pair of equations \eqref{to.solve} and \eqref{key.link2'}
 provide the 2x2 system of 
equations in our two unknowns. We recall also that on $r^m_*(t,\theta)$ we have also imposed the boundary conditions $\partial_\theta r^m_*=0$ at 
$\theta=0,\pi$. 

To obtain a solution to this system we proceed as follows: 

\begin{enumerate}
\item The very first step is to re-express our two equations in terms of $\delta r^m_*(t,\theta)=r^m_*(t,\theta)-\e$, $\tilde{K}^m_{12}(t,\theta)$. This 
permits to seek solutions via a perturbative argument, since the unknowns $\e^{-1} \delta r^m_*(t,\theta), \e^{-3/2}\tilde{K}^m_{12}(t,\theta)$
will be \emph{small}, in suitable norms. Also, the two equations have fixed RHSs, which are now known to be suitably small in suitable norms. 

\item The first difficulty is that the system \eqref{to.solve}, \eqref{key.link2'} is not in  a standard form,\footnote{A ``standard form'' would be, e.~g.~a pair of second order elliptic equations, or an ODE coupled to an algebraic equation.} where we can derive solvability. The main reason for this is 
the 
second equation, where we note the presence of a term $\te^m_2(\delta r^m_*) \te^m_2[\tilde{K}^m_{12}]$ in an equation that is otherwise zeroth order 
in $\tilde{K}^m_{12}$. To achieve a solution of this equation, we introduce a 
\emph{weak formulation} thereof. The weak formulation requires it to hold up 
to integration against test functions, via integration by parts of the extra $\te^m_2$ derivatives from the term $\te^m_2\te^m_2(\delta r^m_*)$
in the LHS of \eqref{to.solve} (having replaced this using \eqref{rstarlowODE.opt}), as well as an integration by parts of the $\te^m_2$ derivative of the most dangerous term 
in \eqref{key.link2'}. 
This leads us to the second element in the proof: 

\item It is more suitable to perform this integration by parts over the region $[0,\pi]\times \mathbb{R} $ (with the usual volume form) instead of 
$\mathbb{S}^2\times \mathbb{R}$ (with the usual volume form. The reason is the degeneration of ${\rm div}\te^m_2$ at the poles. The vanishing 
conditions $\te^m_2(\delta r^m_*)$ at the poles is captured directly 
in this weak formulation. Since we use this ``enhanced at the poles'' volume form, we will seek solutions only in the space $H^{{\rm low}-1}(d\theta dt)$ instead of $H^{\rm low}(sin\theta d\theta dt)$. 
 
 \item Even with this weak formulation, one does not obtain a solution to our system straightaway. This requires two extra steps: Firstly, we consider a 
 viscosity-enhanced version of the two equations. The \emph{linearization} of our equations would in that case be manifestly solvable, in view of a certain 
 surjectivity of the corresponding (viscosity-enhanced) linear system. We then observe that the surjectivity of the linearized system is in fact 
 \emph{uniformly} true for any value of the unknowns $\delta r^m_*$, $\tilde{K}^m_{12}$ in small enough balls balls around zero. This allows us to 
 derive the existence of a solution, given that the RHS are also suitably small. The solutions depend on the viscosity parameter ($\zeta>0$, below). 
 Our bounds at the orders $H^{{\rm low}-1}$ are in fact independent of the viscosity parameter. 
 
 \item Having obtained solutions for the viscosity-enhanced version of our weak formulation, and using the uniformity above, we can pass to the limit
  $\zeta\to 0$, and obtain the desired solutions to our equations. The regularity of these solutions at higher orders will be derived in the next subsection. 
 
\end{enumerate}

In particular, we expand out the equation \eqref{to.solve}, making use of
the formulas \eqref{Korthexp}, \eqref{K12exp} 
 \eqref{K22m.int}, \eqref{K12m.int},  as well as \eqref{tAm.def}.

 Writing $\tilde{e}_2$ for $\tilde{e}^m_2$ for short, 
 we derive:

\beq\label{rstarlowODE.opt}
-(\frac{2M}{{r^m_*}}-1)^{-\frac{1}{2}}{\tilde{e}_2}{\tilde{e}_2}{r^m_*}-(\frac{2M}{r_*}-1)^{-\frac{3}{2}}\frac{2M}{(r_*^m)^2}({\tilde{e}_2}{r_*})^2+(^mq)^2K^m_{22}(t,\theta,r^m_*)
=(^mq)F^{22}_{t,\theta}
\big{(}\tilde{K}^m_{12}\big{)}
\eeq

Recall the formula \eqref{K22m.int}:

\beq
\label{K22m.form}
K^m_{22}(t,\theta,r^m_*)=e^{w(r,t,\theta)}\int_0^{{r^m_*}}
e^{-w(s,t,\theta)}[{\rm RHS}[K^m_{22}(s,t,\theta)] ds
\eeq

Next, we recall the formula for $F^{22}_{t,\theta}(\cdot)$:

\beq
\label{F22.expn}
F^{22}_{t,\theta}\big{(} \tilde{K}_{12}(t,\theta)\big{)}=
\frac{1}{2}\bigg[1+\sqrt{1-\frac{4\tilde{K}_{12}^2}{({\bf K}_{22}-{\bf K_{11}})^2}}\bigg]\K_{22}+
 \frac{1}{2}\bigg[1-\sqrt{1-\frac{4\tilde{K}_{12}^2}{({\bf K}_{22}-{\bf K_{11}})^2}}\bigg]\K_{11}.
\eeq

Next, we subtract $K^S_{22}(\e,t,\theta)$ from both sides of the above 
equation. 
To analyze the resulting differences, we introduce the notation:  
\beq
\label{breakr*}
r^m_*-\e= \delta r^m_*. 
\eeq
\medskip

We remark that we will be seeking the variables $\delta r^m_*$, 
$\tilde{K}^m_{12}(t,\theta)$ in the certain suitably small  balls in the Banach space 
$H^{{\rm low}-1}(d\theta dt)\times H^{{\rm low}-1}(d\theta dt)$, defined in \eqref{balls.sol} below.
\medskip

We also introduce a  piece of notation: 
\begin{definition}
\label{O.low}
$O^{{\rm low}-1}(f)$ below will stand for a quantity bounded in the relevant norm $\| \cdot \|_{H^{{\rm low}-1}}$ by 
$B\| f\|_{H^{{\rm low}-1}}$ for $B$ the universal constant introduced in our introduction.
 We let $O^{{\rm low}-1}_{\eta^q\e^p}(f)$ stand for a generic function $G$ 
of the variable
 $f$ (and possibly other parameters too, such as $t,\theta$, 
 which satisfies a uniform bound for all $l\le {\rm low}-1$):
 
\[
\sum_{|I|=0}^l\sqrt{\int_{t,\theta} |\partial^I G|^2  d\theta dt }
\le B\eta^q \e^p \sum_{|I|=0}^l\sqrt{\int_{t,\theta} |\partial^I f|^2
 d\theta dt},
\] 
where $B>0$ is the uniform constant from the introduction. We can analogously define $O^{K}(f)$ for any $K\le {\rm low}-1$ 
in the obvious modification of the above definition. 
\end{definition}

We next  make the simple but key observation that: 

\beq
\label{TaylorK22}
\begin{split}
&K^m_{22}(r^m_*(t,\theta),t,\theta)-K^m_{22}(\e,t,\theta)= \sqrt{2M}
d^m_2(t,\theta)[(r^m_*(t,\theta))^{-\frac{3}{2}}-\e^{-\frac{3}{2}}]+
O^{\rm low}_{\e^{-5/2+\frac{1}{4}}}(\delta r^m_*).
\end{split}
\eeq

Prior to deriving this, let us expand out the first term on the RHS and make a note about its derivatives: 

\beq
\begin{split}
&\partial^I\big
[d^m_2(t,\theta)[(r^m_*(t,\theta))^{-\frac{3}{2}}-\e^{-\frac{3}{2}}]\big]=
\partial^I\big
[d^m_2(t,\theta)\e^{-3/2}[(1+\frac{\delta r^m_*}{\e})^{-3/2}-1]\big]
\\&=
\partial^I\big
[d^m_2(t,\theta)\e^{-3/2}[\sum_{q=1}^\infty {-\frac{3}{2}\choose q}\frac{\delta r^m_*}{\e})^{q}]\big]
\\&=\sum_{I_1\bigcup I_2=I} \partial^{I_1}[d^m_2(t,\theta)]\e^{-3/2}[\sum_{q=1}^\infty {-\frac{3}{2}\choose q}\partial^{I_2}(\frac{\delta r^m_*}{\e})^{q}].
\end{split}
\eeq

Let us note that one term (the ``main term'' for us)  in the RHS is when $I_2=I$ and $q=1$.
 For those values the ``main term'' we obtain  is: 

\[
-\frac{3}{2}\sqrt{2M} d^m_2(t,\theta)\e^{-\frac{5}{2}}\partial^I(\delta r^m_*).
\]
We note that all \emph{other} terms in the above, for $\delta r^m_*$ satisfying the bound $\|\delta r^m_* \|_{H^{\rm low}}\le DC\eta \e$
 add up to a term of the form: 

\[
O_{\eta \e^{-\frac{5}{2}}}(\partial^I \delta r^m_*).
\]
This is immediate form the explicit expression above. 
\medskip

Now let us derive \eqref{TaylorK22}:
\medskip

{\it Proof of \eqref{TaylorK22}:}
The thing that needs proof is that if we define $D(\delta r_*^m)$ via the formula: 

\beq
[D^I(\delta r_*^m)](t,\theta):=\partial^I [u^m_{22}(r^m_*(t,\theta),t,\theta)-u^m_{22}(\e,t,\theta)]
\eeq
then we need to show: 

\beq
\label{toshow}
D^I(\delta r_*^m)=O^{{\rm low}-|I|}_{\e^{-5/2+\frac{1}{4}}}(\delta r^m_*).
\eeq
To show this, let us express:

\[
[D(\partial^I\delta r_*^m)](t,\theta)=e^{w(r_*^m(t,\theta))}
\int_\e^{r_*^m(t,\theta)}
e^{-w(s)}[RHS [\partial^I \partial_s[u^m_{22}]](s,t,\theta)]ds
\]

We can then replace the terms $RHS [\partial^I  \partial_s[u^m_{22}]]$ using the expression 
in the RHS of \eqref{K22m.diff}. Each of these RHSs 
(evaluated in $L^2$ on any hypersurface 
$\Sigma_{s(t,\theta)}, s(t,\theta)\in(r^m_*(t,\theta), \e)$ 
is uniformly 
bounded by $B \e^{-3+\frac{1}{4}}$ by virtue of  the inductive  bounds we have verified on $\gamma^m$. Recalling the expression \eqref{int.factor} on $w(s,t,\theta)$, we derive the claim \eqref{toshow}. $\Box$
\medskip

 In particular, for all $|I|\le{\rm low}-1$  we have derived the equation:
\beq
\label{K22.exp.r=e}
\partial^IK^m_{22}(r^m_*(t,\theta),t,\theta)-\partial^IK^m_{22}(\e,t,\theta)=-\frac{3}{2}\sqrt{2M}d^m_2(t,\theta)\e^{-\frac{5}{2}}\partial^I(\delta r^m_*)
+O^{{\rm low}-|I|}_{\eta \e^{-5/2}}(\delta r^m_*)+O^{{\rm low}-|I|}_{\e^{-5/2+\frac{1}{4}}}(\delta r^m_*).
\eeq

We define a key term in \eqref{TaylorK22} 
to be a function $V(t,\theta)$, and note a pointwise lower bound for it: 

\beq
\label{def.V}
V(t,\theta):= -(2M)^{1/2}\frac{3}{2}d^m_2(t,\theta)\e^{-\frac{5}{2}}\ge (2M)^{1/2}\frac{3}{2}[1-\frac{1}{8}] \e^{-\frac{5}{2}}. 
\eeq
(The last inequality follows from the expression
\eqref{d2.m}
 for $d_2^m=d^m_2(t,\theta)$ 
in terms of $\alpha^m(t,\theta)$ as well as the $L^\infty$ 
bound $|\alpha^m(t,\theta)-1|\le C\eta$ and the Lipschitz bound for the function $d_2^m(\alpha^m)$).

With this remark, we observe that  equation  \eqref{to.solve} can be expressed in the following 
form, if we subtract $K^m_{22}(\e,t,\theta)$ from both sides: 
\beq
\label{pre.Phi1}
\Phi^1[\delta r^m_*, \tilde{K}^m_{12}]= F^{22}_{t,\theta}[\tilde{K}^m_{12}(t,\theta)]-K^{m}_{22}(\e,t,\theta);
\eeq 
 here the operator on the LHS is precisely: 
 
 \beq
 \label{Phi1}
 \Phi^1[\delta r^m_*, \tilde{K}^m_{12}]=
 -q^{-1}[(\frac{2M}{{r^m_*}}-1)^{-\frac{1}{2}}{\tilde{e}_2}{\tilde{e}_2}{r^m_*}
 +(\frac{2M}{r^m_*}-1)^{-\frac{3}{2}}
 \frac{2M}{r_*^2}({\tilde{e}_2}{(r_*^m)})^2]+
 V(t,\theta)\cdot (\delta r^m_*)+O^{\rm low}_{\eta \e^{-\frac{5}{2}}}(\delta r^m_*)
 +O^{\rm low}_{\e^{-5/2+\frac{1}{4}}}[(\delta r^m_*)].
  \eeq
We note also that the RHS of \eqref{pre.Phi1} can be decomposed into a \emph{fixed} term, and a term that depends on
 $\tilde{K}^m_{12}(t,\theta)$, by recalling the form of $F^{22}[\dots]$ from \eqref{F22.form}: 

\beq
\label{RHS.exp1}
 F^{22}_{t,\theta}[\tilde{K}^m_{12}(t,\theta)]-K^{m}_{22}(\e,t,\theta)= 
[\K_{22}(t,\theta)-K^{m}_{22}(\e,t,\theta)]-\frac{O^{\rm low}[(\tilde{K}^m_{12})^2]}{(\K_{11}-\K_{22})^2}\K_{11}+\frac{O^{\rm low}[(\tilde{K}^m_{12})^2]}{(\K_{11}-\K_{22})^2}\K_{22}.
\eeq 
Thus \eqref{pre.Phi1}  can be re-expressed by moving all terms that depend on $\tilde{K}^m_{12}$ to the LHS:

\beq
\label{Phi1.again} 
\begin{split}
& \tilde{\Phi}^1[\delta r^m_*, \tilde{K}^m_{12}]=
 -[(\frac{2M}{{r^m_*}}-1)^{-\frac{1}{2}}{\tilde{e}_2}{\tilde{e}_2}{r_*}
 +(\frac{2M}{r^m_*}-1)^{-\frac{3}{2}}
 \frac{2M}{r_*^2}({\tilde{e}_2}{(r_*^m)})^2+
 V(t,\theta)\cdot (\delta r^m_*)+O^{\rm low}_{\eta \e^{-\frac{5}{2}}}[\delta r^m_*]
 \\&+O^{low}_{\e^{-5/2+\frac{1}{4}}}[(\delta r^m_*)]+\frac{O^{\rm low}[(\tilde{K}^m_{12})^2]}{(\K_{11}-\K_{22})^2}\K_{11}-\frac{O^{\rm low}[(\tilde{K}^m_{12})^2]}{(\K_{11}-\K_{22})^2}\K_{22}
 = 
[\K_{22}(t,\theta)-K^{m}_{22}(\e,t,\theta)].
\end{split}
\eeq

 We should recall that $\te_2=\te^m_2$ is a vector field that  
 \emph{depends} on the 
 unknown 
 $\tilde{K}^m_{12}$, via the formula \eqref{te1m.tK12}.
  Thus \eqref{pre.Phi1} 
 (with its LHS given by \eqref{Phi1}) is one equation involving only 
  the two unknowns $r^m_*(t,\theta), \tilde{K}^m_{12}(t,\theta)$ in the LHS, and the RHS being \emph{fixed}.

We also recall the bound derived in Lemma \ref{ground.est}, which implies:
\[
\int_{t,\theta} |\partial^I(K^m_{22}(\e,t,\theta)-\K_{22}(t,\theta))|^2  d\theta dt  \le 2C^2\eta^2 \e^{-3},
\]
for all $|I|\le {\rm low}-1$.\footnote{Below $e_i$ is shorthand for $e^{m-1}_i$. }

Now, we perform the same analysis on the second equation 
\eqref{key.link2'}.   
\medskip

Our first aim is to express 
$u^m_{12}(r^m_*(t,\theta),t,\theta)=K^m_{12}(r^m_*(t,\theta),t,\theta)$
as a function of $\delta r^m_*$. For this, we refer to the integral representation \eqref{K12m.int} and again utilize \eqref{breakr*}, along with Taylor's theorem and 
\eqref{K12m.again} to obtain: 

\beq
\begin{split}
\label{K12.pert}
&[K^m_{12}(r^m_*,t,\theta)-K^m_{12}(\e,t,\theta)]=
 ( \delta r_*^m)\cdot \sqrt{\e}M^{1/2}[(-2K_{22}^m\cdot +e_0(^m\gamma))
 \cdot 
K^m_{12}](\e,t,\theta)
\\&+[^{m-1}\overline{\nabla}^2_{12}(^m\gamma) +\frac{1}{2}\big{(}
e_1(\gamma^{m-1})e_2(
 \gamma^m)+e_2(\gamma^{m-1})e_1(
 \gamma^m)\big{)}]
 (\e,t,\theta)
 +O^{{\rm low}}_{\e^{-5/2+\frac{1}{4}}}[(\delta r_*^m)].
\end{split}
\eeq

The fact that the ``remainder term'' is of the form 
$O^{{\rm low}-|I|}_{\eta\e^{-5/2+\frac{1}{4}}}[(\delta r_*^m)]$
follows readily 
by the same argument as for $\partial^I K^m_{22}$, applied this time to 
\eqref{K12m.int}
and invoking the (verified for the step $m$) inductive claim on the 
quantities
\[
\partial^I\overline{\nabla}_{ij}(^m\gamma), \partial^I(e_i(^m\gamma)), i,j
\in \{1,2\}. 
\]
In particular, in analogy with \eqref{K22.exp.r=e} we derive the following 
expression for derivatives $\partial^I, |I|\le \rm low$ 
of \eqref{K12.pert}: 

\beq
\begin{split}
\label{K12.pert.partial}
&\partial^I[K^m_{12}(r^m_*,t,\theta)-K^m_{12}(\e,t,\theta)]=
 \partial^I( \delta r_*^m)\cdot \sqrt{\e}M^{1/2}[(-2K_{22}^m\cdot +e_0(^m\gamma))
 \cdot 
 K^m_{12}](\e,t,\theta)
 \\&+\partial^I[^{m-1}\overline{\nabla}^2_{12}(^m\gamma) 
 +\frac{1}{2}\big{(}e_1(\gamma^{m-1})e_2(
 \gamma^m)+e_1(\gamma^{m})e_2(
 \gamma^{m-1})\big{)}]
 (\e,t,\theta)
\\& +O^{{\rm low}-|I|}_{\e^{-5/2+\frac{1}{4}}}[(\delta r_*^m)].
\end{split}
\eeq

We note now that the explicit terms in the RHS of the above are themselves of the form 
$O^{{\rm low}-|I|}_{\e^{-5/2+\frac{1}{4}}}[(\delta r_*^m)]$; so from this point onwards we will absorb it into that term. 

Now, we recall our second equation \eqref{key.link2'}, which we re-express as:


\beq
\begin{split}
\label{2nd.eq}
&q(\frac{2M}{r^m_*}-1)^{\frac{1}{2}}\tilde{K}^m_{12}(r_*^m(t,\theta),t,\theta)
= q(\frac{2M}{r^m_*}-1)^{\frac{1}{2}}K^m_{12}(r_*^m(t,\theta),t,\theta)
\\&+\te^m_2(r_*)\cdot \bigg{(} 1+O^{\rm low}(\frac{\tilde{K}^m_{12}}{\K_{22}-\K_{11}})  
\bigg{)}
\te^m_2[  \frac{\tilde{K}^m_{12}}{\K_{22}-\K_{11}}](t,\theta)
+({\tilde{e}^m_2}{r^m_*})F_2\big{(}
\frac{\tilde{K}^m_{12}(t,\theta), t,\theta}{\K_{22}-\K_{11}}\big{)}.
\end{split}
\eeq

Thus, 
we find that using the  expression 
\eqref{K12.pert} as well as \eqref{tAm.def} 
  we derive that   our second equation \eqref{2nd.eq} can be re-expressed in the form:

\beq
\begin{split}
\label{2nd.eq.short}
&q(\frac{2M}{r^m_*}-1)^{\frac{1}{2}}\tilde{K}^m_{12}(r_*^m(t,\theta),t,\theta))
 +O^{\rm low}_{\eta \e^{-5/2+\frac{1}{4}}}[(\delta r_*^m)]
-\te_2(\delta r^m_*)\cdot \te_2[(\K_{22}-\K_{11})^{-1}\cdot \tilde{K}^m_{12}](t,\theta)
\\&+({\tilde{e}^m_2}{r^m_*})F_2\big{(}
\frac{\tilde{K}^m_{12}(t,\theta), t,\theta}{\K_{22}-\K_{11}}\big{)}=q(\frac{2M}{r^m_*}-1)^{\frac{1}{2}}K^m_{12}(\e,t,\theta).
\end{split}
\eeq

We denote the LHS of the above by 

\beq
\label{Phi2}
\Phi^2[\delta r^m_*, \tilde{K}^m_{12}].
\eeq
 
 Thus matters are reduced to solving the system of equations 
 
\begin{align}
\label{1stPhi}
\tilde{\Phi}^1[\delta r^m_*, \tilde{K}^m_{12}]=\K^m_{22}(t,\theta)-K^m_{22}(\e,t,\theta),\\
\label{2ndPhi}
\Phi^2[\delta r^m_*, \tilde{K}^m_{12}]=qK^m_{12}(\e,t,\theta).
\end{align} 
This system is to be solved over $\theta\in [0,\pi], t\in (-\infty,+\infty)$,
 subject to the boundary conditions 
$\te^m_2(\delta r^m_*)=0$
at the poles $\theta=0,\theta=\pi$. The function $\tilde{K}^m_{12}$ is not assigned a
 boundary condition at those poles--we can derive as a consequence of the equations
 that $\tilde{K}^m_{12}=0$ at those boundaries. 
 \medskip

  We will solve this 2x2 system of equations by treating a 
  \emph{weak formulation} of this system of equations. To obtain a solution in
   this weak formulation, we will use a variant of the inverse function theorem. 
   Moreover, since our goal is to derive estimates  for our solutions $\delta r^m_*, \te^m_2(\delta r^m_*), \tilde{K}^m_{12}$ etc in the space 
   $H^{{\rm low}-1}(dt d\theta)$, 
we must  also consider the differentiated version of these equations. Thus our weak formulation will involve up to ${\rm low}-1$ of the parameters we solve for. 
It is useful to introduce a piece of notation:

\begin{remark}
\label{lots.def}
Given a multi-index $I$ with $|I|=k\in \{1,\dots, {\rm low}-1\}$,  
the terms $(l.o.t.'s)_I=(l.o.t.'s)_I(\delta r^m_*, \tilde{K}^m_{12})$ in any equation appearing below 
 will be a  linear combination of products of lower-order terms $\partial^p (\delta r^m_*)$ with $p<k$
 multiplied by terms of the form 
$\partial^q K^m_{22}(\e,t,\theta), \partial^q K^m_{12}(\e,t,\theta)$, $\partial^q \overline{\nabla}(^m\gamma)(\e,t,\theta)$, $q<k$. 

Unless mentioned otherwise, the $L^2$ norm of these terms relative to the volume form $ d\theta dt$ will be  bounded by 

\[
B \e^{-5/2+1/4}[\|\delta r^m_* \|_{H^{k-1}}+\|\tilde{K}^m_{12}]
 \|_{H^{k-1}},
\]
when appearing in the equation $\Phi^1=\dots$, and by 

\[
B \e^{1/4}\|\te_2\delta r^m_* \|_{H^{k-1}}+B\e\|\tilde{K}^m_{12}
 \|_{H^{k-1}},
\]
when appearing in the equation $\Phi^2=\dots$.
\medskip

Now, the differentiated version of \eqref{1stPhi}, \eqref{2ndPhi} is:
\end{remark}
\beq
  \begin{split}
 \label{1st.diffd}
&\partial^I[ -(\frac{2M}{{r^m_*}}-1)^{-\frac{1}{2}}{\tilde{e}^m_2}{\tilde{e}^m_2}{r_*}
+(\frac{2M}{r_*}-1)^{-\frac{3}{2}}\frac{2M}{(r^m_*)^2}({\tilde{e}^m_2}{r^m_*})^2]+
[V(t,\theta)+O(\eta \e^{-\frac{5}{2}})+O(\e^{-\frac{5}{2}+\frac{1}{4}})]\cdot  
\partial^I(\delta r^m_*)
 \\&+\partial^I \tilde{K}^m_{12}\frac{O(\tilde{K}^m_{12})}{\K_{11}-\K_{22}}
 +(l.o.t.'s)^1_I(\delta r^m_*, \tilde{K}^m_{12})
 = \partial^I [\K_{22}(t,\theta)-K^m_{22}(\e,t,\theta)], 
\end{split}
\eeq

\beq
\begin{split}
\label{2nd.diffd}
&q(\frac{2M}{r^m_*}-1)^{\frac{1}{2}}\partial^I\tilde{K}^m_{12}(t,\theta)
=
\te^m_2(r_*)\cdot \bigg{(} 1+O^{\rm low}(\frac{\tilde{K}^m_{12}}{\K_{22}-\K_{11}})  
\bigg{)}
[  \frac{\te^m_2\partial^I\tilde{K}^m_{12}}{\K_{22}-\K_{11}}](t,\theta)
\\&+\te^m_2\partial^I(r_*)\cdot \bigg{(} 1+O^{\rm low}
(\frac{\tilde{K}^m_{12}}{\K_{22}-\K_{11}})  
\bigg{)}
\te^m_2[  \frac{\tilde{K}^m_{12}}{\K_{22}-\K_{11}}](t,\theta)
\\&+({\tilde{e}^m_2}\partial^I{r^m_*})F_2\big{(}
\frac{\tilde{K}^m_{12}(t,\theta)}{\K_{22}-\K_{11}}\big{)}+
+({\tilde{e}^m_2}{r^m_*})(F_2)^\prime\big{(}
\frac{\tilde{K}^m_{12}(t,\theta)}{\K_{22}-\K_{11}}\big{)}\cdot 
\frac{\partial^I\tilde{K}^m_{12}(t,\theta)}{\K_{22}-\K_{11}}
\\&+O_{\eta \e^{-\frac{5}{2}}}(\partial^I(\delta r^m_*))
+(l.o.t.'s)^2_I(\delta r^m_*, \tilde{K}^m_{12})= \partial^I[q(\frac{2M}{r^m_*}-1)^{\frac{1}{2}} K^m_{12}(\e,t,\theta)].
\end{split}
\eeq

(The terms  $(l.o.t.'s)^1_I(\dots)$, $(l.o.t.'s)^2_I(\dots)$
are quadratic terms that arise from the product rule. We do not record their precise form here but note that they satisfy the bounds in remark 
\eqref{lots.def}). 
  \medskip

  We will solve the pair of equations \eqref{1stPhi}, \eqref{2ndPhi} 
  in the two unknowns by using a weak formulation of these same equations. To do this, 
we recall that the RHSs in  \eqref{1stPhi}, \eqref{2ndPhi} satisfy the following bounds 
in 
$H^k( d\theta dt)$, $k= {\rm low}-1$ in view of Lemma \ref{ground.est}:
\begin{align}
 \int_0^\pi\int_{-\infty}^\infty|\partial^I {\rm RHS} [\eqref{1stPhi}]|^2
  d\theta dt\le 2C^2 \eta^2\e^{-3},
\int_0^\pi\int_{-\infty}^\infty|\partial^I {\rm RHS} [\eqref{2ndPhi}]|^2
 d\theta dt\le C^2\eta^2\e^{-3+\frac{1}{2}}.
\end{align}

  We now specify the notion of solution that we will use: 
 The \emph{solution} to  the system of  equations \eqref{1stPhi}, \eqref{2ndPhi} is  in the sense of integration by parts. Specifically 
if we expand out those two equations in terms of the sought-after variables $\delta r^m_*, \tilde{K}^m_{12}$, multiply against test functions $v, u$ 
and integrate by parts one of the derivatives $\te^m_2$ from the first term in 
\eqref{1st.diffd} and the one derivative $\te^m_2$ from the term 
$\te^m_2\partial^I[ \tilde{K}^m_{12}](t,\theta)\cdot (\K_{22}-\K_{11})^{-1}$ from \eqref{2nd.diffd}
then the corresponding \emph{integral} identity should hold. 

To put forward the weak formulation of this system of equations, 
let us derive the integration by parts of the vector fields $\te^m_2$ that we perform, 
as well as the boundary terms at the poles $\theta=0,\theta=\pi$ that it gives rise to. 

We recall that: 
\beq
\label{te2.exp}
\te^m_2[F]= ((a^m)^{2\theta}\partial_\theta+(a^m)^{2t}\partial_t)[F].
\eeq

Recall that $(a^m)^{2\theta}$, $(a^m)^{2t}$ are determined by $\tilde{K}^m_{12}$ via the coefficients in the formula \eqref{te1m.tK12}.

Thus, integration by parts of the outside $\te^m_2$ (integrating against a general test function $v$) works as follows: 

\beq
\begin{split}
&-\int_{-\infty}^\infty\int_0^\pi \te^m_2\te^m_2[F]\cdot v d\theta dt= 
\int_{-\infty}^\infty\int_0^\pi \te^m_2[F]\cdot \te^m_2v d\theta dt
-\int_{-\infty}^\infty\int_0^\pi(\partial_\theta (a^m)^{2\theta}+\partial_t(a^m)^{2t})
\te^m_2 [F] \cdot v d\theta dt
\\&+\int_{-\infty}^\infty (a^m)^{2\theta} \te^m_2 \partial^I[F](\theta=0,t)
\cdot v  dt- \int_{-\infty}^\infty (a^m)^{2\theta} \te^m_2 \partial^I[F](\theta=\pi,t)
\cdot v  dt.
\end{split}
\eeq

In view of this and the system \eqref{1stPhi}, \eqref{2ndPhi}, 
 then up to ${\rm low}-1$ derivatives 
of the solution should satisfy the following  \emph{weak} version of the two equations, 
where 
\[
({\rm l.o.t.'s})^1_I[\delta r^m_*, \tilde{K}^m_{12}], ({\rm l.o.t.'s})^2_I
[\delta r^m_*, \tilde{K}^m_{12}]
\]
stand for  a general linear combination of lower-order terms as above, which now include 
such terms that arise from the commutations of derivatives in order to perform the 
integrations by parts we just described: 

We also introduce a further piece of notation: We let $H^1_{\rm even}[d\theta dt]$ 
over $[0,\pi]\times \mathbb{R}$ to stand for the space of functions whose 
(distributional) derivatives $\partial^{I_1,I_2}_{\theta\dots \theta t\dots t} v$ 
with $I_1$ being odd vanishing at $\{\theta=0\}, \{\theta=\pi\}$, in the sense of
 distributions. 

\begin{definition}
\label{weak.soln.def} 
We consider functions 
$(\delta r^m_*, \tilde{K}^m_{12})\in H^{{\rm low}-1}(d\theta dt)\times H^{{\rm low}-1}(d\theta dt)$. 
We call this pair a \emph{weak solution} of the system 
\eqref{1stPhi}, \eqref{2ndPhi} if any only if for 
 any  multi-index $I, |I|=k\le {\rm low}-1$ and any  fixed pair of functions 
$v\in H^1_{\rm even}((0,\pi)\times\mathbb{R}), u\in H^1((0,\pi)\times
\mathbb{R})$. 
We have:

\beq
\begin{split}
\label{weak.sol}
&\int_{t,\theta} (\frac{2M}{{r^m_*}}-1)^{-\frac{1}{2}}\te^m_2\partial^I (\delta r^m_*)\cdot \te^m_2 v 
-(\frac{2M}{r^m_*}-1)^{-\frac{3}{2}}\frac{2M}{r_*^2}\partial^I[({\tilde{e}_2}{r^m_*})^2]v 
\\&+[V+O(\eta\e^{-\frac{5}{2}})+O(\e^{-\frac{5}{2}+\frac{1}{4}}))-\te^m_2\te^m_2(\frac{2M}{r^m_*}-1)^{-\frac{1}{2}}+(\frac{2M}{r^m_*}-1)^{-\frac{1}{2}}[\tilde{div}(\te^m_2)]^2+\te^m_2[\tilde{div}(\te^m_2)] ]\cdot 
[\partial^I(\delta r^m_*)]\cdot v
\\&+[\frac{O(\tilde{K}^m_{12})}{\K_{22}-\K_{11}}] \partial^I \tilde{K}_{12}^m \cdot  v    d\theta dt
+ \int_{t,\theta} 
({\rm l.o.t.'s})^1_I[\delta r^m_*, \tilde{K}^m_{12}] v  d\theta dt
\\&= \int_{t,\theta} 
q_m^2\partial^I[\K_{22}(t,\theta)-K^{m}_{22}(\e,t,\theta)]vdt  d\theta,
\\&  \int_{t,\theta}\partial^I \tilde{K}^m_{12}(t,\theta)\cdot u dtd\theta
+\int_{t,\theta}\te^m_2(\delta r^m_*)\cdot \partial^I[(\K_{22}-\K_{11})^{-1}\cdot \tilde{K}^m_{12}]
(t,\theta)\cdot \te^m_2 u dtd\theta
\\&+\int_{t,\theta}[\te^m_2\te^m_2(\delta r^m_*)+\tilde{div}(\te^m_2) \te^m_2(\delta r^m_*) ]
\cdot \partial^I[(\K_{22}-\K_{11})^{-1}\cdot \tilde{K}^m_{12}]
(t,\theta)\cdot (\frac{2M}{r_*^m}-1)^{-\frac{1}{2}} u dtd\theta 
\\&-\int_{t,\theta} (\frac{2M}{r_*^m}-1)^{-\frac{1}{2}}\te^m_2[(\K_{22}-\K_{11})^{-1}\cdot \tilde{K}_{12}]\cdot \partial^I \te^m_2(\delta r^m_*)
\\&+
(F^1)^\prime
[(\K_{22}-\K_{11})^{-1}\tilde{K}^m_{12}](\K_{22}-\K_{11})^{-1}\partial^I
\tilde{K}^m_{12}(t,\theta)\cdot u dtd\theta 
\\&+\int_{t,\theta} [O(\eta\e^{-\frac{5}{2}})+O(\e^{-\frac{5}{2}+\frac{1}{4}})]\cdot \partial^I(\delta r^m_*)(t,\theta)\cdot u
+(l.o.t.'s)^2_I(\delta r^m_*, \tilde{K}^m_{12})\cdot u dtd\theta 
\\&=
\int_{t,\theta} \partial^I K^m_{12}(\e,t,\theta)\cdot u dtd\theta 
\end{split}
\eeq
\end{definition}

\begin{remark}
\label{tdiv.def}
Let us check how strong solutions of the original system \eqref{1stPhi}, 
\eqref{2ndPhi} would yield solutions in this weak sense: Strong solutions in 
$H^{{\rm low}-1}$  would be differentiated by $\partial^I$, yielding solutions of 
\eqref{1st.diffd}, \eqref{2nd.diffd}. In these, we would integrate by parts one of the derivatives $\te^m_2$ in the first term $\partial^I \te^m_2\te^m_2r_*$, as well as the derivative $\te^m_2$ in \eqref{2nd.diffd} from the term 
$\te^m_2\partial^I\tilde{K}^m_{12}$. This thus ``offloads'' extra derivatives onto the test functions.

In particular, in the above $\tilde{div}(\te_2)$ is defined as the divergence of the vector field 
$\te^m_2$ \emph{with respect to the Euclidean metric $dt^2+d\theta^2$}, where the vector field $\te^m_2$  is expressed as a linear combination of $\partial_\theta, \partial_t$ 
via formula \eqref{te1m.tK12}, see \eqref{div.te2m} below. These terms are some factors that arise in the integrations by parts mentioned above. 
\end{remark}

\begin{remark}
\label{no.bdry.term}The above definition  is a fairly `canonical' definition of  a weak solution using duality: Indeed, the above is equivalent to choosing functions 
$V\in H^1_{\rm even}[dtd\theta]$, $U \in H^1[dtd\theta]$
considering their (distributional) derivatives $\partial^I U, \partial^IV$, and acting by these distributions on the two  equations \eqref{1stPhi}, 
\eqref{2ndPhi}.

The system above is then obtained by integrations by parts of the derivatives
 $\partial^I$, but \emph{also} by integrations by parts of one derivative $\te^m_2$ from the term $\te^m_2\te^m_2(\delta^m_*)$ in \eqref{1stPhi}, \emph{as well as} 
 integrations by parts of the derivative $\te^m_2$ on the term containing 
 $\te^m_2\tilde{K}^m_{12}$ in the second equation, as explained
  in the previous remark. 
 \medskip
 
We note that the imposed boundary conditions $\partial_\theta (\delta r^m_*)=0$ at the two boundaries $\theta=0, \theta=\pi$ are captured in the above weak formulation, 
by the \emph{absence} of boundary terms 
\[
\int_{-\infty}^\infty \partial^{I'}\te^m_2(\delta r^m_*)\cdot v(\theta,t)dt, 
\int_{-\infty}^\infty \te^m_2(\delta r^m_*)\cdot \partial^J\tilde{K}^m_{12}\cdot u(\theta, t) dt,
\]
$|I'|$ even, $|J|\le |I|-1$ 
at the boundaries $\theta=0,\theta=\pi$. Such terms would arise in the integrations by parts of $\partial^I$; their absence in the above precisely imposes the imposed boundary conditions. (The terms that arise with $|I'|$ odd would also have a $\partial^{I''}V$ with $|I''|$ odd, and thus vanish by the choice of test functions $V$).
\end{remark}

 We will be deriving a solution in the weak sense described above. This 
solution will be obtained by combining the inverse function theorem 
with a viscosity-type modification of the system \eqref{1stPhi}, 
\eqref{2ndPhi} [re-cast as 
\eqref{weak.sol}].

 We start with setting the ground for the inverse function theorem. In particular, we will be seeking solutions
 $(\delta r^m_*)(t,\theta), \tilde{K}^m_{12}(t,\theta)$, in the following subsets of 
 the Banach space $H^{{\rm low}-1}$:\footnote{The notion of ${\cal B}_{\rm even}$ 
 naturally extends to this setting.} 
 
 \beq
 \label{balls.sol}
\begin{split}
{\cal B}:=\{&||(\delta r^m_*)||_{\dot{H}^l(\Sigma_{r_*})}\le DC \eta \epsilon, \|
\te_2(r^m_*-\e)||_{\dot{H}^l(\Sigma_{r^m_*})}\le 2DC \eta \epsilon^{-1/2},
\\&||
\te^m_2\te^m_2\partial ^i(r^m_*-\e)||_{\dot{H}^l(\Sigma_{r^m_*})}\le 5DC \eta \epsilon^{-2},
 || \tilde{K}^m_{12}||_{\dot{H}^l(\Sigma_{r^m_*})}\le D C \eta 
\epsilon^{-3/2+\frac{1}{4}}, \forall l\le {\rm low}-1 \}
\end{split}
 \eeq
(Note again that the vector field $\te^m_2(t,\theta)$ 
is implicitly defined by 
$\tilde{K}^m_{12}(t,\theta)$ via formula \eqref{te1m.tK12}).  
\medskip

We start by  recalling that the definition of the space $\cal B$ and the Sobolev embedding show that over $\cal B$: 
\beq
\label{Sob.appl}
|\te^m_2(^m\delta r^m_*)|\le DC C_{\rm Sob}\eta\e^{-1/2}, |\te^m_2\te^m_2
(\delta r^m_*)|\le D C C_{\rm Sob}\eta\e^{-2}
\eeq

\medskip

Note also that $\tilde{div}(\te^m_2)$ 
as defined in remark \ref{tdiv.def} can be bounded as follows, letting 
$x=\tilde{K}^m_{12}(\K_{22}-\K_{11})^{-1}$: 
\beq
\label{div.te2m}
\tilde{div}(\te^m_2)=\partial_\theta\{
\frac{1}{2}\sqrt{1+\sqrt{1-4x^2}}
(\g_{\theta\theta})^{-1/2}\}+\partial_t\{\frac{{\rm sign}(x)}{2}\sqrt{1-\sqrt{1-4x^2}}
(\g_{tt})^{-1/2}\} 
\eeq
Thus  $\tilde{div}(\te^m_2)$ is pointwise  bounded by 
$2CC_{\rm Sob}\eta\cdot \e^{-1-1/4}$ in $\cal B$.
\medskip

Let us note that these bounds imply that any pair of functions  
$(\delta r^m_*)(t,\theta), \tilde{K}^m_{12}(t,\theta)$ 
in the above space we have the bound: 
\[
[V(t,\theta)+O(\eta\e^{-\frac{5}{2}})+O(\e^{-\frac{5}{2}+\frac{1}{4}})
-\te^m_2\te^m_2(\frac{2M}{{r^m_*}}-1)^{-\frac{1}{2}}]+
 (\frac{2M}{{r^m_*}}-1)^{-\frac{1}{2}}\cdot \tilde{div}(\te^m_2)\ge \sqrt{2M}(\frac{3}{2}-\frac{1}{4})\e^{-5/2}.
\] 

We will use the above  momentarily to derive the existence of a solution to the system \eqref{1stPhi},
 \eqref{2ndPhi}, in the weak sense described in \eqref{weak.sol}, in the space \eqref{balls.sol}. Thus 
 we will be deriving the claimed bounds \eqref{bd1}, \eqref{bd2}.

\begin{remark}
 We note that the standard Sobolev embedding on $[0,\pi]\times\mathbb{R}$ 
 yields that any weak  solutions belong to ${\cal C}^{{\rm low}-3}$, and the 
 fact 
 that they satisfy the solutions weakly (in the sense of \eqref{weak.sol})  then implies that they 
 also solve the equations in the classical sense in that space. This follows by an approximation of the identity argument for the test functions, 
and an integration 
by parts off of $v, u$.  
  \end{remark}
\medskip

  To obtain our solutions and our bounds, we work not with the equations
  \eqref{weak.sol} 
  directly, but rather with a viscosity version of these equations.
  (We do this to make use of the Lax-Milgram theorem 
  further down--it is more convenient to apply this theorem instead 
  of working on spaces that are dependent on the solution-dependent vector field 
  $\te^m_2$).  
  
  Letting $\Delta$ be the standard Laplacian on $[0,\pi]\times\mathbb{R}$, 
  (with $t\in\mathbb{R},\theta\in [0,\pi]$)
  $\partial^2_{tt}+\partial^2_{\theta\theta}$,   for any 
  $\zeta>0$ we consider the modified equations that arise from \eqref{1stPhi}, \eqref{2ndPhi}  by adding  a 
  term 
  $-\zeta \cdot \Delta$ leading order term to each side. 
We do not write out the \emph{modified} equations here, but rather skip to the (modified) \emph{weak solutions} that we consider for 
the new equations:

  Note that the solutions depend on $\zeta>0$ now, and in particular the vector fields 
  $\te_2^\zeta$ (being given by 
  $^m\tilde{K}_{12}^\zeta$ via formula \eqref{te1m.tK12}) depend on $\zeta$ also, since 
  $^m\tilde{K}_{12}^\zeta$
    depends on $\zeta$.

We again use the same notation as in definition \ref{lots.def} for the lower-order terms $({\rm l.o.t.'s})^j_I, j=1,2$.

The notion of solution for this viscosity-altered system of 
 equations is then again in the following weak sense, for any $|I|\le {\rm low}-1$ and any 
 $(v,u)\in [H^1_{\rm even}\times H^1](\mathbb{S}^2\times\mathbb{R})$:  
  \beq
\begin{split}
\label{weak.sol.delta}
&\int_{t,\theta}\zeta\cdot \sum_{i=t,\theta}\partial^I\partial_i(^m\delta r_*^\zeta)
\partial_i 
v  d\theta dt
+\int_{t,\theta} (\frac{2M}{{r^{m,\zeta}_*}}-1)^{-\frac{1}{2}}\te^{m,\zeta}_2\partial^I (\delta r^{m,\zeta}_*)\cdot \te^{m,\zeta}_2 v 
\\&-(\frac{2M}{r^{m,\zeta}_*}-1)^{-\frac{3}{2}}\frac{2M}{(r^{m,\zeta}_*)^2}\partial^I[({\tilde{e}^{m,\zeta}_2}{r^{m,\zeta}_*})^2]v 
\\&+[V(t,\theta)+O(\eta\e^{-\frac{5}{2}})+O(\e^{-\frac{5}{2}+\frac{1}{4}})-\te^{m,\zeta}_2\te^{m,\zeta}_2(\frac{2M}{r^{m,\zeta}_*}-1)^{-\frac{1}{2}}
\\&+(\frac{2M}{r^{m,\zeta}_*}-1)^{-\frac{1}{2}}[\tilde{div}(\te^{m,\zeta}_2)]^2+\te^{m,\zeta}_2[\tilde{div}(\te^{m,\zeta}_2)] ]\cdot 
[\partial^I(\delta r^{m,\zeta}_*)]\cdot v
\\&+[\frac{O(\tilde{K}^{m,\zeta}_{12})}{\K_{22}-\K_{11}}] \partial^I \tilde{K}_{12}^{m,\zeta} \cdot  v    d\theta dt
+ \int_{t,\theta} 
({\rm l.o.t.'s})^1_I[\delta r^{m,\zeta}_*, \tilde{K}^{m,\zeta}_{12}] v  d\theta dt
\\&= \int_{t,\theta} 
q_m^2\partial^I[\K_{22}(t,\theta)-K^{m}_{22}(\e,t,\theta)]vdt  d\theta,
\\& \int_{t,\theta}\zeta\cdot \sum_{i=t,\theta}\partial^I\partial_i
(\tilde{K}^{m,\zeta}_{12})
\partial^I\partial_i u  d\theta dt+
\\&  \int_{t,\theta}\partial^I \tilde{K}^{m,\zeta}_{12}(t,\theta)\cdot u dtd\theta
+\int_{t,\theta}\te^m_2(\delta r^{m,\zeta}_*)\cdot \partial^I[(\K_{22}-\K_{11})^{-1}\cdot \tilde{K}^{m,\zeta}_{12}]
(t,\theta)\cdot \te^{m,\zeta}_2 u dtd\theta
\\&+\int_{t,\theta}[\te^{m,\zeta}_2\te^{m,\zeta}_2(\delta r^{m,\zeta}_*)+\tilde{div}(\te^{m,\zeta}_2) \te^{m,\zeta}_2(\delta r^{m,\zeta}_*) ]
\cdot \partial^I[(\K_{22}-\K_{11})^{-1}\cdot \tilde{K}^{m,\zeta}_{12}]
(t,\theta)\cdot (\frac{2M}{r_*^{m,\zeta}}-1)^{-\frac{1}{2}} u dtd\theta 
\\&-\int_{t,\theta} (\frac{2M}{r_*^{m,\zeta}}-1)^{-\frac{1}{2}}\te^{m,\zeta}_2[(\K_{22}-\K_{11})^{-1}\cdot \tilde{K}_{12}]\cdot \partial^I \te^{m,\zeta}_2(\delta r^{m,\zeta}_*)
\\&+
(F^1)^\prime
[(\K_{22}-\K_{11})^{-1}\tilde{K}_{12}](\K_{22}-\K_{11})^{-1}\partial^I
\tilde{K}^{m,\zeta}_{12}(t,\theta)\cdot u dtd\theta 
\\&+\int_{t,\theta} [O(\eta\e^{-\frac{5}{2}})+O(\e^{-\frac{5}{2}+\frac{1}{4}})]\cdot 
\partial^I(\delta r^{m,\zeta}_*)(t,\theta)\cdot u
+(l.o.t.'s)^2_I(\delta r^{m,\zeta}_*, \tilde{K}^{m,\zeta}_{12})\cdot u dtd\theta 
\\&=
\int_{t,\theta} \partial^I K^{m}_{12}(\e,t,\theta)\cdot u dtd\theta 
\end{split}
\eeq
We note that in the second equation, the coefficient in front of 
 $\partial^I\tilde{K}^m_{12}\cdot u$ (after we group up terms) 
 is of the form $(1+O(\eta))\ge \frac{7}{8}$.

Let us derive the existence of a weak solution to this viscosity-enhanced system in 
detail, for completeness:
\medskip

 Recall that the factors in the above two equations that do not involve the functions
  $v,u$ are derivatives of fixed expressions; in particular 
  if $\partial^J= \partial^{I'}\partial^{I_2}$  then the equation\ above 
  with $I=J$ arises from the equation with $I=I_2$ by replacing $v, u$ by 
  $(-1)^{|I'|}\partial^{I'}v, (-1)^{|I'|}\partial^{I'}u$ and performing integrations by 
  parts for those terms. 
In particular, this allows us to consider the LHSs of the above two operators (where 
$\delta r^{m,\zeta}_*,\tilde{K}^{m,\zeta}_{12}$ are two fixed functions,
in a space that is soon to be specified) as elements of the dual space to 
$(H^{2-\rm low}\times H^{2-\rm low}) $:
  
Note that for 
  $ ((H^{{\rm low}-2})^*, (H^{{\rm low}-2})^*)=(H^{1-({\rm low}-1)}$, 
 $ H^{1-({\rm low}-1)})$ we consider the above system of equations with $I=\emptyset$, and choose 
  $v=\partial^JV, u=\partial^JU$ for any fixed $V\in H^1_{\rm even},U\in H^1)$ and 
  $|J|\le {\rm low}-1$. Integrating by parts in $\partial^J$ yields the above 
  system with $I=J$. 
  
  We can thus think of \eqref{weak.sol.delta} (for functions 
  $\delta r^{m,\zeta}_*, \tilde{K}^{m,\zeta}_{12}$ in $H^{{\rm low}}$ and for 
  $I=\emptyset$) as providing 
  an element in the space $((H^{1-({\rm low}-1)})_{\rm even}\times 
  (H^{1-({\rm low}-1)}))^*$, via the above 
  procedure. We note that the dependence of the element on $\delta r^{m,\zeta}_*, \tilde{K}^{m,\zeta}_{12}$
  is non-linear. 
  Denote this operator in $((H^{1-({\rm low}-1)})_{\rm even}\times 
  (H^{1-({\rm low}-1)}))^*$ by 
  $\{\Psi_\zeta(\delta r^{m,\zeta}_*, \tilde{K}^{m,\zeta}_{12})\}[v,u]$. In fact for each choice $v=\partial^IV, u=\partial^IU$ we can denote it for short by 
  $(\Psi^1_{\zeta,I}[v], \Psi^2_{\zeta,I}[u])_{|I|\le {\rm low}-1}$, where 
  
  \beq\begin{split}
 & \Psi^1_{\zeta,I}[V]=\int_{t,\theta}\zeta\cdot \sum_{i=t,\theta}\partial^I\partial_i(^m\delta r_*^\zeta)
\partial_i 
V  d\theta dt
+\int_{t,\theta} (\frac{2M}{{r^{m,\zeta}_*}}-1)^{-\frac{1}{2}}\te^{m,\zeta}_2\partial^I (\delta r^{m,\zeta}_*)\cdot \te^{m,\zeta}_2 V
\\&-(\frac{2M}{r^{m,\zeta}_*}-1)^{-\frac{3}{2}}\frac{2M}{(r^{m,\zeta}_*)^2}\partial^I[({\tilde{e}^{m,\zeta}_2}{r^{m,\zeta}_*})^2]V 
\\&+[V(t,\theta)+O(\eta\e^{-\frac{5}{2}})+O(\e^{-\frac{5}{2}+\frac{1}{4}})-\te^{m,\zeta}_2\te^{m,\zeta}_2(\frac{2M}{r^{m,\zeta}_*}-1)^{-\frac{1}{2}}\\&+(\frac{2M}{r^{m,\zeta}_*}-1)^{-\frac{1}{2}}[\tilde{div}(\te^{m,\zeta}_2)]^2+\te^{m,\zeta}_2[\tilde{div}(\te^{m,\zeta}_2)] ]\cdot 
[\partial^I(\delta r^{m,\zeta}_*)]\cdot V
\\&+[\frac{O(\tilde{K}^{m,\zeta}_{12})}{\K_{22}-\K_{11}}] \partial^I \tilde{K}_{12}^{m,\zeta} \cdot  V   d\theta dt
+ \int_{t,\theta} 
({\rm l.o.t.'s})^1_I[\delta r^{m,\zeta}_*, \tilde{K}^{m,\zeta}_{12}] V d\theta dt,
\\&\Psi^2_{\zeta,I}[U]=\int_{t,\theta}\zeta\cdot \sum_{i=t,\theta}\partial^I\partial_i
(\tilde{K}^{m,\zeta}_{12})
\partial^I\partial_i U  d\theta dt+
\\&  \int_{t,\theta}\partial^I \tilde{K}^{m,\zeta}_{12}(t,\theta)\cdot U dtd\theta
+\int_{t,\theta}\te^m_2(\delta r^{m,\zeta}_*)\cdot \partial^I[(\K_{22}-\K_{11})^{-1}\cdot \tilde{K}^{m,\zeta}_{12}]
(t,\theta)\cdot \te^{m,\zeta}_2 U dtd\theta
\\&+\int_{t,\theta}[\te^{m,\zeta}_2\te^{m,\zeta}_2(\delta r^{m,\zeta}_*)+\tilde{div}(\te^{m,\zeta}_2) \te^{m,\zeta}_2(\delta r^{m,\zeta}_*) ]
\cdot \partial^I[(\K_{22}-\K_{11})^{-1}\cdot \tilde{K}^{m,\zeta}_{12}]
(t,\theta)\cdot (\frac{2M}{r_*^{m,\zeta}}-1)^{-\frac{1}{2}} U dtd\theta 
\\&-\int_{t,\theta} (\frac{2M}{r_*^{m,\zeta}}-1)^{-\frac{1}{2}}\te^{m,\zeta}_2[(\K_{22}-\K_{11})^{-1}\cdot \tilde{K}_{12}]\cdot \partial^I \te^{m,\zeta}_2(\delta r^{m,\zeta}_*)
\\&+
(F^1)^\prime
[(\K_{22}-\K_{11})^{-1}\tilde{K}_{12}](\K_{22}-\K_{11})^{-1}\partial^I
\tilde{K}^{m,\zeta}_{12}(t,\theta)\cdot U dtd\theta 
\\&+\int_{t,\theta} [O(\eta\e^{-\frac{5}{2}})+O(\e^{-\frac{5}{2}+\frac{1}{4}})]\cdot 
\partial^I(\delta r^{m,\zeta}_*)(t,\theta)\cdot U
+(l.o.t.'s)^2_I(\delta r^{m,\zeta}_*, \tilde{K}^{m,\zeta}_{12})\cdot U dtd\theta 
\end{split}
  \eeq
   We denote the norm of this operator in
  this latter space by 
  \[  
  \| \Psi_\zeta \|_{B}=sup_{\|V\|^2_{H^1}+\|U\|^2_{H^1}\le 1} 
\{  \sum_{|I|\le {\rm low}-1}
  \Psi^1_{\zeta,I}[V]+\Psi^2_{\zeta,I}[U]\}.
    \] 

 Thus, we are reduced to proving the existence of functions 
 $ \delta r_*^{m,\zeta}, \tilde{K}_{12}^{m,\zeta}\in ({\cal B}_{\rm even}, {\cal B}):=\Omega$
 so that: 
 
 \beq
 \label{new.symb}
 \big\langle\Psi^1_\zeta(\delta r^{m,\zeta}_*, \tilde{K}^{m,\zeta}_{12}), v\big\rangle= \int_{t,\theta} q_m^2[\K_{22}(t,\theta)-K^m_{22}(\e,t,\theta)]\cdot
  v dtd\theta , \big\langle\Psi^2_\zeta (\delta r^{m,\zeta}_*, \tilde{K}^{m,\zeta}_{12}), u\big\rangle=
  \int_{t,\theta} K^m_{12}(\e,t\theta) \cdot udtd\theta, 
 \eeq
  for all $u\in H^{2-{\rm low}}= (H^{1-({\rm low}-1)})$.
 
 We note that the $\Psi^1_\zeta, \Psi^2_\zeta$ depend nonlinearly on 
 $\delta r^{m,\zeta}_*, \tilde{K}^{m,\zeta}_{12}$. So we will show this 
 claim by considering the linearization of $\Psi_\zeta$ in the parameters 
 $\delta r^{m,\zeta}_*, \tilde{K}^{m,\zeta}_{12}$.
 
 We denote this linearization by $(D\Psi_\zeta)_{\tilde{v},\tilde{u}}$,
 where $\tilde{v}$ stands for a first variation of $\delta r^{m,\zeta}_*$ and 
 $\tilde{u}$ for a first variation of $\tilde{K}^{m,\zeta}_{12}$.
Thus we can consider  $(D\Psi^i_\zeta)_{\tilde{v}}, D\Psi^i_\zeta)_{\tilde{u}}$, 
$i=1,2$; each of 
these operators can be evaluated against $\partial^IV\in H^{2-\rm low}_{\rm even}$ for $i=1$ and against 
$\partial^IU\in H^{2-\rm low}$ for $i=2$.

We wish to show the existence of solutions to \eqref{new.symb}. This can be achieved 
by showing that this linearization  is uniformly surjective for all 
$\delta r^{m,\zeta}_*$, $\tilde{K}^{m,\zeta}_{12}\in \Omega$.

In particular we consider the linearization matrix: 
 \[
 \begin{bmatrix}
(D\Psi^1_\zeta)_{\tilde{v}} & (D\Psi^1_\zeta)_{\tilde{u}}\\ 
(D\Psi^2_\zeta)_{\tilde{v}} & (D\Psi_\zeta)_{\tilde{u}}
 \end{bmatrix}
 \]
 acting on pairs of functions of the form $(v,u)\in H^{2-\rm low}_{\rm even}\times H^{2-\rm low} $.


 In particular if we choose $v=\tilde{v}:=\partial^I(V)$ and 
 $u=\tilde{u}:=\partial^I U$, 
 for all $|I|\le {\rm low}-1$ and then sum in $I$ then we obtain lower 
  bounds on the 
 norms of the components of  $D\Psi_\zeta$ by these 
  evaluations  against those test functions:

 \beq
\begin{split}\label{coercivity}
 &\sum_{|I|=0}^{{\rm low}-1}  \big\langle (D\Psi^1_\zeta)_{\tilde{v}},v\big\rangle
 \ge \sum_{|I|=0}^{{\rm low}-1}  \int_{t,\theta }\zeta[\partial_i
  v]^2+
[\te_2 v]^2 
\\&+[V+O(\e^{-\frac{5}{2}+\frac{1}{4}})
 +\tilde{div}(\te_2^\zeta)+\te_2^\zeta\te_2^\zeta(\delta r^{m,\zeta}_*)](t,\theta)\cdot
[ v]^2 +O(C\eta\e^{\frac{1}{4}})\cdot  v\cdot u
 d\theta dt,
  \\&\sum_{|I|=0}^{{\rm low}-1}\big\langle(D\Psi^2_\zeta))_{\tilde{u}},u\big\rangle
\ge \sum_{|I|s=0}^{{\rm low}-1}  \int_{t,\theta } \zeta[\partial_iu]^2 +\te_2^\zeta(\delta r^{m,\zeta}_*)\cdot [u]^2
+ [u]^2    d\theta dt, 
\\& \sum_{|I|=0}^{{\rm low}-1} |\big\langle(D\Psi^1_\zeta)_{\tilde{u}},v\big\rangle|\le C\eta 
\e^{\frac{1}{4}}\|u \|\cdot \| v\|,  
\sum_{|I|=0}^{{\rm low}-1} |\big\langle(D\Psi^1_\zeta)_{\tilde{v}},u\big\rangle|\le 
[O(\eta \e^{-\frac{5}{2}}+O(\e^{-\frac{5}{2}+\frac{1}{4}})]\|u \|\cdot \| v\|.
\end{split} 
 \eeq 
 These bounds will be proven below. Let us first note that these bounds imply the existence of a solution to our system \eqref{new.symb}. 
  
\begin{lemma}
\label{lem:surj}
Consider the map $\Psi_\zeta[\delta r^{m,\zeta}_*, \tilde{K}^{m,\zeta}_{12}]$ defined for $(\delta r^{m,\zeta}_*,  \tilde{K}^{m,\zeta}_{12})\in\Omega$. Then, assuming \eqref{coercivity},  the image of $\Omega$ contains the product of balls of radii
$\frac{5}{4}C\eta \e^{-\frac{5}{2}}$, $\frac{7}{8}C\eta \e^{-\frac{3}{2}}$ 
 in $H^{{\rm low}-1}_{\rm even}\times H^{{\rm low}-1}$. 
\end{lemma}  

Note that this Lemma implies  the existence of a 
 weak solution  $\delta r_*^{m,\zeta}, \tilde{K}_{12}^{m,\zeta}$ to the system \eqref{weak.sol.delta} since 
 the RHSs of the two equations in \eqref{weak.sol.delta} are bounded (respectively) as 
 follows:

 \begin{equation}
 \label{RHS.bds.2}
  \|    \partial^I    [ K^m_{22}  (\epsilon, t, \theta) -K^S_{22}(\epsilon,t,\theta) ]   \|_{L^2}\cdot
  \|v \|_{L^2}\le C\eta \e^{-3/2} \cdot \| v\|_{L^2},
  \|\partial^I K^m_{12}(\e,t,\theta)\|_{L^2}\cdot \|u\|_{L^2}\le 
  C\eta \e^{-1-\frac{1}{4}} \cdot \|u\|_{L^2}.
 \end{equation}
\eqref{RHS.bds.2} 
implies that the pair of 
RHSs of our equations is contained in that product of balls. This can be seen by choosing the norms  $\|v \|\sim \e, \|u\|\sim \e^{-\frac{3}{2}}$, for the cross terms 
to be absorbed into the main terms.  This shows 
the existence of a solution to \eqref{new.symb}.
  
  \begin{proof}
  
  We will consider the space of functions $^m\delta r_*^\zeta, ^m\tilde{K}_{12}^\zeta$ 
  which lie in
the space  $\Omega$.
 We observe 
that the 
map $\Psi_\zeta$ is uniformly (independently of $\zeta>0$) bounded in ${\cal C}^1$ between 
$H^{{\rm low}-1}\times H^{{\rm low}-1}\to (H^{{\rm low}-1}\times H^{{\rm low}-1})^*$.

On the other hand, the coercivity estimates claimed in \eqref{coercivity} and a simple Cauchy-Schwarz estimate for the 
cross terms 
imply that $\Psi_\zeta[\Omega]$ contains the product of balls claimed in our Lemma. 
 \end{proof}
\medskip

The coercivity bounds 
\eqref{coercivity} also 
imply  that our solutions are unique in the domain $\Omega$.

Moreover, we claim certain bounds on this solution: 
\begin{lemma}
\label{final.bounds}
The solutions  $^m\delta r^\zeta_*,^m\tilde{K}^\zeta_{12}$ to
the above system satisfy the bounds, for every $I, |I|\in \{0,1,\dots, {\rm low}-1\}$:

\beq\label{bds.spac.delta.1st}
\int_{t,\theta}|\partial^I (^m\delta r_*^\zeta)|^2   d\theta dt\le 2 C^2
\eta^2 \e^2,
\int_{t,\theta}|\partial^I (^m\tilde{K}_{12}^\zeta)|^2   d\theta dt\le C^2
\eta^2 
\e^{-3+\frac{1}{2}}. 
\eeq

Moreover, our solution satisfies the following bound (also independent of $\zeta>0$),
for all $I, |I|\le {\rm low}-1$:
  
 \beq
\label{spaces.delta}
\begin{split}
&\| \te_2^\zeta \partial^I(^m\delta r_*^\zeta)||^2_{L^2_{t,\theta}}\le 2 C^2
\eta^2 \epsilon^{-1}, 
|| \partial^I \te_2^\zeta\te_2^\zeta (^m\delta r_*^\zeta)||^2_{L^2_{t,\theta}}
\le 5(D+1)^2C^2\eta^2 \epsilon^{-4}.
\end{split} 
\eeq
\end{lemma}

We prove this Lemma together with \eqref{coercivity}. The proof is in 
fact essentially the same; 
the proof we provide below for Lemma \ref{final.bounds} proves 
\eqref{coercivity} by replacing
 $\delta r_*^{m,\zeta}, \tilde{K}_{12}^{m,\zeta}$
by $v,u$. 
\medskip

\begin{proof}

The Lemma   follows by using the functions $\delta r_*^{m,\zeta}, 
\tilde{K}_{12}^{m,\zeta}$
themselves as test functions $v,u$ (respectively)  in the definition of 
weak solution, \eqref{weak.sol}, and using  a finite induction to move 
the 
lower-order terms (which have already been controlled) 
 to the RHS. 
\medskip

One potentially problematic term is the one with $\te^m_2u$ in the RHS of the second
 equation  in \eqref{weak.sol.delta}, due to the extra derivative on the test function $u$ (where now $u=\tilde{K}^m_{12}$). 
Using  integrations by parts (with respect to the volume form $d\theta dt$) on the 
$\te_2$ derivative in the second factor, we derive
 (writing $\te_2$ instead of $\te^m_\zeta$, for brevity):

\beq
\begin{split}\label{CS2}
 &| \int_{t,\theta} (\frac{2M}{{r^m_*}}-1)^{-\frac{1}{2}}(\K_{22}-\K_{11})^{-1}\cdot \te_2(^m\delta 
r_*^\zeta) \partial^I(^m\tilde{K}^\zeta_{12}) 
\te_2 \partial^I(^m \tilde{K}^\zeta_{12})    d\theta dt|
\\&= |\int_{t,\theta} \{
(\K_{22}-\K_{11})^{-1}\cdot[\te_2((\frac{2M}{{r^m_*}}-1)^{-\frac{1}{2}}\te_2 (^m\delta 
r_*^\zeta)) +(\frac{2M}{{r^m_*}}-1)^{-\frac{1}{2}} \tilde{div}(\te_2) \cdot \te_2 (^m\delta 
r_*^\zeta)]
\\& +\te_2[(\K_{22}-\K_{11})^{-1}](\frac{2M}{{r^m_*}}-1)^{-\frac{1}{2}}\cdot \te_2(^m\delta r_*^\zeta)]\}\cdot |\partial^I(^m \tilde{K}^\zeta_{12})|^2  d\theta dt|.
\end{split}
\eeq
Note the following bound on the coefficient of $|\partial^I\tilde{K}^m_{12}|^2$ in the RHS of the above: 

Using the pointwise bounds \eqref{bds.spac.delta.1st}, \eqref{spaces.delta}, \eqref{div.te2m} as well as the expression \eqref{te1m.tK12} on $\te_2$ and the definition of our ball $\cal B$:
we derive: 

\beq
\label{Kbd.eta}
\begin{split}
&|(\K_{22}-\K_{11})^{-1}\cdot[\te_2(\frac{2M}{{r^m_*}}-1)^{-\frac{1}{2}}\te_2 (^m\delta 
r_*^\zeta) +(\frac{2M}{{r^m_*}}-1)^{-\frac{1}{2}} \tilde{div}(\te_2) \cdot \te_2 (^m\delta 
r_*^\zeta)]|
\\& +|\te_2(\K_{22}-\K_{11})^{-1}(\frac{2M}{{r^m_*}}-1)^{-\frac{1}{2}}\cdot \te_2(^m\delta r_*^\zeta)|\le 2DC\eta +4(DC\eta)^2\le \frac{1}{2}. 
\end{split}
\eeq

We note that the RHSs of the equations \eqref{1stPhi},
\eqref{2ndPhi}
we are solving,
are thought of as linear operators on $L^2$ 
satisfy the following  bounds when we test them on functions $v=\partial^I (\delta r_*^m)(t,\theta), u=\partial^I\tilde{K}^m_{12}(t,\theta)$:

\beq
\begin{split}
& |\int_{t,\theta} \partial^I(\K_{22}(t,\theta)-K^m_{22}(\e,t,\theta))
\cdot \partial^I(^m\delta r^\zeta_*)  d\theta dt|\le 
\\&
(1/4)\cdot \e^{-5/2} \int_{t,\theta}|\partial^I(^m\delta r^\zeta_*)|^2 d\theta dt+2  
 \e^{5/2}\int_{t,\theta} |\partial^I(\K_{22}(t,\theta)-K^m_{22}(\e,t,
 \theta))|^2 d\theta dt,
 \end{split}
\eeq

\beq
 |\int_{t,\theta} \partial^I(K^m_{12}(\e,t,\theta))\cdot 
 \partial^I(\tilde{K}^m_{12}(t,\theta))  d\theta dt|\le 
(1/4)\cdot  \int_{t,\theta}|\partial^I(\tilde{K}^m_{12})|^2  d
\theta dt+
 \int_{t,\theta} |\partial^I(K^m_{12}(\e,t,\theta))|^2 d\theta 
 dt,
\eeq

We now recall the bounds from Lemma \ref{1st.bds}: 

\beq
|\int_{t,\theta} |\partial^I(\K_{22}(t,\theta)-K^m_{22}(\e,t,\theta))|^2 d\theta dt|\le 2C^2 \eta^2\e^{-3},
\eeq

\beq
|\int_{t,\theta} |\partial^I(K^m_{12}(\e,t,\theta))|^2 d\theta dt|\le C^2
\eta^2\e^{-3+\frac{1}{2}}. 
\eeq

Thus, considering the two equations in \eqref{weak.sol.delta} with 
$v=(^m\delta r_*)(t,\theta)$ and $u=\tilde{K}^m_{12}(t,\theta)$, the above two 
inequalities give bounds on the RHSs of   \eqref{weak.sol.delta}. We also use the Cauchy-Schwarz
 inequalities to absorb all cross terms in addition to 
\eqref{Kbd.eta} into the main positive quadratic 
 terms

\[
\int_{t,\theta} [ \te^m_2(\partial^I\delta {}^mr_*^\zeta)]^2  d\theta dt, 
 \int_{t,\theta} [\partial^I(\delta {}^mr_*^\zeta)]^2  d\theta dt,
 \int_{t,\theta} [\partial^I (^m\tilde{K}^\zeta_{12})]^2  d\theta dt.
\] 
  in the LHSs of the two equations in \eqref{weak.sol.delta}.   

This directly yields  the first bounds claimed in Lemma \ref{final.bounds}, 
as well as the first bound in \eqref{spaces.delta}.

 We note that the equation \eqref{1st.diffd} coupled with the already derived bound on $\partial^I (\delta r^m_*)$  implies directly that: 
 
 \[
 \int_{t,\theta}|(\frac{2M}{{r^m_*}}-1)^{-\frac{1}{2}}\partial^I\te_2\te_2 (^m\delta 
r_*^\zeta)|^2  d\theta dt\le \int_{t,\theta} |\partial^I(^m\tilde{K}^\zeta_{22}-^mK^\zeta_{22})|^2 d\theta dt\le 5C^2\eta^2\e^{-3}. 
 \]
This completes the proof of our Lemma. Thus we have derived the existence of a weak solution and of the 
desired bounds. 
\medskip

\end{proof}
\medskip

Finally, we derive a solution of the original system \eqref{weak.sol} 
by a limiting argument:

Given the uniform bounds in the spaces \eqref{spaces.delta}
we have (independently of $\zeta>0$), we can pass to a limit 
$\zeta\to 0$ and derive    the existence of a solution to the system \eqref{weak.sol}, 
which holds for all 
$v,u\in H^1_{\rm even}\times H^1$. 

These solutions satisfy the bounds \eqref{bd1}, 
\eqref{bd2}, and these are independent of $\zeta>0$.
 Moreover they solve the system 
\eqref{weak.sol} for all $I, |I|\le {\rm low}-1$.

\subsection{Estimates for the solved-for
 $r^m_*(t,\theta), \tilde{K}^m_{12}$ at the 
higher derivatives.}
\label{sec:r*.high}

We now prove the bounds \eqref{highbd1}, \eqref{highbd2}, \eqref{highbd3} in this 
subsection. For these orders our claims are with respect to the usual volume form 
${\rm sin}\theta d\theta dt$ on axi-symmetric functions on 
$\mathbb{S}^2\times \mathbb{R}$.

We recall that the functions $r^m_*(t,\theta),\tilde{K}^m_{12}(t,\theta)$ 
have \emph{already} been solved for at this point,
 and we have also obtained the bounds 
\eqref{bd1}, \eqref{bd2}   at the lower orders $|I\le {\rm low}-1$, with respect to the 
volume form $d\theta dt$--the latter of course implies our claims at the same orders 
$|I|\le {\rm low}-1$ with respect to the (weaker) volume form $sin\theta d\theta dt$. 
 We now seek to derive 
bounds for $ r^m_*(t,\theta),\tilde{K}^m_{12}(t,\theta)$ at the remaining 
higher orders.

We remark that the method of proof at the lower orders could in fact be adapted to the 
higher ones, \emph{except} when $|I|=s-3$. The reason for this is precisely the 
transition to the stronger volume form $d\theta dt$ which ``cost'' us one derivative in 
terms of the bounds on the RHSs of the equations \eqref{1stPhi}, \eqref{2ndPhi}. 
We will thus employ a new method to derive our higher-order bounds. We highlight two 
aspects of the method here:

One difference with the lower orders is that we will \emph{first} bound the functions 
$\partial^IK^m_{22}(t,\theta,r^m_*(t,\theta))$, 
$\partial^IK^m_{12}(t,\theta,r^m_*(t,\theta))$ on 
$\Sigma_{r^m_*}$, in the $L^2({\rm sin}\theta d\theta dt)$ norm. (This is possible, 
since the functions $r^m_*(t,\theta)$ have already been solved for at the point).
\emph{After} those bounds have been derived, 
we will then bound the sought-after parameters 
 $\partial^I(r^m_*(t,\theta))$,
 $\partial^I \tilde{K}^m_{12}(t,\theta,r_*(t,\theta))$. A second difference again 
 has to do with the degeneration of the volume form ${\rm sin}\theta d\theta dt$ at the 
 poles, and the singular terms that this can generate in integrations by parts, coupled 
 with the necessity of ``seeing'' the imposed vanishing conditions $\partial_\theta r^m_*$ at 
 those poles. The method we use is to (essentially) treat an ``integrated'' version of 
 the first equation \eqref{1stPhi}, where the integration is along integral curves of 
 $\te^m_2$ (recall that at the poles $\te^m_2$ is parallel to $\partial_\theta$). 
 The required vanishing of $\te^m_2(r^m_*)$ at the poles is seen via the vanishing of 
 boundary terms there that would be present when we invoke the fundamental theorem of 
 calculus. 
 \medskip

{\bf Bounds on $K^m_{22}(t,\theta,r^m_*(t,\theta))$, 
$K^m_{12}(t,\theta,r^m_*(t,\theta))$ at the higher norms:}

Having already solved for $r^m_*(t,\theta)$ in the previous subsection, 
we can now invoke the bounds \eqref{Km22.high.var} (and its analogue for $K^m_{12}$)
 applied to 
$\delta(t,\theta)= r_*^m(t,\theta)$. This then implies
\eqref{Km22.high.var.rhom}, \eqref{Km12.high.var.rhom}. 
\medskip

{\bf Derivation of the bounds on $\partial^I\te^m_2(r^m_*), \partial^I(r^m_*-\e), \partial^I\tilde{K}^m_{12}$, $|I|\ge low$:}

Instead of treating equations \eqref{1stPhi}, \eqref{2ndPhi}  \emph{directly} 
 we \emph{first} consider 
 a  suitable \emph{integrated} version of the first  equation, and then derive our desired 
 estimates using that re-cast equation. (As we will see, using the integrated version 
 of the equation will also capture the imposed $\partial_\theta (\delta r^m_*)=0$
 condition at the two poles $\theta=0,\theta=\pi$).

In view of the expression \eqref{te2.exp} for the vector field $\te^m_2$, we also 
put down the formula for the \emph{integral} along the integral curves of $\te^m_2$: 
We consider integrals originating at $\theta=0$ until $\theta=\pi/2$. Denote this integral operator by $(\te^m_2)^{-1}$. We let $s$ stand for the parameter with 
$\te^m_2(s)=1$, $s=0$ at $\theta =0$. In particular for any $s\in [0,\pi/2), t=t_*\in\mathbb{R}$ we let $\theta_{t_*}(s),t_{t_*}(s)$ be the point that arises by flowing along $\te^m_2$ from $(0,t_*)$ for parameter $s$.   Then define: 

\[
(\te^m_2)^{-1}[F](\theta,t_*)=\int_0^\theta F(\theta_{t_*}(s), t_{t_*}(s)) ds.
\]
Note of course that with this definition: 

\[
(\te^m_2)^{-1}[\te^m_2\te^m_2\partial^I  \big(\delta r^m_*\big)](\theta,t_*)= \te^m_2\partial^I  
\big(\delta r^m_*\big)(\theta,t_*).
\]

Note that the absence of a boundary term at $\theta=0$ captures precisely the 
imposed condition that $\te^m_2(\delta r^m_*)=0$ at $\theta=0$. We can define 
$(\te^m_2)^{-1}[F](\theta, t_*)$ for $\theta\in [\frac{\pi}{2}, \pi]$
 in precisely the same way, only with our line integral starting at $\pi$. 
 
Now acting on the equation \eqref{1st.diffd} by $(\te^m_2)^{-1}$ (and performing some commutations to cancel $\te^m_2$ with $(\te^m_2)^{-1}$
in the first term below, which introduce new lower-order terms)   we derive the new, equivalent 
equation: 
\newcommand{\oV}{\overline{V}}

\beq
  \begin{split}
 \label{1st.diffd.te-1}
&\partial^I[ -(\frac{2M}{{r^m_*}}-1)^{-\frac{1}{2}}{\tilde{e}^m_2}{r^m_*}
+2(\te^m_2)^{-1}[(\frac{2M}{r^m_*}-1)^{-\frac{3}{2}}\frac{2M}{(r^m_*)^2 }(\partial^I{\tilde{e}^m_2}{r^m_*})\te^m_2(r^m_*))]]+
(\te^m_2)^{-1}[\oV(t,\theta)
\partial^I(\delta r^m_*)]
\\& +(\te^m_2)^{-1}[\partial^I \tilde{K}^m_{12}\frac{O(\tilde{K}^m_{12})}{\K_{11}-\K_{22}}]
 +(\te^m_2)^{-1}[(l.o.t.'s)^1_I(\delta r^m_*, \tilde{K}^m_{12})]
 = (\te^m_2)^{-1}[\partial^I [\K_{22}(t,\theta)-K^m_{22}(\e,t,\theta)]. 
\end{split}
\eeq
(recall that 
$\oV(t,\theta)=V(t,\theta)+O(\eta \e^{-\frac{5}{2}})+O(\e^{-\frac{5}{2}+\frac{1}{4}})\ge \e^{-\frac{5}{2}}$)--from this point onwards we write $V$ instead of $\oV(t,\theta)$, slightly abusing notation). 

\medskip

We will then be 
multiplying the above equation by $-\te^m_2\partial^I  \big(\delta r^m_*\big)$
and integrating over $[0,\pi]\times \mathbb{R}$ 
\emph{with respect to the volume form $sin\theta d\theta dt $}. (Note that strictly speaking, we should break up the interval $[0,\pi]$ into $[0,\pi/2]$ and $[\pi/2, \pi]$
and add the resulting two expressions--we skip the obvious details here). 
\medskip

Let us first consider the ``main terms'' in the resulting integral identity, first in the LHS and then in the RHS. We will subsequently briefly discuss how all the remaining terms can be absorbed into these main terms. 

The first main term we  obtain from the procedure just outlined is:
\[
\int_{-\infty}^\infty \int_0^\pi |\te^m_2\partial^I  \big(\delta r^m_*\big)|^2 
sin\theta d\theta dt.
\]
This is one of the terms we are seeking to bound, and we keep this term as is. 
 The next key  term in the integral  is: 

\[
-\int_{-\infty}^\infty \int_0^\pi  (\te^m_2)^{-1}[V\cdot \partial^I  \big(\delta r^m_*\big)]  \te^m_2\partial^I  \big(\delta r^m_*\big)sin\theta d\theta dt
\]
For this term, we perform an integration by parts with respect to $\te^m_2$. We obtain the main term: 

\beq
\label{main.term.IBPs}
+\int_{-\infty}^\infty \int_0^\pi  V| \partial^I  \big(\delta r^m_*\big)|^2
sin\theta 
d\theta dt.
\eeq
This term we wish to keep. There is  also a second term arising from $-\te^m_2(sin\theta)=(a^m)^{2\theta }cos\theta $; that term is of the form: 

\[
\frac{1}{2}\int_{-\infty}^\infty \int_0^\pi (a^m)^{2\theta }\cdot  V^{-1}\cdot 
\te^m_2[(\te^m_2)^{-1}\big{(}V\cdot \partial^I \big(\delta r^m_*\big)\big{)}]^2 cos\theta d\theta dt.
\]
In this term we again integrate by parts the $\te^m_2$ derivative, and we derive: 

\beq
\begin{split}
&\frac{1}{2}\int_{-\infty}^\infty \int_0^\pi (a^m)^{2\theta }\cdot  V^{-1}\cdot 
\te^m_2[(\te^m_2)^{-1}\big{(}V\partial^I  \big(\delta r^m_*\big)
\big{)}]^2 cos\theta d\theta dt
\\&=\frac{1}{2}\int_{-\infty}^\infty \int_0^\pi[(a^m)^{2\theta}]^2\cdot 
 V^{-1}\cdot 
[(\te^m_2)^{-1}\big{(}V\partial^I  \big(\delta r^m_*\big)\big{)}]^2 
sin\theta d\theta dt
\\&-\frac{1}{2}\int_{-\infty}^\infty \int_0^\pi \partial_\theta
 [[(a^m)^{2\theta })]2\cdot  V]
\cdot 
[(\te^m_2)^{-1}\big{(}V\partial^I \big(\delta r^m_*\big)\big{)}]^2
 cos\theta d\theta dt.
\end{split}
\eeq
The first term in the RHS has a favourable sign. For the second term, 
we  recall the pointwise  bounds on $(a^m)^{2\theta}$ and its first 
$\partial_\theta$-derivative from formula \eqref{te1m.tK12}, as well as the bounds on $V(t,\theta)$ and its first 
derivatives. Combining these with the  standard Hardy inequalities on the interval 
$[0,\pi]$, we find that the RHS of the above is bounded in absolute value  by:

\[
(\e^{1/4}) \int_{-\infty}^\infty \int_0^\pi \e^{-\frac{5}{2}}
[\partial^I \big(\delta r^m_*\big{)}]^2
 sin\theta d\theta dt.
\]

Thus this term can be absorbed into the main term \eqref{main.term.IBPs}.

Let us also derive some useful bounds on the product term appearing in the RHS of 
the  equation \eqref{1st.diffd.te-1}.

For brevity, let 
$v:=\partial^I [(\delta r^m_*)(t,\theta)-(\delta r^m_*)(t,0)]$ if $I$ contains a 
$\partial_\theta$ derivative and $v:=\partial^I [(\delta r^m_*)(t,\theta)]$ otherwise. 
We must control the term: 

\beq
\int_{-\infty}^\infty \int_0^{\pi} (\te^m_2)^{-1}[ \partial^I  (K^m_{22}(\e,\cdot ,t)-\K_{22}(\cdot , t))](t,\theta) \cdot \te^m_2( v)sin\theta d\theta dt.
\eeq

Again we perform an integration by parts of the derivative $\te^m_2$. The main term we obtain is: 

\[
\int_{-\infty}^\infty \int_0^{\pi}  \partial^I  [K^m_{22}(\e,\cdot ,t)-\K_{22}(\cdot , t)](t,\theta) \cdot  vsin\theta d\theta dt.
\]
This term can be controlled by Cauchy-Schwarz in absolute value by: 
\[
2\e^2 \int_{-\infty}^\infty \int_0^{\pi} [\big( \partial^I (K^m_{22}(\e,\cdot ,t)-\K_{22}(\cdot , t)](t,\theta))]^2 sin\theta d\theta dt+ 
\frac{1}{8}\e^{-2}\int_{-\infty}^\infty \int_0^{\pi}|v|^2 sin\theta d\theta dt. 
\]

There is a correction term  from our previous integration by parts, which is 
of the form: 

\[
\int_{-\infty}^\infty \int_0^{\pi} a^{2\theta }[(\te^m_2)^{-1}\big( \partial^I  [K^m_{22}(\e,\cdot ,t)-\K_{22}(\cdot , t)])(t,\theta) \cdot  v cos\theta d\theta dt.
\]

To control this term, we break the inner integral $\int_0^\pi$ 
into $\int_0^{\pi/2}+\int_{\pi/2}^\pi$ and write $v=\te^m_2[(\te^m_2)^{-1}v]$,
where now in the first interval $(\te^m_2)^{-1}[F](\theta,t_*)=\int_0^\theta F(\theta_{t_*}(s), t_{t_*}(s))ds$ and in the second interval 
$(\te^m_2)^{-1}[F](\theta,t_*)=\int_\pi^\theta F(\theta_{t_*}(s), t_{t_*}(s))ds$. 
In both of these intervals 
we integrate by parts again. The  resulting expression in each of the intervals
 is essentially the same, so we just perform it on the first interval
  $\theta\in [0,\pi/2]$ and we   derive: 

\beq
\label{trick.IBP}
\begin{split}
&\int_{-\infty}^\infty \int_0^{\pi/2} a^{2\theta }(\te^m_2)^{-1}[ \partial^I  (K^m_{22}(\e,\cdot ,t)-\K_{22}(\cdot , t)](t,\theta) \cdot ( v)cos\theta d\theta dt
\\&= \int_{-\infty}^\infty \int_0^{\pi/2} [a^{2\theta }]^2(\te^m_2)^{-1}[ \partial^I  (K^m_{22}(\e,\cdot ,t)-\K_{22}(\cdot , t)](t,\theta) \cdot ((\te^m_2)^{-1} v)sin \theta d\theta dt
\\&+ 
\int_{-\infty}^\infty \int_0^{\pi/2} a^{2\theta }[\big( \partial^I (K^m_{22}(\e,\cdot ,t)-\K_{22}(\cdot , t)](t,\theta)] \cdot 
((\te^m_2)^{-1} v)cos \theta d\theta dt
\\&-\int_{-\infty}^\infty \int_0^{\pi/2} \partial_\theta[a^{2\theta }]^2[\big( \partial^I (K^m_{22}(\e,\cdot ,t)-\K_{22}(\cdot , t)](t,\theta)] \cdot 
((\te^m_2)^{-1} v)cos \theta d\theta dt.
\end{split}
\eeq
(Note that the two boundary terms vanish, since at $\theta=0$ the factor 
$(\te^m_2)^{-1} v$  vanishes, while at $\theta=\pi/2$ $cos\theta$ vanishes).  
Now the first term can be controlled by Cauchy-Schwarz: 

\beq
\begin{split}
&| \int_{-\infty}^\infty \int_0^{\pi} [a^{2\theta }]^2(\te^m_2)^{-1}\big( \partial^I 
 (K^m_{22}(\e,\cdot ,t)-\K_{22}(\cdot , t)\big) (t,\theta) \cdot 
((\te^m_2)^{-1} v)sin \theta d\theta dt|
\\& \le | \int_{-\infty}^\infty \int_0^{\pi/2} [a^{2\theta }]^2[(\te^m_2)^{-1}\big( 
\partial^I (K^m_{22}(\e,\cdot ,t)-\K_{22}(\cdot , t))]^2(t,\theta) 
sin \theta d\theta dt|
\\&+ \frac{1}{4}\int_{-\infty}^\infty \int_0^{\pi/2} [a^{2\theta }]^2
[(\te^m_2)^{-1} v]^2 sin\theta d\theta dt. 
\end{split}
\eeq
Using Hardy's inequality, (recalling that $\te^m_2(\theta)=a^{2\theta}$)
 both terms can be controlled by: 

\beq
\frac{1}{4}| \int_{-\infty}^\infty \int_0^{\pi} [\big( \partial^I 
(K^m_{22}(\e,\cdot ,t)-\K_{22}(\cdot , t)]^2(t,\theta) 
sin \theta d\theta dt|+
\frac{1}{4}\int_{-\infty}^\infty \int_0^{\pi} [ v]^2 sin\theta d\theta dt.
\eeq

Finally, we now bound the second term in the RHS of \eqref{trick.IBP} again by 
Cauchy-Schwarz and then Hardy, and introducing $sin\theta$ into the volume form of the
 term 
involving $K^m_{22}-\K_{22}$: 

\beq
\begin{split}
&\int_{-\infty}^\infty \int_0^{\pi/2} a^{2\theta }[\big(\partial^I  (K^m_{22}(\e,\cdot ,t)-\K_{22}(\cdot , t))](t,\theta) \cdot 
((\te^m_2)^{-1} v)cos \theta d\theta dt
\\&\le 4\e^{\frac{5}{2}}\int_{-\infty}^\infty \int_0^{\pi} [\big(\partial^I (K^m_{22}(\e,\cdot ,t)-\K_{22}(\cdot , t))]^2(t,\theta) sin^2 \theta d\theta dt+
\frac{\e^{-\frac{5}{2}}}{4} \int_{-\infty}^\infty \int_0^{\pi} [a^{2\theta }]^2 
((\te^m_2)^{-1} v)^2 (sin \theta)^{-2} d\theta dt
\\&\le 4\e^{\frac{5}{2}}\int_{-\infty}^\infty \int_0^{\pi} [\big(\partial^I  (K^m_{22}(\e,
\cdot ,t)-\K_{22}(\cdot , t))]^2(t,\theta) sin \theta d\theta dt+
\frac{\e^{-\frac{5}{2}}}{4} \int_{-\infty}^\infty \int_0^{\pi} 
( v)^2  d\theta dt.
\end{split}
\eeq
The third term in \eqref{trick.IBP} is bounded similarly. 
\medskip

We separately consider the product 

\[
\int_{-\infty}^\infty \int_0^{\pi} (\te^m_2)^{-1}[\partial^I\tilde{K}^m_{12}
\frac{O(\tilde{K}^m_{12})}{\K_{22}-\K_{11}}]\cdot 
\te^m_2\partial^I (\delta r^m_*)sin\theta d\theta dt
\]
which arises in the integrated equation out of \eqref{1st.diffd.te-1}; this term 
(using the $L^\infty$ bounds for $\tilde{K}^m_{12}$ in our space) 
is also bounded via integration by parts and then  Cauchy-Schwarz by:

\[
\frac{\e^{2+\frac{3}{4}}}{2}\int_{-\infty}^\infty \int_0^{\pi} [\partial^I\tilde{K}^m_{12}]^2 sin\theta d\theta dt+
\frac{\e^{-2-\frac{1}{4}}}{2}\int_{-\infty}^\infty \int_0^{\pi} |\partial^I (\delta r^m_*)|^2 sin\theta d\theta dt.
\]

This concludes our treatment of the main terms in the integral 
identity we derived from \eqref{1st.diffd.te-1}. We note the secondary terms 
that also arise from the same equation, which arise from the terms we have not 
considered. These are straightforwardly bounded by Cauchy-Schwarz and absorbed into our main terms; also all products of lower-order terms with $v$ are bounded by Cauchy-Schwarz, with the $L^2$-norm 
of the lower-order term placed in the right-hand side. 



 Our claimed bounds  then follow directly from the system \eqref{1stPhi}, \eqref{2ndPhi} 
 as described above: From \eqref{1st.diffd} and all estimates derived after it, we 
 derive: 
 
 \beq
 \label{conseq1}
 \begin{split}
& \int_{t,\theta} (\frac{2M}{r^m_*}-1)^{-\frac{1}{2}}|\partial^I   \te_2 r_*^m|^2 
+[V(t,\theta)+o(\e^{-2})
  ]\partial^I  (\delta r_*^m)|^2
 +o(1) | \partial^I  \tilde{K}^m_{12}\cdot \partial^I 
 (\delta r^m_*) sin\theta d\theta dt
\\& \le 4\e^{\frac{5}{2}} \int_{\Sigma_{\e}}|\partial^I(K^m_{22}(\e,t,\theta)-\K_{22}
 (t,\theta)|^2dt sin\theta d\theta|+\e^{\frac{5}{2}}\int_{\Sigma_{\e}}|(l.o.t.'s)|^2dt 
 sin
 \theta d\theta+  \e^{-\frac{5}{2}}\int_{\Sigma_{\e}}|\partial^I
 (\delta r_*^m)|^2
 sin\theta dt.
 \end{split}
 \eeq

 In this case we recall the  bounds from \eqref{Km22.high.var.rhom}, 
 \eqref{Km12.high.var.rhom} we now have on the terms 
 
 \beq
 \label{key.terms}\int_{\Sigma_{\e}}|\partial^I(K^m_{22}(\e,t,\theta)-\K_{22}(t,\theta)|^2dt sin\theta d\theta|+\int_{\Sigma_{\e}}|(l.o.t.'s)|^2dt sin\theta d\theta\le C^2\eta^2 \e^{-3-p(k)}. 
 \eeq

On the other hand, we consider  \eqref{2nd.diffd}, and multiply it by 
 $\partial^I \tilde{K}^m_{12}$  and again 
  integrate in $t,\theta$  with respect to the volume form 
  $sin\theta d\theta dt$. 
We derive:

   \beq
   \label{conseq2}
\begin{split}   
  & \int_{t,\theta}|\partial^I  \tilde{K}^m_{12}|^2-
  (\frac{2M}{r^m_*}-1)^{-\frac{1}{2}} (\te_2\delta r^m_*) 
   (\K_{22}-\K_{11})^{-1} \partial^I  \te_2\tilde{K}^m_{12}
   \cdot \partial^I  \tilde{K}^m_{12}sin\theta 
   d\theta dt=
   \\&\int_{t,\theta}o(1)\partial^I \tilde{K}^m_{12}\cdot 
   [\partial^I (K^m_{12}(\e,t,\theta)+(l.o.t,'s)]sin\theta 
   d\theta dt.  
   \end{split}
   \eeq
  
The delicate term is the second one in the LHS, which we deal 
 with via integration by parts (and commutation terms which give rise to 
 $(l.o.t.)'s$):
 
 \beq
 \begin{split}
&- \int_{t,\theta}
(\frac{2M}{r^m_*}-1)^{-\frac{1}{2}}   (\te^m_2\delta r^m_*) 
   (\K_{22}-\K_{11})^{-1} \partial^I \te_2\tilde{K}^m_{12}
   \cdot \partial^I  \tilde{K}^m_{12}sin\theta 
   d\theta dt
   \\&= \int_{t,\theta} [(\frac{2M}{r^m_*}-1)^{-\frac{1}{2}}
  \te^m_2\te^m_2(\delta r^m_*)\cdot 
   (\K_{22}-\K_{11})^{-1} + 
  M (\frac{2M}{r^m_*}-1)^{-\frac{3}{2}}
  [\te^m_2(\delta r^m_*)]^2\cdot 
   (\K_{22}-\K_{11})^{-1}   
\\&+   (\frac{2M}{r^m_*}-1)^{-\frac{1}{2}}\te^m_2(\delta r^m_*)\cdot 
   div(\te^m_2) (\K_{22}-\K_{11})^{-1}
   +(\frac{2M}{r^m_*}-1)^{-\frac{1}{2}}\te^m_2(\K_{22}-\K_{11})^{-1}]|\partial^I  \tilde{K}^m_{12}|^2
   +(l.o.t's)
      \end{split}
 \eeq
The coefficient \beq\begin{split}
&(\frac{2M}{r^m_*}-1)^{-\frac{1}{2}}
  \te^m_2\te^m_2(\delta r^m_*)\cdot 
   (\K_{22}-\K_{11})^{-1} + 
  M (\frac{2M}{r^m_*}-1)^{-\frac{3}{2}}
  [\te^m_2(\delta r^m_*)]^2\cdot 
   (\K_{22}-\K_{11})^{-1}  \\& +
[div(\te_2) (\K_{22}-\K_{11})^{-1}
   +\te_2(\K_{22}-\K_{11})^{-1}]
   +
    (\frac{2M}{r^m_*}-1)^{-\frac{1}{2}}\te^m_2(\delta r^m_*)\cdot 
   div(\te^m_2) (\K_{22}-\K_{11})^{-1}
  \\& +(\frac{2M}{r^m_*}-1)^{-\frac{1}{2}}\te^m_2(\K_{22}-\K_{11})^{-1}]
   \end{split}\eeq
    is $o(1)$; hence the term can be 
   absorbed in the first term in \eqref{conseq2}. 
   The RHS can be controlled by Cauhcy-Schwarz 
as before.

 Then adding \eqref{conseq1}, \eqref{conseq2} 
 and choosing $\kappa= \frac{1}{2}\e^{-5/2}$ and absorbing the first 
  term into the LHS, we derive our claims \eqref{highbd1}, 
  \eqref{highbd2}, \eqref{highbd3}. Using these bounds, we then derive
   \eqref{extra.bound.te2K12} using \eqref{2ndPhi}, since all the other terms in that 
   equation have now been bounded. $\Box$
\medskip

 We have thus derived the existence of our the hypersurface 
 $\Sigma_{{r^m_*}}$
 on which our initial data will be induced, along with the key component
 $\tilde{K}^m_{12}(t,\theta)$ on that hypersurface. 
  \medskip

  The rest of this section is devoted to deriving the inductive assumptions for 
  the remaining term $K^m_{11}(r,t,\theta)$ in the REVESNGG system.  
We commence by deriving the data for this \emph{on the just-solved-for 
hypersurface} $\Sigma_{{r^m_*}}$. 

\subsection{The initial data for the remaining connection coefficients on 
the initial data hypersurface $\Sigma_{{r^m_*}}$. }
\label{sec:init.dat.nxt}

Having solved for $r^m_*(t,\theta)$ and $\tilde{K}^m_{12}(t,\theta)$ and derived the 
claimed inductive estimates in the relevant spaces, we now proceed to \emph{impose} 
the required initial conditions on the remaining parameter
 $\tilde{K}^m_{11}(t,\theta)$,  \emph{on} the initial data hypersurface 
$\Sigma_{r^m_*}$. (This parameter is with respect to the \emph{adapted frame} 
$\te^m_0, \te^m_1=e_1, \te^m_2$ on $\Sigma_{r^m_*}$--it is in fact determined purely 
from the value of $\tilde{K}^m_{12}(t,\theta)$ as we will recall). After these, we will 
also impose the required initial conditions on 
 the corresponding parameters that depend on the frame 
 $e^m_0, e^m_1, e^m_2$. (This will depend also on $r^m_*(t,\theta)$ and on 
 $\te^m_2(r^m_*)(t,\theta)$).

 We also express the new frames $\te^m_0, \te^m_1, \te^m_2$ and 
 $e^m_0,  e^m_1, e^m_2$ 
 with respect to the background coordinate vector fields $\partial_t,\partial_\theta$, and find initial values for the parameters $a^m_{Ai}, (a^m)^{iA}$:
the vector fields $\te^m_1, \te^m_2$ are prescribed 
in terms of $\partial_t,\partial_\theta$ via the requirements 
\eqref{te1m.tK12}.


Then,  invoking \eqref{F11.form}  and \eqref{tAm.def}  we impose the values

\begin{align}
\label{K11.again}
\tilde{K}^m_{11}(t,\theta)= F^{11}_{t,\theta}(\tilde{K}^m_{12}(t,\theta)),\\
\label{A112.again}
\tilde{A}^m_{11,2}=\bigg(1-\frac{4(\tilde{K}^m_{12})^2}{({\bf K}_{22}-{\bf K_{11}})^2}\bigg)^{-1/2}\te^m_1(\frac{\tilde{K}^m_{12}}{\K_{22}-\K_{11}})+
F^1(\frac{\tilde{K}^m_{12}}{\K_{22}-\K_{11}}),\\
\label{A221.again}
\tilde{A}^m_{22,1}=\bigg(1-\frac{4(\tilde{K}^m_{12})^2}{({\bf K}_{22}-{\bf K_{11}})^2}\bigg)^{-1/2}\te^m_2(\frac{\tilde{K}^m_{12}}{\K_{22}-\K_{11}})+
F^2(\frac{\tilde{K}^m_{12}}{\K_{22}-\K_{11}});
\end{align}
  here the functions $F^{11}$ 
$F^1, F^2$ are given by formulas \eqref{F11.form},
 \eqref{F1.form}, \eqref{F2.form}.

Next we derive estimates for this quantity (on $\Sigma_{r^m_*}$).

First, using the formula \eqref{K11.again} and the form \eqref{F11.form} of the function 
$F^{11}_{t,\theta}$ we directly derive: 
\beq
\label{tilK.m.low}
||(\tilde{K}^m_{11}({r^m_*}(t,\theta), t,\theta)-\K_{11}(t,\theta))||_{\dot{H}^{|I|}_{t,
\theta}}\le
 \|\tilde{K}^m_{12}(\K_{11}-\K_{22})^{-1}\|_{L^\infty} \cdot \|\big(  \tilde{K}^m_{12}
 \big) \|_{{H}^k_{t,\theta}}\le  D
C\eta \e^{-\frac{3}{2}+\frac{1}{4}}
\eeq
for $I, |I|\le {\rm low}$ and for $|I|\in \{{\rm low}+1,\dots, s-3\}$: 
\beq
\label{tilK.m.med}
\begin{split}
&||\partial^I(\tilde{K}^m_{11}({r^m_*}(t,\theta), t,\theta)-
\K_{11}(t,\theta))||_{{L}^2_{t,\theta}}\le\|\tilde{K}^m_{12}
(\K_{11}-\K_{22})^{-1}
\|_{L^\infty} \cdot \|\partial^I\big(  \tilde{K}^m_{12}\big) 
\|_{{L}^2_{t,\theta}}
\\&\le
DC\eta  \e^{-3/2-\frac{|I|-{\rm low}}{4}+\frac{1}{4}}
\end{split}
\eeq


\medskip

Next, in view of the verified inductive bounds \eqref{tK12.ind.claim.low}
 and 
\eqref{tK12.ind.claim.high}, as well as \eqref{frames.coords}
 we derive the following  bounds on 
 $\tilde{A}^m_{22,1}, \tilde{A}^m_{11,2}$:
for all $k\le {\rm low}$: 
\beq
\label{tA.bd.low}
||\partial^I \tilde{A}^m_{22,1}||_{L^2_{t,\theta}}\le B
\e^{-3/2+\frac{3}{8}}\le \eta\e^{-5/4},
 ||\partial^I \tilde{A}^m_{11,2}||_{L^2_{t,\theta}}\le 
 B
 \e^{-3/2+\frac{3}{8}}\le \eta\e^{-5/4}, 
\eeq
and the following bounds on the higher derivatives ($h=k-{\rm low}$ below)
for $|I|\le s-4$: 
\beq
\label{tA.bd.high}
\|\partial^I\tilde{A}^m_{22,1}||_{L^2}
\le 
B\e^{-3/2+\frac{3}{8}+\frac{1}{4}h},
 ||\partial^I \tilde{A}^m_{11,2}||_{L^2}
 \le 
 B \e^{-3/2+\frac{3}{8}+\frac{h}{4}}\le \eta \e^{-3/2+1/4+\frac{h}{4}}. 
\eeq

The key to deriving the above four bounds  is the control (by the same 
bounds) 
of the terms $\te^m_1(\frac{\tilde{K}^m_{12}}{\K_{22}-\K_{11}})$, 
$\te^m_2(\frac{\tilde{K}^m_{12}}{\K_{22}-\K_{11}})$. These follow by 
invoking the
 bounds \eqref{tK12.ind.claim.low} and 
\eqref{tK12.ind.claim.high} we have on 
$\partial^{I'}\tilde{K}^m_{12}(t,\theta)$, $|I'|=k+1$,
and expressing $\tilde{e}^m_1, \tilde{e}^m_2$ in terms of $\partial_t,\partial_\theta$ 
using \eqref{frames.coords}. This gives a \emph{gain} of a power 
$\e^{3/8}$ at the 
minimum. The rest of the terms are easily seen to also be bounded by the RHSs of \eqref{tA.bd.low}, \eqref{tA.bd.high}. 
\medskip

Then the function $K^m_{11}({r^m_*}(t,\theta),t,\theta)$,
\emph{on} the hypersurface $\Sigma_{{r^m_*}(t,\theta)}$ is defined by the equation
\eqref{K11exp}.
\medskip

 In view of the inductive bounds \eqref{r*.ind.claim.low}, \eqref{tK12.ind.claim.low},
 \eqref{r*.ind.claim.high}, \eqref{tK12.ind.claim.high}  on the quantities 
 ${r^m_*}(t,\theta), \tilde{K}^m_{12}(t,\theta)$ that were verified 
 in the previous subsection, we next derive bounds on 
  $K^m_{11}({r^m_*}(t,\theta),t,\theta)$.

We  use \eqref{K11exp}, along with \eqref{tilK.m.low} and the assumed closeness of $\K_{11}$ to $K^S_{11}$  to derive for all $I, |I|\le {\rm low}$: 
\beq
||\partial^I(K^m_{11}({r^m_*}(t,\theta), t,\theta)
-\K_{11}(t,\theta))||^2_{L^2_{t,\theta}(\Sigma_{r^m_*)}}\le  ((D+1)C\eta)^2\e^{-3+\frac{1}{4}}+\eta^2\e^{-3}\le 2\eta^2\e^{-3}.
\eeq
 While at the higher orders $k>{\rm low}, k\le s-4$ 
 (invoking the closeness of $\K_{11}$ to $K^S$) we similarly derive:

\beq
\begin{split}
& ||\partial^I(K^m_{11}({r^m_*}(t,\theta), t,\theta)-K^S_{11}(t,\theta))
||^2_{H^k_{t,\theta}}\le  2(C\eta)^2\e^{-3-\frac{k-{\rm low}}{2}}+ \eta^2 \e^{-3-\frac{k-{\rm low}}{2}}\le 2\eta^2 \e^{-3-\frac{k-{\rm low}}{2}}.
\end{split}
\eeq
The desired bound at the top orders is derived in the last subsection of this section. 
\medskip





We have thus derived the desired bounds for 
$K^m_{11}$  on the 
initial data 
hypersurface $\Sigma_{{r^m_*}}$.

\medskip
 We proceed in the next subsections to derive the claimed bounds for 
 $K^m_{11}(r,t,\theta)$
  \emph{off} of this hypersurface, for all $r\in (0,2\e]$.

\subsection{Energy estimates for $K^m_{11}(r,t,\theta)$ 
and the 
asymptotically CMC property of the surfaces $\Sigma_{r^m}$.}
\label{backforwestK11m}

Here we verify the inductive assumptions 
\eqref{inductiontrKlow}, 
\eqref{inductiontrKopt}, \eqref{inductiontrKtopmixed}, \eqref{inductiontrKtopmixed.sing2} for the connection 
coefficient 
$K^m_{11}(r,t,\theta)$, and also verify \eqref{trKheur}. 
\medskip

We just imposed  the initial data for $K^m_{11}$ on $\Sigma_{r^m_*}$ 
and 
also derived the desired bounds for it on that hypersurface. 
The evolution equation we have imposed is 
  the  
equation (\ref{finredEVERic11it}). Having already solved for and bounded 
$K^m_{22}(r,t,\theta), K^m_{12}(r,t,\theta)$, this equation is purely a 
non-linear ODE
 in $K^m_{11}(r,t,\theta)$. Contrary to $K^m_{22}, K^m_{12}$, this ODE 
 will be 
 solved \emph{forwards}, towards the singularity.

 \begin{remark}
 The reason why $K^m_{11}(r,t,\theta)$ can be solved forwards without 
 potentially violating our ``asymptotically CMC'' claim at the level of derivatives  is that in 
 contrast to
  $K^m_{22}(r,t,\theta)$ 
 the non-linear term $(K_{11}^m)^2$ in the Riccati-type 
  equation 
  is non-focusing (given the sign of the initial datum is close to 
  $\K_{11}$).  
  This implies that the solution remains smooth for all $r>0$, and 
  becomes singular at
   $r=0$ purely because of the singular behaviour of the RHS in the Riccati equation. Moreover, 
    all solutions that we obtain (regardless of the initial 
   condition we impose, provided it is close enough to $\K_{11}$) will 
   have a ``free branch'' of the form
    $O(r^{\zeta^m(t,\theta)})$ with 
    $|\zeta^m(t,\theta)| \le \frac{1}{8}$. 
    In particular, 
    it is much less singular than the leading-order behaviour 
  $(2M)^{\frac{1}{2}}d_1^m(t,\theta)r^{-3/2}$ which is the contribution of the singular 
 forcing terms on the  RHS;
  this is in contrast to the behaviour of $K^m_{12}(r,t,\theta)$. 
  \emph{Crucially} this same free branch appears for the differentiated variable 
  $\partial K^m_{11}(r,t,\theta)$; this is in contrast with $K^m_{22}(r,t,\theta)$, which admitted a singular branch of the form 
  $r^{-3+\e^m(t,\theta)}$ for the differentiated variable 
  $\partial K^m_{22}(r,t,\theta)$. Recall that it was to set these two singular 
  branches (which would destroy our desired ``asymptotically CMC'' property) 
   to zero that we solved for $K^m_{12}(r,t,\theta)$ and $K^m_{22}(r,t,\theta)$ \emph{backwards} from the singularity.    
  \end{remark}
\medskip

Beyond the question of where initial data are imposed, the 
derivation of the inductive claim follows the same outline as for 
$K^m_{22}, K^m_{12}$:

At the lower orders, 
we show \eqref{K11heur}, recalling the expression:
\begin{align}\label{K11expn}
&K_{11}^m=:\frac{d_1^m(t,\theta)\sqrt{2M}}{r^\frac{3}{2}}+u_{11}^m
\end{align}
where according to Proposition \ref{prop:gammam/logr}, 
$\|d_1^m-\frac{1}{2}\|_{L^\infty}\leq DC\eta$. The 
function $u^m_{11}(r,t,\theta)$ is again claimed to be  less singular, as $r\rightarrow0$, see \eqref{inductiontrKopt}.

 Recall that from the-already verified inductive step for $\gamma^m$,  
  (\ref{gammam1enest})  it follows that 
\begin{align}\label{um33enest}
\|e_0^{J_0}u^m_{33}\|_{H^{s-3-c}}\leq Br^{-1-\frac{1}{4}-\frac{3}{2}|J_0|},&&r
\in(0,2\epsilon], |J_0|\le 2.
\end{align}
For the energy estimates of $u^m_{ij},K^m_{ij}$, $i,j=1,2$, we start at the lower orders 
validating (\ref{Kmijexp}) for $K^m_{11}$ 
and work ourselves up to the top order energy estimates.

\subsubsection{Lower-order  estimates for $K^m_{11}(r,t,\theta)$: 
The asymptotic expansion.}

Substitute (\ref{K11heur}) in the LHS of (\ref{finredEVERic11it}) and 
(\ref{K33heur}) in the $e_0^2\gamma^m,e_0\gamma^m$ 
terms:
\begin{align}
\label{u11mRicit.again}&e_0 u_{11}^m+(2d_1^m-\alpha^m)\frac{\sqrt{2M}}{r^\frac{3}{2}}
u_{11}^m+(u_{11}^m)^2+3(u_{12}^{m})^2+u_{33}^mu_{11}^m\\
\notag=&\,\overline{\nabla}_{11}\gamma^m+(e_1\gamma^{m-1})(e_1\gamma^m)
-e_0u^m_{33}-(u_{33}^m)^2+(2\alpha^m-d_1^m)\frac{\sqrt{2M}}{r^\frac{3}{2}}u^m_{33}
+O(\frac{1}{r^2})
\end{align}
where $O(\frac{1}{r^2})$ stands for a term that is bounded by $Br^{-2}$ in the
 $H^{s-3-4c}$ norm.

Notice that considering the model homogenous equation 

\[
e_0 y+(2d_1^m-\alpha^m)\frac{\sqrt{2M}}{r^\frac{3}{2}}y=0
\]
we observe that the solutions of this equation behave to  leading order as  
$r^{2d_1^m(t,\theta)-\alpha^m(t,\theta)}$. Note that in view of the formula 
\eqref{d1.m}  this behaviour is 
much less singular than $r^{2d_2^m(t,\theta)-\alpha^m(t,\theta)}$. 
 In fact \emph{this} is the real asymptotic 
behaviour for the ``homogenous free part'' of the \emph{true}
equation \eqref{u11mRicit.again}. 

\begin{proposition}\label{prop:um11enest}
Given $u^m_{11}({r^m_*},t,\theta)\in H^{s-3-4c}$, there exists a unique
 solution $u^m_{11}(r,t,\theta)\in H^{s-3-4c}$, $i,j=1,2$, to the 
equation (\ref{u11mRicit.again}).
This satisfies the bounds: 
\begin{align}\label{um11enest}
\|u^m_{11}\|_{H^{s-3-4c}}\leq \frac{C}{2}r^{-DC\eta} \|u^m_{11}
(r^m_*,t,\theta)\|_{H^{s-3-4c}}+Br^{-1-\frac{1}{4}},
\end{align}
for all $r\in(0,2\epsilon]$.
\end{proposition}
Observe that the conclusion of the previous proposition validates the inductive assumption (\ref{inductiontrKopt}) for $u^m_{11}$, 
 in view of 
the bounds on $K^m_{11}(r^m_*(t,\theta),t,\theta)$ in the previous subsection. 
\begin{proof}
Rewrite the equation (\ref{u11mRicit.again}) in the form
\begin{align}
\label{partial_ru11m}\partial_r(r^{\alpha^m-2d_1^m}u^m_{11})=&-(\frac{2M}{r}-1)^{-\frac{1}{2}}r^{\alpha^m-2d_1^m}\bigg[\overline{\nabla}_{11}\gamma^m+
(\overline{\nabla}_1\gamma^{m-1})(\overline{\nabla}_1\gamma^m)
-e_0u^m_{33}-(u_{33}^m)^2+O(1)u^m_{11}\\
\notag&+(2\alpha^m-d_1^m)\frac{\sqrt{2M}}{r^\frac{3}{2}}u^m_{33}
+O(\frac{1}{r^2})-(u_{11}^m)^2-3(u_{12}^{m})^2-u_{33}^mu_{11}^m\bigg].
\end{align}
We proceed by integrating (\ref{partial_ru11m}) in $[r,r^m_*]$
to obtain:
\begin{align}
\label{partial_ru11m.int}r^{\alpha^m-2d_1^m}u^m_{11}\big|^r_{{r^{m}_*}}=&\int^{{r^{m-1}_*}}_r(\frac{2M}{\tau}-1)^{-\frac{1}{2}}\tau^{\alpha^m-2d_1^m}\bigg[\overline{\nabla}_{11}\gamma^m+(\overline{\nabla}_1\gamma^{m-1})(\overline{\nabla}_1\gamma^m)
-e_0u^m_{33}-(u_{33}^m)^2+O(1)u^m_{11}\\
\notag&+(2\alpha^m-d_1^m)\frac{\sqrt{2M}}{\tau^\frac{3}{2}}u^m_{33}
+O(\frac{1}{\tau^2})-(u_{11}^m)^2-(u_{12}^{m})^2-u_{33}^mu_{11}^m\bigg]
d\tau.
\end{align}
Utilising the estimates (\ref{um33enest}) 
for $u^m_{33}$, as well as the estimates 
in  Lemmas \ref{hessian.bounds.again}, 
 we infer that
\begin{align}
\label{u11mLinftyineq} u^m_{11}=&\,\frac{({r^m_*})^{\alpha^m-2d_1^m}}{r^{\alpha^m-2d_1^m}} u^m_{11}({r^m_*},t,\theta)
+O(r^{-1-\frac{1}{4}})+\frac{1}{r^{\alpha^m-2d_1^m}}\int^{r^m_*}_r\tau^{\alpha^m-2d_1^m} O(\frac{1}{\tau^{1-\frac{1}{4}}})u_{11}^m d\tau\\
\notag&-\frac{1}{r^{\alpha^m-2d_1^m}}\int^{r^m_*}_r(\frac{2M}{\tau}-1)^{-\frac{1}{2}}\tau^{\alpha^m-2d_1^m}\big[(u^m_{11})^2\big]d\tau.
\end{align}
Standard ODE theory then furnishes a solution with the prescribed initial condition at $r^m_*(t,\theta)$; 
a simple bootstrap argument then yields  the bounds %
\begin{align}\label{u11mLinftyest}
\|u^m_{11}\|_{L^\infty}\leq \frac{C}{2}r^{-DC\eta} \|u^m_{11}({^{m}r_*},t,\theta)\|_{L^\infty}+Br^{-1-\frac{1}{4}},
\end{align}
for $r\in(0,2\epsilon]$. 

Next, we derive the lower order  order energy estimates for $u^m_{11},i,j=1,2$, proving (\ref{um11enest}). We argue by finite induction, assuming the estimate
\begin{align}\label{um11enest2}
\|u^m_{11}\|_{H^{l-1}}\leq \frac{C}{2}r^{-DC\eta} \|u^m_{11}({{r^{m-1}_*}},t,\theta)\|_{H^{l-1}}+Br^{-1-\frac{1}{4}},
\end{align}
is valid for $0\leq l-1< s-3-4c$ and proceed to show that the analogous estimate holds for $\partial^Iu^m_{11}$, where $|I|=l\leq s-3-4c$. Note that for $l=1$, (\ref{um11enest2}) holds true by virtue of (\ref{u11mLinftyineq}).

We derive:
\begin{align}\label{u11menineq}
&-\partial_{\rho^{m-1}}\int_{\Sigma_{\rho^{m-1}}}(\partial^Iu^m_{11})^2\mathrm{vol}_{Euc}=\int_{\Sigma_{\rho^{m-1}}}2\partial^Iu^m_{11}\partial^I\bigg[\frac{e_0u^m_{11}}{(\frac{2M}{r}-1)^{\frac{1}{2}}[1-\partial_r\chi(r)({{r^{m-1}_*}}-\epsilon)]}\bigg]\mathrm{vol}_{Euc}\\
\tag{by (\ref{gammamlowGron2}),(\ref{u33m-1enest}), and  \ref{inductionHess}) }\leq& \int_{\Sigma_{\rho^{m-1}}}\frac{2\alpha^m-4d_1^m}{r}(\partial^Iu^m_{11})^2\mathrm{vol}_{Euc}+\frac{B}{r^{2+2DC\eta}}\|\partial^Iu^m_{11}\|_{L^2}\\
\notag&+\frac{B}{r^{1-\frac{1}{4}}}\|u^m_{11}\|^2_{H^l}+\frac{B}{r^{1-\frac{1}{4}}}\|u^m_{11}\|_{H^l}\|u^m_{12}\|_{H^l}\\
\tag{$2DC\eta\leq\frac{1}{4}$}\leq& \frac{DC\eta}\rho^{m-1}\|\partial^Iu^m_{11}\|_{L^2}^2+\frac{B}{r^{2+\frac{1}{4}}}\|\partial^Iu^m_{11}\|_{L^2}
+\frac{C}{r^{1-\frac{1}{4}}}\|u^m_{11}\|^2_{H^l}+\frac{B}{r^{1-\frac{1}{4}}}\|u^m_{11}\|_{H^l}\|u^m_{12}\|_{H^l},
\end{align}
and hence
\begin{align}\label{u11menineq2}
\notag-\partial_{\rho^{m-1}}({^{m-1}\rho^{DC\eta}}\|\partial^Iu^m_{11}\|_{L^2}^2)\leq&\, {^{m-1}\rho^{DC\eta}}\big[\frac{B}{r^{2-\frac{1}{4}}}\|\partial^Iu^m_{11}\|_{L^2}
+\frac{B}{r^{1-\frac{1}{4}}}\|u^m_{11}\|^2_{H^l}+\frac{B}{r^{1-\frac{1}{4}}}\|u^m_{11}\|_{H^l}\|u^m_{12}\|_{H^l}\big]\\
\|\partial^Iu^m_{11}\|_{L^2}^2\leq&\, 
\frac{\epsilon^{DC\eta}}{{^{m-1}\rho^{DC\eta}}}\|\partial^Iu^m_{11}({{r^{m-1}_*}},t,\theta)\|_{L^2}^2+
\int^{\epsilon}_{\rho^{m-1}}\frac{\tau^{DC\eta}}{{^{m-1}\rho^{DC\eta}}}\big[\frac{B}{\tau^{2+\frac{1}{4}}}\|\partial^Iu^m_{11}\|_{L^2}\\
\notag&+\frac{B}{\tau^{1-\frac{1}{4}}}\|u^m_{11}\|^2_{H^l}+\frac{B}{\tau^{1-\frac{1}{4}}}\|u^m_{11}\|_{H^l}\|u^m_{12}\|_{H^l}\big]d\tau
\end{align}

Then applying Lemma \ref{lem:Gron} to the above we derive the claim 
\eqref{um11enest2} for $|I|=l$. Given the bounds we have imposed on 
$\e,\eta$ in terms of $C$, this yields the inductive claim 
\eqref{inductiontrKopt} for $u^m_{11}$. 
\end{proof}

\subsubsection{Higher order estimates for $K^m_{11}$}\label{subsec:K11high}

The higher order estimates for $\partial^IK^m_{11}$, 
$s-3-4c<|I|\leq s-4$, are derived in a similar manner to those
 for $u^m_{11}$ in the previous subsubsection. 
We differentiate  the Riccati equation (\ref{finredEVERic11it}) to 
the required higher orders.
We subtract from this the corresponding equation for the Schwarzschild component $K^S_{11}(r,t,\theta)$. 
For the coefficients (depending on $K^m_{11},e_0\gamma^m$) of the highest order terms  $\partial^I(K^m_{11}-K^S_{11})$ we replace the expressions
 (\ref{K11heur}), (\ref{K33heur}). This enables us to distinguish the leading-order behaviour of those coefficients from the lower-order terms; the latter can be readily absorbed into the main estimate. 
 We derive: 
\begin{align}
\label{highRicm11}   &e_0\partial^I
 [K_{11}^m-K^S_{11}]+\frac{(2d_1^m-\alpha^m)\sqrt{2M}}{r^\frac{3}{2}}
\partial^I [K^m_{11}-K^S_{11}]+O(\frac{1}{r^{1+\frac{1}{4}}})\partial^I [K_{11}^m-K^S_{11}]\\
\notag&=\partial^I \big[
\overline{\nabla}_{11}\gamma^m+(\overline{\nabla}_1\gamma^{m-1})(\overline{\nabla}_1\gamma^m)
-e_0^2[\gamma^m-\gamma^S] -[(e_0\gamma^m)^2-(e_0\gamma^S)^2]-3
(K_{12}^{m})^2\big]\\
\notag& -\{\sum_{I_1\cup I_2,\,|I_1|< |I|}\partial^{I_1}K_{11}^m \partial^{I_2}e_0\gamma^m_{\rm rest}
 +\sum_{I_1\cup I_2=I,\,|I_1|<|I|,\,|I_2|<|I|}\partial^{I_1}K_{11}^m\partial^{I_2}K_{11}^m\}\\
\notag& +\sum_{I_1\cup I_2,\,|I_1|< |I|}\partial^{I_1}K_{11}^S \partial^{I_2}e_0\gamma^S
 +\sum_{I_1\cup I_2=I,\,|I_1|<|I|,\,|I_2|<|I|}
\partial^{I_1}K_{11}^S\partial^{I_2}K_{11}^S\}.
\end{align}
This equation holds for all orders. 

Our estimate at the higher orders is then the following: 
\begin{proposition}\label{prop:Kmenest}
Given the value  $K^m_{11}({r^m_*},t,\theta)$ prescribed via 
\eqref{K11exp}, there exists a unique solution $K^m_{11}$,  to  
(\ref{highRicm11}) until $r=0$. At the higher orders it 
satisfies the following  estimates on level sets of $\rho^m$: 
\begin{align}\label{Kmenest}
\|\partial^I{}
[K_{11}^m-K^S_{11}]\|_{L^2}\leq Cr^{-DC\eta}\|[K^m_{11}-K^S_{11}]
({r^m_*},t,\theta)\|_{H^{|I|}}
+C\eta\rho^{-\frac{3}{2}+(s-3-|I|)\frac{1}{4}-c},
\end{align}
for all $r\in(0,2\epsilon]$, $s-3-4c<|I|\leq s-4$. The same estimate holds 
for $|I|=s-3$, but for $\partial^J \partial_\Th K^{m}_{11}{\rm cot}\Th$ instead of 
$\partial^IK^m_{11}$ as claimed in \eqref{inductiontrKtopmixed2}.




\end{proposition}
The estimates (\ref{Kmenest}), confirm the inductive 
claim on $K^m_{11}$
 at the higher derivatives. At the top derivatives we will prove 
 \eqref{inductiontrKtopmixed2} right below, using the already-derived 
 \eqref{K.em.flux} below. 
%
\begin{proof}

We show (\ref{Kmenest}) is valid in increasing order in $|I|=l$, assuming 
the estimate (\ref{Kmenest}) is valid for every $|I|<l$. Note that in the 
case 
$|I|<l=s-3-4c+1$, the estimate (\ref{Kmenest}) is valid by  
(\ref{K11heur}),(\ref{um11enest}). \\
In the derivations below we make use of the energy estimates 
\eqref{inductiongammaopt}, \eqref{inductiongammalow}, \eqref{inductiongammatopmixed}
 for $\gamma^m$ and 
 \ref{inductionHess}.

The result of these (and the product inequality) is that the RHS of \eqref{highRicm11}
is bounded in the $L^2$ norm by 
$\rho^{-3-\frac{|I|-low}{4}}$.

Thus, the energy inequality for $\partial^I{}K^m_{11}$ reads: 
\begin{align}\label{K11menineq}
\notag&-\partial_{\rho}\int_{\Sigma_{\rho^m}}
(\partial^I{}K^m_{11})^2\mathrm{vol}_{Euc}=
\int_{\Sigma_{\rho^m}}2(\partial^I{}K^m_{11})\partial^I
\bigg[\frac{e_0K^m_{11}}{(\frac{2M}{r}-1)^{\frac{1}{2}}[1-\partial_r\chi(r)({r^m_*}-\epsilon)]}\bigg]\mathrm{vol}_{Euc}\\
\tag{plugging in (\ref{highRicm11})}\leq& \int_{\Sigma_{\rho^m}}\frac{2\alpha^m-4d_1^m}{r}(\partial^IK^m_{11})^2\mathrm{vol}_{Euc}
+\frac{B}{r^{1-\frac{1}{4}}}\|
\partial^IK^m_{11}\|^2_{L^2}\\
\notag&+(\frac{B}{r^{\frac{1}{2}+\frac{1}{4}+D\eta}}
)
r^{-\frac{3}{2}-c+(s-3-l)\frac{1}{4}}\|\partial^I
K^m_{11}\|_{L^2}
+\frac{B}{r^{1-\frac{1}{4}}}\|\partial^I{}K^m_{11}\|_{L^2}\|\partial^I{}K^m_{12}
\|_{L^2}\\
\notag\leq& \frac{D\eta}{\rho^m}\|\partial^I
K^m_{11}\|_{L^2}^2+\frac{B}{r^{1-\frac{1}{4}}}\|
\partial^I K^m_{11}\|^2_{L^2}
\tag{$D\eta\leq \frac{1}{8}$}+(\frac{B}{r^{1-\frac{1}{8}}})
r^{-\frac{3}{2}-c+(s-3-l)\frac{1}{4}}\|
\partial^I K^m_{11}\|_{L^2}
+\frac{B}{r^{1-\frac{1}{4}}}\|\partial^I
K^m_{11}\|_{L^2}\|\partial^I{}K^m_{12}\|_{L^2}
\end{align}
and hence, using integrating factors and integrating in $[{\rho}^m,\epsilon]$ we obtain:
\begin{align}\label{K11menineq2}
&\|\partial^I{}K^m_{11}\|^2_{L^2[\Sigma_{\rho^m=\rho}]}
\leq \frac{\epsilon^{DC\eta}}{{(^{m}\rho)^{DC\eta}}}\|
\partial^I{}K^m_{11}(\epsilon,t,
\theta)\|^2_{L^2[\Sigma_{\rho^m=\e}]}
+\int^\epsilon_{\rho}\frac{\tau^{D\eta}}{^{m-1}\rho^{D\eta}}\bigg[
\frac{B}{\tau^{1-\frac{1}{4}}}\|
\partial^IK^m_{11}\|^2_{L^2[\Sigma_{\rho^m=\tau}]}\\
\notag&+(\frac{B}{\tau^{1-\frac{1}{8}}})
\tau^{-\frac{3}{2}-c+(s-3-l)\frac{1}{4}}\|
\partial^I{}K^m_{11}\|_{L^2[\Sigma_{\rho^m=\tau}]}
+\frac{B}{r^{1-\frac{1}{4}}}\|\partial^IK^m_{11}\|_{L^2}\|
\partial^I{}K^m_{12}\|_{L^2[\Sigma_{\rho^m=\tau}]}\bigg]d\tau
\end{align}
The above, combined with the Gronwall inequality imply the claim 
\eqref{inductiontrKtopmixed} at the higher orders, when $|I|\le s-4$. 
\medskip

We now provide the proof of the claim for the top order regarding $K^m_{11}$. 
Certain aspects of this analysis are precisely analogous to the one performed in 
\ref{sec:K22K11.top}; these parts we just outline. The parts which differ we spell out
 in more detail. 
 
 We recall in particular that $K^m_{11}$ is defined on the initial data hypersurface
  $\Sigma_{r^m_*}$ via \eqref{K11exp}; notably the factor 
  $\tilde{A}_{11,2}^m $ in the second summand is determined in terms of 
  $\tilde{K}^m_{12}$ via \eqref{tAm.def}; in particular that term does \emph{not} 
  apriori 
  have $s-3$ derivatives in $L^2$, precisely due to this term.
  \emph{However} we note that if one of those derivatives is of the form $\partial_\Th$
  then we \emph{do} have bounds on the term. The term which is non-obvious is: 
  
  \beq\label{bad.term}
O(1)  \partial^J_{\dots \Th} \te^m_1(\frac{\tilde{K}^m_{12}}{\K_{22}-\K_{11}})\cdot 
  (\frac{2M}{r}-1)^{-1/2}(\te^m_2r^m_*). 
  \eeq
  However, recall that $(a^m)_{\Th 1}=0$, which allows us to re-express $\partial_\Th$ 
  as a multiple of $\te^m_2$, and then $\te^m_1$ can be re-expressed in
 terms of the derivatives $\partial_T, \partial_\Th$:
\[
\partial_\Th= (a^m)_{\Th 2} \te^m_2, \te^m_1= (a^m)^{1T}\partial_T+(a^m)^{1\Th}
\partial_\Th. 
\]
  Thus, (up to lower-order terms) the term \eqref{bad.term} can be expressed as follows, where $|J|=s-4$:

  \beq
O(1)  \partial^J_{\dots \Th} [\te^m_1(\frac{\tilde{K}^m_{12}}{\K_{22}-\K_{11}})]\cdot 
  (\frac{2M}{r}-1)^{-1/2}(\te^m_2r^m_*)=
\sum_{A=T,\Th}(a^m)^{1A} \partial_A\partial^J [\te^m_2 
(\frac{\tilde{K}^m_{12}}{\K_{22}-\K_{11}})]    (\frac{2M}{r}-1)^{-1/2}(\te^m_2r^m_*).
  \eeq
  In view of the bound \eqref{extra.bound.te2K12}, we observe that the above term is bounded by 
  $\e^{-3/2-c}$. 
In particular the initial data for $\partial^IK^m_{11}$ at the top order satisfy the
 required bounds. 
 
 For all these top-order derivatives we can then repeat the proof in section 
 \ref{sec:K22K11.top}  for the 
 \emph{true} connection coefficients $K^m_{11}, K^m_{12}, K^m_{22}$; we use   the 
 already-derived bounds for $K^m_{22}, K^m_{12}$ and we derive the claim for 
 ${\rm tr}K^m$, as in that section. Since ${\rm tr}K^m=K^m_{11}+K^m_{22}+e_0\gamma^m$ and the claim 
 has already been derived for $K^m_{22}$, $e_0\gamma^m$ 
 we derive our claim for these top-order terms
  for $K^m_{11}$ also.

 \end{proof}

\appendix

\section{Appendix}

\subsection{Proof of Lemma \ref{lem:Weyl}.}\label{pfWeyl}
\begin{proof}
In $1+2$ dimensions the Weyl curvature vanishes, hence, the following formula holds \cite[(3.2.28)]{Wald}:
\begin{align}\label{Weyl}
\text{R}_{abcd}(h)=2\big(h_{a[c}\text{R}_{d]b}(h)-h_{b[c}\text{R}_{d]a}(h)
\big)
-\text{R}(h)h_{a[c}h_{d]b}
\end{align}
Let $\pi_{ab}$ be the second fundamental form of an orthogonal hypersurface to $\partial_{\phi}$, the induced metric on which is by definition $h$. Then the twice contracted Gauss equation reads
\begin{align}\label{Gausseqh}
\text{R}(g)-2g^{\phi\phi}\text{R}_{\phi\phi}(g)=&\,\text{R}(h)+|\pi|^2_h-(\text{tr}_h
\pi)^2,
\end{align}
Since $\partial_{\phi}$ is Killing, it follows that $\pi_{ab}$ is an anti-symmetric 2-tensor. However, $\pi_{ab}$ is also symmetric, being the second fundamental form of a hypersurface, hence, it vanishes: $\pi_{ab}=0$. By virtue of the EVE, the 
identity (\ref{Gausseqh}) then reduces to $\text{R}(h)=0$. This completes the proof of 
the lemma.
\end{proof}

\subsection{Solutions to REVESNGG yield solutions of the vacuum Einstein equations.}
\label{retrEVE}

Let us show how a solution to the REVESNGG system, with initial data that satisfy the vacuum constraint equations yield a metric $g$ that satisfies the vacuum Einstein equations:

We are given initial data ${\bf g,K}$ for the EVE on an initial 3-dim hypersurface $\Sigma^3$ satisfying the constraint equations
\begin{align}\label{const}
\left\{
\begin{array}{ll}
\big({\text{R}_{0b}(g)}\big|_\Sigma=\big)\;{\bf D}^b{\bf K}_{ab}-{\bf D}_a\text{tr}_{\bf g}{\bf K}=0,\quad b=1,2,3\\
\big(2{\text{R}_{00}(g)}\big|_\Sigma+\text{R}(g)\big|_\Sigma=\big)\;\text{R}({\bf g})-|{\bf K}|^2+(\text{tr}_{\bf g}{\bf K})^2=0
\end{array}
\right.,
\end{align}
where ${\bf D}$ is Levi-Civita connection of ${\bf g}$.
Let $\gamma,K_{ij},i,j=1,2,$ solve the wave-Riccati system of equations (\ref{redEVEwav}), \eqref{finredEVERic11pre}-\eqref{finredEVERic12pre}, 
for an orthonotmal frame $\{e_i\}^2_0$
satisfying \eqref{almpar.trans}.

These parameters, together with the initial configurations satisfying \eqref{const}, produce a $1+3$-metric $g$ \eqref{metric} via the 
coordinate-to frame coefficients $a_{Ai}$; the latter are complemented by the parameters $r_*(t,\theta)$ which see the \emph{location}
 of the initail data hypersurface in our chosen gauge, along with the  position of our chosen frame $e_1, \te_2$ on our inital data surface.
 We will show that $g$ is in fact a solution to the EVE \eqref{EVE}.
\medskip

Axi-symmetric and polarized metrics (\ref{metric}) satisfy \cite[Appendix VII]{ChoqBook} the relations $\text{R}_{a3}(g)=0$, $a=0,1,2$, $\text{R}_{33}(g)=-\square_g\gamma$. Since the wave equation is part of the system (\ref{redEVEwav})-(\ref{redEVERic}), it remains to show the vanishing of 
 the Ricci components $\text{R}_{ab}(g)$, $a,b=0,1,2$. For this purpose, we make use of the general geometric formula:
\begin{align}\label{Rijg}
\text{R}_{ab}(g)=\text{R}_{ab}(h)-\nabla_{ab}\gamma-\nabla_a\gamma\nabla_b\gamma,\qquad a,b=0,1,2.
\end{align}
By \eqref{finredEVERic11pre}-\eqref{finredEVERic12pre} and \eqref{R0i0j2}, we obtain the identities:
\begin{align}\label{R0i0j3}
\text{R}_{0i0j}=-\nabla_{ij}\gamma-
e_i\gamma e_j\gamma
+\delta_{ij}\big(\nabla_{00}\gamma+(\nabla_0\gamma)^2\big),\qquad i,j=1,2.
\end{align} 
Hence, $\text{R}_{12}(h)=-\text{R}_{0102}(h)=\nabla_{12}\gamma+e_1\gamma e_2\gamma$, giving $\text{R}_{12}(g)=0$. Contracting indices in (\ref{R0i0j3}), we obtain:
\begin{align}\label{R00}
\text{R}_{00}(h)=\text{R}_{0101}(h)+\text{R}_{0202}(h)=-\square_h\gamma-|\nabla\gamma|^2_h+\nabla_{00}\gamma+(e_0\gamma)^2\overset{(\ref{redEVEwav2})}{=}\nabla_{00}\gamma+(e_0\gamma)^2,
\end{align}
verifying that $\text{R}_{00}(g)=0$. 
We also have from (\ref{Weyl}) the identity:
\begin{align}\label{Weyl2}
\text{R}_{0i0j}(h)=&-\text{R}_{ij}(h)+\delta_{ij}\text{R}_{00}(h)+\frac{1}{2}\delta_{ij}\text{R}(h),\qquad i,j=1,2.
\end{align}
Evaluating (\ref{Weyl2}) for $i=j=1$, and plugging in \eqref{R0i0j},\eqref{R00}, we deduce that
\begin{align}\label{restEVE1}
\text{R}_{11}(h)-\frac{1}{2}\text{R}(h)=\nabla_{11}\gamma+(e_1\gamma)^2,
\end{align}
and similarly for $i=j=2$:
\begin{align}\label{restEVE2}
\text{R}_{22}(h)-\frac{1}{2}\text{R}(h)=\nabla_{22}\gamma+(e_2\gamma)^2.
\end{align}
Hence, by (\ref{Rijg}), the vanishing $\text{R}_{11}(g),\text{R}_{22}(g)$ reduces to proving the vanishing of $\text{R}(h)$. Note that by tracing \eqref{Rijg}, we also have
\begin{align}\label{Rg=Rh}
\text{R}(g)=\text{R}(h).
\end{align}
We may thus rewrite \eqref{restEVE1}-\eqref{restEVE2} in the form
\begin{align}\label{restEVE3}
\text{R}_{11}(g)=\frac{1}{2}\text{R}(g),\qquad \text{R}_{22}(g)=\frac{1}{2}\text{R}(g).
\end{align}

Next, we derive evolution equations for the Ricci components $\text{R}_{01}(g),\text{R}_{02}(g)$, utilising the contracted second Bianchi identity: [$i,i_*=1,2$, $i\neq i_*$]
\begin{align}\label{nablaR0i}
e_0\text{R}_{0i}(g)\overset{(\ref{almpar.trans})}{=}&\,D_0\text{R}_{0i}(g)+(-1)^{i_*}K_{12}\text{R}_{0i_*}=\sum_{j=1}^3[D_j\text{R}_{ji}(g)-\frac{1}{2}D_i\text{R}(g)]+(-1)^{i_*}K_{12}\text{R}_{0i_*}\\
\tag{using \eqref{restEVE3} and $\text{R}_{ij}(g)=0$, for $i\neq j$}=&-\sum_{j=1}^3[K_{jj}\text{R}_{0i}(g)+K_{ij}\text{R}_{j0}(g)]+(-1)^{i_*}K_{12}\text{R}_{0i_*}.
\end{align}
Since by (\ref{const}), $\text{R}_{01}(g)=\text{R}_{02}(g)=0$ on $\Sigma$ and $\text{R}_{01}(g),\text{R}_{02}(g)$ satisfy the homogeneous ODE system (\ref{nablaR0i}), they vanish everywhere. 

Lastly, by virtue of the second Bianchi identity and the vanishing of the components $\text{R}_{0a}(g),\text{R}_{ab}(g)$, $a\neq b$, it follows that 
\begin{align}\label{e0Rg}
e_0\text{R}(g)=2D^a\text{R}_{0a}(g)=-2K_{11}\text{R}_{11}(g)-2K_{22}\text{R}_{22}(g)=-(K_{11}+K_{22})\text{R}(g),
\end{align}
where in the last equality we made use of \eqref{restEVE3}. Note that according to \eqref{const} (and $\text{R}_{00}(g)=0$), $\text{R}(g)$ vanishes on $\Sigma$. This implies that $\text{R}(g)$ vanishes everywhere, which in turn also yields the vanishing of $\text{R}_{11}(g),\text{R}_{22}(g)$.

\subsection{Proof of Corollary \ref{thm:stabPen}}\label{app:stabPen}

We sketch the proof of this Corollary. 

The conditions imposed on the initial data in \cite{KS}, along the hypersurface $\Sigma$ 
(see Figure \ref{stabPen}), induce initial data along the event horizons $\mathcal{H}^+$ 
that are compatible with the ones in \cite{DL2}. Next, the stability of the inner red-shift
 regions \cite{DL2} induces initial data on space-like pieces in the interior of 
the black hole, emanating from the timelike infinities. In particular, given the (small)  
$\e>0$ and $\eta>0$  
that are needed for our theorem,  if the  initial perturbation is chosen small enough, 
the data induced on $\Sigma_\e$ will be $\frac{\eta}{6}$-close to 
two (different, in principle)  Schwarzschild 
initial data near each (asymptotically cylindrical) end. 

In particular, we obtain a hypersurface $\Sigma_\e$ covered by coordinates 
$t,\theta,\phi$ where the data on  $\{ t\le -H\}$ is a $\eta/6$-perturbation (in the
 norms of our main theorem) of the Schwarzschild 
data with mass $M_2$ on the corresponding portion of 
 $\{r=\e\}$, and on $\{ t\ge H\}$ is an $\eta/6$-perturbation
of the Schwarzschild 
data with mass $M_1$ on the corresponding portion of $\{r=\e\}$. (The $\eta/3$-closeness of our data to the Schwarzschild metric of mass $M$ 
on the compact region $\{-H\le T\le H\}$ follows by the argument in \cite{DL}). 
 Moreover $M_1, M_2$ 
can be taken to be $\eta/10$-close to the background $M$, by taking the initial 
perturbations small enough. 

We now wish to apply  Theorem \ref{thm_rough} to this initial data set to proceed
 further
 towards the singularity. 

Although Theorem \ref{thm_rough} is stated in the case of a single Schwarzschild metric, 
i.e. $M_1=M_2$, it can be easily adapted to the more general  two-mass-limits scenario 
$M_1\neq M_2, M_1,M_2\sim M_0,$ by arguing as follows: \\
Let $({\bf g,K})$ denote the perturbed initial data for the EVE on 
$\Sigma_\e\cong (-\infty,+\infty)_t\times\mathbb{S}^2$. Also, let  
$({\bf g}_{{\rm S},i},{\bf K}_{{\rm S},i})$ be the Schwarzschild initial data on 
$\Sigma_\e$ of mass $M_i$, $i=1,2$. Consider the subset of $\Sigma_\e$, 
$[-C,C]_t\times\mathbb{S}^2$, $C>0$, whose domain of dependence in Schwarzschild of
 mass $M_0$ intersects the singularity at  $[-B,B]_t\times\mathbb{S}^2$, $B>0$, see
  Figure \ref{twoDom}. Then, we define 
\begin{align}\label{giKi}
\begin{split}
{\bf g}_1=h_{[-C-1,+\infty)}{\bf g}+h_{(-\infty,-C]}{\bf g}_{{\rm S},1},\qquad {\bf K}_1=h_{[-C-1,+\infty)}{\bf K}+h_{(-\infty,-C]}{\bf K}_{{\rm S},1},\\
{\bf g}_2=h_{(-\infty,C+1)}{\bf g}+h_{[C,+\infty)}{\bf g}_{{\rm S},2},\qquad {\bf K}_2=h_{(-\infty,C+1]}{\bf K}+h_{(C,+\infty)}{\bf K}_{{\rm S},2},
\end{split}
\end{align}
where the pairs of functions $h_{[-C-1,+\infty)},h_{(-\infty,-C]}$ and $h_{(-\infty,C+1]},h_{[C,+\infty)}$ are both partitions of unity, satisfying
\begin{align}\label{partuni}
h_{[-C-1,+\infty)}=\left\{
\begin{array}{ll}
1,&[-C,+\infty)\\
0,&(-\infty,-C-1]
\end{array}\right.
,\qquad h_{(-\infty,C+1]}=\left\{
\begin{array}{ll}
1,&(-\infty,C]\\
0,&[C+1,+\infty)
\end{array}\right..
\end{align}
In particular, the two pairs of initial data agree with $({\bf g,K})$ on 
$[-C,C]\times\mathbb{S}^2$. 

The pairs $({\bf g}_i,{\bf K}_i)$ 
are not initial data for the EVE, since they obviously 
do not satisfy the constraint equations in the regions $[-C-1, -C]$ and $[C, C+1]$.
However, we note that the satisfy the $\eta$-closeness to the background Scwarzschild 
data, given that  $M_1, M_2$ are both $\eta/10$-close to $M$, the initial data for these ``cut-and-paste'' initial data sets. 
 Thus, we can run our iteration algorithm in Section \ref{suc.iter}, since each pair 
 of initial data converges to Schwarzschild of the same mass $M_i$ at both ends 
 $|t|=+\infty$. Hence, we may pass to the limit, producing space-time ($1+3$)-metrics
  $g_i$, having induced data $({\bf g}_i,{\bf K}_i)$ on $\Sigma'$,  but do not exactly 
  solve the EVE. On the other hand, by \eqref{giKi}-\eqref{partuni}, the initial data 
  for the two metrics $g_1,g_2$ agree on $\Sigma_{\e,C}:=\Sigma_\e\cap\{-C\leq t\leq C\}$: 
  $({\bf g}_1,{\bf K}_1)=({\bf g}_2,{\bf K}_2)=({\bf g},{\bf K})$, verifying as well the 
  constraint equations on this portion of the initial hypersurface $\Sigma_\e$. Hence, 
  they are both the same solution to the EVE in the domain of dependence 
  region\footnote{The domain of dependence considered with respect to $g_1=g_2$.} 
  $\mathcal{D}_{com}:=\mathcal{D}(\Sigma_\e\cap\{-C\leq t\leq C\})$, arising from 
  $({\bf g}\big|_{\Sigma_{\e,C}},{\bf K}\big|_{\Sigma_{\e,C}})$.

Hence, the metric 
\begin{align*}
g:=\left\{\begin{array}{lll}
g_1,&\mathcal{D}(\Sigma'\cap \{t\ge -C\})\\
g_2,&\mathcal{D}(\Sigma'\cap \{ t\leq C\})
\end{array}\right.
\end{align*}
is well-defined and its induced data on $\Sigma_\e$ are ${\bf g,K}$. Moreover, it satisfies the constraint equations on $\Sigma_\e$ and the reduced equations \eqref{redEVEwav},\eqref{finredEVERic11pre}-\eqref{finredEVERic12pre}. Thus, by the derivations in \S\ref{retrEVE}, we conclude that $g$ is in fact a solution to the EVE, consistent with Theorem \ref{thm_rough} in the future of $\Sigma_\e$, which we desired to prove.
\begin{remark}
In the previous proof, we conveniently exploited the fact that our treatment of the reduced equations \eqref{redEVEwav},\eqref{finredEVERic11pre}-\eqref{finredEVERic12pre}, via the iteration scheme outlined in Section \ref{suc.iter}, does not make any further use of the constraint equations \eqref{const} for the EVE. Otherwise, one would have to adapt our derivations to explicitly deal with the different two-mass-limits at $|t|=+\infty$, see Figure \ref{stabPen}.
\end{remark}

\section{The optimal coordinates at the singularity: Derivation of Theorem 
\ref{thm_strict}, and its consequence.}
\label{subsec:REVESNGGimpl}
In order to derive Theorem \ref{thm_strict} from Theorem \ref{thm_REVESNGG} we need to 
establish the following Lemma, which asserts a \emph{stronger} bound on $K_{12}$ than 
what 
we claimed as part of our Theorem \ref{thm_REVESNGG}.

\begin{lemma}
\label{K12.better}
The component $K_{12}(r,t,\theta)$ satisfies the \emph{stronger} 
 (compared to the lower-order bounds claimed in Lemma \ref{a.oa.bds}), 
 estimate:

\[
\|K_{12}\|_{H^{{\rm low}}[\Sigma_r[{\rm sin}\theta d\theta dt]]}\le DC\eta r^{-\frac{1}{2}-2DC\eta}.
\]
\end{lemma}
Once this has been proven, we derive the improved estimates for $a_{T2}(r,t,\theta)$
\[
\|a_{T2}\|_{H^{{\rm low}}[\Sigma_r[{\rm sin}\theta d\theta dt]]}\le 
DC\eta r^{\frac{1}{2}-\frac{1}{8}}.
\]
Using the bounds on the other coordinate-to-frame components (which are as in Lemma \ref{a.oa.bds}), together with the formulas \eqref{g.from.a.pre}, \eqref{g.cross}, we 
derive the bounds on the metric components claimed  in Theorem \ref{thm_strict}so 
matters are reduced to proving this Lemma.

\begin{proof} Our bound follows by controlling the RHS in \eqref{finredEVERic12pre} 
in $H^l$ by 
$(DC\eta)r^{-2-2DC\eta }$. This in turn follows by 
re-expressing the RHS of that 
equation in terms of the frame-to-coordinates coefficients $a^{iA}(r,t,\theta)$, and 
using the expansion \eqref{gammam-1exp} of $\gamma$ as well as the fact that $e_2$
 annihilates $r$: 
 
 The term that yield the more singular term (in terms of powers of $r$) arises from the 
 transition from $\onabla_{12}\gamma$ to $\nabla_{\overline{e}_1\overline{e}_2}\gamma$:
 
 \[
 \onabla_{12}\gamma=O(\sqrt{r})e_1(r)\nabla_{e_2e_0}\gamma+{\rm l.o.t.'s}
 =O(\sqrt{r})e_1(r)
a^{2\theta } \partial_\theta(e_0\gamma)+{\rm l.o.t.'s}= O(DC\eta r^{-2-2DC\eta}).
 \]
 Our claim on $K_{12}$ then follows directly in view of \eqref{finredEVERic12pre}. 
\end{proof}

We remark that ultimately the more singular term in the RHS of \eqref{finredEVERic12pre}
is multiplied by $e_1(r)$; in particular it ``sees'' the  non-tangency (in an
 asymptotic sense) of $e_1$ to the singularity. 
 \medskip
 
 This leads us to another consequence of our theorem \ref{thm_strict}, expressing 
 the space-time metric in a different (geodesic, still) gauge. In this different gauge
 we obtain the optimal form of the metric, in the sense that the coordinate vector
  fields capture the principal directions of contraction and expansion, and with the off-diagonal terms in these coordinates vanishing at the optimal (fastest) rates. 
  We sketch the construction of this new coordinate system, where our metric applies its ``optimal'' form
 \medskip

\newcommand{\tilr}{{\tilde{r}}} 
 
 We construct a \emph{new} family of affine geodesics with vector field ${\bf e}_0$
 (with associated parameter $\tilde{r}$ so that 
 \[{\bf e}_0(\tilr)=(\frac{2M}{\tilr}-1)^{1/2},
 \]
 for which the associated  normal frame ${\bf e}_1, {\bf e}_2$ are \emph{both} 
 normal to the singularity, captured by the two conditions 
 ${\bf e}_1(\tilr)={\bf e}_2(\tilr)=0$. The evolution equations for $K_{ij}$ and
  $e_i(\tilr)$ are
  the same as for the REVESNGG system. Also the function $\gamma$ satisfies the
   expansion \eqref{gammam-1exp} still, with $r$ replaced by $\tilr$. 
  It is the initialization of $e_1(\tilr)$ that
   changes.

Our claim is then the following:

\begin{corollary}\label{lem:tthetatilde}
There exist \emph{new} coordinates 
$\tilr,\tilde{t},\tilde{\theta}$, with values 
$\ttheta\in (0,\pi), \tttt\in\mathbb{R}$ 
whose coordinate fields 
$\partial_{\tilde{t}},\partial_{\tilde{\theta}}\in T\Sigma_{\rho}$
 define frame-to-coordinates and coordinate-to-frame coefficients 
 $(a)^{i\tilde{t}}$, 
 $(a)^{i\tilde{\theta}}$,
 $a_{\tilde{t}i},a_{\tilde{\theta}i}$, $i=1,2$, as in 
 \eqref{tthetatransebar}, which  satisfy the following \emph{improved} 
 estimates: Let $H^l$  be the Sobolev space of order $l$ with respect to the volume form 
 ${\rm sin}\theta d\theta dt $ on $\Sigma_{\tilr}$. Then for all orders $l\le {\rm low}-2$
 we have: 

\begin{align}\label{a.oa.tilde.est}
\|{\rm log}{a}_{\tilde{t}1}(\tilr,t,\theta)-{\rm log}a^S_{t1}(\tilr,t,\theta)
\|_{H^l[\Sigma_\tilr]}\le  
DC\eta ,\quad 
\|{\rm log}a_{\tilde{\theta}2}(\tilr,t,\theta)-{\rm log}a^S_{\theta 2}(\tilr,t,\theta)
\|_{H^{l} }\leq 
DC\eta,\quad 
{\rm log}a_{\tilde{\theta}1} \equiv0\\
\notag \| {\rm log}a_{\tttt 2}\|_{H^{l} }\leq (2-DC\eta)\cdot 
|({\rm log}\tilr)|,
\end{align}
for all $r\in(0,\epsilon/2]$. Moreover the coordinate vector fields $\partial_{\tttt}, \partial_{\ttheta}$ are normal to 
$\partial_\tilr$ in the same region. 

\end{corollary}

In particular, in the new coordinates $\tilr, \ttheta,\tttt,\phi $ our space-time metric acquires the following expansion near the singularity: 

\beq
\bs
\label{best.expn}
g= -(\frac{2M}{\tilr}-1)d\tilr^2+ A(t,\theta) \tilr^{\alpha(t,\theta)} ({\rm sin}\theta)^2d\phi^2+ B(\tttt,\theta) \tilr^{\beta(t,\theta)}d\ttheta^2+C(\tttt,\ttheta) \tilr^{\delta(t,\theta)}d\tttt^2+O(\tilr^{3-\frac{1}{4}})d\tttt d\ttheta. 
\end{split}
\eeq
The functions $A,B,C, \alpha,\beta,\delta$ satisfy the same bounds as those claimed in Theorem \ref{thm_strict}. The strengthening here comes in the much-improved 
behaviour of the $d\ttheta d\ttt$ component, and the \emph{absence} of $d\ttheta d\tilr$ and $d\tttt d\tilr$ components. 
\medskip

\begin{figure}[h]
	\centering
	\def\svgwidth{9cm}
	\includegraphics[scale=1.2]{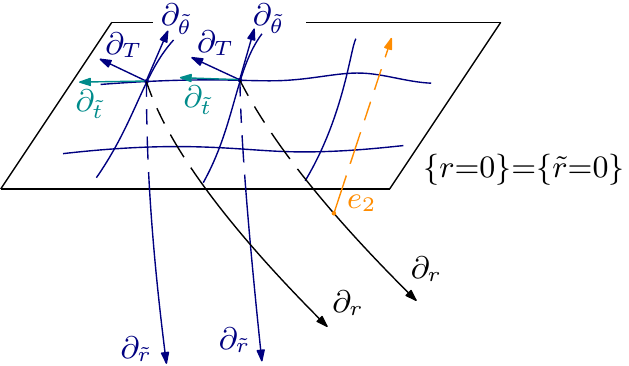}
	\caption{We display the coordinate vector fields of the old coordinates 
 $\{r,T,\Theta \}$ and $\{\tilde{r},\tilde{t},\tilde{\theta}\}$ on the final singularity hypersurface $\{r=0\}=\{\tilde{r}=0\}$ and slightly nearby. The geodesics corresponding to $\partial_{\tilde{r}}, \partial_r$ are \emph{not} the same. This new coordinate systems captures all the principal directions of contraction/expansion, and as a consequence the off-diagonal terms enjoy much stronger rates of decay.}
 \label{two-coords-sing}
\end{figure}

\begin{proof}
The key parameters of the REVESNGG can again be solved for as before. The only term that 
now satisfies a \emph{stronger} bound than in the gauge of Theorem 
\ref{thm_REVESNGG} is ${\bf K}_{12}=\langle \nabla_{{\bf e}_1}{\bf e}_0,{\bf e}_2 \rangle$; in this setting, 
we have that ${\bf e}_1(r)=0$; as a consequence the vector fields $\bf{e}_1,{\bf e}_2$ 
are \emph{both} tangent to the level sets of $\tilr$. 

This yields improved bounds on the 
the RHS in equation \eqref{finredEVERic12pre}: Since \eqref{gammam-1exp} still holds, we 
find that up to less singular terms the RHS has the following expansion:

\[
{\bf e}_1[{\bf e}_2](\gamma)={\bf e}_1[{\bf e}_2](\alpha(t,\theta))\cdot {\rm log}r=
a^{1\tilde{t}}a^{2\tilde{\theta}}\partial_{t\theta}\alpha(t,\theta)\cdot {\rm log}(r)= 
O(\tilr^{-\frac{1}{2}-DC\eta});
\]
the same bound holds for $[{\bf e}_1(\gamma)]\cdot [{\bf e}_2\gamma]$. In particular the term 
${\bf K}_{12}(\tilr,\tttt,\ttheta)$ is bounded in $H^{\rm low}$ by 
$DC\eta \tilr^{1-DC\eta}$.

This allows for improved estimates for one of the coordinate-to-frame coefficient 
$a_{\tilde{t} 2}$, which will be used to construct our new coordinates $\tttt,\ttheta$.:
\medskip

The frame coefficients $a_{\tilde{t}i},a_{\tilde{\theta}i}$, $i=1,2$, satisfy the ODEs \eqref{e0.a}, with $\tilde{t}, \tilde{\theta}$ in place of $t,\theta$. In this case, however, we will solve for \emph{all}  these parameters \emph{backwards} from the singularity. 

Notice that $a_{\tilde{\theta}1},{a}_{\tilde{t}2}$ satisfy separate  ODEs (the one for $a_{\tilde{\theta}1}$ 
is homogenous). We choose the free branches of both these ODEs to zero. 
Hence, this eliminates the variable
 $a_{\tilde{\theta}1}\equiv0$. This
choice implicitly imposes that  that $\partial_{\tilde{\theta}}$ is parallel to unique collapsing direction, \emph{and} that $\partial_\tttt$ 
should be the corresponding 
\emph{principal} dual direction.

Then, the equations for $a_{\tilde{t}1},a_{\tilde{\theta}2}$ decouple as 
well, having general solutions of the form:
\beq
\begin{split}
\label{a.tildes.formula}
& a_{\tilde{t}1}(\tilr, t,\theta)= c_{\tilde{t}1}(t,\theta) 
e^{\int_\e^{\tilr}(1-\frac{2M}{r})^{1/2} {\bf K}_{11}(s) ds }, 
a_{\tilde{\theta}2}(\tilr, t,\theta)= c_{\tilde{\theta}2}(t,\theta) 
e^{\int_\e^{\tilr}(1-\frac{2M}{r})^{1/2} {\bf K}_{22}(s) ds },
\\&a_{\tilde{\theta}1} =0, a_{\tilde{t}2}(\tilr,t,\theta)= 
e^{-\int_\e^{\tilr}(1-\frac{2M}{s})^{1/2} {\bf K}_{22}(s) ds } 
\int_0^{\rho} e^{-\int_\e^{\tau}(1-\frac{2M}{\tilr})^{1/2} {\bf K}_{22}(s,t,\theta) ds }  2{\bf K}_{12}
(\tau)a_{\tttt 1}(\tau,t,\theta)
\cdot (1-\frac{2M}{\tilr})^{1/2} d\tau  
\end{split}
\eeq
(In the last term we have used the function $a_{\tttt 1}(\rho,t,\theta)$ that was solved 
for first).

As noted, in this coordinate system , the 
parameter ${\bf K}_{12}(\tilr,\tttt,\ttheta)$ satisfies the \emph{stronger} bound, for all
 $l\le {\rm low}-2$: 

\[
\| {\bf K}_{12}\|_{{\cal C}^l[\Sigma_r]}\le DC\eta r^{1-DC\eta}
\]

So, up to specifying the magnitude of the coefficient $c_{\tttt 1}(t,\theta)$, 
the magnitude of $a_{\tilde{t}2}(\tilr,t,\theta)$ will be of the order: 

\[
\| a_{\tilde{t}2}(\tilr,t,\theta)\|_{H^l[\Sigma_\tilr]}\lesssim O(1)DC\eta 
\tilr^{2-\frac{1}{4}}.
\]

Thus matters are reduced to bounding the coefficients $c_{\tilde{t} 1}(t,\theta), c_{\ttheta 2}(t,\theta)$. We do this next, by proving they can be specified to be 
$O(1)$
in the low norms.  
\medskip

We are free to choose the coefficients 
$c_{\tilde{t} 1}(t,\theta), c_{\ttheta 2}(t,\theta)$; the only restriction is that 
the resulting vector fields 
\[
\partial_{\tttt}=\sum_{i=1,2} a_{\tttt i}\overline{e}_i, \partial_\ttheta =\sum_{i=1,2}
 a_{\ttheta i}\overline{e}_i
\]
should commute;
it suffices to check this condition on $\Sigma_{r_*}$, where we recall that 
by construction $a_{\tttt 1}(r_*(t,\theta),t,\theta)=c_{\tttt 1}(t,\theta) $ and $a_{\ttheta 2}(r_*(t,\theta),t,\theta)=c_{\ttheta 2}(t,\theta)$;
 in particular on that hypersurface we will be requiring: 
\beq
\label{commut.reqt}
\bigg{[}c_{\ttheta 2}(t,\theta)\te_2, c_{\tttt 1}(t,\theta)\cdot 
\te_1
+a_{\tttt 2}[c_{\tttt 1}, {\bf K}_{12}]\cdot \te_2\bigg{]}=0.
\eeq
(We use the notation $a_{\tttt 2}[c_{\tttt 1}, {\bf K}_{12}]$
to highlight the dependence of the variable $a_{\tttt 2}$ only on the two 
parameters $c_{\tttt 1}, {\bf K}_{12}$).

Our freedom comes in choosing the values of the two functions $c_{\ttheta 2}(t,\theta), c_{\tttt 1}(t,\theta)$
along two curves  $\{t=0\}$ and $\{\theta=0\}$ respectively. 
In fact that requirement fixes the values of the coordinates $\tttt, \ttheta$ on the 
lines $\{\theta=0\}, \{ t=0\}$.
For definiteness, we will set $\tttt=t$ and $\ttheta=\theta$ on those lines; the 
resulting solution to 
  \eqref{commut.reqt} then specifies the values of the functions 
  $c_{\ttheta 2}(t,\theta), c_{\tttt 1}(t,\theta)$
  everywhere. 
  
Let us in fact expand \eqref{commut.reqt} into: 

\beq
\begin{split}
&\bigg{(}c_{\ttheta 2}(t,\theta)\cdot \te_2[c_{\tttt 1}(t,\theta)]\bigg{)}\cdot \te_1- 
\bigg{(}c_{\tttt 1}(t,\theta)\te_1[c_{\ttheta 2}(t,\theta)]+c_{\ttheta 2}(t,\theta)\cdot
\te_2(c_{\tttt 2}[c_{\tttt 1},{\bf K}_{12}])\bigg{)}\cdot \te_2
\\&+[c_{\ttheta 2}(t,\theta)\cdot c_{\tttt 1}(t,\theta)]\cdot [\tilde{A}_{21,2}\te_2-
\tilde{A}_{12,1}\te_1]=0.
\end{split}
\eeq
Thus we derive a  system of 1st order transport equations:
\beq
\begin{split}
&\bigg{(}c_{\ttheta 2}(t,\theta)\cdot \te_2[c_{\tttt 1}(t,\theta)]\bigg{)}=
c_{\ttheta 2}(t,\theta)\cdot 
c_{\tttt 1}(t,\theta)
 \tilde{A}_{12,1},
\\& \bigg{(}c_{\tttt 1}(t,\theta)\te_1[c_{\ttheta 2}(t,\theta)]+c_{\ttheta 2}(t,\theta)
\cdot
\te_2(a_{\ttheta 1}[c_{\tttt 1},{\bf K}_{12}])\bigg{)}=
c_{\ttheta 2}(t,\theta)\cdot c_{\tttt 2}(t,\theta)\cdot \tilde{A}_{21,2}.
\end{split}
\eeq

At this point we recall the expressions \eqref{tAm.def} for 
$\tilde{A}_{21,2},\tilde{A}_{12,1}$ in terms of (0th and 1st derivatives of) 
the variable  $\tilde{K}_{12}$;  using also the expressions \eqref{te1m.tK12} for
 $\te_1, \te_2$ in terms of
 the background coordinates $t,\theta$ we can view the above as a 2x2 1st order 
 system in the 
 two parameters $c_{1\tttt}, c_{2\ttheta}$.



Now, we will use the above system to solve for $c_{1\tttt}(t,\theta),
 c_{2\ttheta}(t,\theta)$. To do this, we need to impose conditions on these functions; 
 we choose to do so on the two curves $\{\theta=0\}$ and 
 $\{t=0\}$. As noted, we set 
$\tttt=t$ and $\ttheta=\theta$ respectively on those two curves. 
 
Implicitly using the formulas \eqref{tthetatransebar} this prescribes the values 
of $c_{\tttt 1}(t,\theta),
 c_{\ttheta 2}(t,\theta)$ on these two curves, respectively. Then using the 
 expressions second line formulas in \eqref{tthetatransebar}, we solve for 
 $c_{\tttt 1}(t,\theta)$ first 
 along the integral curves of $\te_2$, and then for $c_{\ttheta 2}$ along the 
 integral curves of $\te_1$.  This shows that in the norms $H^{{\rm low}-1}$
 we can bound $c_{\ttheta 2}$, $c_{\tttt 1}$ by $O(\e^{1-DC\eta})$, 
 $O(\e^{-\frac{1}{2}-DC\eta})$, respectively. 
 
From these formulas we can directly derive our claimed bounds on the parameters $a_{Ai}$
\emph{on} $\Sigma_{r_*}$,
 utilizing our derived bounds on 
$\tilde{K}_{12}$.  Given the bounds on 
 the
 components of $K_{ij}(r,t,\theta)$ (in particular 
 with the \emph{improved} bound on ${\bf K}_{12}(r,t,\theta)$),
  we then invoke the integral representations \eqref{a.tildes.formula} to derive the claimed 
 bounds for $a_{Ai}$, $A=\tttt,\ttheta$ and $i=1,2$  off of $\Sigma_{r_*}$. 
\medskip

\end{proof}

\bigskip
Spyros Alexakis\\
{\itshape Address:} {\scshape\footnotesize Department of Mathematics, University of Toronto, Room 6290, 40 St. George Street, Toronto, Ontario M5S 2E4, Canada}\\
{\itshape Email:} {\ttfamily alexakis@math.utoronto.ca}

\smallskip
Grigorios Fournodavlos\\
{\itshape Address:} {\scshape\footnotesize Laboratoire Jacques-Louis Lions, Sorbonne Universit\'e, 4 place Jussieu, 75005 Paris, France }\\
{\itshape Email:} {\ttfamily grigorios.fournodavlos@sorbonne-universite.fr}

\end{document}